\documentclass{article}

\usepackage{amsmath}
\usepackage{amssymb}
\usepackage{latexsym}
\usepackage{verbatim}
\usepackage[final]{graphicx}
\usepackage{psfrag}
\usepackage{mathrsfs}  
\usepackage{amsthm}
\usepackage{fancybox}
\usepackage{enumitem}
\usepackage[export]{adjustbox}
\usepackage{caption,subcaption}

\newtheorem{theorem}{Theorem}
\newtheorem{proposition}{Proposition}
\newtheorem{lemma}{Lemma}
\newtheorem{corollary}{Corollary}

\theoremstyle{remark}
\newtheorem{remark}{Remark}
\newtheorem{assumption}{Assumption}
\newtheorem{definition}{Definition}
\newtheorem{fact}{Fact}

{\begin{list}               
    {$\bullet$ \hfill}{
        \setlength{\leftmargin}{\parindent}
        \setlength{\parsep}{0.04\baselineskip}
        \setlength{\itemsep}{0.5\parsep}
        \setlength{\labelwidth}{\leftmargin}
        \setlength{\labelsep}{0em}}
    }
{\end{list}}

\providecommand{\eref}[1]{\eqref{eq:#1}}  
\providecommand{\cref}[1]{Chapter~\ref{chap:#1}}
\providecommand{\sref}[1]{Section~\ref{sec:#1}}
\providecommand{\appref}[1]{Appendix~\ref{appendix:#1}}
\providecommand{\fref}[1]{Figure~\ref{fig:#1}}

\providecommand{\thref}[1]{Theorem~\ref{thm:#1}}
\providecommand{\defref}[1]{Definition~\ref{def:#1}}
\providecommand{\lemref}[1]{Lemma~\ref{lem:#1}}

\providecommand{\assumpref}[1]{Assumption~\ref{assump:#1}}
\providecommand{\factref}[1]{Fact~\ref{fact:#1}}
\providecommand{\propref}[1]{Proposition~\ref{prop:#1}}
\providecommand{\corref}[1]{Corollary~\ref{cor:#1}}

\providecommand{\R}{\ensuremath{\mathbb{R}}}
\providecommand{\C}{\ensuremath{\mathbb{C}}}
\providecommand{\N}{\ensuremath{\mathbb{N}}}

\providecommand{\W}{\ensuremath{\mathbb{N}_0}}

\providecommand{\bydef}{\overset{\text{def}}{=}}
\renewcommand{\dim}{N}
\providecommand{\sigpsi}{\sigma_{\psi}^2}

\renewcommand{\vec}[1]{\ensuremath{\boldsymbol{#1}}}
\providecommand{\mat}[1]{\ensuremath{\boldsymbol{#1}}}

\providecommand{\calS}{\mathcal{S}}

\providecommand{\calV}{\mathcal{V}}

\providecommand{\mA}{\mat{A}} \providecommand{\mB}{\mat{B}}
\providecommand{\mC}{\mat{C}} 
\providecommand{\mD}{\mat{D}}
\providecommand{\mF}{\mat{F}}
\providecommand{\mH}{\mat{H}}
\providecommand{\mI}{\mat{I}} \providecommand{\mJ}{\mat{J}} 
  
\providecommand{\mM}{\mat{M}} \providecommand{\mP}{\mat{P}} 
\providecommand{\mQ}{\mat{Q}} \providecommand{\mR}{\mat{R}}
\providecommand{\mS}{\mat{S}} \providecommand{\mU}{\mat{U}} 

\providecommand{\mW}{\mat{W}}

\providecommand{\mZ}{\mat{Z}}
\providecommand{\mSigma}{\mat{\Sigma}}
 
\providecommand{\mPsi}{\mat{\Psi}}
\providecommand{\mPsi}{\mat{\Psi}}
\providecommand{\mX}{\mat{X}}

\providecommand{\va}{\vec{a}} \providecommand{\vb}{\vec{b}}
 \providecommand{\ve}{\vec{e}}
\providecommand{\vg}{\vec{g}}
\providecommand{\vh}{\vec{h}} 
\providecommand{\vm}{\vec{m}}

\providecommand{\vs}{\vec{s}}

\providecommand{\vu}{\vec{u}} \providecommand{\vw}{\vec{w}}
\providecommand{\vx}{\vec{x}} \providecommand{\vy}{\vec{y}}
\providecommand{\vz}{\vec{z}} 
 \providecommand{\vzero}{\vec{0}}

\providecommand{\vv}{\vec{v}}

\providecommand{\vDelta}{\vec{\Delta}}
\providecommand{\valpha}{\vec{\alpha}}
\providecommand{\vbeta}{\vec{\beta}}
\providecommand{\vepsilon}{\vec{\epsilon}}
\providecommand{\vsigma}{\vec{\sigma}}
\providecommand{\fieldvec}{\vec{1}}

\providecommand{\vone}{\vec{1}}

\newcommand{\unif}[1]{\mathsf{Unif}(#1)} 
\newcommand{\ortho}{\mathbb{O}} 
\newcommand{\explain}[2]{\overset{\text{\tiny{#1}}}{#2}}
\newcommand{\iter}[2]{{#1}^{(#2)}} 
\newcommand{\E}{\mathbb{E}} 
\renewcommand{\P}{\mathbb{P}} 
\newcommand{\Var}{\mathrm{Var}}
\usepackage[numbers]{natbib} 
\newcommand{\ip}[2]{\left\langle {#1}, {#2} \right\rangle} 
\newcommand{\gauss}[2]{\mathcal{N}\left( #1,#2 \right)} 
\newcommand{\nonlin}{f}

\newcommand{\degree}{D}
\newcommand{\fourier}[3]{\hat{\nonlin}_{#1}(#2,#3)}
\newcommand{\pweight}[1]{p_{#1}}
\newcommand{\qweight}[1]{q_{#1}}
\newcommand{\effect}[1]{\omega_{#1}}
\newcommand{\height}[1]{h_{#1}}
\newcommand{\valid}{\mathtt{VALID}}
\newcommand{\hermite}[1]{H_{#1}}
\newcommand{\trees}[2]{\mathscr{T}_{#1}(#2)}
\newcommand{\forests}[2]{\mathscr{F}_{#1}(#2)}
\newcommand{\colorings}[2]{\mathscr{C}_{#1}(#2)}

\newcommand{\leaves}[1]{\mathscr{L}(#1)}
\newcommand{\nullleaves}[1]{\mathscr{L}_0(#1)}
\newcommand{\enodes}{\mathscr{E}}
\newcommand{\gnodes}{\mathscr{G}}
\newcommand{\onodes}{\mathscr{O}}
\newcommand{\bnodes}{\mathscr{B}}
\newcommand{\eleaves}{\mathscr{L}_e}
\newcommand{\gleaves}{\mathscr{L}_g}
\newcommand{\tileaves}{\mathscr{L}_1}
\newcommand{\tiileaves}{\mathscr{L}_2}
\newcommand{\tiiileaves}{\mathscr{L}_3}
\newcommand{\nleaves}{\mathscr{L}_0}
\newcommand{\map}{\mathscr{M}}
\renewcommand{\part}[1]{\mathscr{P}(#1)}
\newcommand{\chrn}[1]{c_{#1}}
\newcommand{\wcolorings}[2]{\mathscr{WC}_{#1}(#2)}

\newcommand{\rootset}{\mathscr{R}}
\newcommand{\gibbs}{\mu_{\dim}}
\newcommand{\temp}{\beta}
\newcommand{\field}{\theta}
\newcommand{\invpoly}{\Gamma}
\newcommand{\universal}{simple }
\newcommand{\exponent}{\eta}

\newcommand{\mgnt}{\chi}
\newcommand{\vmgnt}{\vec{\mgnt}}
\newcommand{\pw}{$\mathrm{PW}_2$}

\makeatletter
\newcommand{\myitem}[1]{%
\item[#1]\protected@edef\@currentlabel{#1}%
}
\makeatother



\newenvironment{fminipage}%
  {\begin{Sbox}\begin{minipage}}%
  {\end{minipage}\end{Sbox}\fbox{\TheSbox}}

\newenvironment{algbox}[0]{\vskip 0.2in
\noindent 
\begin{fminipage}{6.3in}
}{
\end{fminipage}
\vskip 0.2in
}

\DeclareFontFamily{U}{mathx}{\hyphenchar\font45}
\DeclareFontShape{U}{mathx}{m}{n}{
      <5> <6> <7> <8> <9> <10>
      <10.95> <12> <14.4> <17.28> <20.74> <24.88>
      mathx10
      }{}
\DeclareSymbolFont{mathx}{U}{mathx}{m}{n}
\DeclareFontSubstitution{U}{mathx}{m}{n}
\DeclareMathAccent{\widecheck}{0}{mathx}{"71}
\DeclareMathAccent{\wideparen}{0}{mathx}{"75}

\providecommand{\ctr}[1]{\widecheck{#1}}

\usepackage[margin=1in]{geometry}

\usepackage{cite}

\usepackage{color}

\usepackage{booktabs,multirow}

\usepackage{enumitem}
\usepackage{pmat}
\usepackage{commath}
\usepackage{mathtools,xcolor,authblk}

\usepackage{algorithm,algpseudocode}

\usepackage{hyperref}
 \hypersetup{
     colorlinks=true,
     linkcolor=blue,
     filecolor=blue,
     citecolor = blue,      
     urlcolor=cyan,
     }

\usepackage{esdiff}

\usepackage{pgf,tikz,psfrag}

\usepackage{dsfont}

\newcommand{\Ortho}{\mathbb{O}}

\newcommand{\mLambda}{\mat{\Lambda}}

\DeclareMathOperator{\diag}{diag}

\DeclareMathOperator{\Tr}{Tr}

\DeclareMathOperator{\op}{op}
\DeclareMathOperator{\fr}{Fr}

\newcommand*{\tran}{^{\mkern-1.5mu\mathsf{T}}}

\providecommand{\vbeta}{\vec{\beta}}

\newcommand{\new}[1]{#1}

\renewcommand*\diff{\mathop{}\!\mathrm{d}}

\title{Universality of Approximate Message Passing with \\Semi-Random Matrices}
\author[]{Rishabh Dudeja\thanks{rd2714@columbia.edu}}
\author[]{Yue M. Lu\thanks{yuelu@seas.harvard.edu}}
\author[]{Subhabrata Sen\thanks{subhabratasen@fas.harvard.edu}}
\affil[]{Harvard University}

\begin{document}

\maketitle

\begin{abstract}
Approximate Message Passing (AMP) is a class of iterative algorithms that have found applications in many problems in high-dimensional statistics and machine learning. In its general form, AMP can be formulated as an iterative procedure driven by a matrix $\mM$. Theoretical analyses of AMP typically assume strong distributional properties on $\mM$---for example, $\mM$ has i.i.d. sub-Gaussian entries or is drawn from a rotational invariant ensemble. However, numerical experiments suggest that the behavior of AMP is \emph{universal}, as long as the eigenvectors of $\mM$ are generic. In this paper, we take the first step in rigorously understanding this universality phenomenon. In particular, we investigate a class of ``memory-free'' AMP algorithms (proposed in \citet{ccakmak2019memory} for mean-field Ising spin glasses), and show that their asymptotic dynamics is universal on a broad class of  ``semi-random matrices''. In addition to having the standard rotational invariant ensemble as a special case, the class of semi-random matrices that we define in this work also includes matrices constructed with very limited randomness. One such example is a randomly signed version of the Sine model, introduced in \citet{marinari1994replica} and \citet{parisi1995mean} for spin glasses with fully deterministic couplings.
\end{abstract}

\tableofcontents

\section{Introduction}

Approximate Message Passing (AMP) algorithms are low-complexity iterative algorithms that have attracted considerable attention recently in statistics and machine learning. These algorithms were originally introduced for solving the TAP equations for mean-field spin glasses \citep{bolthausen2014iterative} and in the context of compressed sensing \citep{donoho2009message}. They are also intricately connected to classical iterative inference algorithms such as belief propagation \citep{mezard2009information,kabashima2003cdma}, and expectation propagation \citep{minka2013expectation,opper2005expectation}. Since their inception, AMP algorithms have found applications in diverse situations---on the one hand, they are directly used as computationally efficient inference algorithms in compressed sensing \citep{donoho2009message} and coding theory \citep{rush2017capacity}; on the other hand, these algorithms have been used as constructive proof devices to characterize the asymptotic performance of statistical procedures such as the LASSO \citep{bayati2011dynamics}, M-estimators \citep{bean2013optimal,donoho2016high,pmlr-v125-gerbelot20a,gerbelot2020asymptotic}, maximum likelihood \citep{sur2019likelihood,sur2019modern}, and spectral methods \citep{montanari2021estimation,mondelli2021pca} in high-dimensions. 

Given a data matrix $\mM \in \R^{N \times N}$, an AMP algorithm  in its general form  consists of the following iterative updates:
\begin{equation}\label{eq:general-amp}
\begin{aligned}
\iter{\vz}{t} &= \mM F_{t}(\iter{\vz}{0}, \iter{\vz}{1}, \dotsc, \iter{\vz}{t-1})  + G_t(\iter{\vz}{0}, \iter{\vz}{1}, \dotsc, \iter{\vz}{t-1}).
\end{aligned}
\end{equation}
where $F_{t}: \R^{\dim \times t} \mapsto \R^\dim$ and $G_{t}: \R^{\dim \times t} \mapsto \R^\dim$ are well-chosen vector-valued functions. AMP algorithms are particularly attractive due to their theoretical tractability. Specifically, when the data matrix $\mM$ is drawn from a rotationally-invariant ensemble (such as the Gaussian orthogonal ensemble), and if the function $G_t$ (called the Onsager correction) is suitably chosen based on $F_t$, the joint empirical distributions of the iterates $\iter{\vz}{1}, \ldots, \iter{\vz}{t}$ can be shown to converge to a mean-zero Gaussian process as $N \to \infty$. Moreover, the covariance of this limiting Gaussian process can be explicitly computed via a deterministic recursion known as \emph{state evolution} \citep{bolthausen2014iterative,bayati2011dynamics,javanmard2013state,rangan2019vector,takeuchi2017rigorous,berthier2020state,fan2020approximate}.

Theoretical analyses of AMP algorithms typically make strong assumptions on the distribution of the matrix $\mM$---for example, one might assume that the entries of $\mM$ are i.i.d. Gaussian \citep{bolthausen2014iterative,bayati2011dynamics,javanmard2013state,berthier2020state}. Another widely used model assumes that $\mM$ is \emph{rotationally invariant}, i.e., its distribution is invariant under conjugation by any deterministic orthogonal matrix \citep{rangan2019vector,takeuchi2017rigorous,fan2020approximate}. 

While the idealistic statistical models mentioned above are convenient for mathematical analysis, they do not resemble the data matrices encountered in practice, which are often structured or exhibit strong correlations among the matrix entries. Interestingly, numerical experiments suggest that the behavior of AMP algorithms does not depend too strongly on the precise distribution of the matrix $\mM$. In fact, it has been observed that \citep{ccakmak2019memory,abbara2020universality,ma2021spectral} the theoretical characterizations obtained under idealistic statistical models remain true for many semi-random (or even deterministic) matrix ensembles. Establishing this universality phenomenon is thus of intrinsic importance, as it allows practitioners to use these theoretical characterizations of AMP with greater confidence in real-life statistical and machine learning applications.

There has been some important recent progress in understanding the universality of AMP algorithms for i.i.d matrices $\mM$. Specifically, it is now well-understood (see \citep{bayati2015universality,chen2021universality}) that the Gaussianity of the entries of $\mM$ is unnecessary---the distribution of the AMP iterates can be tracked using the same state evolution recursion as long as the entries $\mM$ are i.i.d. mean-zero, unit-variance \emph{sub-Gaussian} random variables.

Unfortunately, the existing  guarantees do not capture the full-scope of the universality phenomenon observed in practice. A striking example in this regard is the Sine model of \citet{marinari1994replica} and \citet{parisi1995mean}, an Ising Model where the coupling matrix is the Discrete Sine Transform Matrix. Using non-rigorous techniques, physicists conjecture that the behavior of this completely deterministic model should be the same as the \emph{Random Orthogonal Model} (ROM), a fully disordered Ising model whose coupling matrix is given by the rotationally invariant matrix $\mU \diag(b_{1:\dim}) \mU \tran$, where $\mU \sim \unif{\ortho(\dim)}$ is a Haar matrix and $b_{1:\dim} \explain{i.i.d.}{\sim} \unif{\{\pm1\}}$.  In the context of AMP algorithms, numerical simulations also suggest the equivalence between the Sine model and the ROM, in that they can be characterized by the same state evolution recursion (see Section \ref{sec:universality_figs} for supporting numerical evidence). More generally, numerical studies reported in the literature \citep{abbara2020universality,ccakmak2019memory} suggest that AMP algorithms exhibit universality properties as long as the eigenvectors of $\mM$ are generic.  Formalizing this conjecture remains squarely beyond existing techniques, and presents a fascinating challenge. 

In this paper, we take the first step in understanding this universality phenomenon. In particular, we investigate a sub-class of AMP algorithms that take the form 
\begin{align} \label{eq:memory-free-amp} 
    \iter{\vz}{t+1} & = \mM  \cdot \left( \nonlin_{t+1}(\iter{\vz}{t}) - \frac{\langle{\nonlin_{t+1}(\iter{\vz}{t})}\; , \; {\iter{\vz}{t}} \rangle}{\|\iter{\vz}{t}\|^2} \cdot \iter{\vz}{t} \right).
\end{align}
In the above display, the functions $\nonlin_t: \R \mapsto \R$ and act entry-wise on their arguments. We wish to understand the dynamics of the above algorithm under general assumptions on the matrix $\mM$, the (coordinate-wise) non-linearities $\nonlin_t$ and the initialization $\iter{\vz}{0}$. The algorithm in \eref{memory-free-amp} is a special case of \eref{general-amp}. The specific choice of $F_t$ made in \eqref{eq:general-amp} to obtain \eqref{eq:memory-free-amp} ensures that the iterates of the resulting algorithm converge to a Gaussian process as $\dim \rightarrow \infty$ without any Onsager correction (given by the function $G_t$ in \eqref{eq:general-amp}).  We choose to focus on \eref{memory-free-amp} due to its simple ``memory-free'' structure, namely, the iterate $\iter{\vz}{t+1}$ only depends on its immediate predecessor $\iter{\vz}{t}$. In contrast, the general AMP algorithms in \eref{general-amp} might have to maintain a long memory (that grows with $t$) to ensure that the empirical distributions of $\iter{\vz}{t}$ are asymptotically Gaussian. The AMP algorithm in \eref{memory-free-amp} was introduced in \citep{ccakmak2019memory} to approximate the magnetization of Ising spin glasses with orthogonally invariant coupling matrices. Similar memory-free variants of AMP algorithms for rectangular data matrices have been proposed under the names ``orthogonal AMP'' \citep{ma2017orthogonal} and ``vector approximate message passing'' \citep{rangan2019vector,takeuchi2017rigorous}.

\subsection{Notation}
We begin by collecting some notations that will be used throughout this paper. 

\emph{Some common sets:} $\N$ and $\R$ denote the set of positive integers and the set of real numbers respectively. $\W \explain{def}{=} \N \cup \{0\}$ is the set of non-negative integers. For each $\dim \in \N$, $[\dim]$ denotes the set $\{1, 2, 3, \dotsc, \dim\}$ and $\ortho(\dim)$ denotes the set of $\dim \times \dim$ orthogonal matrices. 

\emph{Asymptotics:} Given a sequence $a_{\dim}$ and a non-negative sequence $b_{\dim}$ indexed by $\dim \in \N$ we say $a_\dim \ll b_\dim$ or $a_{\dim} = o(b_\dim)$ if $a_\dim/b_\dim \rightarrow 0$. Similarly we say $a_{\dim} \lesssim b_{\dim}$ or $a_\dim = O(b_\dim)$ if there exist fixed constants $\alpha \geq 0$ and $N_0 \in \N$, such that $|a_\dim| \leq \alpha b_\dim$ for all $\dim \geq N_0$.

\emph{Asymptotics for random variables:} We use  $\explain{P}{\rightarrow}$ to denote convergence in probability. 

\emph{Linear Algebra:} For a vector $\vv \in \R^{\dim}$, $\|\vv\|_1, \|\vv\|, \|\vv\|_\infty$ denote the $\ell_1$, $\ell_2$ and $\ell_\infty$ norms respectively and $\|\vv\|_0$ denotes the number of non-zero coordinates (or sparsity) of $\vv$. \new{For a matrix $\mA \in \R^{\dim \times \dim}$, we denote the $(i,j)$ entry of $\mA$ using the corresponding lowercase letter $a_{ij}$. To refer to the $(i,j)$ entry of the matrix product $\mA \mB$ we use the notation $(\mA \mB)_{ij}$.}  $\|\mA\|_{\op}, \|\mA\|$ denote the operator (spectral) norm and Frobenius norm of $\mA$ respectively. On the other hand $\|\mA\|_\infty \explain{def}{=} \max_{i,j \in [\dim]} |a_{ij}|$ denotes the entry-wise $\ell_\infty$ norm.  $\vone$ denotes the vector $(1, 1, \dotsc, 1)$, $\vzero$ denotes the vector $(0, 0, \dotsc, 0)$, and $\ve_1, \ve_2, \dotsc, \ve_{\dim}$ denote the standard basis vectors in $\R^\dim$. $\mI_{\dim}$ is the $\dim \times \dim$ identity matrix.  

\emph{Gaussian Distributions and Hermite Polynomials:} The univariate Gaussian distribution on $\R$ with mean $\mu$ and variance $\sigma^2$ is denoted by $\gauss{\mu}{\sigma^2}$. The multivariate Gaussian distribution on $\R^{\dim}$ with mean vector $\vec{\mu}$ and covariance matrix $\mSigma$ is denoted by $\gauss{\vec{\mu}}{\mSigma}$. For each $i \in \W$, $H_i : \R \rightarrow \R$ is denotes the Hermite polynomial of degree $i$. The Hermite polynomials are orthogonal polynomials for the standard Gaussian measure $\gauss{0}{1}$. This means that for $Z \sim \gauss{0}{1}$, $\E H_i^2(Z) = 1$ for each $i \in \W$ and $\E[H_i(Z) H_j(Z)] = 0$ for $i, j \in \W$ and $i \neq j$ (note that we assume throughout that the Hermite polynomials are normalized to have unit norm under the standard Gaussian measure). The first few Hermite polynomials are $H_0(z) = 1, \; H_1(z) = z, \; H_2(z) = (z^2-1)/\sqrt{2}$. We refer the reader to \citet[Chapter 11]{o2014analysis} for additional background on Hermite polynomials. 

\emph{Miscellaneous:} For a finite set $A$, $\unif{A}$ denotes the uniform distribution on $A$. Hence $\unif{\{\pm 1\}}$ and $\unif{\{\pm 1\}^\dim}$ denote the uniform distributions on $\{-1,1\}$ and the $\dim$-dimensional Boolean hypercube $\{-1,1\}^\dim$, respectively. We use $\unif{\ortho(\dim)}$ to denote the Haar measure on the orthogonal group $\ortho(\dim)$. For any $i,j \in \W$, $\delta_{ij}$ denotes the Kronecker delta function, with $\delta_{ij} = 1$ if $i = j$ and $\delta_{ij} = 0$ otherwise.

\subsection{Main Result}
Our main result establishes the universality of the AMP algorithm \eref{memory-free-amp} for a wide class of \emph{semi-random matrix ensembles} $\mM$, defined below. 
\begin{definition}[Semi-random Matrix Ensemble] \label{def:matrix-ensemble-relaxed}  A semi-random matrix ensemble $\iter{\mM}{\dim} \in \R^{\dim \times \dim}$ is a sequence of random matrices of the form $\iter{\mM}{\dim} = \iter{\mS}{\dim} \iter{\mPsi}{\dim} \iter{\mS}{\dim}$ where, \looseness=-1
\begin{enumerate}
    \item $\iter{\mS}{\dim} = \diag(s_1, s_2, \dotsc, s_\dim)$ with $s_i \explain{i.i.d.}{\sim} \unif{\{\pm 1\}}$.
    \item $\iter{\mPsi}{\dim}$ is a sequence of deterministic matrices that satisfy:
    \begin{enumerate}
        \item $\|\iter{\mPsi}{\dim}\|_\infty \lesssim \dim^{-\frac{1}{2} + \epsilon}$ for all fixed $\epsilon > 0$.
        \item $\|\iter{\mPsi}{\dim}\|_{\op} \lesssim 1$.
        \item $\max_{i \neq j} \left|\big(\iter{\mPsi}{\dim}{\iter{\mPsi}{\dim}}\tran\big)_{ij} \right|  \lesssim  \dim^{-\frac{1}{2} + \epsilon}$ for all fixed $\epsilon > 0$.
        \item There is a fixed constant $\sigpsi \in (0,\infty)$ (independent of $\dim$) such that,
        \begin{align*}&\max_{i \in [\dim]} \left|\big(\iter{\mPsi}{\dim}{\iter{\mPsi}{\dim}}\tran\big)_{ii} - \sigpsi \right| \ll 1. \end{align*}
    \end{enumerate}
\end{enumerate}
\new{If $\iter{\mPsi}{\dim}$ is a sequence of \emph{random matrices} that satisfy the requirements (2a-2d) on an event with probability $1$:
\begin{subequations}\label{eq:borel-cantelli}
\begin{align}
&\P\left( \|\iter{\mPsi}{\dim}\|_\infty \lesssim \dim^{-\frac{1}{2} + \epsilon} \;\; \forall \;  \epsilon \; > 0\right) = 1,  \;\; \P\left( \max_{i \in [\dim]} \left|\big(\iter{\mPsi}{\dim}{\iter{\mPsi}{\dim}}\tran\big)_{ii} - \sigpsi \right| \ll 1 \right) = 1 , \\
&\P\left( \max_{i \neq j} \left|\big(\iter{\mPsi}{\dim}{\iter{\mPsi}{\dim}}\tran\big)_{ij} \right|  \lesssim  \dim^{-\frac{1}{2} + \epsilon} \;\; \forall \;  \epsilon \; > 0\right) = 1, \;\; \P\left( \|\iter{\mPsi}{\dim}\|_{\op} \lesssim 1 \right)  = 1,
\end{align}
\end{subequations}
we say $\iter{\mM}{\dim}$ is a semi-random ensemble with probability $1$.}
\end{definition}
\begin{remark}
\new{The notion of a semi-random matrix ensemble is defined only for a \emph{sequence of $\dim \times \dim$ matrices of increasing dimension} of the form $ \iter{\mM}{\dim} = \iter{\mS}{\dim} \iter{\mPsi}{\dim} \iter{\mS}{\dim}$.} For notational clarity, we will suppress the dependence of $\mM$ and $\mPsi$ on $N$ in our subsequent discussion. This dependence will be assumed implicitly throughout. \new{We will often use the phrase ``\emph{$\mM$ is semi-random}''  as a shorthand for ``\emph{the sequence of random matrices $\iter{\mM}{\dim}$ forms a semi-random matrix ensemble}''.}
\end{remark}
\begin{remark}
We call matrix ensembles that satisfy the above definition semi-random because the only randomness in these matrices arises from the random sign diagonal matrix $\mS$. The conditions (2a)-(2d) on $\mPsi$ are fully deterministic. These requirements ensure the entries of $\mPsi$ are delocalized, and the rows of $\mPsi$ are approximately orthogonal with almost equal norms. In \sref{matrix-ensemble-eg}, we show that these assumptions are satisfied for many matrix ensembles. 
\end{remark}
\begin{remark} \new{The conditions in \eqref{eq:borel-cantelli} pertain to tail events and can often be verified using the Borel-Cantelli lemma. See \lemref{sign-perm-inv-ensemble} and its proof in \appref{sign-perm-inv-ensemble} for an illustration.}
\end{remark}
We study the iteration \eqref{eq:memory-free-amp} under the following assumption on the initialization.
\begin{assumption}[Gaussian Initialization] \label{assump:initialization} The iteration \eref{memory-free-amp} is initalized with $\iter{\vz}{0} \sim \gauss{\vzero}{\sigma_0^2 \mI_\dim}$ for some positive constant $\sigma_0^2>0$ (independent of $\dim$). 
\end{assumption}

In order to state our main result, we need to introduce the state evolution recursion, which characterizes the dynamics of \eqref{eq:memory-free-amp}.

\paragraph{State Evolution Recursion. } Fix a $T \in \N_0$. Define the state evolution recursion associated with $T$ iterations of \eqref{eq:memory-free-amp} as:
\begin{subequations}\label{eq:SE-memory-free}
\begin{align}
    \sigma_{t+1}^2 & = \sigpsi \cdot \E[\ctr{\nonlin}_{t+1}^2(Z_t)] \\
    \rho_{{s,t+1}} & = \sigpsi \cdot \E[ \ctr{\nonlin}_{s}(Z_{s-1})\ctr{\nonlin}_{t+1}(Z_t) ] \; \forall \; s \leq t.
\end{align}
In the above display:
\begin{enumerate}
\item The recursion is initialized with $\sigma_0^2: = \sigma_0^2$, the parameter from \assumpref{initialization} and $\rho_{0,i} := 0$ for each $i \geq 1$. 
\item $\sigpsi$ is the parameter from \defref{matrix-ensemble-relaxed}.\item For each $t \in \{0, 1, \dotsc, T\}$, the vector $(Z_0, Z_1, \dotsc, Z_{t}) \sim \gauss{\vzero}{\mSigma_t}$, where $\mSigma_t \in \R^{t+1 \times t+1}$ is defined as:
\begin{align}
    \mSigma_t & \explain{def}{=} \begin{bmatrix} \sigma_0^2 & \rho_{0,1} & \rho_{0,2} & \hdots  & \rho_{0,t}  \\ \rho_{0,1} & \sigma_1^2 & \rho_{1,2} & \hdots &\rho_{1,t} \\ \vdots & \vdots & \vdots & \vdots & \vdots \\ \rho_{0,t} & \rho_{1,t} & \rho_{2,t} & \hdots & \sigma_t^2 \end{bmatrix}. \label{def:covariance_matrix} 
\end{align}
\item For each $t \in \{1, 2, \dotsc, T\}$, $\ctr{\nonlin}_t:\R \rightarrow \R$ is given by:
\begin{align} \label{eq:divergence-removal}
    \ctr{\nonlin}_t(x) = {\nonlin}_t(x) - \frac{\E[Z \nonlin_t(\sigma_{t-1}Z)]}{\sigma_{t-1}} \cdot x, \; Z \sim \gauss{0}{1}.
\end{align}
\end{enumerate}
\end{subequations}

\begin{remark}[Non-degenerate Non-linearities] Throughout this paper, we will assume that the non-linearities $\nonlin_t$ are non-degenerate in the sense that $\nonlin_t$ is not the linear function $x \mapsto \alpha x$ for any $\alpha \in \R$. Since all of our results additionally assume that $\nonlin_{1:T}$ are continuous functions, this ensures that the variance sequence $\sigma_{1:T}^2$ is strictly positive. Degenerate non-linearities are not useful for applications since if $\nonlin_t(x) = \alpha x$ for some $t \in [T]$ an inspection of \eqref{eq:memory-free-amp} shows that the corresponding iterate $\iter{\vz}{t} = \vzero$.
\end{remark}

Our results characterize the dynamics of \eqref{eq:memory-free-amp} using the following notion of convergence.
\begin{definition}[Convergence of Empirical Distributions] A collection of $k$ random vectors $(\iter{\vv}{1}, \dotsc, \iter{\vv}{k})$ in $\R^{\dim}$ converges with respect to the Wasserstein-$2$ metric to a random vector $(V_1, V_2, \dotsc, V_k) \in \R^k$ in probability as $\dim \rightarrow \infty$, if for any fixed test function $h: \R^{k} \rightarrow \R$  (independent of $\dim$) that satisfies:
\begin{align} \label{eq:P-conv}
    |h(\vx) - h(\vy)| & \leq L \|\vx - \vy\| (1 + \|\vx\| + \|\vy\|),
\end{align}
for some finite constant $L$, we have,
\begin{align} \label{eq:PL1}
    \frac{1}{\dim} \sum_{i=1}^\dim h(\iter{v}{1}_i, \iter{v}{2}_i, \dotsc, \iter{v}{k}_i) \explain{P}{\rightarrow} \E h(V_1, V_2, \dotsc, V_k).
\end{align}
We denote convergence in this sense using the notation $(\iter{\vv}{1}, \iter{\vv}{2}, \dotsc, \iter{\vv}{k}) \explain{\pw}{\longrightarrow} (V_1, V_2, \dotsc, V_k)$. 
\end{definition}

The following is our main result.

\begin{theorem}\label{thm:SE-alt} Fix a non-negative integer $T \in \W$ and functions $\nonlin_1, \nonlin_2, \dotsc, \nonlin_T : \R \rightarrow \R$. Consider the iteration \eqref{eq:memory-free-amp} initialized at $\iter{\vz}{0}$. Suppose that:
\begin{enumerate}
    \item $\iter{\vz}{0}$ satisfies \assumpref{initialization},
    \item $\mM$ is a semi-random matrix ensemble in the sense of  \defref{matrix-ensemble-relaxed},
    \item  The non-linearities $\nonlin_t$ are {continuously differentiable} Lipschitz functions.
\end{enumerate}
Then, $(\iter{\vz}{0}, \iter{\vz}{1}, \dotsc, \iter{\vz}{T}) \explain{\pw}{\longrightarrow} (Z_0, Z_1, \dotsc, Z_T) \sim \gauss{\vzero}{\mSigma_T}$, where $\mSigma_T$ is as defined in \eqref{eq:SE-memory-free}. 
\end{theorem}
\begin{remark}
If $\mM$ is drawn from a rotationally invariant ensemble, the conclusion of the above theorem follows from the work of \citet{fan2020approximate}. As we show in \lemref{sign-perm-inv-ensemble},  rotationally invariant matrices are special cases of the semi-random ensemble in the sense that they satisfy the requirements of \defref{matrix-ensemble-relaxed} with probability $1$. \thref{SE-alt} shows that the state evolution actually holds under significantly weaker assumptions than rotational invariance. Indeed,  it has identified a much broader class of matrices $\mM$ such that the associated AMP algorithm has the same asymptotic dynamics. In this sense, this result can be interpreted as a universality theorem.  
\end{remark}

\begin{remark}
For mean-field Ising models, \citet{ccakmak2019memory} have proposed algorithms of the form \eqref{eq:memory-free-amp} to compute the magnetization vector. In this application, $\mM$ is a suitably centered resolvent of the coupling matrix for the Ising model. For rotationally invariant coupling matrices (as studied in \citet{ccakmak2019memory} and \citet{fan2021replica}), the resolvent is also rotationally invariant. Hence, the previously mentioned result of \citet{fan2020approximate} can be used to analyze the dynamics of this algorithm. In this context, our results show that the exact rotational invariance of the coupling matrix is unnecessary for the validity of the state evolution. Instead, this characterization is valid as soon as the relevant resolvent matrix is semi-random in the sense of \defref{matrix-ensemble-relaxed}. This is indeed valid for many coupling matrices---we provide some examples in Section \ref{sec:examples}. 
\end{remark}

In many applications, the non-linearities $\nonlin_t$ are chosen adaptively so that they have the following convenient property.

\begin{assumption}[Divergence-Free Non-Linearities]\label{assump:divergence-free} The functions $\nonlin_t: \R \rightarrow \R$ satisfy:
\begin{align*}
    \E[Z \nonlin_{t+1}(\sigma_t Z)] = 0. 
\end{align*}
\end{assumption}
A simple choice of non-linearities that satisfy the above divergence-free property are the non-linearities $\ctr{\nonlin}_t$ defined in \eqref{eq:divergence-removal}. For divergence-free non-linearities, the iteration \eqref{eq:memory-free-amp} can be simplified without changing its dynamics by observing that if the non-linearities have the divergence-free property, the coefficient of the correction term:
\begin{align*}
    \frac{\langle{\nonlin_{t+1}(\iter{\vz}{t})}\; , \; {\iter{\vz}{t}} \rangle}{\|\iter{\vz}{t}\|^2} \explain{(a)}{\approx} \frac{\E[(\sigma_t Z) \cdot \nonlin_{t+1}(\sigma_t Z)]}{\sigma_t^2} \explain{(b)}{=} 0,
\end{align*}
where the approximation in (a) follows from \thref{SE-alt} and the equality in (b) follows from the divergence-free property. Hence, we also have the following closely related result. 

\begin{theorem}\label{thm:SE-final} Fix a non-negative integer $T \in \W$ and functions $\nonlin_1, \nonlin_2, \dotsc, \nonlin_T : \R \rightarrow \R$. Consider the iteration:
\begin{align}\label{eq:memory-free-amp-simple}
    \iter{\vz}{t+1} & = \mM   \nonlin_{t+1}(\iter{\vz}{t}),
\end{align}
initialized at $\iter{\vz}{0}$. Suppose that in addition to all the assumptions of \thref{SE-alt}, the non-linearities $\nonlin_t$ satisfy \assumpref{divergence-free}. Then, $(\iter{\vz}{0}, \iter{\vz}{1}, \dotsc, \iter{\vz}{T}) \explain{\pw}{\longrightarrow} (Z_0, Z_1, \dotsc, Z_T) \sim \gauss{\vzero}{\mSigma_T}$, where $\mSigma_T$ is as defined in \eqref{eq:SE-memory-free}. 
\end{theorem}
\begin{remark}
Our choice of the random matrix ensemble $\mM$ is inspired by recent progress in free probability. Specifically, \citep{anderson2014asymptotically} established that delocalized orthogonal matrices with sign and permutation symmetries behave like Haar matrices in the sense that conjugation by these matrices also induces freeness \citep{voiculescu1991limit,voiculescu1998strengthened}. We emphasize that although our choice is motivated by these results, to the best of our knowledge, the result does not follow from existing results in the free probability literature. Here we design a new approach specifically tailored to the AMP algorithms under consideration.  
\end{remark}
\begin{remark} \new{\thref{SE-final} also holds in the situation when the iteration \eqref{eq:memory-free-amp-simple} is initialized with the deterministic initialization $\iter{\vz}{0} = c \vone$ for any $c \in \R$. To see this, consider an AMP algorithm of the form in \eqref{eq:memory-free-amp-simple} with a deterministic initialization:
\begin{align*}
   \iter{\vz}{0} & = c \vone, \quad  \iter{\vz}{t+1}  = \mM   \nonlin_{t+1}(\iter{\vz}{t}).
\end{align*}
Such an algorithm can be implemented using another AMP algorithm of the form in \eqref{eq:memory-free-amp-simple} with a random Gaussian initialization to which \thref{SE-final} applies:
\begin{align*}
   \iter{\vw}{0} & \sim \gauss{\vzero}{\mI_\dim}, \quad  \iter{\vw}{t+1}  = \mM   \nonlin_{t+1}(\iter{\vw}{t}).
\end{align*}
In order to do so, we choose the non-linearity for the first iteration $g_1$ as the constant function $g_1(w) = f_1(c) \; \forall \; w \; \in \; \R$. This choice is divergence-free in the sense of \assumpref{divergence-free} and ensures that  
$\iter{\vw}{1} = f_1(c) \cdot  \mM \vone = \iter{\vz}{1}$. For subsequent iterations, we can take $g_t = f_t \; \forall \; t \geq 2$, ensuring that $\iter{\vw}{t} = \iter{\vz}{t} \; \forall \; t \geq 1$.
}
\end{remark}
\subsection{Examples of Semi-random Matrix Ensembles} \label{sec:matrix-ensemble-eg}
We provide some examples of matrix ensembles which are semi-random in the sense of \defref{matrix-ensemble-relaxed}. These examples consist of random matrices that neither have i.i.d. entries nor are rotationally invariant. Consequently, none of the existing results on the state evolution of AMP algorithms apply to these ensembles. 

\paragraph{Example 1.} The following lemma shows that any symmetric, delocalized orthogonal matrix  conjugated by a random sign diagonal matrix is semi-random in the sense of \defref{matrix-ensemble-relaxed}.

\begin{lemma} \label{lem:eg1} Let $\mPsi \in \R^{\dim \times \dim}$ be a symmetric orthogonal matrix with $\|\mPsi\|_{\infty} \lesssim \dim^{-1/2+\epsilon}$ for any $\epsilon > 0$. Let $\mS = \diag(s_{1:\dim})$ be a uniformly random signed diagonal matrix with $s_{1:\dim} \explain{i.i.d.}{\sim} \unif{\{\pm 1\}}$. Then, $\mM = \mS \mPsi \mS$ is semi-random with constant $\sigpsi = 1$ in the sense of \defref{matrix-ensemble-relaxed}. 
\end{lemma}
\begin{proof}
Observe that $\mM$ satisfies requirement (1) of \defref{matrix-ensemble-relaxed} by construction. Requirement (2a) follows from the delocalization hypothesis. Since $\mPsi^2 = \mI_{\dim}$, requirements (2c) and (2d) are also verified. Furthermore, since the spectral measure of symmetric orthogonal matrices is supported on $\{-1,1\}$, $\|\mPsi\|_{\op} = 1$, which verifies (2b).
\end{proof}

Since there are many well-known examples of deterministic, symmetric, delocalized orthogonal matrices (such as the Discrete Cosine Transform matrix, the Discrete Sine Transform matrix, and the Hadamard-Walsh Transform matrices), \lemref{eg1} shows that our results (\thref{SE-alt} and \thref{SE-final}) apply to matrices constructed with very limited randomness ($\dim$ random bits). In contrast, prior state evolution results applied exclusively to matrices constructed using $O(\dim^2)$ random variables.

\paragraph{Example 2 (Sign and Permutation Invariant Ensembles).} Next, we show that any matrix with delocalized eigenvectors and a ``sign and permutation invariance" is also semi-random in the sense of \defref{matrix-ensemble-relaxed}. 

\begin{lemma} \label{lem:sign-perm-inv-ensemble}Suppose that $\mM = \mF \mLambda \mF\tran$ where:
\begin{enumerate}
    \item $\mF$ is a random orthogonal $\dim \times \dim$ matrix which satisfies:
    \begin{description}[font=\normalfont\emph]
    \item [Delocalization:] $\|\mF\|_{\infty} \lesssim \dim^{-1/2+\epsilon}$ for any fixed $\epsilon>0$,
    \item [Invariance: ] $\mS\mF \mP \explain{d}{=} \mF$ for any signed diagonal matrix $\mS = \diag{(\vs)}, \vs \in \{\pm 1\}^{\dim}$ and any $\dim \times \dim$ permutation matrix $\mP$.
    \end{description}
    \item $\mLambda = \diag(\lambda_1, \lambda_2, \dotsc \lambda_{\dim})$ is a deterministic diagonal matrix such that:
    \begin{align} \label{eq:lambda-hypothesis}
        \|\mLambda\|_{\op} \lesssim 1, \quad {\Tr(\mLambda)}/{\dim} \lesssim \dim^{-1/2 + \epsilon} \; \forall \; \epsilon > 0,  \quad  \lim_{\dim \rightarrow \infty} {\Tr(\mLambda^2)}/{\dim} = \sigpsi. 
    \end{align}
\end{enumerate}
\new{Then, there exists a random matrix $\widetilde{\mM}$, which is semi-random with probability $1$ (cf. \defref{matrix-ensemble-relaxed}), and satisfies $\widetilde{\mM} \explain{d}{=} {\mM}$.}
\end{lemma}
We provide the proof of \lemref{sign-perm-inv-ensemble} in \appref{sign-perm-inv-ensemble} using a concentration inequality for permutation statistics developed by \citet{bercu2015concentration}.

\new{Observe that \lemref{sign-perm-inv-ensemble} implies that the conclusions of \thref{SE-alt} and \thref{SE-final} also apply to AMP algorithms of the form \eqref{eq:memory-free-amp} and \eqref{eq:memory-free-amp-simple} driven by a sign and permutation invariant matrix $\mM$ (cf. \lemref{sign-perm-inv-ensemble}). In order to see this, observe that \lemref{sign-perm-inv-ensemble} guarantees the existence of a matrix $\widetilde{\mM}$ which is semi-random with probability $1$ and satisfies $\mM \explain{d} = \widetilde{\mM}$. Let $\iter{\vz}{0}, \dotsc, \iter{\vz}{T}$ denote the AMP iterates generated by matrix $\mM$ and $\iter{\widetilde{\vz}}{0}, \dotsc, \iter{\widetilde{\vz}}{T}$ denote the AMP iterates generated by matrix $\widetilde{\mM}$. Since $\widetilde{\mM}$ is semi-random with probability 1 (cf. \lemref{sign-perm-inv-ensemble}), by \thref{SE-alt} and \thref{SE-final}, $(\iter{\widetilde{\vz}}{0}, \dotsc, \iter{\widetilde{\vz}}{T}) \explain{\pw}{\longrightarrow} (Z_0, Z_1, \dotsc, Z_T) \sim \gauss{\vzero}{\mSigma_T}$, where $\mSigma_T$ is as defined in \eqref{eq:SE-memory-free}. Because $(\iter{\vz}{0}, \dotsc, \iter{\vz}{T}) \explain{d}{=}  (\iter{\widetilde{\vz}}{0}, \dotsc, \iter{\widetilde{\vz}}{T})$ (recall $\mM \explain{d} = \widetilde{\mM}$), we conclude that $(\iter{{\vz}}{0}, \dotsc, \iter{{\vz}}{T}) \explain{\pw}{\longrightarrow} (Z_0, Z_1, \dotsc, Z_T) \sim \gauss{\vzero}{\mSigma_T}$, as claimed. }

The sign and permutation invariant model in \lemref{sign-perm-inv-ensemble} is a natural \emph{generalization} of the rotationally invariant model. This follows as a rotationally invariant matrix is of the form $\mM = \mU \mLambda \mU \tran$ where $\mU \sim \unif{\Ortho(\dim)}$ and $\mLambda$ is a deterministic diagonal matrix. Since the Haar measure on $\Ortho(\dim)$ is invariant to left and right multiplication by arbitrary orthogonal matrices (and in particular sign or permutation matrices), $\mU \explain{d}{=} \mS \mU \mP$ for any diagonal sign matrix $\mS$ and any permutation matrix $\mP$. Hence, rotationally invariant matrices satisfy the assumptions of \lemref{sign-perm-inv-ensemble}. Moreover, given a deterministic diagonal matrix $\mLambda$ satisfying the hypothesis in \eqref{eq:lambda-hypothesis}, a delocalized $\dim \times \dim$ orthogonal matrix $\mH$ with $\|\mH\|_{\infty} \lesssim \dim^{-1/2+\epsilon}$ for any $\epsilon>0$, a uniformly random $\dim \times \dim$ permutation matrix $\mQ$, and a uniformly random sign diagonal matrix $\mD$, the matrix $$\mM = (\mD \mH \mP) \cdot  \mLambda \cdot (\mP\tran \mH \tran \mD)$$ satisfies the requirements of \lemref{sign-perm-inv-ensemble} by construction. Consequently, sign and permutation invariant matrices can be constructed with significantly less randomness than what is required to construct rotationally invariant matrices.  
\paragraph{} In many applications, the matrix $\mM$ used in the AMP algorithm \eqref{eq:memory-free-amp} is the resolvent of another random matrix $\mJ$, centered to have zero trace, that is,
\begin{align} \label{eq:eg-centered-resolvent}
    \mM(\lambda) = (\lambda \mI_{\dim} - \mJ )^{-1} - \frac{\Tr((\lambda \mI_{\dim} - \mJ )^{-1})}{\dim} \cdot \mI_{\dim}.
\end{align}
This is true in the case of the AMP algorithm used by \citet{ccakmak2019memory} to compute the magnetization of mean-field Ising models with rotationally invariant couplings and for Vector Approximate Message Passing (VAMP) algorithms used in compressed sensing \citep{rangan2019vector,takeuchi2017rigorous}. In these situations, the local law \citep{erdHos2017dynamical} for the random matrix $\mJ$ (if available) can be readily used to verify that $\mM$ is semi-random (\defref{matrix-ensemble-relaxed}). The following two examples show that resolvents of Wigner and sample covariance matrices satisfy the requirements of \defref{matrix-ensemble-relaxed}. 

\paragraph{Example 3 (Resolvent of Wigner Matrices).}  Let $\mJ$ be a Wigner matrix with symmetric entries, that is, $\mJ = \mW/\sqrt{\dim}$ for a symmetric matrix $\mW$ whose entries $W_{ij}$ are i.i.d. symmetric ($W_{ij} \explain{d}{=} - W_{ij})$ random variables with $\E W_{ij} = 0$, $\E W_{ij}^2 = 1 + \delta_{ij}$ and finite moments of all orders. These hypotheses are sufficient to guarantee that the spectral measure of $\mJ$ converges to  semi-circle distribution supported on $[-2,2]$ \citep{wigner1958distribution}  and the largest eigenvalue $\lambda_1(\mJ) \rightarrow 2$ \citep{bai1988necessary}. In this situation, the centered resolvent $\mM(\lambda)$ in \eqref{eq:eg-centered-resolvent} can be shown to satisfy the requirements of \defref{matrix-ensemble-relaxed} for any fixed $\lambda > 2$. Indeed, since the entries of $\mW$ are assumed to be symmetric, requirement (1) of \defref{matrix-ensemble-relaxed} holds. The remaining requirements can be verified using the local law for Wigner matrices. Optimal local laws for Wigner matrices were first obtained by \citet{erdHos2009local} and we refer the reader to the \citet[Section 18.2]{erdHos2017dynamical}  for additional historical context. In \appref{local-law}, we show how the following lemma follows as a consequence of a variant of the local law for Wigner matrices derived by \citet{benaych2016lectures}.

\begin{lemma} \label{lem:wigner} Let $\mJ$ be a $\dim \times \dim$ Wigner matrix with symmetric entries as defined above. For any $\lambda > 2$, the matrix  $\mM(\lambda)$ in \eqref{eq:eg-centered-resolvent} with this $\mJ$ satisfies the requirements of \defref{matrix-ensemble-relaxed} with constant $\sigpsi = - G_{\mathrm{sc}}^\prime(\lambda) - G_{\mathrm{sc}}^2(\lambda)$ with probability $1-o_{\dim}(1)$, where $G_{\mathrm{sc}}: (2,\infty) \rightarrow \R$ denotes the Cauchy transform of the semi-circle distribution on $[-2,2]$. 
\end{lemma}

\paragraph{Example 4 (Resolvent of Sample Covariance Matrices).} Consider the situation where $\mJ$ is a covariance matrix of the form $\mJ = \mX \tran \mX / \sqrt{M \dim}$ for a $M \times \dim$ matrix $\mX$ with a converging aspect ratio $M/N \rightarrow \phi \in (0, \infty)$ whose entries $X_{ij}$ are i.i.d. symmetric ($X_{ij} \explain{d}{=} - X_{ij}$) random variables with $\E X_{ij} = 0$, $\E X_{ij}^2 = 1$ and finite moments of all orders. These hypotheses are sufficient to guarantee that the spectral measure of $\mJ$ converges to the Marchenko-Pastur distribution \citep{marvcenko1967distribution} and $\lambda_1(\mJ)$, the largest eigenvalue of $\mJ$ satisfies $\lambda_1(\mJ) \explain{P}{\rightarrow} \lambda_+^{\mathrm{MP}}$ \citep{bai2008limit} where,
\begin{align}
    \lambda_+^{\mathrm{MP}} \explain{def}{=} \sqrt{\phi} + \frac{1}{\sqrt{\phi}} + 2.
\end{align}
For this $\mJ$, we can verify that the centered resolvent matrix $\mM(\lambda)$ defined in \eqref{eq:eg-centered-resolvent} satisfies the requirements of \defref{matrix-ensemble-relaxed} for any $\lambda > \lambda_+^{\mathrm{MP}}$. Indeed, since the entries of $\mX$ are assumed to be symmetric, $\mM(\lambda)$ satisfies requirement (1) of \defref{matrix-ensemble-relaxed}. The remaining requirements can be verified using the local law for sample covariance matrices obtained by \citet{alex2014isotropic}. In particular, we have the following result, whose proof appears in \appref{local-law-mp}. 

\begin{lemma} \label{lem:wishart} Let $\mJ$ be a $\dim \times \dim$ sample covariance matrix as defined above. For any $\lambda > \lambda_+^{\mathrm{MP}}$, the matrix  $\mM(\lambda)$ in \eqref{eq:eg-centered-resolvent} with this $\mJ$ satisfies the requirements of \defref{matrix-ensemble-relaxed} with constant $\sigpsi = - G_{\mathrm{MP}}^\prime(\lambda) - G_{\mathrm{MP}}^2(\lambda)$ with probability $1-o_{\dim}(1)$, where $G_{\mathrm{MP}}: (\lambda_+^{\mathrm{MP}},\infty) \rightarrow \R$ denotes the Cauchy transform of the Marchenko-Pastur distribution.
\end{lemma}

\section{Applications to Mean-Field Ising Models}\label{sec:examples}

As our main application, we discuss how \thref{SE-final} can be used to obtain a characterization of the dynamics of an iterative algorithm proposed by \citet{ccakmak2019memory}  to compute the magnetization of mean-field Ising spin glass models for several random coupling matrices. 

\subsection{Background} 
\paragraph{Mean-Field Ising Spin Glasses.} The mean-field Ising spin glass model is described by the random Gibbs measure on the discrete hypercube $\{\pm 1\}^\dim$: 
\begin{subequations}\label{eq:ising-model}
\begin{align}
    \gibbs(\vsigma) &\explain{def}{=} Z^{-1} \cdot \exp\left( \frac{\temp}{2} \cdot  \vsigma\tran \mJ \vsigma  + \field \cdot \ip{\fieldvec}{\vsigma}  \right), \\
    Z &\explain{def}{=}  \sum_{\vsigma \in \{\pm1\}^\dim} \exp\left( \frac{\temp}{2} \cdot  \vsigma\tran \mJ \vsigma  + \field \cdot \ip{\fieldvec}{\vsigma}  \right).
\end{align}
\end{subequations}
In the above display, $\mJ$ is a symmetric random coupling matrix and the vector $\fieldvec = (1, 1, \dotsc, 1)$ is the external field. The parameter $\beta \geq 0$ is the inverse temperature and the parameter $\theta \geq 0$ regulates the strength of the external field.  The thermodynamic properties of this model depend on the spectral measure of the random matrix $\mJ$. Consequently, to study this model in the high-dimensional limit, it is assumed that $\xi_{\dim}$, the empirical distribution of the eigenvalues $\lambda_1(\mJ), \lambda_2(\mJ), \dotsc, \lambda_{\dim}(\mJ)$ converges in distribution to a compactly supported, limiting probability measure $\xi$:
\begin{subequations}\label{eq:weak-convergence-with-support}
\begin{align}\label{eq:ESD-convergence}
    \xi_\dim \explain{def}{=} \frac{1}{\dim} \sum_{i=1}^\dim \delta_{\lambda_i(\mJ)} \explain{d}{\rightarrow} \xi. 
\end{align}
Furthermore, this convergence is such that the largest eigenvalue of $\mJ$ converges to the rightmost edge of the support of $\xi$ and the smallest eigenvalue remains bounded from below:
\begin{align}\label{eq:support-convergence}
    \max_{i \in [\dim]}(\lambda_i(\mJ)) \rightarrow \lambda_+ \explain{def}{=} \max\{ x : x \in \mathrm{Supp}(\xi)\}< \infty, \;
   \liminf_{\dim \rightarrow \infty} \min_{i \in [\dim]}(\lambda_i(\mJ)) > -\infty.
\end{align}
\end{subequations}
\paragraph{Magnetization and the TAP Equations.} A key object of interest for this model is the magnetization, the mean vector of the Gibbs measure:
\begin{align}\label{eq:magnetization}
    \vmgnt \explain{def}{=} \sum_{\vsigma \in \{\pm 1\}^\dim} \gibbs(\vsigma) \cdot  \vsigma. 
\end{align}
At high temperatures (when $\beta$ is sufficiently small), The magnetization was conjectured to approximately satisfy a system of non-linear fixed point equations, called the Thouless-Anderson-Palmer (TAP) equations, derived by \citet{parisi1995mean} and generalized by \citet{opper2001adaptive}:
\begin{align} \label{eq:TAP} 
     \vmgnt & \approx \tanh\Big( \theta  \cdot \fieldvec + \beta \cdot \mJ \cdot \vmgnt - \beta \cdot  R(\beta-\beta q_\star) \cdot \vmgnt  \Big),
\end{align}
where:
\begin{enumerate}
    \item the function $\tanh(\cdot)$ acts entry-wise on its vector argument. 
    \item $R(\cdot)$ is the R-transform of the measure $\xi$. This $R$-transform of a probability measure is defined in terms of its Cauchy transform $G: [\lambda_+ , \infty) \rightarrow (0,\infty)$:
    \begin{align} \label{eq:cauchy-transform}
        G(z) \explain{def}{=}  \int \frac{\xi(\diff \lambda)}{z - \lambda}.
    \end{align}
    The Cauchy transform is a strictly decreasing function on $(\lambda_+, \infty)$ and hence, has a well defined inverse $G^{-1}: (0, G(\lambda_+)) \rightarrow (\lambda_+, \infty)$. The R-transform is defined on the domain $(0,G(\lambda_+))$ and is given by:
    \begin{align}\label{eq:R-transform}
        R(z) = G^{-1}(z) -1/z.
    \end{align}
       \item $q_{\star} \in [0,1)$ is the unique solution (guaranteed to exist for small $\beta$ \citep[Proposition 1.2]{fan2021replica}) to the fixed point equation (in q):
    \begin{subequations}\label{eq:q-star-sigma-star}
    \begin{align}\label{eq:q-star}
        q  & = \E\left[\tanh^2\left(\theta + \sigma_\star(q) \cdot G\right)\right], \; G \sim \gauss{0}{1},
    \end{align}
    where,
    \begin{align}\label{eq:sigma-star}
        \sigma^2_\star(q) \explain{def}{=} \beta^2 \cdot {q \cdot R^{\prime}(\beta - \beta q)}.
    \end{align}
    \end{subequations}
\end{enumerate}
This conjecture has been established recently at high-temperature (i.e. for $\beta$ sufficiently small) for rotationally invariant coupling matrices by the third author, in joint work with Yufan Li and Zhou Fan \citep{fan2022tap}. 
\paragraph{Solving the TAP equations.} The TAP equations provide a way to compute the magnetization $\vmgnt$ by solving the fixed point equation in \eref{TAP}, which avoids the evaluation of the high-dimensional integral in \eref{magnetization}. \citet{ccakmak2019memory} have proposed the following iterative scheme to compute an approximate solution for the TAP equation \eref{TAP} at high temperatures (small $\beta$):
\begin{align}\label{eq:opper-iteration} 
    \iter{\vz}{t+1} & = \mM(\lambda_\star)  \cdot g(\iter{\vz}{t}), \\
    \iter{\vz}{0} &\sim \gauss{\vzero}{\sigma^2_\star(q_{\star}) \cdot  \mI_{\dim}}. \label{eq:opper-iteration-init}
\end{align}
In the above display,
\begin{enumerate}
    \item $\mM(\lambda)$ denotes the resolvent of the interaction matrix $\mJ$ centered to have zero trace:
    \begin{subequations}\label{eq:resolvent}
    \begin{align} 
        \mM(\lambda) & = \left( \lambda \cdot  \mI_{\dim} - \mJ \right)^{-1}  - \frac{\Tr[( \lambda \cdot  \mI_{\dim} - \mJ )^{-1}]}{\dim} \cdot \mI_{\dim}, 
    \end{align}
    and the special value $\lambda_\star$ is given by:
    \begin{align}
        \lambda_\star \explain{def}{=} G^{-1}(\beta - \beta q_\star).
    \end{align}
    \end{subequations}
      \item The function $g: \R \rightarrow \R$ is defined as:
    \begin{align} \label{eq:g-func}
        g(z) & = \frac{1}{\beta - \beta q_\star} \cdot \left( \frac{\tanh(\theta +z)}{1-q_\star} - z\right).
    \end{align}
    This function acts entry-wise on its vector arguments in \eref{opper-iteration}. 
    \item $q_\star$ and $\sigma_\star^2(q_{\star})$ are as defined in \eqref{eq:q-star-sigma-star}. 
\end{enumerate}
Intuitively, this algorithm can be used to construct a solution to the TAP equations \eref{TAP} because any fixed point $\iter{\vz}{\infty}$ of iteration \eref{opper-iteration} (if it exists) satisfies:
\begin{align*}
    \iter{\vz}{\infty} =  \mM(\lambda_\star)  \cdot g(\iter{\vz}{\infty}).
\end{align*}
Recalling the definitions of $\mM(\lambda_\star)$, $g$, the relationship $R(z) = G^{-1}(z) - 1/z$ and using the approximation $\xi_{\dim} \approx \xi$, the above equation can be re-expressed as:
\begin{align*}
    \iter{\vz}{\infty}   \approx \beta \mJ \tanh(\theta \cdot \fieldvec + \iter{\vz}{\infty}) - \beta \cdot  R(\beta - \beta q_\star) \cdot \tanh(\theta \cdot \fieldvec + \iter{\vz}{\infty}). 
\end{align*}
In particular, $\iter{\vm}{\infty} \explain{def}{=} \tanh(\theta \cdot \fieldvec + \iter{\vz}{\infty})$ satisfies the TAP equations \eref{TAP} approximately:
\begin{align*}
    \iter{\vm}{\infty} & \approx \tanh\Big( \theta  \cdot \fieldvec + \beta \cdot \mJ \cdot \iter{\vm}{\infty} - \beta \cdot  R(\beta-\beta q_\star) \cdot \iter{\vm}{\infty}  \Big).
\end{align*}

\paragraph{Prior Results.} \citet{ccakmak2019memory} characterized the asymptotic dynamics of iteration \eref{opper-iteration} with non-rigorous statistical physics methods. For rotationally invariant coupling matrices $\mJ$, these predictions were rigorously established in \citet{fan2021replica}.  They used the earlier results of \citet{fan2020approximate}, which provides a characterization of a broad class of iterative algorithms involving a rotationally invariant random matrix. 
The associated state evolution recursion is obtained by instantiating that general state evolution recursion \eqref{eq:SE-memory-free} with the following parameters:
\begin{subequations} \label{eq:SE-Opper-choices}
\begin{enumerate} 
    \item The parameter $\sigma_0^2$ which determines the variance of the initialization is set as
    \begin{align}
        \sigma_0^2: = \sigma_\star^2(q_\star),
    \end{align}
    where $\sigma_\star^2(q_\star)$ is as defined in \eqref{eq:sigma-star}.
    \item The non-linearities $\nonlin_t$ are set to \begin{align}
        \nonlin_t = g \quad \forall \; t \; \geq \; 1,
    \end{align} where the function $g$ is as defined in \eqref{eq:g-func}.
    \item The parameter $\sigpsi$ corresponding to the matrix $\mM(\lambda_\star)$ is given by:
    \begin{align}
    \sigpsi \explain{def}{=} \lim_{\dim \rightarrow \infty} \frac{\Tr(\mM^2(\lambda_\star))}{\dim} \explain{(a)}{=} \frac{\beta^2 \cdot (1-q_\star)^4 \cdot \sigma_\star^2(q_\star)}{q_\star - (1-q_\star)^2 \sigma_\star^2(q_\star)}.  \label{eq:sigma-xi-final}
\end{align}
    In the above display, the equality marked (a) is readily derived by recalling the formula for $\mM(\lambda_\star)$ from \eqref{eq:resolvent}, the definitions of Cauchy and R transforms from \eqref{eq:cauchy-transform} and \eqref{eq:R-transform}  and the formula for $\sigma_\star^2(q_\star)$ from \eqref{eq:sigma-star}. 
\end{enumerate}
\end{subequations}
 Because of the special choice of the variance of the initialization in \eqref{eq:opper-iteration-init}, the recursion for $\sigma_{t+1}^2$ can be simplified significantly. Indeed, \citet[Appendix A1, Equation A7-A8]{ccakmak2019memory} have shown that the variance $\sigma_{t}^2$ remains constant in the recursion:
\begin{align}\label{eq:SE-opper-sigma-simple}
\sigma_t^2 = \sigma_\star^2(q)  \; \forall \; t \; \geq \; 0.
\end{align}
The result obtained by \citet{fan2020approximate} and \citet{fan2021replica} is quoted below.

\begin{proposition}[\citet{fan2020approximate,fan2021replica}]\label{prop:tap-se-fan} Suppose the coupling matrix $\mJ = \mU \mLambda \mU\tran$ where:
\begin{enumerate}
    \item $\mU \sim \unif{\Ortho(\dim)}$
    \item $\mLambda = \diag{(\lambda_1, \lambda_2, \dotsc, \lambda_{\dim})}$ is a deterministic diagonal matrix whose empirical spectral distribution converges to a compactly supported measure $\xi$ in the sense of \eqref{eq:weak-convergence-with-support}
\end{enumerate}
Then, for any fixed $T \in \N$, the iteration \eqref{eq:opper-iteration} satisfies $(\iter{\vz}{0}, \dotsc, \iter{\vz}{T}) \explain{\pw}{\longrightarrow} (Z_0,  \dotsc, Z_T) \sim \gauss{\vzero}{\mSigma_T}$ where $\mSigma_T$ is the covariance matrix generated by the state evolution recursion \eqref{eq:SE-memory-free} with the parameters set as given in \eqref{eq:SE-Opper-choices}.  
\end{proposition}
\begin{remark} The result obtained by \citet{fan2020approximate} and \citet{fan2021replica} is stronger than the result stated in \propref{tap-se-fan} in some aspects. For example, these works allow for a general external field (not necessarily the vector $\vone$). \new{Moreover, they show convergence in a stronger sense than \pw{ }by obtaining almost sure convergence of empirical averages in \eqref{eq:P-conv} (instead of convergence in probability), and by allowing broader classes of test functions which satisfy a weaker analog of the requirement \eqref{eq:PL1}.} 
\end{remark}

\subsection{Consequences of \thref{SE-final}}
As a consequence of \thref{SE-final}, we can show that the characterization of the dynamics of the iteration \eqref{eq:opper-iteration} given in \propref{tap-se-fan} continues to hold for many matrices beyond the rotationally invariant family. Specifically, we introduce the following four Ising spin glass models.

\paragraph{The Signed Sine Model.} In the signed sine model, the entries of the coupling matrix $\mJ$ are given by:
\begin{align} \label{eq:semirandom-sine-model}
    J_{ij} =\frac{2 s_i s_j}{\sqrt{2\dim+1}} \cdot \sin\left( \frac{2\pi i j}{2N + 1} \frac{}{} \right), \; i \in \; [\dim], \; j \in \; [\dim]. 
\end{align}
where $\vs \sim \unif{\{\pm 1\}^{\dim}}$ is a uniformly random sign vector. Note that $\mJ = \mS \mC \mS$ where $\mS = \diag(\vs)$ and $\mC$ is the Discrete Sine Transform (DST) matrix \citep[Section 2.7]{britanak2010discrete}. Since the DST matrix $\mC$ is a symmetric orthogonal matrix, $\mJ^2 = \mI_{\dim}$ (see for e.g. \citep[Section 2.8]{britanak2010discrete}). As a consequence, the spectrum of $\mJ$ is supported on $\{\pm 1\}$. Furthermore, since $\Tr(\mJ)/\dim \lesssim \dim^{-1/2}$, the empirical spectral distribution of $\mJ$ converges to $\xi = \unif{\{\pm 1\}}$ in the sense of \eqref{eq:weak-convergence-with-support}.  Our motivation to consider this model comes from the work of \citet{marinari1994replica} and \citet{parisi1995mean} who used statistical physics techniques (high-temperature expansions) and Monte Carlo simulations to demonstrate that many thermodynamic properties of the Ising model with the deterministic coupling matrix $\mC$ (known as the Sine model) are identical to the corresponding thermodynamic properties of the Ising model where the coupling matrix is $\mU \mB \mU \tran$ where $\mU \sim \unif{\Ortho(\dim)}$ is a uniformly random orthogonal matrix and $\mB = \diag(\vb), \; \vb \sim \unif{\{\pm 1\}^{\dim}}$ is a uniformly random sign diagonal matrix (known as the Random Orthogonal Model). The model in \eqref{eq:semirandom-sine-model} can be thought of as a semi-random analog of the Sine model of \citet{marinari1994replica}: while it is not fully deterministic, it is significantly less random than the random orthogonal model.

\paragraph{Sign and Permutation Invariant Model.} In the sign and permutation invariant model, the eigenvectors of the coupling matrix $\mJ$ are assumed to be \emph{sign and permutation invariant} and \emph{delocalized}. Formally, we assume that $\mJ = \mF \mLambda \mF\tran$ where:
\begin{enumerate}
    \item $\mF$ is a random orthogonal $\dim \times \dim$ matrix which satisfies:
    \begin{description}[font=\normalfont\emph]
    \item [Delocalization:] $\|\mF\|_{\infty} \lesssim \dim^{-1/2+\epsilon}$ for any fixed $\epsilon>0$,
    \item [Invariance: ] $\mS\mF \mP \explain{d}{=} \mF$ for any signed diagonal matrix $\mS = \diag{(\vs)}, \vs \in \{\pm 1\}^{\dim}$ and any $\dim \times \dim$ permutation matrix $\mP$.
    \end{description}
    \item $\mLambda = \diag(\lambda_1, \lambda_2, \dotsc \lambda_{\dim})$ is a deterministic diagonal matrix whose spectral distribution converges to a compactly supported measure $\xi$ in the sense of \eqref{eq:weak-convergence-with-support}.
\end{enumerate}
Our motivation to study this model comes from the work of \citet{ccakmak2019memory} who observed experimentally that such matrices behave like rotationally invariant matrices with regard to the dynamics of the iteration \eqref{eq:opper-iteration}.
\paragraph{Sherrington-Kirkpatrick Model.} We also consider the Sherrington-Kirkpatrick (SK) Model \citep{sherrington1975solvable} where the coupling matrix $\mJ = \mW/\sqrt{\dim}$ for a symmetric matrix $\mW$ whose entries $W_{ij}$ are i.i.d. symmetric ($W_{ij} \explain{d}{=} - W_{ij})$ random variables with $\E W_{ij} = 0$, $\E W_{ij}^2 = 1 + \delta_{ij}$ and finite moments of all orders. These hypotheses are sufficient to guarantee that $\xi_{\dim}$, the spectral measure of $\mJ$ converges to $\xi_{\mathrm{sc}}$, the semicircle distribution supported on $[-2,2]$ in the sense of \eqref{eq:weak-convergence-with-support} \citep{wigner1958distribution,bai1988necessary}. 
\paragraph{Hopfield Model.} Finally, we consider the Hopfield Model \citep{hopfield1982neural} where the coupling matrix $\mJ$ is a covariance matrix of the form $\mJ = \mX \tran \mX / \sqrt{M \dim}$ for a $M \times \dim$ matrix $\mX$ with a converging aspect ratio $M/N \rightarrow \phi \in (0, \infty)$ whose entries $X_{ij}$ are i.i.d. symmetric ($X_{ij} \explain{d}{=} - X_{ij}$) random variables with $\E X_{ij} = 0$, $\E X_{ij}^2 = 1$ and finite moments of all orders. These hypotheses are sufficient to guarantee that $\xi_{\dim}$, the spectral measure of $\mJ$ converges to $\xi_{\mathrm{MP}}$, the Marchenko-Pastur distribution in the sense of \eqref{eq:weak-convergence-with-support} \citep{marvcenko1967distribution,bai2008limit}.

\paragraph{}Observe that the coupling matrices $\mJ$ corresponding to spin glass models introduced above are not rotationally invariant. Consequently, prior results of \citet{fan2020approximate} and \citet{fan2021replica} (quoted in \propref{tap-se-fan}) cannot be used to characterize the dynamics of the iteration \eqref{eq:opper-iteration} for these models. As a corollary of \thref{SE-final}, we obtain the following result, which shows that the asymptotic characterization of the dynamics of \eqref{eq:opper-iteration} given in \propref{tap-se-fan} continues to hold for these models. This demonstrates a universality phenomenon in the sense that a large class of coupling matrices behave like rotationally invariant coupling matrices with the same limiting spectral distribution as far as the dynamics of  \eqref{eq:opper-iteration} is concerned.

\begin{corollary} \label{cor:TAP-AMP}  For any fixed $T \in \N$, the iterations \eqref{eq:opper-iteration} corresponding to:
\begin{enumerate}
    \item the signed sine model,
    \item the sign and permutation invariant model,
    \item the Sherrington-Kirkpatrick model, and
    \item the Hopfield model 
\end{enumerate}
as defined above, satisfy $(\iter{\vz}{0}, \dotsc, \iter{\vz}{T}) \explain{\pw}{\longrightarrow} (Z_0,  \dotsc, Z_T) \sim \gauss{\vzero}{\mSigma_T}$ where $\mSigma_T$ is the covariance matrix generated by the state evolution recursion \eqref{eq:SE-memory-free} with the parameters set as given in \eqref{eq:SE-Opper-choices}.  
\end{corollary}
\begin{proof} The claim of the corollary is immediate from \thref{SE-final}, once all the assumptions of \thref{SE-final} have been verified. We begin by observing that the iteration \eqref{eq:opper-iteration} already satisfies many of the properties required by \thref{SE-final}:
\begin{enumerate}
    \item The matrix $\mM(\lambda_\star)$ has bounded operator norm since,
    \begin{align} \label{eq:op-norm-bound}
        \|\mM(\lambda_\star)\|_{\op}  \explain{\eqref{eq:resolvent}}{\leq} \frac{2}{G^{-1}(\beta - \beta q_\star) - \max_{i\in[\dim]}(\lambda_i(\mJ))} \explain{\eqref{eq:support-convergence}}{\rightarrow} \frac{2}{G^{-1}(\beta - \beta q_\star) - \lambda_+} < \infty,
    \end{align}
    where the final inequality follows from the fact that the inverse Cauchy transform $G^{-1}$ maps its domain $(0,G(\lambda_+))$ to the range $(\lambda_+, \infty)$.
    \item The initialization  $\iter{\vz}{0} \sim \gauss{\vzero}{\sigma_\star^2(q_\star) \mI_{\dim}}$ satisfies \assumpref{initialization}. 
    \item The non-linearity used in the iteration \eqref{eq:opper-iteration} $g(\cdot)$ satisfies is Lipschitz and divergence-free (\assumpref{divergence-free}) since for $Z \sim \gauss{0}{1}$,
    \begin{align*}
        \E[Z g(\sigma_t Z)] \explain{\eqref{eq:SE-opper-sigma-simple}}{=} \E[Z g(\sigma_{\star} Z)] &\explain{\eqref{eq:g-func}}{=} \frac{1}{\beta - \beta q_\star} \cdot \left( \frac{\E Z \tanh(\theta + \sigma_\star Z)}{1-q_\star} - \sigma_\star \right) \\&= \frac{\sigma_\star}{\beta - \beta q_\star} \cdot \left( \frac{1-\E\tanh^2(\theta + \sigma_\star Z)}{1-q_\star} - 1 \right) \explain{\eqref{eq:q-star}}{=} 0.
    \end{align*}
\end{enumerate}
In order to verify the remaining requirements on $\mM(\lambda_\star)$ required by \defref{matrix-ensemble-relaxed}, we consider each of the models individually. 
\begin{description}
    \item [Signed Sine Model.] For the signed sine model, since $\mJ^2 = \mI_{\dim}$, the resolvent $(\lambda I_{\dim} - \mJ)^{-1}$ can be expressed as a linear polynomial in $\mJ$. \footnote{More generally, if $\mJ$ has $k$ distinct eigenvalues, the resolvent $(\lambda \mI_{\dim} - \mJ)^{-1}$ can be expressed as polynomial in $\mJ$ of degree at most $k-1$. This is a consequence of the polynomial interpolation. If $\gamma_{1:k}$ denote the $k$ distinct eigenvalues of $\mJ$, then for any $\lambda \in \R$ there is a polynomial $Q_\lambda: \R \rightarrow \R$ of degree at most $k-1$ such that $Q_\lambda(\gamma_i) = (\lambda - \gamma_i)^{-1} \; \forall \; i \; \in \; [k]$. Hence the resolvent $(\lambda I_{\dim} - \mJ)^{-1}$ can be expressed as $(\lambda I_{\dim} - \mJ)^{-1} = Q_{\lambda}(\mJ)$, as claimed.} Hence, we have,
    \begin{align*}
        \mM(\lambda) & \explain{def}{=} (\lambda \mI_{\dim} - \mJ)^{-1} -  \frac{\Tr[( \lambda \cdot  \mI_{\dim} - \mJ )^{-1}]}{\dim} \cdot \mI_{\dim} = \frac{1}{\lambda^2 - 1} \cdot \left( \mJ - \frac{\Tr(\mJ)}{\dim} \mI_{\dim} \right).
    \end{align*}
    Recall that for the signed sine model, $\mJ = \mS \mC \mS$ where $\mC$ is the DST matrix and $\mS$ is a random sign diagonal matrix. Since the DST matrix is a delocalized and symmetric orthogonal matrix, $\mM(\lambda)$ is semi-random in the sense of  \defref{matrix-ensemble-relaxed} for $\lambda = \lambda_\star  > 1$ by \lemref{eg1}. 
    \item [Sign and Permutation Invariant Model.] Let $\mJ = \mF \mLambda \mF \tran$ denote the eigendecomposition of $\mJ$. Observe that:
    \begin{align*}
        \mM(\lambda) & \explain{def}{=} (\lambda \mI_{\dim} - \mJ)^{-1} -  \frac{\Tr[( \lambda \cdot  \mI_{\dim} - \mJ )^{-1}]}{\dim} \cdot \mI_{\dim} = \mF \cdot \left( \mLambda - \frac{\Tr(\mLambda)}{\dim} \mI_{\dim} \right) \cdot \mF \tran.
    \end{align*}
    Hence, $\mM(\lambda)$ inherits the sign and permutation invariance of $\mJ$. Consequently, $\mM(\lambda)$ is semi-random in the sense of \defref{matrix-ensemble-relaxed} for $\lambda = \lambda_\star > \lambda_+$ by \lemref{sign-perm-inv-ensemble}. 
    \item [SK Model.] Here, $\mM(\lambda)$ satisfies \defref{matrix-ensemble-relaxed} for $\lambda = \lambda_\star > 2$ by \lemref{wigner}. 
    \item [Hopfield Model.] Here, $\mM(\lambda)$ satisfies \defref{matrix-ensemble-relaxed} for $\lambda = \lambda_\star > \lambda_+^{\mathrm{MP}}$ by \lemref{wishart}.
\end{description}
This concludes the proof of the corollary. 
\end{proof}

\begin{remark} For the SK model with i.i.d. sub-Gaussian (but not necessarily symmetric) entries, the universality results of \citet{bayati2015universality} and \citet{chen2021universality} can be used to characterize the dynamics of a different iterative algorithm designed by \citet{bolthausen2014iterative}  to solve the TAP equation \eqref{eq:TAP}. Unlike the iteration in \eqref{eq:opper-iteration}, which involves multiplication by the centered resolvent $\mM(\lambda_{\star})$ at each step, the iteration of \citet{bolthausen2014iterative} involves multiplication by the coupling matrix $\mJ$ at each step.  
 \end{remark}

\subsection{Experimental Demonstration of Universality}
\label{sec:universality_figs} 

We end this section with an empirical demonstration of the universality phenomenon studied in \corref{TAP-AMP}.  We simulate the dynamics of \eqref{eq:opper-iteration} for three different coupling matrices $\mJ$ of dimension $\dim = 2^{15}$:
\begin{description}[font=\normalfont\emph]
    \item [Random Orthogonal Ensemble.] Here, we take $\mJ = \mU \diag(\lambda_{1:\dim}) \mU \tran$ where $\mU \sim \unif{\ortho(\dim)}$ and $\lambda_{1:\dim} \explain{i.i.d.}{\sim} \unif{\{\pm1\}}$. Since generating and manipulating Haar matrices of this dimension is prohibitive in terms of memory and run-time, we used the Householder Dice algorithm of the second author \citep{lu2021householder} to simulate the dynamics of \eqref{eq:opper-iteration} in this case. This algorithm does not require sampling the entire Haar matrix $\mU$. 
    \item [Signed Hadamard Ensemble.] Here, we take $\mJ = \diag(s_{1:\dim})\mH \diag(\lambda_{1:\dim}) \mH \tran \diag(s_{1:\dim})$ where $\mH$ is the $\dim \times \dim$ Hadamard-Walsh matrix  and $s_{1:\dim}, \lambda_{1:\dim} \explain{i.i.d.}{\sim} \unif{\{\pm1\}}$. The Hadamard-Walsh matrix is a deterministic orthogonal matrix with entries in $\{-1/ \sqrt{\dim}, 1/ \sqrt{\dim}\}$. 
    \item [Signed Sine Ensemble.] Here, we take $\mJ = \diag(s_{1:\dim})\mC \diag(s_{1:\dim})$ where $\mC$ is the $\dim \times \dim$ Discrete Sine Matrix (DST) and $s_{1:\dim}\explain{i.i.d.}{\sim} \unif{\{\pm1\}}$. The DST matrix is a deterministic, symmetric orthogonal matrix  with entries $C_{ij} = 2 \sin(2\pi ij/(2\dim + 1))/\sqrt{2\dim +1}$. 
\end{description}
\fref{univplot} shows the results of this experiment. Observe that each of the above matrices has the property that $\mJ^2 = \mI_{\dim}$. Hence $\xi_{\dim}$, the spectral distribution of $\mJ$ is supported on the set $\{-1,+1\}$. Furthermore since $\Tr(\mJ)/\dim \explain{P}{\rightarrow} 0$ for each of these matrices, $\xi_{\dim} \rightarrow \unif{\{\pm 1\}}$ in the sense of \eqref{eq:weak-convergence-with-support} for all of the above ensembles. While \propref{tap-se-fan} only applies to the Random Orthogonal Ensemble, \fref{univplot} shows that the state evolution accurately describes the dynamics for Signed Hadamard and the Signed Sine ensembles, even though these matrices are significantly less random. Similar empirical observations have been made in previous works \citep{marinari1994replica,ccakmak2019memory}. \corref{TAP-AMP} provides a theoretical explanation for this empirical phenomenon. 
\begin{figure}[t]
        \centering
         \includegraphics[width=\textwidth,trim={1.9cm 0 1.6cm 0},clip]{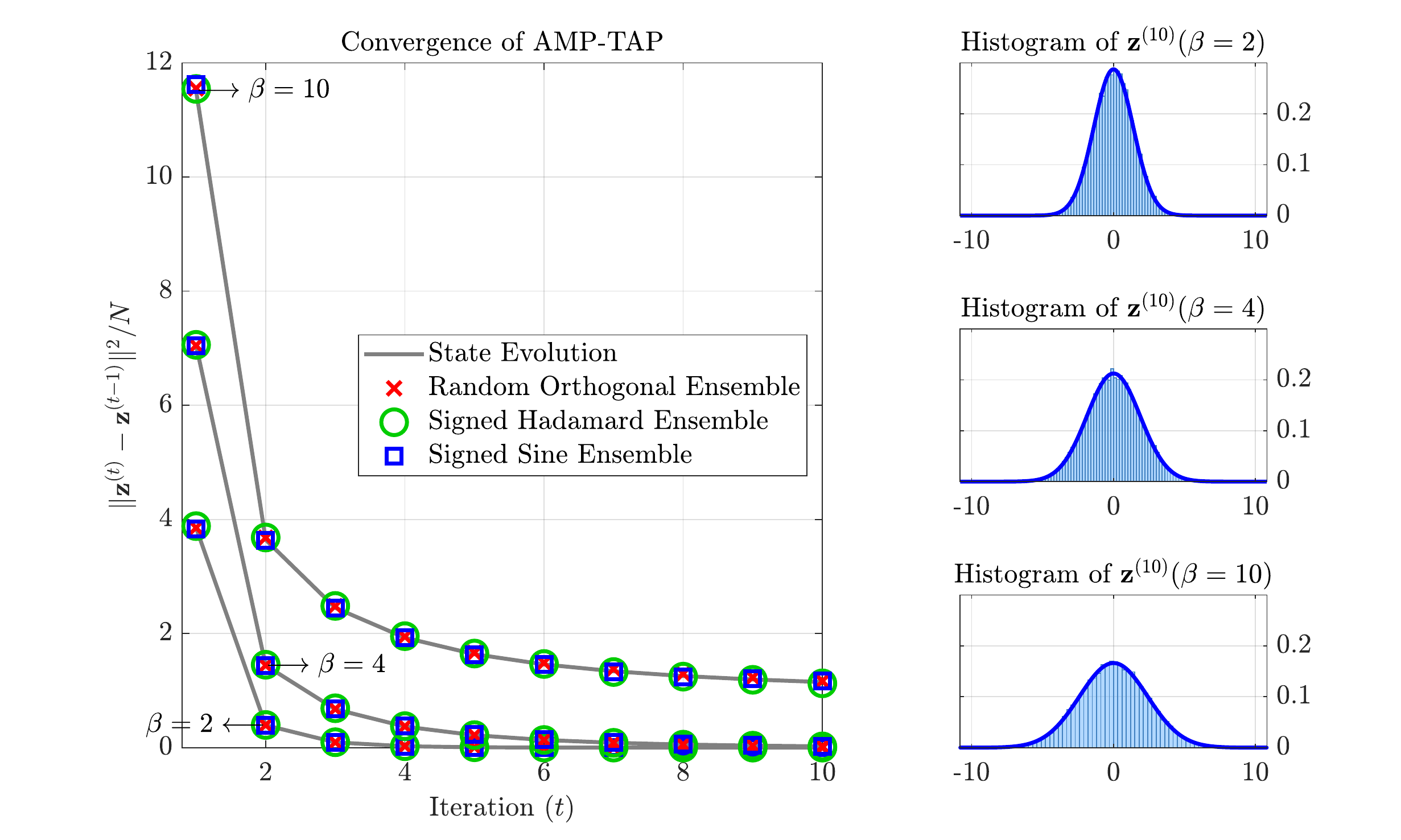}
         \caption{Left: Dynamics of $\|\iter{\vz}{t} - \iter{\vz}{t-1}\|^2/\dim$ for the \eqref{eq:opper-iteration} iteration for the Random Orthogonal Ensemble ($\times$), Signed Hadamard Ensemble ($\circ$) and Signed Sine Ensemble ($\square$) compared with predictions from the State Evolution (gray lines) at 3 inverse temperatures $\beta \in \{2, 4, 10\}$ and external field strength $\theta = 2$. Right: Histogram of the empirical distribution of the entries of $\iter{\vz}{10}$ for the Signed Sine Ensemble compared with the density function of $\gauss{0}{\sigma_{10}^2}$ (dark blue curve) for $\theta = 2$ and at 3 inverse temperatures $\beta \in \{2, 4, 10\}$.}
         \label{fig:univplot}
\end{figure}

\new{\subsection{Ising Models with Deterministic Couplings and Random External Fields} \corref{TAP-AMP} describes the dynamics of the AMP algorithm of \citet{ccakmak2019memory} in \eqref{eq:opper-iteration} for Ising models with coupling matrices constructed with very little randomness (such as the signed sine model in \eqref{eq:semirandom-sine-model}). However, it appears to fall short of describing the dynamics of \eqref{eq:opper-iteration} for Ising model with fully deterministic coupling matrices, such as the sine model considered by \citet{marinari1994replica} and \citet{parisi1995mean}, where the coupling matrix is given by:
\begin{align} \label{eq:sine-model}
    J_{ij} =\frac{2}{\sqrt{2\dim+1}} \cdot \sin\left( \frac{2\pi i j}{2N + 1} \frac{}{} \right), \; i \in \; [\dim], \; j \in \; [\dim].
\end{align}
For such deterministic coupling matrices, \emph{if the external field is not present or is random}, \corref{TAP-AMP} can still be used to analyze the dynamics of the AMP algorithm of \citet{ccakmak2019memory} via a change-of-variables argument. In order to see this, consider the mean-field Ising spin glass model with the following Gibbs measure on the discrete hypercube $\{\pm 1\}^\dim$: 
\begin{subequations}\label{eq:ising-model-general}
\begin{align}
    \gibbs(\vsigma| \mJ, \vh) &\explain{def}{=} \frac{1}{Z(\mJ, \vh)} \exp\left( \frac{\temp}{2} \cdot  \vsigma\tran \mJ \vsigma  + \field \cdot \ip{\vh}{\vsigma}  \right), \\
    Z(\mJ, \vh) &\explain{def}{=}  \sum_{\vsigma \in \{\pm1\}^\dim} \exp\left( \frac{\temp}{2} \cdot  \vsigma\tran \mJ \vsigma  + \field \cdot \ip{\vh}{\vsigma}  \right).
\end{align}
\end{subequations}
In the above display, $\mJ$ is a symmetric coupling matrix, the vector $\vh \in \{\pm 1\}^\dim$ is the external field, the parameter $\beta \geq 0$ is the inverse temperature, and the parameter $\theta \geq 0$ controls the field strength. Observe that the above definitions generalize those in \eqref{eq:ising-model}, which correspond to the special case where the external field is given by $\vh = \fieldvec$. Consider situation when the external field is random:
\begin{align*}
\vh & \sim \unif{\{\pm 1\}^\dim},
\end{align*}
but the coupling matrix is deterministic (for instance, the sine model of \eqref{eq:sine-model}). In this situation, the AMP algorithm of \citet{ccakmak2019memory} takes the form:
\begin{align}\label{eq:opper-iteration-general}
     \iter{\vz}{t+1}(\mJ, \vh) &= \frac{1}{\beta - \beta q_\star} \cdot \mM(\lambda_\star)  \cdot  \left( \frac{\tanh(\theta \vh +\iter{\vz}{t}(\mJ, \vh))}{1-q_\star} - \iter{\vz}{t}(\mJ, \vh)\right).
\end{align}
In the above display, $q_\star$ is as defined in \eqref{eq:q-star-sigma-star} and the matrix $\mM(\lambda_\star)$ is the centered resolvent of $\mJ$, as defined in \eqref{eq:resolvent}. Observe that setting $\vh = \vone$ in \eqref{eq:opper-iteration-general}, one obtains the AMP algorithm introduced in \eqref{eq:opper-iteration}. Many properties of the Ising model with a deterministic coupling matrix $\mJ$  and a random external field $\vh$ can be inferred from the corresponding properties of an Ising model with random coupling matrix $\overline{\mJ} \explain{def}{=} \diag(\vh) \mJ \diag(\vh)$ and a deterministic external field $\vone$. Indeed, introducing a change of variables $\vsigma \leftrightarrow \diag(\vh) \vsigma$ in the summation of \eqref{eq:ising-model-general} gives:
\begin{align*}
    Z(\mJ, \vh) & = Z(\diag(\vh) \cdot \mJ \cdot \diag(\vh), \vone) = Z(\overline{\mJ}, \vone)
\end{align*}
Similarly, using the fact that $\tanh(\cdot)$ is odd in \eqref{eq:opper-iteration-general} one obtains:
\begin{align}\label{eq:guage-transform}
    \iter{\vz}{t}(\mJ, \vh) & = \diag(\vh) \cdot  \iter{\vz}{t}(\diag(\vh)\mJ\diag(\vh), \vone) = \diag(\vh) \cdot  \iter{\vz}{t}(\overline{\mJ}, \vone).
\end{align}
In particular, for any \emph{even} test function $\phi: \R^{T+1} \mapsto \R$ which satisfies:
\begin{align}\label{eq:test-symmetry}
    \phi(\vz ; h) & =  \phi(-\vz ; -h) \quad \forall \; \vz\in \R^T, \;  h \in \;  \R,
\end{align}
and 
\begin{align*}
    |\phi(\vz ; h) - \phi(\vz^\prime ; h)| & \leq L \cdot \|\vz - \vz^\prime\| \cdot (1 + \|\vz\| + \|\vz^\prime\|) \quad \forall \; \vz, \vz^\prime \; \in \; \R^T, \; h \; \in \; \R
\end{align*}
for some finite constant $L$, we have:
\begin{align} \label{eq:symmetric-observables}
    \frac{1}{\dim} \sum_{i=1}^\dim \phi(\iter{z}{1}_i(\mJ, \vh), \dotsc, \iter{z}{T}_i(\mJ, \vh) ; h_i) & \explain{\eqref{eq:guage-transform}}{=}  \frac{1}{\dim} \sum_{i=1}^\dim \phi(h_i\iter{z}{1}_i(\overline{\mJ}, \vone), \dotsc, h_i\iter{z}{T}_i(\overline{\mJ}, \vone) ; h_i) \\
    & \explain{\eqref{eq:test-symmetry}}{=} \frac{1}{\dim} \sum_{i=1}^\dim \phi(\iter{z}{1}_i(\overline{\mJ}, \vone), \dotsc, \iter{z}{T}_i(\overline{\mJ}, \vone) ; 1) \\
    & \explain{P}{\rightarrow} \E[\phi(Z_1, \dotsc, Z_T; 1)].
\end{align}
In the above display, the convergence in the last step follows by appealing to \corref{TAP-AMP}. Lastly, we note that many natural quantities of interest can be computed using test functions that satisfy the symmetry requirement in \eqref{eq:guage-transform}. Examples include:
\begin{itemize}
    \item Taking $\phi(z_1, \dotsc, z_T; h) = h z_t$ in \eqref{eq:symmetric-observables} characterizes the limiting value of the overlap $\ip{\vh}{\iter{\vz}{t}}/\dim$ between the AMP iterate $\iter{\vz}{t}$ and the external field vector $\vh$. 
    \item Taking $\phi(z_1, \dotsc, z_T; h) = (z_t - z_{s})^2$ in \eqref{eq:symmetric-observables} characterizes the limiting value of  $\|\iter{\vz}{t} - \iter{\vz}{s}\|^2/\dim$, which can be used to detect the convergence/non-convergence of the AMP algorithm in \eqref{eq:opper-iteration-general}.
\end{itemize}}

\section{Related Work}
\thref{SE-final} can be thought of as an instance of the following general universality  principle.
\begin{equation}
  \tag{\#}\label{principle}
  \parbox{\dimexpr\linewidth-4em}{%
    \begin{center}
    \emph{Properties of a high-dimensional system driven by a generic random matrix $\mM$ can be accurately predicted by modeling $\mM$ as a rotationally invariant matrix with a matching spectrum.}
    \end{center}
  }
\end{equation}
Indeed, the characterization of the dynamics obtained for \eqref{eq:memory-free-amp} in \thref{SE-alt} is precisely the characterization one would ``guess'' by modeling $\mM$ as a rotationally invariant matrix and appealing to the results of \citet{fan2020approximate}. Hence, \thref{SE-alt}  can be interpreted as a formalization of the universality principle $\eqref{principle}$ in the context of the iterative algorithm \eqref{eq:memory-free-amp}. Many other instances of this universality principle have been observed and studied, which we will discuss next. 

\paragraph{Gaussian Universality.} The simplest instance of the universality principle \eqref{principle} is observed when the matrix $\mM$ is a Wigner matrix with i.i.d. entries. In this situation, $\mM$ often behaves like a Wigner matrix with i.i.d. \emph{Gaussian} entries, also known as the Gaussian Orthogonal Ensemble. Since this ensemble is rotationally invariant and has the same limiting spectral measure as Wigner matrices (the semi-circle distribution), Gaussian universality can be thought of as a special case of the general universality principle \eqref{principle}. Gaussian universality has been studied from a mathematical perspective in the context of spin-glass models \citep{carmona2006universality,chatterjee2005simple}, empirical risk minimization in statistical learning \citep{korada2011applications,panahi2017universal}, and the dynamics of AMP algorithms \citep{bayati2015universality,chen2021universality} and general first-order methods \citep{celentano2021high}. More recently, Gaussian universality has been studied in the case when matrix $\mM$ has independent rows, with possible correlations within each row. In this situation, $\mM$ often behaves like a Gaussian matrix with independent rows and matching row covariance. Examples of works that prove Gaussian universality in this situation include the work of \citet{hastie2019surprises} and \citet{mei2022generalization} for ridge regression, the work of the second author with Hong Hu \citep{hu2020universality} for empirical risk minimization with convex loss functions, and the work of \citet{montanari2022universality} for general loss functions.

\paragraph{}However, the universality principle \eqref{principle} appears to extend far beyond examples that can be explained by previously mentioned results on Gaussian universality. Indeed, it has been observed empirically that the universality principle \eqref{principle} seems to be valid even if the matrix $\mM$ has very limited randomness. We discuss some of these empirical observations next. 

\paragraph{Empirical Observations of Universality.} One instance of the universality principle \eqref{principle} appears in the context of spin glasses, where \citet{marinari1994replica} used numerical simulations and high-temperature expansions to demonstrate that many thermodynamic properties of the Sine Model, a mean-field Ising model with a \emph{deterministic} coupling matrix are nearly identical to the corresponding properties of the Random Orthogonal Model. The coupling matrices of the two models share the same limiting spectral measure. Another instance of the universality principle \eqref{principle} appears in the field of compressed sensing, where the goal is to reconstruct an unknown sparse signal vector from $n \leq \dim$ linear measurements of the form, specified using a $n \times \dim$ sensing matrix $\mA$. In this context, it has been observed that the performance of many estimators for this problem is nearly identical when $\mA : = \mA_{\mathrm{PDFT}}$ where $\mA_{\mathrm{PDFT}}$ denotes the matrix obtained by randomly sub-sampling $n$ rows of $\dim \times \dim$ DFT matrix and when $\mA := \mA_{\mathrm{Haar}}$ where $\mA_{\mathrm{Haar}}$ denotes the matrix obtained by sub-sampling $n$ rows of $\dim \times \dim$ uniformly random (or Haar distributed) orthogonal matrix. Observations of this type were first reported by \citet{donoho2009observed} and then subsequently in other works \citep{monajemi2013deterministic,abbara2020universality}. This empirical phenomenon is remarkably robust and is observed even in the presence of measurement noise \citep{oymak2014case} and for non-linear inference problems beyond compressed sensing like phase retrieval \citep{ma2021spectral,maillard2020phase}. The performance of many estimators for compressed sensing depends only on the Gram matrix $\mM = \mA \tran \mA$. Since $\mA_{\mathrm{Haar}} \tran \mA_{\mathrm{Haar}}$ is a rotationally invariant matrix with the same spectrum as $\mA_{\mathrm{PDFT}} \tran \mA_{\mathrm{PDFT}}$, this empirical observation can be viewed as an instance of the universality principle \eqref{principle}. 
\paragraph{} Though the universality principle \eqref{principle} has been observed in many different contexts, its mathematical understanding is limited. In what follows we discuss works that study this universality principle from a mathematical viewpoint in specific contexts.

\paragraph{A result from compressed sensing.} In the context of certain linear programming-based estimators for noiseless compressed sensing, \citet{donoho2010counting} provided a proof for the observed universality using results from the theory of random polytopes \citep{wendel1962problem,cover1965geometrical,winder1966partitions}. 

\paragraph{Results from free probability.} A different line of work \citep{tulino2010capacity,farrell2011limiting,anderson2014asymptotically} from free probability can also be viewed as a formalization of \eqref{principle}. In this area, a well-known result of \citet{voiculescu1991limit,voiculescu1998strengthened} shows that given two deterministic matrices $\mA, \mB$, the matrix $\mA$ is asymptotically freely independent from the rotationally invariant matrix $\mU \mB \mU \tran$ constructed using a uniformly random orthogonal matrix  $\mU \sim \unif{\ortho(\dim)}$. As a consequence, the limiting spectral measure of $\mA + \mU \mB \mU \tran$ (and more generally of arbitrary matrix polynomials in $\mA, \mU \mB \mU \tran$) can be derived from the limiting spectral measure of $\mA, \mB$ using an operation called the free additive convolution. \citet{tulino2010capacity} obtained a suprising extension of the results of \citet{voiculescu1991limit,voiculescu1998strengthened} by showing that in the situation that $\mA, \mB$ are independent diagonal matrices with i.i.d. diagonal entries, the matrix $\mA$ is also asymptotically free of the significantly less random matrix $\mF \mB \mF \tran$ where $\mF$ is the deterministic $\dim \times \dim$ Fourier matrix. As a consequence, the limiting spectral distribution $\mA + \mM$ is identical in the situation when $\mM = \mF \mB \mF \tran$ and in the situation when $\mM = \mU \mB \mU \tran$, a rotationally invariant matrix with the same spectrum as $\mF \mB \mF \tran$. Hence, the result of \citet{tulino2010capacity} can also be viewed as an instance of the universality principle \eqref{principle}. Subsequently, \citet{anderson2014asymptotically} have obtained a far-reaching generalization of the results of \citet{tulino2010capacity} by showing that conjugation by any \emph{delocalized} orthogonal matrix with certain \emph{sign and permutation symmetries} is sufficient to induce asymptotic freeness.  

\paragraph{Linearized Approximate Message Passing.} In recent joint work with M. Bakhshizadeh \citep{dudeja2020universality}, the first author formalized this universality principle  \eqref{principle} for the dynamics of \emph{linearized} AMP algorithms in the context of the phase retrieval problem. Specifically, the authors of \citep{dudeja2020universality} establish that the dynamics of linearized AMP algorithms for phase retrieval are identical if the sensing matrix $\mA : = \mA_{\mathrm{PHWT}}$ or $\mA := \mA_{\mathrm{Haar}}$, where $\mA_{\mathrm{PHWT}}$ denotes the matrix obtained by randomly sub-sampling $n$ columns of $\dim \times \dim$ Hadamard-Walsh matrix and when $\mA_{\mathrm{Haar}}$ denotes the matrix obtained by sub-sampling $n$ columns of $\dim \times \dim$ uniformly random (or Haar distributed) orthogonal matrix. Linearized AMP algorithms have the convenient property that $\iter{\vz}{t}$, which denotes the output of the algorithm at iteration $t$, is a linear transformation of the initialization $\iter{\vz}{0}$. That is, $\iter{\vz}{t} = \mR_t \iter{\vz}{0}$, for a certain random matrix $\mR_t$.  Consequently, understanding the dynamics of linearized AMP algorithms boils down to analyzing the trace and certain quadratic forms of the random matrices $\mR_t, \mR_t^2$. The authors analyze the relevant spectral properties of these matrices by leveraging and extending the proof techniques of the previously discussed work of \citet{tulino2010capacity}. The non-linear AMP algorithms studied in this paper do not have this convenient linear structure, making their analysis more challenging.

\section{Proof Outline} 
It turns out that \thref{SE-alt} and \thref{SE-final} are equivalent in the sense that any one of them implies the other. Since we will derive \thref{SE-alt} from \thref{SE-final}, we record one side of this equivalence in the following lemma.
\begin{lemma}\label{lem:equivalence} \thref{SE-final} implies \thref{SE-alt}. 
\end{lemma}
\begin{proof} See \appref{equivalence}. 
\end{proof}
As a consequence, we focus on providing a road-map to the proof of \thref{SE-final} in this section. The complete proof is deferred to Section \ref{proof_SE-final}.  
\begin{itemize}
    \item[(i)] We first establish that without loss of generality, we can normalize some of the parameters involved. In particular, we can assume $\iter{\vz}{0} \sim \gauss{0}{\mI_\dim}$, i.e., $\sigma_0^2=1$. In addition, we can assume that $\mM$ is semi-random with $\sigma_{\psi}^2=1$ (\defref{matrix-ensemble-relaxed}) and $\mathbb{E}[f_t^2(Z)]=1$ for $Z \sim \mathcal{N}(0,1)$ for all $t \geq 1$. We collect these assertions formally in \lemref{rescaling}.   
    
    \item[(ii)] Armed with this normalization, we turn to the main comparison argument in our universality proof: consider any two iteration sequences 
    \begin{align*}
    \iter{\vz}{t+1} = \mM \nonlin_{t+1}(\iter{\vz}{t}), \\
    \iter{\vw}{t+1} = \mQ \nonlin_{t+1}(\iter{\vw}{t}),
\end{align*}
started at $\iter{\vz}{0}= \iter{\vw}{0} \sim \gauss{\vzero}{\mI_\dim}$. We show that if $\mM$ and $\mQ$ both satisfy \defref{matrix-ensemble-relaxed}, for suitably ``nice" test functions $h:\mathbb{R}^{t+1}\to \mathbb{R}$, 
\begin{align}
    \frac{1}{\dim} \sum_{i=1}^\dim h(\iter{z}{0}_i, \iter{z}{1}_i, \dotsc, \iter{z}{T}_i) - \frac{1}{\dim} \sum_{i=1}^\dim h(\iter{w}{0}_i, \iter{w}{1}_i, \dotsc, \iter{w}{T}_i) \explain{P}{\rightarrow} 0. \label{eq:comparison} 
\end{align}
Indeed, this establishes that the empirical distribution of the AMP iterates is identical for \emph{any} matrix ensemble that satisfies \defref{matrix-ensemble-relaxed}.
    \item[(iii)] Finally, we choose $ \mQ = \mU \mLambda \mU\tran$, where $\mU$ is an $N \times N$ Haar matrix,  and $\mLambda = \diag(\lambda_{1:\dim})$ is a diagonal matrix with $\lambda_{1:\dim} \explain{i.i.d.}{\sim} \unif{\{\pm 1\}}$. This ensemble is semi-random in the sense of \defref{matrix-ensemble-relaxed}, and is an ``integrable" member in this class. In particular, this ensemble is \emph{orthogonally invariant}, i.e., $\mQ \stackrel{d}{=} \mat{O} \mQ \mat{O}\tran$ for any orthogonal matrix $\mat{O}$, and the empirical distribution for this iteration can be directly derived from 
    \citet{fan2020approximate} and \citet{fan2021replica}. These results show that 
\begin{align}
    \frac{1}{\dim} \sum_{i=1}^\dim h(\iter{w}{0}_i, \iter{w}{1}_i, \dotsc, \iter{w}{T}_i) \explain{P}{\rightarrow} \E h(Z_0, Z_1, \dotsc, Z_T), \label{eq:h_emp} 
\end{align}
where the law of the random vector $(Z_0, Z_1, \dotsc Z_T)$ is as defined in \thref{SE-final}. 
\end{itemize}

\noindent
Theorem~\ref{thm:SE-final} follows upon combining \eqref{eq:comparison} and \eqref{eq:h_emp}. Our main contribution in this paper is the comparison result \eqref{eq:comparison}. We turn to some key ideas underlying its proof. To this end, fix $T \geq 1$. We first approximate the Lipschitz non-linearities $\{f_t: 1\leq t \leq T\}$ and the test function $h$ by an appropriate sequence of polynomials $f_t^{(k)}: \mathbb{R}\to \mathbb{R}$ and $h^{(k)}:\mathbb{R}^{t+1} \to \mathbb{R}$. In addition, the matrix $\mM$ is perturbed slightly to construct $\hat{\mM}$ such that $\hat{\mM} = \mS \hat{\mat{\Psi}} \mS$, and $(\hat{\mPsi} \hat{\mPsi}\tran)_{11}= (\hat{\mPsi} \hat{\mPsi}\tran)_{22} = \dotsb = (\hat{\mPsi} \hat{\mPsi}\tran)_{\dim \dim} = 1$. The polynomial approximations $\{f_t^{(k)}:1\leq t \leq T\}, h^{(k)}$ and the perturbed matrix $\hat{\mM}$ are chosen so that the proxy iteration 
\begin{align*}
     \iter{\hat{\vz}}{t+1;k} = \hat{\mM} \iter{\nonlin}{k}_{t+1}(\iter{\hat{\vz}}{t;k}),
\end{align*}
initialized at $ \iter{\hat{\vz}}{0;k} = \iter{\vz}{0}$ accurately tracks the output of the original iteration. Formally, we establish that  
\begin{align}
    \lim_{k \rightarrow \infty}  \limsup_{\dim \rightarrow \infty} \E\left[ \left| \frac{1}{\dim} \sum_{i=1}^\dim h(\iter{z}{0}_i, \iter{z}{1}_i, \dotsc, \iter{z}{T}_i) - \frac{1}{\dim} \sum_{i=1}^\dim \iter{h}{k}(\iter{\hat{z}}{0;k}_i, \iter{\hat{z}}{1;k}_i, \dotsc, \iter{\hat{z}}{T;k}_i)  \right| \right] & = 0.\label{eq:approx} 
\end{align}
We establish this approximation in Proposition~\ref{prop:approx}. 
Note that given \eqref{eq:approx}, to derive the general claim \eqref{eq:comparison}, it suffices to establish \eqref{eq:comparison} under the additional assumptions that the non-linearities $\{f_t: 1\leq t \leq T\}$, the test function $h$ are polynomials, and that $(\mPsi \mPsi\tran)_{11}= (\mPsi \mPsi\tran)_{22} = \dotsb = (\mPsi \mPsi\tran)_{\dim \dim} = 1$. Given these additional assumptions, we first establish that for any matrix $\mM$ satisfying \defref{matrix-ensemble-relaxed} and $(\mPsi \mPsi\tran)_{11}= (\mPsi \mPsi\tran)_{22} = \dotsb = (\mPsi \mPsi\tran)_{\dim \dim} = 1$, 
\begin{align*}
    \Var\left[ \frac{1}{N} \sum_{i=1}^\dim h(\iter{z}{0}_i, \iter{z}{1}_i, \dotsc, \iter{z}{T}_i) \right] \to 0.  
\end{align*}
The variance bound is derived by an application of the Efron-Stein inequality \citep{boucheron2013concentration}, and is established in \thref{concentration}. In turn, this implies
\begin{align*}
    \frac{1}{N} \sum_{i=1}^\dim h(\iter{z}{0}_i, \iter{z}{1}_i, \dotsc, \iter{z}{T}_i) - \mathbb{E}\Big[ \frac{1}{N} \sum_{i=1}^\dim h(\iter{z}{0}_i, \iter{z}{1}_i, \dotsc, \iter{z}{T}_i)\Big] &\explain{P}{\rightarrow} 0, \\  \frac{1}{N} \sum_{i=1}^\dim h(\iter{w}{0}_i, \iter{w}{1}_i, \dotsc, \iter{w}{T}_i) - \mathbb{E}\Big[ \frac{1}{N} \sum_{i=1}^\dim h(\iter{w}{0}_i, \iter{w}{1}_i, \dotsc, \iter{w}{T}_i)\Big] &\explain{P}{\rightarrow} 0.
\end{align*}
At this point, to complete the proof of \eqref{eq:comparison}, it suffices to establish that the limit of the expected empirical average in the display above is equal for any two matrix ensembles $\mM$ and $\mQ$ satisfying \defref{matrix-ensemble-relaxed} and $(\mPsi \mPsi\tran)_{11}= (\mPsi \mPsi\tran)_{22} = \dotsb = (\mPsi \mPsi\tran)_{\dim \dim} = 1$:
\begin{align}
    \lim_{\dim \rightarrow \infty} \E\left[\frac{1}{\dim} \sum_{i=1}^\dim h(\iter{z}{0}_i, \iter{z}{1}_i, \dotsc, \iter{z}{T}_i) \right] & = \lim_{\dim \rightarrow \infty} \E\left[ \frac{1}{\dim} \sum_{i=1}^\dim h(\iter{w}{0}_i, \iter{w}{1}_i, \dotsc, \iter{w}{T}_i) \right]. \label{eq:comparison-expectation} 
\end{align}This is accomplished by the following theorem, which is the main technical contribution of the paper. To formally state this result, we first record the normalization assumption mentioned above for ease of reference in the subsequent discussion. 

\begin{assumption}[Normalization] \label{assump:normalization} The initialization $\iter{\vz}{0}$, the matrix ensemble $\mM = \mS \mPsi \mS$ and the non-linearities $\nonlin_t$ satisfy:
\begin{enumerate}
    \item $\iter{\vz}{0} \sim \gauss{\vzero}{\mI_\dim}$ i.e. $\sigma_0^2 = 1$.
    \item  $(\mPsi \mPsi\tran)_{11}= (\mPsi \mPsi\tran)_{22} = \dotsb = (\mPsi \mPsi\tran)_{\dim \dim} = 1$. In particular, $\sigpsi = 1$.
    \item $\E \nonlin^2_t(Z) = 1$ for $Z \sim \gauss{0}{1}$.
\end{enumerate}
\end{assumption}
To establish \eqref{eq:comparison-expectation} for a general polynomial polynomial $h$, it suffices to prove \eqref{eq:comparison-expectation} for any basis for the collection of $(T+1)$-variable polynomials. We will establish \eqref{eq:comparison-expectation} for polynomials $h$ that can be expressed as a product of the univariate Hermite polynomials $\{H_{k}: k \geq 0\}$ (\citet{o2014analysis}). The Hermite polynomials form a basis for all univariate polynomials, and thus any $(T+1)$-variable polynomial $h$ can be expressed as a linear combination of these terms. We now present a formal statement.
\begin{theorem}\label{thm:SE-normalized} Fix non-negative integers $T \in \W$, $D \in \W$ and $k_0, k_1, \dotsc, k_T \in \W$ and functions $\nonlin_1, \nonlin_2, \dotsc, \nonlin_T : \R \rightarrow \R$. There is a constant $c$ that depends only on $T, k_{0:T}, \nonlin_{1:T}$ such that if,
\begin{enumerate}
    \item $\iter{\vz}{0}$ satisfies \assumpref{initialization} and \assumpref{normalization},
    \item The matrix ensemble $\mM = \mS \mPsi\mS$ is semi-random in the sense of \defref{matrix-ensemble-relaxed} and satisfies \assumpref{normalization},
    \item  The non-linearities $\nonlin_{1:T}$ satisfy \assumpref{divergence-free}, \assumpref{normalization} and are polynomials of degree at most $\degree$,
\end{enumerate}
then, the iterates:
\begin{align*}
    \iter{\vz}{t+1} = \mM \nonlin_{t+1}(\iter{\vz}{t}),
\end{align*}
initialized at $\iter{\vz}{0}$ satisfy:
\begin{align}
   \lim_{\dim \rightarrow \infty} \E\left[ \frac{1}{\dim} \sum_{i=1}^\dim \hermite{k_0}(\iter{z}{0}_i) \cdot \hermite{k_1}(\iter{z}{1}_i) \cdot \dotsb \cdot \hermite{k_T}(\iter{z}{T}_i)  \right] & = c. \label{eq:multivariate-result}
\end{align}
{
Furthermore, for each $t \in [T]$
\begin{align} \label{eq:univariate-result}
    \lim_{\dim \rightarrow \infty} \E\left[ \frac{1}{\dim} \sum_{i=1}^\dim \hermite{k_t}(\iter{z}{t}_i)  \right] & = \begin{cases} 1 &: k_t = 0 \\ 0 &: k_t \geq 1 \end{cases}.
\end{align}}
\end{theorem}
Note that \eqref{eq:comparison-expectation} follows directly from Theorem \ref{thm:SE-normalized} as the constant $c$ is independent of the matrix  ensemble $\mathbf{M}$ (as long as the assumptions of Theorem \ref{thm:SE-normalized} are satisfied). The reader might naturally wonder why we work with the Hermite polynomials instead of the usual monomials. This choice provides us a simple way to exploit the divergence-free assumption on the non-linearities $\nonlin_{1:T}$ (\assumpref{divergence-free}). Specifically, for any fixed $t\geq 0$, consider the expansion of the non-linearity $\nonlin_{t}$ in the \emph{orthonormal basis} of Hermite polynomials:
\begin{align*}
    \nonlin_{t}(z) & = \sum_{i=0}^D \vartheta_i \cdot \hermite{i}(z), \quad \vartheta_i \explain{def}{=} \E[\hermite{i}(Z) \nonlin_t(Z)], \quad Z \sim \gauss{0}{1}.
\end{align*}
Under the assumptions of \thref{SE-normalized}, $\iter{\vz}{t} \explain{\pw}{\longrightarrow} \gauss{0}{1}$. Consequently, the divergence-free assumption is equivalent to the condition that the first Hermite coefficient $\vartheta_1 \explain{def}{=} \E[Z \nonlin_t(Z)]$ vanishes, which is easy to exploit in our combinatorial analysis. Another convenient by-product of using the Hermite basis is that it leads to a simple formula for the limit value in \eqref{eq:univariate-result} and streamlines our analysis. In random matrix theory, \citet{sodin2014several} has similarly observed that using orthogonal polynomials for the correct limiting distribution ensures that the resulting combinatorial analysis is better conditioned while implementing the moment method.  
\subsection{Proof ideas for \thref{SE-normalized}}
We start with a warm-up example to illustrate some of the ideas involved in the general proof. We study the empirical distribution of the single iterate $\iter{\vz}{2}$ using the test function $h(x)=x$ (the Hermite polynomial of degree $1$). Consider the simple situation where the non-linearities $f_1(x)= f_2(x) = x^2/\sqrt{3}$. Note that $f_1$ and $f_2$ are divergence free (cf. \assumpref{divergence-free}) since they are even functions. Furthermore, for $Z\sim \gauss{0}{1}$,  $\E[f_1^2(Z)] = \E[f_2^2(Z)] =1$. Our goal is to establish that 
\begin{align}\label{eq:simple-example}
   \lim_{\dim \rightarrow \infty} \E\left[ \frac{1}{N} \sum_{i=1}^{N} z_i^{(2)} \right] & = 0. 
\end{align}
\paragraph{Step 1: Unrolling.} Our first step is to ``unroll" the iterations and express the LHS of \eqref{eq:simple-example} as a polynomial in the initialization $\iter{\vz}{0}$. Specifically, using that $\iter{\vz}{2} = \mM \nonlin_2(\iter{\vz}{1})$, we have
\begin{align}
    \frac{1}{N} \sum_{i=1}^{N} z_i^{(2)} = \frac{1}{N} \sum_{i=1}^{N} \sum_{j=1}^{N} m_{ij} \frac{(z_i^{(1)})^2}{\sqrt{3}}. \nonumber 
\end{align}
The RHS of the above display is a polynomial in $\iter{\vz}{1}$. In order to obtain a polynomial of $\iter{\vz}{0}$, we further unroll $\iter{\vz}{1}$ using the fact that $\iter{\vz}{1} = \mM f_1(\iter{\vz}{0})$:
\begin{align}
    \frac{1}{N} \sum_{i=1}^{N} z_i^{(2)} & = \frac{1}{\sqrt{3}N} \sum_{i,j=1}^{N} m_{ij} \cdot \left( \sum_{j_1=1}^{N} m_{ij_1} \frac{(z_{j_1}^{(0)})^2}{\sqrt{3}} \right)^2 \nonumber \\
    &= \frac{1}{3\sqrt{3}N} \sum_{i,j,j_1,j_2=1}^{N} m_{ij} m_{i j_1} m_{i j_2} (z_{j_1}^{(0)})^2 (z_{j_2}^{(0)})^2. \label{eq:after-unrolling}
\end{align}
\paragraph{Step 2: Computing Expectations.} Next, we compute the expectation of  \eqref{eq:after-unrolling} with respect to the randomness in the initialization $\iter{\vz}{0} \sim \gauss{\vzero}{\mI_\dim}$ and the random signs $\mS = \diag(s_{1:\dim})$ used to generate $\mM = \mS \mPsi \mS$. Evaluating the expectation with respect to the $\iter{\vz}{0}$ variables (conditioned on $\mS$) and  using the fact that $\E Z^4 = 3, \E Z^2 =1$ for $Z \sim \gauss{0}{1}$, we obtain, 
\begin{align}
  \mathbb{E}\left[\frac{1}{N} \sum_{i=1}^{N} z_i^{(2)} \Bigg| \mS \right]& = \frac{1}{\sqrt{3}N} \sum_{i,j,j_1} m_{ij} m_{ij_1}^2 + \frac{1}{3 \sqrt{3} N} \sum_{i,j} \sum_{j_1 \neq j_2} m_{ij} m_{ij_1} m_{ij_2}. \nonumber \\
   &= \frac{1}{\sqrt{3}N}\sum_{i,j,j_1} s_i s_j \psi_{ij} \psi_{i,j_1}^2 + \frac{1}{3 \sqrt{3} N} \sum_{i,j} \sum_{j_1 \neq j_2}  s_i s_j s_{j_1} s_{j_2} \psi_{ij} \psi_{ij_1} \psi_{ij_2}. \nonumber
\end{align}
We now evaluate the expectation with respect to the $\mS$.  Starting with the first term, we observe that $\E[s_i s_j] = 0$ unless $i=j$, in which case $\E[s_i^2] = 1$. Similarly, for the second term, $\E[ s_i s_j s_{j_1} s_{j_2}] = 0$ unless $i=j_1$, $j=j_2$ or $i=j_2$, $j=j_1$, in which case $\E[ s_{j_1}^2 s_{j_2}^2] = 1$. Hence, 
\begin{align}
     \mathbb{E}\left[\frac{1}{N} \sum_{i=1}^{N} z_i^{(2)} \right] & = \frac{1}{\sqrt{3}N}\sum_{i,j}  \psi_{ii} \psi_{ij}^2 + \frac{2}{3 \sqrt{3} N}  \sum_{i\neq j}  \psi_{ii} \psi_{ij}^2 \explain{def}{=} \invpoly(\mPsi). \label{eq:warmup-after-E}
\end{align}
\paragraph{Step 3: Combinatorial Estimates.} Observe that at the end of Step 2, we have expressed the empirical first Hermite moment of $\iter{\vz}{2}$ as a polynomial $\invpoly(\mPsi)$. In the next step, we leverage combinatorial estimates to argue that for any $\mPsi$ that satisfies the assumptions of \thref{SE-normalized}, $\lim_{\dim \rightarrow \infty} \invpoly(\mPsi)$ exists and is universal (that is, identical for any $\mPsi$). In the specific case of \eqref{eq:warmup-after-E}, this step is particularly easy. Indeed by the triangle inequality and the fact that $\|\mPsi\|_{\infty} \lesssim \dim^{-1/2 + \epsilon}$, we have 
\begin{align}
   \left| \mathbb{E}\left[\frac{1}{N} \sum_{i=1}^{N} z_i^{(2)} \right] \right|= \frac{1}{\sqrt{3}N}\sum_{i,j}  |\psi_{ii}| |\psi_{ij}|^2 + \frac{2}{3 \sqrt{3} N}  \sum_{i\neq j}  |\psi_{ii}| |\psi_{ij}|^2 \lesssim \frac{1}{N} \cdot N^2 \cdot  N^{-3/2 + 3 \epsilon}  = N^{-1/2 + 3 \epsilon},  \label{eq:crude-bound-warmup}
\end{align}
which converges to zero for $\epsilon>0$ small enough and thus proves \eqref{eq:simple-example}. The example above illustrates \emph{some} of the key ideas involved in the proof of \thref{SE-normalized}. Unfortunately, this simple case does not capture \emph{all} the intricacies involved in the general proof. In the general case, we can still execute Step 1 and Step 2 to express the empirical mixed Hermite moments of the iterates $\iter{\vz}{0:T}$ (cf. \eqref{eq:multivariate-result}) as a polynomial in $\mPsi$. However, the resulting polynomials are more complicated and the crude bound in \eqref{eq:crude-bound-warmup} based on triangle inequality and $\|\mPsi\|_\infty \lesssim \dim^{-1/2+\epsilon}$ is no longer adequate (see Fig.~\ref{fig:nullifying-example} for a concrete counterexample).  Handling the general case requires the following two additional ideas.
\paragraph{Idea 1: Cancellations.} This idea leverages the property $\|\mPsi \mPsi \tran - \mI_{\dim} \|_\infty \lesssim \dim^{-1/2+\epsilon}$ to improve upon the crude bound in \eqref{eq:crude-bound-warmup}. In order to illustrate this technique, consider the polynomial $\invpoly_1(\mPsi)$ defined as:
    \begin{align*}
        \invpoly_1(\mPsi) \explain{def}{=} \frac{1}{\dim}\sum_{i_1 \neq i_2} \sum_{j,k,\ell}  \psi_{i_1 j} \psi_{i_1 k} \psi_{i_1 \ell} \psi_{i_2 j} \psi_{i_2 k} \psi_{i_2 \ell}.
    \end{align*}
    Observe that the naive estimate based on the triangle inequality and $\|\mPsi\|_\infty \lesssim \dim^{-1/2+\epsilon}$ yields:
    \begin{align*}
        |\invpoly_1(\mPsi)| \explain{(a)}{\leq} \frac{1}{\dim}\sum_{i_1 \neq i_2} \sum_{j,k,\ell}  |\psi_{i_1 j}| |\psi_{i_1 k}| |\psi_{i_1 \ell}|  |\psi_{i_2 j}| |\psi_{i_2 k}| |\psi_{i_2 \ell}| \lesssim \frac{1}{\dim} \cdot \dim^5 \cdot \dim^{-3 + 6\epsilon} = \dim^{1+6\epsilon}.
    \end{align*}
    Thus the above estimate fails to show the existence and universality of $\lim_{\dim \rightarrow \infty} \invpoly_1(\mPsi)$. The deficiency in the above estimate lies in the use of the triangle inequality in step (a). A better estimate is obtained by leveraging cancellations that arise from the constraint $\|\mPsi \mPsi \tran - \mI_{\dim} \|_\infty \lesssim \dim^{-1/2+\epsilon}$. Indeed we have,
    \begin{align*}
        |\invpoly_1(\mPsi)| & = \frac{1}{\dim} \left| \sum_{i_1 \neq i_2} \left( \sum_{j=1}^\dim \psi_{i_1 j} \psi_{i_2 j} \right) \cdot \left( \sum_{k=1}^\dim \psi_{i_1 k} \psi_{i_2 k} \right) \cdot \left( \sum_{\ell=1}^\dim \psi_{i_1 \ell} \psi_{i_2 \ell} \right) \right| \\
        & =  \frac{1}{\dim} \left| \sum_{i_1 \neq i_2} (\mPsi \mPsi\tran)_{i_1,i_2}^3 \right| \leq \frac{1}{\dim}  \sum_{i_1 \neq i_2} |(\mPsi \mPsi\tran)_{i_1,i_2}|^3  \lesssim \frac{1}{\dim} \cdot \dim^2 \cdot \dim^{-3/2 + 3 \epsilon} = \dim^{-1/2 + \epsilon},
    \end{align*}
    where step (b) follows from $\|\mPsi \mPsi \tran - \mI_{\dim} \|_\infty \lesssim \dim^{-1/2+\epsilon}$. Hence, this estimate shows that $\lim_{\dim \rightarrow \infty} \invpoly_1(\mPsi)$ exists and is universal for any $\mPsi$ that satisfies  $\|\mPsi \mPsi \tran - \mI_{\dim} \|_\infty \lesssim \dim^{-1/2+\epsilon}$.  
\paragraph{Idea 2: Simplifications.} Here, the idea is that some polynomials can be \emph{simplified} by leveraging the constraint $(\mPsi \mPsi\tran)_{11}= (\mPsi \mPsi\tran)_{22} = \dotsb = (\mPsi \mPsi\tran)_{\dim \dim} = 1$. By simplifying the polynomial before applying the cancellation technique described above, one extracts the maximum benefit out of the cancellations described above. This idea can be illustrated using the following polynomial:
    \begin{align*}
        \invpoly_2(\mPsi) \explain{def}{=} \frac{1}{\dim}\sum_{i_1 \neq i_2} \sum_{j,k,\ell} \sum_{m_1,m_2,m_3}  \psi_{i_1 j} \psi_{i_1 k} \psi_{i_1 \ell} \psi_{i_2 j} \psi_{i_2 k} \psi_{i_2 \ell} \psi_{j m_1}^2 \psi_{k m_2}^2 \psi_{\ell m_3}^2.   
    \end{align*}
    This polynomial can be simplified by evaluating the sum over $m_1,m_2,m_3$ as follows:
    \begin{align*}
        \invpoly_2(\mPsi) &\explain{}{=} \frac{1}{\dim}\sum_{i_1 \neq i_2} \sum_{j,k,\ell}  \psi_{i_1 j} \psi_{i_1 k} \psi_{i_1 \ell} \psi_{i_2 j} \psi_{i_2 k} \psi_{i_2 \ell} \cdot \left( \sum_{m_1}\psi_{j m_1}^2 \right) \cdot \left( \sum_{m_2} \psi_{k m_2}^2 \right) \cdot \left( \sum_{m_3} \psi_{\ell m_3}^2 \right) \\
        & = \frac{1}{\dim}\sum_{i_1 \neq i_2} \sum_{j,k,\ell}  \psi_{i_1 j} \psi_{i_1 k} \psi_{i_1 \ell} \psi_{i_2 j} \psi_{i_2 k} \psi_{i_2 \ell} \cdot (\mPsi \mPsi \tran)_{jj} \cdot (\mPsi \mPsi \tran)_{kk} \cdot (\mPsi \mPsi \tran)_{\ell \ell} \\&\explain{(c)}{=} \frac{1}{\dim}\sum_{i_1 \neq i_2} \sum_{j,k,\ell}  \psi_{i_1 j} \psi_{i_1 k} \psi_{i_1 \ell} \psi_{i_2 j} \psi_{i_2 k} \psi_{i_2 \ell} = \invpoly_1(\mPsi),
    \end{align*}
    where step (c) uses the property $(\mPsi \mPsi\tran)_{11}= (\mPsi \mPsi\tran)_{22} = \dotsb = (\mPsi \mPsi\tran)_{\dim \dim} = 1$. The simplification above reveals that $\invpoly_2(\mPsi) = \invpoly_1(\mPsi)$ and hence $\lim_{\dim \rightarrow \infty} \invpoly_2(\mPsi)$ also exists and is universal. 
 
\subsection{Outline for the remaining paper} 
In light of the discussion in the previous section, the rest of the paper is organized as follows:
\begin{enumerate}
    \item Section \ref{sec:se-normalized_proof} collects the main technical ingredients required for the proof of \thref{SE-normalized}, and completes the proof, assuming these intermediate results. Specifically:
    \begin{enumerate}
        \item \lemref{unrolling_1} and \lemref{key-formula} expand the mixed Hermite moments in terms of the initialization $\iter{\vz}{0}$ (Step 1) and evaluate the expectation to identify the terms which have a non-zero contribution in the limit (Step 2). These lemmas are proved in Section \ref{sec:unrolling_proof}. Together, these results express the expected mixed Hermite moments of the iterates as a polynomial in the matrix $\mPsi$.
        \item Proposition \ref{prop:key-estimate} provides an estimate on the magnitude of a  class of simple polynomials. This result uses the idea of cancellations described previously (Idea 1) to improve on the naive estimates based on the triangle inequality. We prove Proposition \ref{prop:key-estimate} in Section \ref{sec:proof-key-estimate}.
        \item Before appealing to the improved estimate obtained in \propref{key-estimate}, we first simplify the  polynomials obtained by the expansion of the mixed moments (Idea 2) in \propref{decomposition}.  The proof of this result appears in \sref{proof-decomposition}.
        \item \propref{universal-config-estimate} leverages the improved estimate of \propref{key-estimate} to show that the behavior of the resulting simplified polynomials is universal. This result is proved in \sref{proof-universal-config-estimate}.
    \end{enumerate}
    \item Finally, Section \ref{proof_SE-final} explains how the main results of this paper (\thref{SE-alt} and \thref{SE-final}) follow from \thref{SE-normalized}. 
\end{enumerate}

\section{Proof of \thref{SE-normalized}}
\label{sec:se-normalized_proof} 

This section is devoted to the proof of \thref{SE-normalized}. We begin by introducing the key ingredients involved in the proof of this result. We defer the proof of these intermediate results to later sections. 

\subsection{Key Results}

\subsubsection{Unrolling the AMP Iterations.} 
We begin by expressing the observable of interest:
\begin{align} \label{eq:key-observable}
     \frac{1}{\dim} \sum_{i=1}^\dim \hermite{k_0}(\iter{z}{0}_i) \cdot \hermite{k_1}(\iter{z}{1}_i) \cdot \dotsb \cdot \hermite{k_T}(\iter{z}{T}_i) 
\end{align}
as a polynomial of the initialization $\iter{\vz}{0}$. This involves ``unrolling'' the AMP iterations to express each $\iter{\vz}{t}$ as a polynomial of $\iter{\vz}{0}$. The resulting polynomial takes the form of a combinatorial sum over colorings of decorated forests, which we introduce below.  

\begin{definition}[Decorated Trees and Forests] \label{def:decorated-forests} A decorated forest $F$ is given by a tuple $(V,E,\pweight{},\qweight{})$ where:
\begin{enumerate}
    \item $V = \{1, 2, 3, \dotsc, |V|\}$ is the set of vertices.
    \item $E$ is the set of directed edges.
\end{enumerate}
The sets $(V,E)$ are such that the directed graph given by $(V,E)$ is a directed forest. We define the following notions:
\begin{enumerate}
     \item If $u \rightarrow v \in E$, we say $u$ is the parent of $v$ and $v$ is a child of $u$. Each vertex in a directed forest has at most one parent. 
    \item A vertex with no parent is called a root vertex. The set of all root vertices in the forest $F$ is denoted by $\rootset(F)$. We number the roots as $1, 2, \dotsc, $. Hence, $\rootset(F) = \{1, 2, 3, \dotsc, |\rootset(F)|\}$.
    \item A root vertex with no children is called a trivial root. The set of all trivial roots in the forest is denoted by $\rootset_0(F)$.
    \item For every vertex $u$ we define $\chrn{u}(F)$ as the number of children of $u$.
    \item A non-root vertex with no children is called a leaf. The set of all leaves is denoted by  $\leaves{F}$. 
    \item A pair of non-root vertices $u, v \in V \backslash \rootset(F)$ are siblings if they have the same parent.
\end{enumerate}
The forest is decorated with 3 functions 
\begin{align*}\height{} :& V \rightarrow \W = \{0,1, 2, 3, \dotsc \}, \\ \pweight{} :& V \backslash \rootset(F) \rightarrow \N = \{1, 2, 3, \dotsc \}, \\ \qweight{} :& V \rightarrow \W = \{0, 1, 2, 3, \dotsc \},
\end{align*}
such that:
\begin{enumerate}
    \item  The height function $\height{}$ has the following properties:
    \begin{enumerate}
        \item for any $u \rightarrow v \in E$, $\height{v} = \height{u} - 1$. 
        \item If $\height{u} = 0$, then $u$ has no children $(\chrn{u}(F)=0)$. 
        \item For every vertex $u$ with no children and $\qweight{u} \geq 1$, we have $\height{u} = 0$.
    \end{enumerate}
    \item The function $\qweight{}$ satisfies: for any vertex $u \in V \backslash (\leaves{F} \cup \rootset_0(F))$, we have,
    \begin{align}\label{eq:conservation-eq}
        \qweight{u} = \sum_{v\in V: u \rightarrow v } \pweight{v}.
    \end{align}
\end{enumerate}
If a decorated forest has exactly one root vertex, it is called a decorated tree. See \fref{decorated-tree} for an example of a small decorated tree. 
\end{definition}
\begin{figure}
\centering
\includegraphics[width=0.6\textwidth]{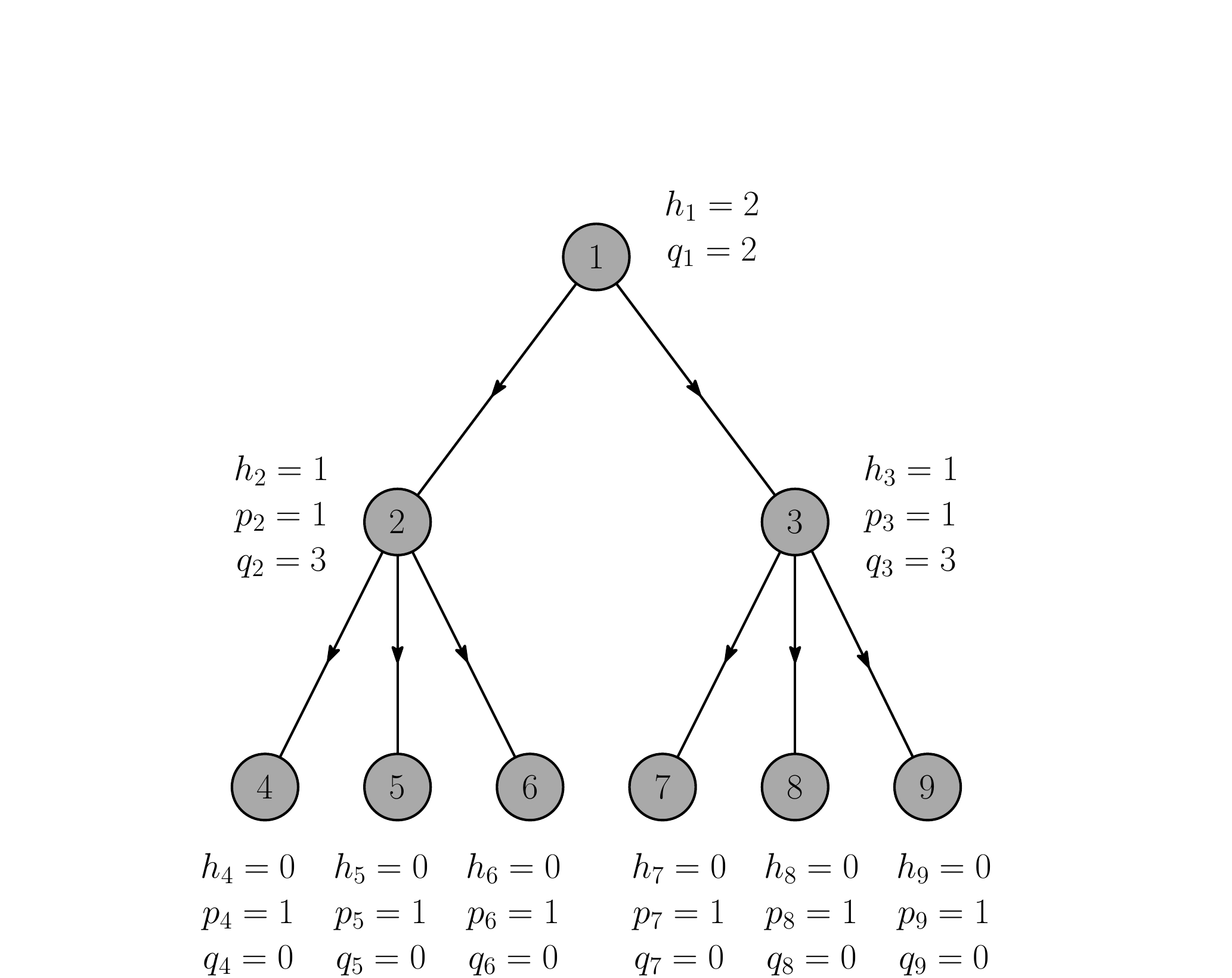}
\caption{An example of a decorated forest (\defref{decorated-forests}) with one root (decorated tree).}  
\label{fig:decorated-tree}
\end{figure}

Next, we introduce the notion of a coloring of a decorated forest. 

\begin{definition}[Coloring of Decorated Trees and Forests] \label{def:coloring}
A coloring of a decorated forest $F$ with vertex set $V$ is a map $\ell : V \rightarrow [\dim]$. The set of all colorings of a forest $F$ with vertex set $V$ is denoted by $[\dim]^{V}$.
\end{definition}

The colored decorated trees that appear in the polynomial expansion of the observable \eqref{eq:key-observable} have additional constraints, which we collect in the definition below. 

\begin{definition}[Valid Decorated Colorings of Decorated Forests] \label{def:valid-tree}  A decorated forest $F = (V, E, \height{}, \pweight{}, \qweight{})$  and a coloring $\ell : V \rightarrow [\dim]$ are valid if:
\begin{enumerate}
    \item For vertices $u, v$ that are siblings in the forest, we have $\ell_u \neq \ell_v$.
    \item $\ell_u = \ell_v$ for every $u, v \in \rootset(F)$.
\end{enumerate}
We denote valid decorated colored forests by defining the indicator function $\valid$ such that $\valid(F,\ell) = 1$ iff $(F,\ell)$ is a valid colored decorated forest and $\valid(F,\ell) = 0$ otherwise. See \fref{colored-tree} for an example of a valid coloring of the decorated tree from \fref{decorated-tree}.
\end{definition}
\begin{figure}
\centering
\includegraphics[width=0.5\textwidth]{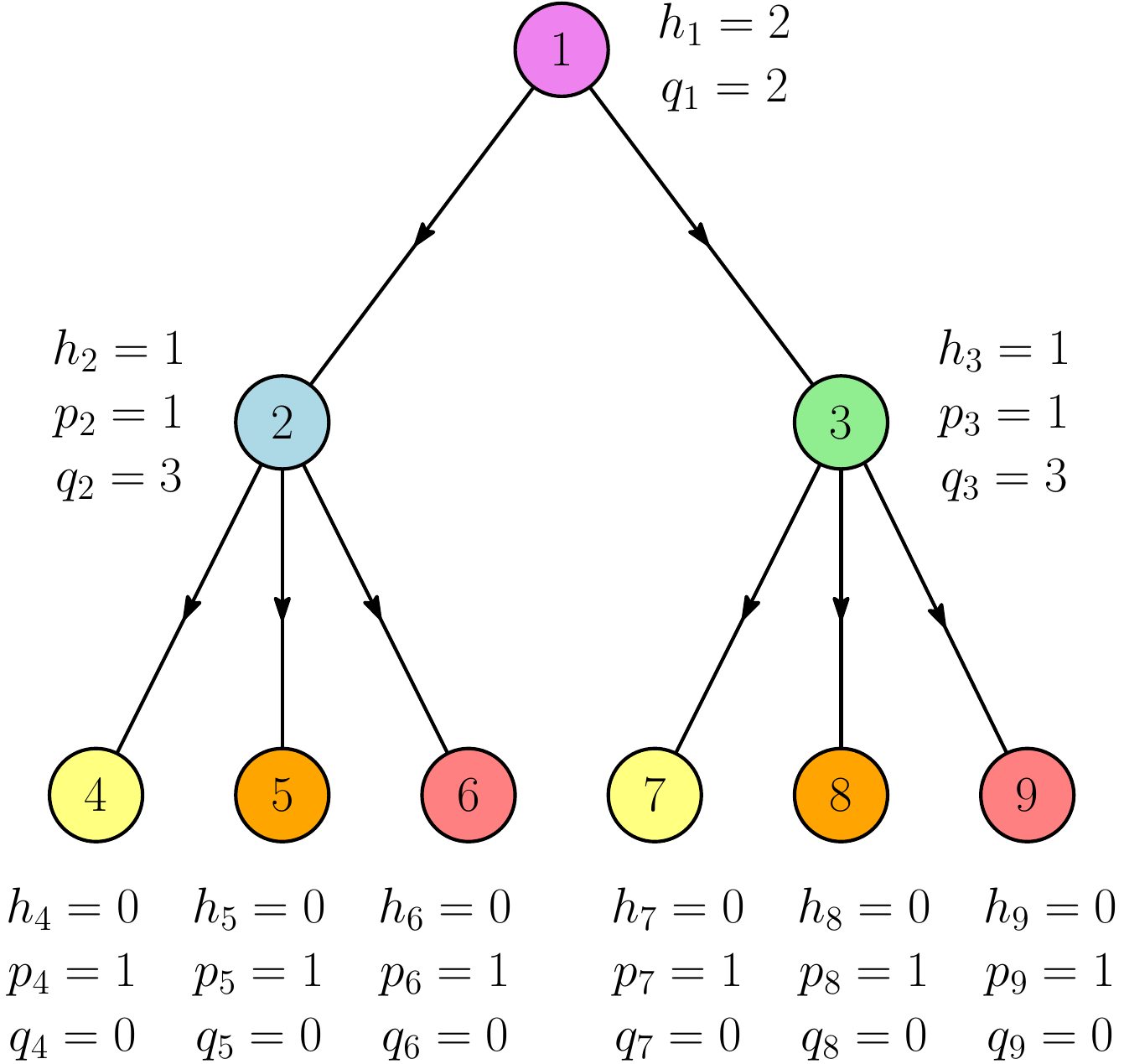}
\caption{An example of a valid coloring (\defref{valid-tree}) of the decorated tree (\defref{decorated-forests}) from \fref{decorated-tree}.}  
\label{fig:colored-tree}
\end{figure}

Armed with these  definitions, we can express the observable \eqref{eq:key-observable} as a polynomial in the initialization $\iter{\vz}{0}$. 

\begin{lemma}
\label{lem:unrolling_1} 
For any $T \in \W$ and $k_0, k_1, \dotsc, k_T \in \N_0$, we have:
\begin{subequations}\label{eq:polynomial-expansion-near-final}
\begin{align} 
     &\frac{1}{\dim} \sum_{i=1}^\dim \hermite{k_0}(\iter{z}{0}_i) \cdot \hermite{k_1}(\iter{z}{1}_i) \cdot \dotsb \cdot \hermite{k_T}(\iter{z}{T}_i)  = \nonumber\\ &  \frac{1}{\dim}\sum_{\substack{F \in \forests{T}{k_0, k_1, \dotsc, k_T}\\ F = (V,E, \height{}, \pweight{}, \qweight{})}} \sum_{\substack{\ell \in [\dim]^V}} \alpha(F) \cdot \valid(F,\ell) \cdot \left( \prod_{u \rightarrow v} m_{\ell_u \ell_v}^{\pweight{v}}\right) \cdot \left(\prod_{\substack{w \in \leaves{F} \cup \rootset_0(F)\\ \qweight{w} \geq 1}} \hermite{\qweight{w}}(\iter{z}{0}_{\ell_w}) \right). \label{eq:polynomial-expansion-near-final-main}
\end{align}
In the above display $\mM = (m_{ij})$, and for $F = (V, E, \height{}, \pweight{}, \qweight{})$ we define the coefficient:
\begin{align}
    \alpha(F) &\bydef \left( \prod_{\substack{u \in V \backslash (\leaves{F} \cup \rootset_0(F))}} \frac{\sqrt{\qweight{u}!}}{\chrn{u}(F)!} \cdot \prod_{v: u \rightarrow v} \frac{1}{\sqrt{\pweight{v}!}} \right) \cdot \left( \prod_{w \in V \backslash \rootset(F)} \fourier{\height{w}+1}{\pweight{w}}{\qweight{w}}\right), \\
        \fourier{t}{p}{q} &\explain{def}{=} \E[ \hermite{p} \circ \nonlin_t(Z) \cdot  \hermite{q}(Z)], \; Z \sim \gauss{0}{1}. \label{eq:alpha_defn} 
\end{align}
Furthermore, $\forests{T}{k_0, k_1, \dotsc, k_T}$ denotes the set of all decorated forests with $T+1$ roots, $\rootset(F) = \{1, 2, \dotsc, T+1\}$ and {$\height{1} = 0, \height{2} = 1, \dotsc \height{T+1} = T$} and $\qweight{1} = k_0, \qweight{2} = k_1, \dotsc, \qweight{T+1} = k_{T}$.
\end{subequations}
\end{lemma}
\noindent 
We defer the proof of Lemma \ref{lem:unrolling_1} to \sref{proof-unrolling-formula}. 
\subsubsection{Partitions on a Decorated Tree}
Our next step is to evaluate the expectation of the formula \eqref{eq:polynomial-expansion-near-final-main} with respect to the randomness in the initialization $\iter{\vz}{0} \sim \gauss{\vzero}{\mI_{\dim}}$ and the uniformly random signed diagonal matrix $\mS$ used to generate the matrix $\mM = \mS \mPsi \mS$. Observe that in order to evaluate the expectation of the RHS of \eqref{eq:polynomial-expansion-near-final-main}, the repetition pattern of the coloring $\ell \in [\dim]^V$ is important---for two vertices $u,v$ that have the same color $\ell_{u} = \ell_{v}$, the corresponding Gaussian and sign variables are identical $\iter{z}{0}_{\ell_u} = \iter{z}{0}_{\ell_v}, s_{\ell_u} = s_{\ell_v}$. On the other hand, for vertices $u, v$ with different colors $\ell_{u} \neq \ell_{v}$,   the corresponding Gaussian and sign variables are independent. The repetition pattern of a coloring in $\ell \in [\dim]^V$ can be encoded by a partition of the vertex set $V$. This motivates the following definitions.

\begin{definition}[Partitions and Configurations] Given a decorated forest $F=(V, E, \height{}, \pweight{}, \qweight{})$, a partition $\pi$ of the vertex set $V$ is a collection of disjoint subsets (called blocks) $\{B_1,B_2, \dotsc ,B_{s}\}$ such that
\begin{align*}
    \bigcup_{j=1}^s B_j = V, \; B_j \cap B_k = \emptyset \; \forall \; j \neq k.
\end{align*}
We define $|\pi|$ to be the number of blocks in $\pi$. For every $v \in V$, we use $\pi(v)$ to denote the unique block $j \in   [|\pi|]$ such that $v \in B_j$. The set of all partitions of $V$ is denoted by $\part{V}$. A configuration is a pair $(F,\pi)$ consisting of a decorated forest $F$ and a partition $\pi$ of its vertices. 
\end{definition}

\begin{definition}[Colorings consistent with a partition] Let $\pi$ be a partition of the vertex set $V$ of a decorated forest $F=(V, E, \height{}, \pweight{}, \qweight{})$. A coloring consistent with $\pi$ is a function $\ell: V \rightarrow [\dim]$ such that, 
\begin{align*}
    \ell_u = \ell_v & \Leftrightarrow \pi(u) = \pi(v).
\end{align*}
The set of all colorings that are consistent with a partition $\pi$ is denoted by $\colorings{}{\pi}$.
\end{definition}
Note that whether a pair $(F,\ell)$ is valid colored forest or not depends only on the $(F,\pi)$ where $\pi$ is a partition corresponding to $\ell$. Hence we introduce the following definition.

\begin{definition}[Valid Configurations] A decorated forest $F = (V, E, \height{}, \pweight{}, \qweight{})$  and a partition $\pi \in \part{V}$ form a valid configuration if:
\begin{enumerate}
    \item For vertices $u, v$ that are siblings in the forest, we have $\pi(u) \neq \pi(v)$.
    \item For every $u,v \in \rootset(F)$, $\pi(u) = \pi(v)$.
\end{enumerate}
We denote valid configurations by defining the indicator function $\valid$ such that $\valid(F,\pi) = 1$ iff $(F,\pi)$ is a valid configuration and $\valid(F,\pi) = 0$ otherwise.
\end{definition}
Observe that as a consequence of the above definitions, \eqref{eq:polynomial-expansion-near-final-main} can be rearranged to the following form:
\begin{align} \label{eq:rearrangement}
    &\frac{1}{\dim} \sum_{\ell=1}^\dim \hermite{k_0}(\iter{z}{0}_i) \cdot \hermite{k_1}(\iter{z}{1}_i) \cdot \dotsb \cdot \hermite{k_T}(\iter{z}{T}_i)  = \nonumber\\ &  \frac{1}{\dim}\sum_{\substack{F \in \forests{T}{k_0, k_1, \dotsc, k_T}\\ F = (V,E, \height{}, \pweight{}, \qweight{})}} \sum_{\pi \in \part{V}} \sum_{\substack{\ell \in \colorings{}{\pi}}} \alpha(F) \cdot \valid(F,\pi) \cdot \left( \prod_{u \rightarrow v} m_{\ell_u \ell_v}^{\pweight{v}}\right) \cdot \left(\prod_{\substack{w \in \leaves{F} \cup \rootset_0(F)\\ \qweight{w} \geq 1}} \hermite{\qweight{w}}(\iter{z}{0}_{\ell_w}) \right).
\end{align}
\subsubsection{The Expectation Formula}
The expansion in \eqref{eq:rearrangement} facilitates the evaluation of the expectation of the polynomial expansion 
 \eqref{eq:polynomial-expansion-near-final-main} obtained in \lemref{unrolling_1}. The following lemma identifies the expected value, and identifies the configurations $(F,\pi)$ which have a non-zero contribution to the expectation.

\begin{lemma}[The Expectation Formula]  \label{lem:key-formula}For any $T \in \W$ and any $k_0, k_2, k_3, \dotsc, k_T \in \W$, we have,
\begin{subequations}\label{eq:key-formula}
\begin{align} 
  &\E\left[\frac{1}{\dim} \sum_{\ell=1}^\dim \hermite{k_0}(\iter{z}{0}_i) \cdot \hermite{k_1}(\iter{z}{1}_i) \cdot \dotsb \cdot \hermite{k_T}(\iter{z}{T}_i) \right]    = \nonumber \\&\hspace{2cm} \sum_{\substack{F \in \forests{T}{k_0, k_1, \dotsc, k_T}\\ F = (V,E, \height{}, \pweight{}, \qweight{})}} \sum_{\pi \in \part{V}} \sum_{\substack{\ell \in \colorings{}{\pi}}} \alpha(F) \cdot \valid(F,\pi) \cdot \beta_1(F,\pi) \cdot \beta_2(F,\pi) \cdot \gamma(\mPsi; F,\ell).
\end{align}
For a decorated forest $F =  (V,E, \height{}, \pweight{}, \qweight{})$ and a partition $\pi = \{B_1, B_2, \dotsc, B_{\pi}\}$ recall $\alpha(F)$ from \eqref{eq:alpha_defn}. Further, we define:
\begin{align}
    \beta_1(F,\pi) &\explain{def}{=} \prod_{j=1}^{|\pi|} \E \left[ \prod_{\substack{v \in (\leaves{F} \cup \rootset_0(F)) \cap B_j\\ \qweight{v} \geq 1}} \hermite{\qweight{v}}(Z) \right], \\
    \beta_2(F,\pi) & \explain{def}{=}   \prod_{j=1}^{|\pi|} \E\left[ \left( \prod_{u \in  (B_j \cap \rootset(F))\backslash \rootset_0(F)} S^{\qweight{u}} \right)  \cdot \left( \prod_{v \in B_j \cap \leaves{F}} S^{\pweight{v}} \right) \cdot \left( \prod_{\substack{u \in B_j \backslash (\leaves{F} \cup \rootset(F))}} S^{\pweight{u} + \qweight{u}} \right)\right], \\
    \gamma(\mPsi; F,\ell) & \explain{def}{=}  \frac{1}{\dim} \prod_{\substack{e \in E \\ e = u \rightarrow v}} \psi_{\ell_u\ell_v}^{\pweight{v}},
\end{align}
where $Z \sim \gauss{0}{1}$ and $S \sim \unif{\{\pm 1\}}$. Furthermore, $\alpha(F) \cdot \valid(F,\pi) \cdot \beta_1(F,\pi) \cdot \beta_2(F,\pi) = 0$, unless $(F,\pi)$ form a relevant configuration, defined below. 
\end{subequations}
\end{lemma}

\begin{definition}[Relevant Configuration] \label{def:relevant-config}A decorated forest $F = (V, E, \height{}, \pweight{}, \qweight{})$ and a partition $\pi = \{B_1, B_2, \dotsc, B_{|\pi|}\}$ form a relevant configuration if they satisfy the following properties:
\begin{description}
    \myitem{\texttt{Root Rule}}\label{root-rule}: For any two root vertices $u, v \in \rootset(F)$, we have $\pi(u) = \pi(v)$.
    \myitem{\texttt{Sibling Rule}}\label{sibling-rule}: For vertices $u, v \in V \backslash \rootset(F)$ that are siblings in the forest, we have $\pi(u) \neq \pi(v)$.
    \myitem{\texttt{Forbidden Weights Rule}}\label{forbidden-weights-rule}: There are no vertices $u \in V \backslash \rootset(F)$ such that $\pweight{u} = 1, \qweight{u} = 1$ or $\pweight{u} = 2, \qweight{u} = 0$.
    \myitem{\texttt{Leaf Rule}}\label{leaf-rule}: There are no leaf vertices $v \in \leaves{F}$ with $q_v \geq 1$ and $|B_{\pi(v)}| = 1$.
    \myitem{\texttt{Trivial Root Rule}}\label{trivial-root-rule}: There are no trivial roots $u \in \rootset_0({F})$ with $q_u \geq 1$ and $|B_{\pi(u)}| = 1$.
    \myitem{\texttt{Parity Rule}}\label{parity-rule}: There is no block $B$ in the partition $\pi$ such that the sum:
    \begin{align*}
       \left(\sum_{(B \cap \rootset(F)) \backslash \rootset_0(F)} \qweight{u}  \right)+ \left(\sum_{u \in B \backslash ( \leaves{F} \cup \rootset(F))} \pweight{u} + \qweight{u} \right) + \left( \sum_{v \in B \cap \leaves{F}} \pweight{v} \right) 
    \end{align*}
    has odd parity.
\end{description}
Notice that the configuration shown in \fref{colored-tree} satisfies all the requirements of a relevant configuration.
\end{definition}
\noindent 
We prove Lemma~\ref{lem:key-formula} in Section \ref{sec:unrolling_proof}. 
\subsubsection{Estimates on Polynomials Associated with a Configuration}
As a consequence of \lemref{key-formula}, one can see that the critical objects of interest are the polynomials $\invpoly(\mPsi; F, \pi)$ defined as follows:
\begin{align} \label{eq:inv-poly}
    \invpoly(\mPsi; F, \pi) \explain{def}{=} \sum_{\substack{\ell \in \colorings{}{\pi}}}  \gamma(\mPsi; F,\ell) = \sum_{\substack{\ell \in \colorings{}{\pi}}} \frac{1}{\dim} \prod_{\substack{e \in E \\ e = u \rightarrow v}} \psi_{\ell_u\ell_v}^{\pweight{v}}.
\end{align}
Indeed, in light of \lemref{key-formula}, we have,
\begin{align}
\label{eq:expectation_final}
    &\E\left[\frac{1}{\dim} \sum_{\ell=1}^\dim \hermite{k_0}(\iter{z}{0}_i) \cdot \hermite{k_1}(\iter{z}{1}_i) \cdot \dotsb \cdot \hermite{k_T}(\iter{z}{T}_i) \right]    \\ \nonumber \\&= \sum_{\substack{F \in \forests{T}{k_0, k_1, \dotsc, k_T}\\ F = (V,E, \height{}, \pweight{}, \qweight{})}} \sum_{\pi \in \part{V}}  \alpha(F) \cdot \valid(F,\pi) \cdot \beta_1(F,\pi) \cdot \beta_2(F,\pi) \cdot \invpoly(\mPsi; F,\pi).
\end{align}
An inspection of the proof of \lemref{unrolling_1} shows that because the non-linearities are assumed to be polynomials of bounded degree (independent of $\dim$), the number of configurations $(F,\pi)$ with a non-zero contribution to \eqref{eq:expectation_final} can be bounded by a finite constant independent of $\dim$. That is,
\begin{align*}
    | \{(F,\pi) : F = (V, E, \height{}, \pweight{}, \qweight{}) \in  \forests{T}{k_0, k_1, \dotsc, k_T}, \; \pi \in \part{V}, \;   \alpha(F)  \neq 0\}| \lesssim 1.
\end{align*}
Consequently, \thref{SE-normalized} follows if we show that for any relevant configuration (\defref{relevant-config}), $\lim_{\dim \rightarrow \infty} \invpoly(\mPsi; F,\pi)$ exists and is identical for any $\mPsi$ that satisfies the assumptions of \thref{SE-normalized}. A simpler preliminary task is to conclude that $\invpoly(\mPsi; F,\pi) \lesssim 1$. An initial naive estimate on $|\invpoly(\mPsi; F,\pi)|$ can be obtained as follows:
\begin{align} \label{eq:naive-estimate}
    |\invpoly(\mPsi; F,\pi)| & \explain{(a)}{\leq} \sum_{\substack{\ell \in \colorings{}{\pi}}}  |\gamma(\mPsi; F,\ell)| \leq |\colorings{}{\pi}| \cdot \dim^{-1} \cdot \max_{\ell \in \colorings{}{\pi}} \gamma(\mPsi; F,\ell)   \explain{(b)}{\leq} \dim^{|\pi|-1-\frac{1}{2} \sum_{v \in V \backslash \rootset(F)} \pweight{v} + \epsilon}.
\end{align}
In the above display, the step (a) follows from triangle inequality and (b) follows from the fact that $|\colorings{}{\pi}| \asymp \dim^{\pi}$ and $\|\mPsi\|_{\infty} \lesssim \dim^{-1/2 + \epsilon}$. However, for many relevant configurations the naive estimate in \eqref{eq:naive-estimate} is insufficient to obtain the conclusion that $|\invpoly(\mPsi; F,\pi)| \lesssim 1$. \fref{nullifying-example} presents a simple example to this end. 
\begin{figure}[ht]
\centering
\includegraphics[width=0.5\textwidth]{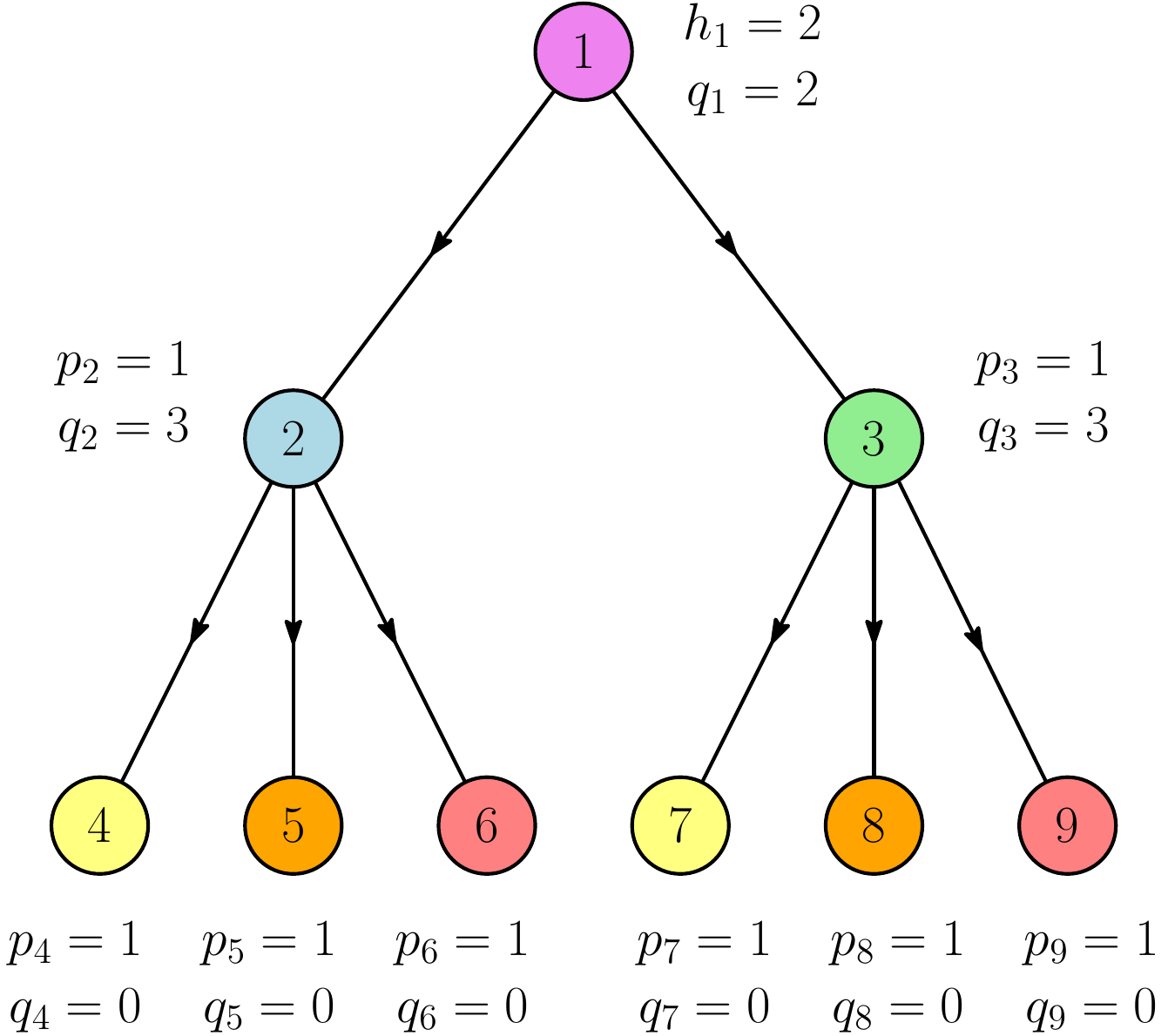}
\caption{A Relevant Configuration for which the naive estimate \eqref{eq:naive-estimate} fails. The colors of the vertices encode the partition $\pi$, vertices with the same color form a block in the partition $\pi$.}  
\label{fig:nullifying-example}
\end{figure}
For the relevant configuration in \fref{nullifying-example}, $|\pi| = 6$ and, $$\sum_{v \in V \backslash \rootset(F)} \pweight{v} = |V| - 1 = 8.$$
Consequently, the naive estimate \eqref{eq:naive-estimate} for this configuration yields the inadequate bound $|\invpoly(\mPsi; F,\pi)| \lesssim \dim^{1+\epsilon}$. The key deficiency of the naive estimate \eqref{eq:naive-estimate} is the use of the triangle inequality in step (a). Many decorated forests have certain structures  which  can be leveraged to improve the naive estimate \eqref{eq:naive-estimate}. We introduce here a special class of such structures, which we call nullifying leaves and edges.
\begin{definition}[Nullifying Leaves and Edges]\label{def:nullifying}A pair of edges $u \rightarrow v$ and $u^\prime \rightarrow v^\prime$ is a pair of nullifying edges for a configuration $(F,\pi)$ with $F = (V, E, \height{}, \pweight{}, \qweight{})$ and $\pi = \{B_1, B_2, \dotsc, B_{|\pi|}\}$ if:
\begin{enumerate}
    \item $v,v^\prime \in \leaves{F}$, $v \neq v^\prime$,
    \item $\pweight{v} = \pweight{v^\prime} = 1$,
    \item $B_{\pi(v)} = B_{\pi(v^\prime)} = \{v,v^\prime\}$, 
    \item $\pi(u) \neq \pi(u^\prime)$.
\end{enumerate}
In this situation, $v, v^\prime$ are referred to as a pair of nullifying leaves and the set of all nullifying leaves of a configuration $(T,\pi)$ is denoted by $\nullleaves{T,\pi}$. Note that $|\nullleaves{T,\pi}|$ is always even (since nullifying leaves occur in pairs) and the number of pairs of nullifying edges in a configuration is given by $|\nullleaves{T,\pi}|/2$.
\end{definition}
Note that in the presence of nullifying edges $u \rightarrow v, u^\prime \rightarrow v^\prime$, summing over the possible colors for $v,v^\prime$ in \eqref{eq:inv-poly} yields the expression:
\begin{align*}
    \Bigg|\sum_{\substack{\ell_v, \ell_{v^\prime} \in [\dim] \\ \ell_{v} = \ell_{v^\prime}}}  \psi_{\ell_u\ell_v} \psi_{\ell_{u^\prime} \ell_{v^\prime}} \Bigg|  = |(\mPsi \mPsi \tran)_{\ell_{u} \ell_{u^\prime}}| \explain{(c)}{\lesssim} \dim^{-1/2+\epsilon},
\end{align*}
where the estimate (c) follows from the assumption $\|\mPsi \mPsi \tran - \mI_{\dim} \|_{\infty} \lesssim \dim^{-1/2+\epsilon}$ and the fact that $\ell_{u} \neq \ell_{u^\prime}$ for a pair of nullifying edges. The above estimate improves upon the naive estimate obtained by the triangle inequality and the assumption $\|\mPsi\|_{\infty} \lesssim \dim^{-1/2+\epsilon}$:
\begin{align*}
    \Bigg|\sum_{\substack{\ell_v, \ell_{v^\prime} \in [\dim] \\ \ell_{v} = \ell_{v^\prime}}}  \psi_{\ell_u\ell_v} \psi_{\ell_{u^\prime} \ell_{v^\prime}} \Bigg| \leq \sum_{\substack{\ell_v, \ell_{v^\prime} \in [\dim] \\ \ell_{v} = \ell_{v^\prime}}} | \psi_{\ell_u\ell_v}| |\psi_{\ell_{u^\prime} \ell_{v^\prime}} | \lesssim \dim \cdot \dim^{-1 + \epsilon} \lesssim \dim^{\epsilon}.
\end{align*}
If a configuration $(F,\pi)$ has $|\nullleaves{F,\pi}|$ nullifying leaves (or $|\nullleaves{F,\pi}|/2$ pairs of nullifying edges), this intuition suggests that one can improve upon the naive estimate in \eqref{eq:naive-estimate} by a factor of $\dim^{|\nullleaves{F,\pi}|/4}$. The following proposition shows that this is indeed correct.
\begin{proposition} \label{prop:key-estimate} Consider a configuration $(F,\pi)$ with $F = (V, E, \height{}, \pweight{}, \qweight{})$ and $\pi = \{B_1, B_2, \dotsc, B_{|\pi|}\}$. For any ${\mPsi} \in \R^{\dim \times \dim}$ such that,
\begin{align*}
   \|{\mPsi}{}\|_\infty &\lesssim \dim^{-\frac{1}{2} + \epsilon}, \;    \|{\mPsi}{} {{\mPsi}{}}\tran- \mI_\dim\|_\infty  \lesssim \dim^{-\frac{1}{2} + \epsilon} \; \forall \; \epsilon > 0,
\end{align*}
we have,
\begin{align*}
   |\invpoly({\mPsi}{}; F, \pi)| \explain{def}{=}  \left|   \sum_{\ell \in \colorings{}{\pi}} \frac{1}{\dim} \prod_{\substack{e \in E \\ e = u \rightarrow v}} \psi_{\ell_u\ell_v}^{\pweight{v}}  \right| & \lesssim \dim^{-\exponent(F,\pi) + \epsilon} \; \forall \; \epsilon > 0,
\end{align*}
where,
\begin{align*}
    \exponent(F,\pi) \explain{def}{=}  \frac{|\nullleaves{F,\pi}|}{4} + 1 - |\pi| + \frac{1}{2} \sum_{v \in V \backslash \rootset(F)}  \pweight{v}.
\end{align*}
\end{proposition}
We prove this result in 
\sref{proof-key-estimate}. 
Returning to the example in \fref{nullifying-example}, the configuration depicted in the figure has $6$ nullifying leaves or $3$ pairs of nullifying edges (these are $\{2 \rightarrow 4, 3 \rightarrow 7\}$, $\{2 \rightarrow 5, 3 \rightarrow 8\}$, and $\{2 \rightarrow 6, 3 \rightarrow 9\}$). Consequently, \propref{key-estimate} yields $|\invpoly(\mPsi; F,\pi)| \lesssim \dim^{-1/2 + \epsilon}$. In particular, for this configuration, $\lim_{\dim \rightarrow \infty} \invpoly(\mPsi; F,\pi) = 0$ for any $\mPsi$ that satisfies the assumptions of \thref{SE-normalized}. 
\subsubsection{Removable Edges and Decomposition into \universal configurations}

Unfortunately, a direct application of the estimate in \propref{key-estimate} is not sufficient to conclude the universality of $\lim_{\dim \rightarrow \infty} \invpoly(\mPsi; F,\pi)$ for every relevant configuration. An example of such a configuration is presented in \fref{removable-example}. This configuration has one pair of nullifying edges $\{2 \rightarrow 5, 3 \rightarrow 8\}$, $|\pi| = 10$ and
\begin{align*}
    \sum_{v \in V \backslash \rootset(F)} \pweight{v} = |V| - 1 = 16.
\end{align*}
Hence, \propref{key-estimate} yields the inadequate bound $|\invpoly(\mPsi; F,\pi) | \lesssim \dim^{1/2+\epsilon}$ for the polynomial $\invpoly(\mPsi; F,\pi)$ corresponding to this configuration. The failure of the estimate in \propref{key-estimate} is due to the presence of structures known as removable edges, which we introduce next. 
\begin{figure}[ht]
\centering
\includegraphics[width=0.5\textwidth]{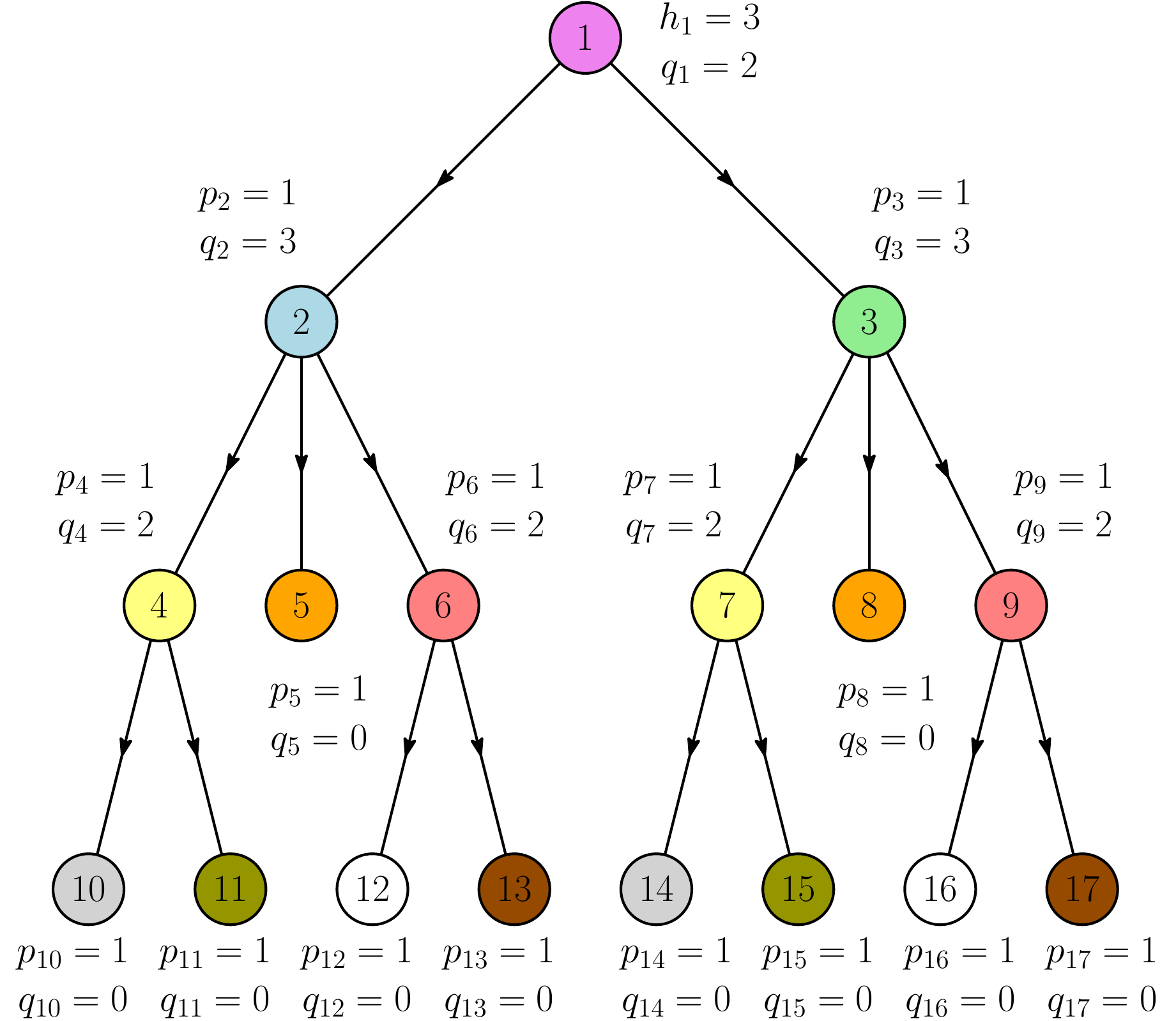}
\caption{A relevant configuration for which the improved estimate in \propref{key-estimate} fails. The colors of the vertices encode the partition $\pi$, vertices with the same color form a block in the partition $\pi$.}  
\label{fig:removable-example}
\end{figure}
\begin{definition} \label{def:removable-pair}  A pair of edges $u \rightarrow v, u^\prime \rightarrow v^\prime$ is called a removable pair of edges for configuration $(F,\pi)$ with $F = (V,E, \height{}, \pweight{}, \qweight{})$  and $\pi=\{B_1, B_2, \dotsc, B_{|\pi|}\}$ if:
\begin{enumerate}
    \item $v,v^\prime \in \leaves{F}$, $v \neq v^\prime$.
    \item $\pweight{v} = \pweight{v^\prime} = 1$.
    \item $B_{\pi(v)} = B_{\pi(v^\prime)} = \{v,v^\prime\}$.
    \item $\pi(u) = \pi(u^\prime)$.
\end{enumerate}
\end{definition}
The configuration in \fref{removable-example} has four pairs of removable edges $\{4 \rightarrow 10, 7 \rightarrow 14\}$, $\{4 \rightarrow 11, 7 \rightarrow 15\}$, $\{6 \rightarrow 12, 9 \rightarrow 16\}$, and  $\{6 \rightarrow 13, 9 \rightarrow 17\}$.  If these edges were absent, this configuration would have been identical to the one depicted in \fref{nullifying-example}, where the estimate given in \propref{key-estimate} was adequate. When a removable edge $u \rightarrow v, u^\prime \rightarrow v^\prime$ is present in a configuration $(F,\pi)$, the configuration $(F,\pi)$ can be simplified while keeping the polynomial $\invpoly(\mPsi; F,\pi)$ (cf. \eqref{eq:inv-poly}) unchanged. This follows since evaluating $\invpoly(\mPsi; F,\pi)$ (cf. \eqref{eq:inv-poly}) involves summing over the possible colors for $v,v^\prime$, which yields an expression of the form:
\begin{align*}
   \sum_{\substack{\ell_v, \ell_{v^\prime} \in [\dim] \\ \ell_{v} = \ell_{v^\prime}}}  \psi_{\ell_u\ell_v} \psi_{\ell_{u^\prime} \ell_{v^\prime}}   = (\mPsi \mPsi \tran)_{\ell_{u} \ell_{u^\prime}} \explain{(d)}{=} 1,
\end{align*}
where the equality (d) follows from the assumption $(\mPsi \mPsi \tran)_{ii} = 1$ for every $i \in [\dim]$ and the fact that $\ell_{u} = \ell_{u}^\prime$ for a pair of removable edges $u \rightarrow v, u^\prime \rightarrow v^\prime$. As a consequence of this simplification, the vertices $v,v^\prime$ and the corresponding edges $u \rightarrow v, u^\prime \rightarrow v^\prime$ can be deleted from the configuration $(F,\pi)$, thus simplifying its structure. By eliminating every pair of removable edges in a relevant configuration $(F,\pi)$ one can express the corresponding polynomial $\invpoly(\mPsi; F, \pi)$ as a linear combination of polynomials associated with \emph{\universal configurations}, which we introduce next.   

\begin{definition}\label{def:universal-configuration} A decorated forest $F = (V, E, \height{}, \pweight{}, \qweight{})$ and a partition $\pi = \{B_1, B_2, \dotsc, B_{|\pi|}\}$ of $V$ form a \universal configuration if:
\begin{description}
    \myitem{\texttt{Root Property}}\label{root-property}: For any two root vertices $u, v \in \rootset(F)$, we have $\pi(u) = \pi(v)$.
   \myitem{\texttt{Singleton Leaf Property}}\label{singleton-leaf-property}: Each leaf $v \in \leaves{F}$ with $|B_{\pi(v)}| = 1$ satisfies $\pweight{v} \geq 4$.
   \myitem{\texttt{Paired Leaf Property}}\label{paired-leaf-property}: Any pair of leaves $v, v^\prime \in \leaves{F}$ with $\pweight{v} = \pweight{v^\prime} = 1$ and $B_{\pi(v)} = B_{\pi(v^\prime)} = \{v,v^\prime\}$ satisfies $\pi(u) \neq \pi(u^\prime)$, where $u,u^\prime$ are the parents of $v,v^\prime$ respectively. 
    \myitem{\texttt{Forbidden Weights Property}}\label{forbidden-weights-property}: There are no vertices $u \in V \backslash \rootset(F)$ such that $|B_{\pi(u)}| = 1, \pweight{u} = 1, \qweight{u} = 1$ or $|B_{\pi(u)}| = 1, \pweight{u} = 2, \qweight{u} = 0$.
    \myitem{\texttt{Parity Property}}\label{parity-property}: There is no block $B$ in the partition $\pi$ such that the sum:
    \begin{align*}
       \left(\sum_{r \in (B \cap \rootset(F)) \backslash \rootset_0(F)} \qweight{r}  \right)+ \left(\sum_{u \in B \backslash ( \leaves{F} \cup \rootset(F))} \pweight{u} + \qweight{u} \right) + \left( \sum_{v \in B \cap \leaves{F}} \pweight{v} \right) 
    \end{align*}
    has odd parity.
\end{description}
\end{definition}

Observe that the \ref{paired-leaf-property} in \defref{universal-configuration} ensures that \universal configurations have no removable edges. The following is the formal statement of our decomposition result.

\begin{proposition}[Decomposition]\label{prop:decomposition} Let $(F,\pi)$ be a relevant configuration. Then there exists $L \in \N$ (independent of $\dim$), \universal configurations $\{(F_i, \pi_i): i \in [L]\}$ and weights $\va \in \{-1,1\}^{L}$ such that,
\begin{align*}
    \sum_{\ell \in \colorings{}{\pi}} \gamma(\mPsi; F,\ell) & = \sum_{i=1}^L  \sum_{\ell \in \colorings{}{\pi_i}} a_i \cdot \gamma(\mPsi; F_i,\ell),
\end{align*}
for any $\mPsi \in \R^{\dim \times \dim}$ such that $(\mPsi \mPsi\tran)_{jj} = 1$ for all $j \in [\dim]$. Furthermore, if the relevant configuration $(F, \pi)$ has exactly one non-trivial root $(|\rootset(F) \backslash \rootset_0(F)| = 1)$, then each  \universal configuration $(F_i,\pi_i), \; i \in [L]$ also has exactly one non-trivial root $(|\rootset(F_i) \backslash \rootset_0({F_i})| = 1)$.
\end{proposition}
\noindent 
The proof of this result is deferred to \sref{proof-decomposition}.

\subsubsection{Universality for \universal configurations}

It turns out that the presence of removable edges (cf. \defref{removable-pair}) is the only barrier that can cause the improved estimate of \propref{key-estimate} to fail. Since \universal configurations do not have removable edges due to the \ref{paired-leaf-property}, we obtain the following result by applying \propref{key-estimate} to such configurations. 

\begin{proposition}[Universality for \universal configurations]  \label{prop:universality-universal-config} Consider a \universal configuration $(F,\pi)$ with $F = (V,E, \height{}, \pweight{}, \qweight{})$ and $\pi=\{B_1, B_2, \dotsc, B_{|\pi|}\}$. For any ${\mPsi}{} \in \R^{\dim \times \dim}$ such that,
\begin{align*}
   \|{\mPsi}\|_\infty &\lesssim \dim^{-\frac{1}{2} + \epsilon}, \;    \|{\mPsi}{{\mPsi}}\tran- \mI_\dim\|_\infty  \lesssim \dim^{-\frac{1}{2} + \epsilon} \; \forall \; \epsilon > 0,
\end{align*}
we have,
\begin{align*}
   \lim_{\dim \rightarrow\infty} \invpoly({\mPsi}; F, \pi) \explain{def}{=}  \lim_{\dim \rightarrow\infty}   \sum_{\ell \in \colorings{}{\pi}} \frac{1}{\dim} \prod_{\substack{e \in E \\ e = u \rightarrow v}} \psi_{\ell_u\ell_v}^{\pweight{v}}  & = \begin{cases} 1 &: E = \emptyset \\ 0 &: E \neq \emptyset \end{cases}.
\end{align*}
\end{proposition}
The above result shows that for \universal configurations, $\lim_{\dim \rightarrow \infty}  \invpoly({\mPsi}; F, \pi)$ exists and is universal (identical for any $\mPsi$ that satisfies the assumptions of \thref{SE-normalized}). 
The proof of this result relies on graph-theoretic arguments to show that any \universal configuration with at least one edge has enough nullifying leaves (cf. \defref{nullifying}) to ensure that the improved estimate of \propref{key-estimate} implies that $\lim_{\dim \rightarrow \infty}  \invpoly({\mPsi}{}; F, \pi) = 0$. The complete proof appears in \sref{proof-universal-config-estimate}.

\subsection{Proof of \thref{SE-normalized}}
We have now introduced all the key ideas involved in the proof of \thref{SE-normalized}. We end this section by presenting the proof of \thref{SE-normalized} assuming \lemref{unrolling_1}, \lemref{key-formula}, \propref{key-estimate}, \propref{decomposition}, and \propref{universality-universal-config}.

\begin{proof}[Proof of  \thref{SE-normalized}] 
Recall the expression for the expectation of the key observable from \eqref{eq:expectation_final}. Since \lemref{key-formula} guarantees that $ \alpha(F) \cdot \valid(F,\pi) \cdot \beta_1(F,\pi) \cdot \beta_2(F,\pi) = 0$ for all non-relevant configurations, one need to compute $\lim_{\dim \rightarrow \infty} \invpoly(\mPsi; F,\pi)$ only for relevant configurations. Let $(F,\pi)$ be any relevant configuration. By \propref{decomposition}, there exists $L \in \N$ (independent of $\dim$), \universal configurations $\{(F_i, \pi_i): i \in [L]\}$ and weights $\va \in \{-1,1\}^{L}$ such that,
\begin{align*}
     \invpoly(\mPsi; F, \pi) & = \sum_{i=1}^L  a_i \cdot  \invpoly(\mPsi; F_i, \pi_i)
\end{align*}
for any $\mPsi \in \R^{\dim \times \dim}$ such that $(\mPsi \mPsi\tran)_{jj} = 1$ for all $j \in [\dim]$. Let $E_i$ denote the edge set of $F_i$. Hence, by \propref{universality-universal-config},
\begin{align*}
    \lim_{\dim \rightarrow \infty}   \invpoly(\mPsi; F, \pi)  & = \sum_{i=1}^L   a_i \cdot \lim_{\dim \rightarrow \infty} \invpoly(\mPsi; F_i, \pi_i) = \sum_{i=1}^L   a_i \cdot \vone_{E_i = \emptyset} \explain{def}{=} \kappa(F,\pi).
\end{align*}
Hence, 
\begin{align} \label{eq:key-formula-application}
     &\E\left[\frac{1}{\dim} \sum_{\ell=1}^\dim \hermite{k_0}(\iter{z}{0}_i) \cdot \hermite{k_1}(\iter{z}{1}_i) \cdot \dotsb \cdot \hermite{k_T}(\iter{z}{T}_i) \right]     \\ \nonumber& \hspace{2cm} =\sum_{\substack{F \in \forests{T}{k_0, k_1, \dotsc, k_T}\\ F = (V,E, \height{}, \pweight{}, \qweight{})}} \sum_{\pi \in \part{V}}  \alpha(F) \cdot \valid(F,\pi) \cdot \beta_1(F,\pi) \cdot \beta_2(F,\pi) \cdot \kappa(F,\pi) \explain{def}{=} c,
\end{align}
which proves the first claim of \thref{SE-normalized}. To prove the second claim of \thref{SE-normalized}, consider the expectation:
\begin{align*}
    \E\left[\frac{1}{\dim} \sum_{\ell=1}^\dim\hermite{k_t}(\iter{z}{t}_i) \right].
\end{align*}
If $k_t = 0$, since $\hermite{0}$ is the constant polynomial $1$, the above expectation is precisely $1$, as claimed. On the other hand, if $k_t \geq 1$, then \eqref{eq:key-formula-application} specializes to:
\begin{align} \label{eq:key-formula-application-2}
    &\E\left[\frac{1}{\dim} \sum_{\ell=1}^\dim\hermite{k_t}(\iter{z}{t}_i) \right]    \\ & \hspace{1.5cm} =\sum_{\substack{F \in \forests{T}{0, 0, \dotsc, k_t, \dotsc, 0}\\ F = (V,E, \height{}, \pweight{}, \qweight{})}} \sum_{\pi \in \part{V}}  \alpha(F) \cdot \valid(F,\pi) \cdot \beta_1(F,\pi) \cdot \beta_2(F,\pi) \cdot \kappa(F,\pi). \nonumber 
\end{align}
Recalling the definition of the set $\forests{T}{0, 0, \dotsc, k_t, \dotsc, 0}$ from \lemref{unrolling_1} observe that any configuration $(F,\pi)$ appearing in the above formula has $T+1$ roots, denoted by $1, 2, \dotsc, T+1$ with $\height{t+1} = t \geq 1$,  $\qweight{r} = 0 \; \forall \; r \neq (t+1)$ and $\qweight{t+1} = k_t \geq 1$. Note that requirement (1c) on the height function $h$ in the definition of decorated forests (\defref{decorated-forests}) guarantees that the root $t+1$ has at least one child and hence, is a non-trivial root.  One the other hand, the conservation equation \eqref{eq:conservation-eq}  in definition of decorated forests (\defref{decorated-forests}) ensures that all other roots $r \neq t+1$ cannot have any children and must be trivial roots (recall $\pweight{v} \geq 1 \quad \forall \; v \; \in \; V \backslash \rootset(F)$).  Consequently, all configurations $(F,\pi)$ appearing in \eqref{eq:key-formula-application-2} have exactly one trivial root. Furthermore, \propref{decomposition} guarantees that the \universal configurations that arise from decomposing such relevant configurations also have exactly one non-trivial root and hence, the edge set of these \universal configurations cannot be empty. Hence as a consequence of the \propref{universality-universal-config}, $\kappa(F,\pi) = 0$ and we obtain,
\begin{align*}
    \lim_{\dim \rightarrow \infty} \E\left[\frac{1}{\dim} \sum_{\ell=1}^\dim\hermite{k_t}(\iter{z}{t}_i) \right] = 0,
\end{align*}
when $k_t \geq 1$, as claimed. This concludes the proof of \thref{SE-normalized}.
\end{proof}
\section{Proof of \lemref{unrolling_1} and \lemref{key-formula}}
\label{sec:unrolling_proof} 
\subsection{Proof of \lemref{unrolling_1}} \label{sec:proof-unrolling-formula}
This section unrolls the AMP iterations to prove \lemref{unrolling_1}. 
\begin{proof}[Proof of \lemref{unrolling_1} ]
\noindent First, we expand:
\begin{align*}
     \hermite{\qweight{1}}(\iter{z}{t}_{\ell_1}),
\end{align*}
as a polynomial of the initialization $\iter{\vz}{0}$ for arbitrary $q_1 \in \N$ and $\ell_1 \in [N]$. The expansion relies crucially on the following property of Hermite polynomials.
\begin{fact}\label{fact:hermite-property} Let $\vu \in \R^\dim$ be such that $\|\vu\| = 1$. We have,
\begin{align*}
    \hermite{q}(\ip{\vu}{\vx}) & = \sum_{\substack{\va \in \W^\dim \\ \|\va\|_1 = q}} \sqrt{\binom{q}{\va}} \cdot \vu^{\va} \cdot \hermite{\va}(\vx).
\end{align*}
In the display above,
\begin{align*}
    \binom{q}{\va} \explain{def}{=} \frac{q!}{a_1 ! a_2! \dotsb a_\dim!}, \; \vu^{\va} \explain{def}{=} \prod_{i=1}^\dim u_i^{a_{i}}, \; \hermite{\va}(\vx) \explain{def}{=} \prod_{i=1}^\dim  \hermite{a_i}(x_i).
\end{align*}
\end{fact}
The property stated above is easily derived using the well known generating formula for Hermite polynomials. We completeness, we provide a proof in \appref{hermite-property}. 

Let $\vm_i$ denote the $i$th row of $\mM$. Using \factref{hermite-property}, we have,
\begin{align*}
    \hermite{\qweight{1}}(\iter{z}{t}_{\ell_1}) & = \hermite{\qweight{1}}\left(\ip{\vm_{\ell_1}}{\nonlin_t(\iter{\vz}{t-1})}\right) \explain{(a)}{=} \sum_{\substack{\va \in \W^\dim \\ \|\va\|_1 = \qweight{1}}} \sqrt{\binom{\qweight{1}}{\va}} \cdot \vm_{\ell_1}^{\va} \cdot \hermite{\va}(\nonlin_t(\iter{\vz}{t-1})).
\end{align*}
In step (a) we used the fact that $\|\vm_{\ell_1}\| = 1$. This is guaranteed by \defref{matrix-ensemble-relaxed} and \assumpref{normalization} since $\|\vm_{\ell_1}\|^2 = (\mM \mM\tran)_{\ell_1 \ell_1} = (\mPsi \mPsi\tran)_{\ell_1 \ell_1} = 1$. We rewrite the above expression as follows. We pick a vector $\va \in \W^n$ with $\|\va\|_1 = \qweight{1}$ in the following steps:
\begin{enumerate}
    \item We first decide the value of $\chrn{1} \explain{def}{=} \|\va\|_0$. Note that $1 \leq \chrn{1} \leq \qweight{1}$. 
     \item We pick labels $\ell_{2}, \ell_3, \dotsc, \ell_{\chrn{1} + 1} \in [\dim]$ with $\ell_2 < \ell_3 < \ell_4 \dotsb < \ell_{\chrn{1}+1}$. These are the locations of the non-zero coordinates of $\va$.
    \item Next we pick a solution $p_2, p_3, \dotsc, p_{\chrn{1}+1} \in \N$, a solution to the integral equation $p_2 + p_3 + \dotsb + p_{\chrn{1}+1} = \qweight{1}$. These are the values of the non-zero coordinates of $\va$. 
    \item We then obtain the vector $\va$ by setting $a_{\ell_i} = p_i$ for all $2 \leq i \leq \chrn{1}+1$ and setting all other coordinates of $\va$ to 0. 
\end{enumerate}
Using this construction we can write,
\begin{align*}
     \hermite{q_1}(\iter{z}{t}_{\ell_1}) & = \sum_{\chrn{1} = 1}^{q_1} \sum_{\substack{p_2, \dotsc, p_{\chrn{1}+1} \in \N \\ p_2+ \dotsb + p_{\chrn{1}+1} = q_1}} \sum_{\substack{\ell_2, \dotsc, \ell_{\chrn{1}+1 \in [\dim]}\\ \ell_2< \dotsb< \ell_{\chrn{1}+1}}} \sqrt{\binom{q_1}{p_2, \dotsc, p_{\chrn{1}+1}}} \cdot \left(\prod_{j=2}^{\chrn{1}+1}  m_{\ell_1\ell_j}^{p_j} \right)  \cdot \left( \prod_{j=2}^{\chrn{1}+1} \hermite{p_{j}} \circ \nonlin_t(\iter{z}{t-1}_{\ell_j}) \right) \\
     & = \sum_{\chrn{1} = 1}^{q_1} \sum_{\substack{p_2, \dotsc, p_{\chrn{1}+1} \in \N \\ p_2+ \dotsb + p_{\chrn{1}+1} = q_1}} \sum_{\substack{\ell_2 \dotsc, \ell_{\chrn{1}+1 \in [\dim]}\\ \ell_2 \neq \ell_3 \neq \dotsb \neq \ell_{\chrn{1}+1}}} \frac{1}{\chrn{1}!} \sqrt{\binom{q_1}{p_2, \dotsc, p_{\chrn{1}+1}}} \cdot \left(\prod_{j=2}^{\chrn{1}+1}  m_{\ell_1\ell_j}^{p_j} \right)  \cdot \left( \prod_{j=2}^{\chrn{1}+1} \hermite{p_{j}} \circ \nonlin_t(\iter{z}{t-1}_{\ell_j}) \right)
\end{align*}
For any $s \in \W$ we can expand the polynomial $\hermite{s} \circ f$ in the Hermite basis:
\begin{align*}
    \hermite{p} \circ \nonlin_t(x) = \sum_{q=0}^{pD} \fourier{t}{p}{q} \hermite{q}(x), 
\end{align*}
where,
\begin{align*}
    \fourier{t}{p}{q} \explain{def}{=} \E[ \hermite{p} \circ \nonlin_t(Z) \cdot  \hermite{q}(Z)].
\end{align*}
Hence,
\begin{align}
     &\hermite{\qweight{1}}(\iter{z}{t}_{\ell_1}) \nonumber \\&= \sum_{\chrn{1} = 1}^{q_1} \sum_{\substack{p_2, \dotsc, p_{\chrn{1}+1} \in \N \\ p_2+ \dotsb + p_{\chrn{1}+1} = q_1}} \sum_{\substack{\ell_2 \dotsc, \ell_{\chrn{1}+1 \in [\dim]}\\ \ell_2 \neq \ell_3 \neq \dotsb \neq \ell_{\chrn{1}+1}}} \frac{1}{\chrn{1}!} \sqrt{\binom{q_1}{p_2, \dotsc, p_{\chrn{1}+1}}} \cdot \left(\prod_{j=2}^{\chrn{1}+1}  m_{\ell_1\ell_j}^{p_j} \right)  \cdot \left( \prod_{j=2}^{\chrn{1}+1} \sum_{q_j=0}^{D q_1} \fourier{t}{p_{j}}{q_j} \cdot  \hermite{q_j}(\iter{z}{t-1}_{\ell_j}) \right) \nonumber \\
     & = \sum_{\chrn{1} = 1}^{q_1} \sum_{\substack{p_2, \dotsc, p_{\chrn{1}+1} \in \N \\ p_2+ \dotsb + p_{\chrn{1}+1} = q_1}} \sum_{\qweight{2}, \dotsc, \qweight{\chrn{1}+1} \in  [0:Dq_0]} \sum_{\substack{\ell_2 \dotsc, \ell_{\chrn{1}+1 \in [\dim]}\\ \ell_2 \neq \ell_3 \neq \dotsb \neq \ell_{\chrn{1}+1}}} \frac{1}{\chrn{1}!} \sqrt{\binom{q_1}{p_2, \dotsc, p_{\chrn{1}+1}}} \cdot \left(\prod_{j=2}^{\chrn{1}+1} \fourier{t}{p_j}{q_j} \cdot m_{\ell_1\ell_j}^{p_j} \cdot  \hermite{q_j}(\iter{z}{t-1}_{\ell_j})\right) \label{eq:one-step-unrolling}
\end{align}
Next, we express each $\hermite{q_j}(\iter{z}{t-1}_{\ell_j})$ in the above formula as a polynomial in $\iter{\vz}{t-2}$ by recursively applying the above formula. We continue this process till we obtain a polynomial in the initialization $\iter{\vz}{0}$. Recalling the definition of decorated forests (\defref{decorated-forests}), colorings of a decorated forest (\defref{coloring}) and the notion of valid colored decorated forest (\defref{valid-tree}), we obtain the following formula for the expansion of $\hermite{k}(\iter{z}{t}_{i})$ in terms of the initialization:
\begin{subequations}\label{eq:polynomial-expansion-preliminary}
\begin{align} 
     \hermite{k}(\iter{z}{t}_{i}) & =  \sum_{\substack{T \in \trees{t}{k}\\ T = (V,E, \height{}, \pweight{}, \qweight{})}} \sum_{\substack{\ell \in [\dim]^V\\ \ell_1 = i}} \alpha(T) \cdot \valid(T,\ell) \cdot \left( \prod_{u \rightarrow v} m_{\ell_u \ell_v}^{\pweight{v}}\right) \cdot \left(\prod_{\substack{w \in \leaves{T}\\ \qweight{w} \geq 1}} \hermite{\qweight{w}}(\iter{z}{0}_{\ell_w}) \right).
\end{align}
In the above display, $\trees{t}{k}$ denotes the set of all decorated trees (decorated forests with exactly one root denoted by $1$, see \defref{decorated-forests}) with $\height{1} = t$ and $\qweight{1} = k$. The coefficient $\alpha(T)$ is as defined in \lemref{unrolling_1}.
\end{subequations}
At this point, we draw the reader's attention to some seemingly arbitrary aspects of the definition of decorated forests (\defref{decorated-forests}) and valid colored decorated forests  (\defref{valid-tree}) that play an important role in ensuring that the above formula is correct:
\begin{enumerate}
    \item In \defref{decorated-forests}, the height function $h$ keeps track of the extent to which the iterations have been unrolled: property (a) of $h$ captures the fact that each step of unrolling expresses the coordinates  of $\iter{\vz}{t}$ as a polynomial in $\iter{\vz}{t-1}$, property (b) captures the fact that the unrolling process stops once a polynomial in $\iter{\vz}{0}$ is obtained and property (c) ensures that the unrolling process continues till every non-trivial polynomial in the iterates has been expressed in terms of the initialization $\iter{\vz}{0}$. 
    \item In \defref{valid-tree}, the second requirement (no two siblings have the same color) captures the $\ell_2 \neq \ell_3 \neq \dotsb$ constraint that appears in \eref{one-step-unrolling}.
\end{enumerate}
Next, for any $T \in \W$ and $k_0, k_1, \dotsc, k_T \in \N$, consider:
\begin{align*} 
     \frac{1}{\dim} \sum_{i=1}^\dim \hermite{k_0}(\iter{z}{0}_i) \cdot \hermite{k_1}(\iter{z}{1}_i) \cdot \dotsb \cdot \hermite{k_T}(\iter{z}{T}_i).
\end{align*}
Using the formula in \eref{polynomial-expansion-preliminary} to expand $\hermite{k_0}(\iter{z}{0}_i)$, $\hermite{k_1}(\iter{z}{1}_i),\dotsc ,\hermite{k_T}(\iter{z}{T}_i)$, we immediately obtain the claim of the lemma
\begin{subequations} \label{eq:unrolling-formula-repeat}
\begin{align} 
     &\frac{1}{\dim} \sum_{i=1}^\dim \hermite{k_0}(\iter{z}{0}_i) \cdot \hermite{k_1}(\iter{z}{1}_i) \cdot \dotsb \cdot \hermite{k_T}(\iter{z}{T}_i)  = \nonumber\\ & \hspace{2.8cm} \frac{1}{\dim}\sum_{\substack{F \in \forests{T}{k_0, k_1, \dotsc, k_T}\\ F = (V,E, \height{}, \pweight{}, \qweight{})}} \sum_{\substack{\ell \in [\dim]^V}} \alpha(F) \cdot \valid(F,\ell) \cdot \left( \prod_{u \rightarrow v} m_{\ell_u \ell_v}^{\pweight{v}}\right) \cdot \left(\prod_{\substack{w \in \leaves{F} \cup \rootset_0(F)\\ \qweight{w} \geq 1}} \hermite{\qweight{w}}(\iter{z}{0}_{\ell_w}) \right).
\end{align}
where, for $F = (V, E, \height{}, \pweight{}, \qweight{})$ we define,
\begin{align}
    \alpha(F) \bydef \left( \prod_{\substack{u \in V \backslash (\leaves{F} \cup \rootset_0(F))}} \frac{\sqrt{\qweight{u}!}}{\chrn{u}(F)!} \cdot \prod_{v: u \rightarrow v} \frac{1}{\sqrt{\pweight{v}!}} \right) \cdot \left( \prod_{w \in V \backslash \rootset(F)} \fourier{\height{w}+1}{\pweight{w}}{\qweight{w}}\right).
\end{align}
\end{subequations}

\end{proof}

We conclude this section with the following remark.
\begin{remark} We draw the reader's attention to the following subtle aspects of the definition of decorated forests (\defref{decorated-forests}) and the formula in \lemref{unrolling_1} (reproduced in \eqref{eq:unrolling-formula-repeat}) which are useful to keep in mind. 
\begin{enumerate}
    \item For a decorated forest $F = (V, E, \height{}, \pweight{}, \qweight{})$ the function $\pweight{}$ takes values in $\N$ and hence $\pweight{u} \geq 1$. In contrast, $\qweight{}$ takes values in $\W$ and hence $\qweight{u} \geq 0$. 
    \item The function $\pweight{}$ is not defined for root vertices. 
    \item Trivial roots, like leaves, have no children. However, a trivial root is not a leaf since the definition of a leaf requires the vertex to be a non-root. Trivial roots behave like leaves in some ways and not in others, and hence have to be treated separately from leaves and non-trivial roots:
    \begin{enumerate}
        \item Like leaves but unlike non-trivial roots, a trivial root $u \in \rootset_0(F)$:  (i) contributes a factor $\hermite{\qweight{u}}(\iter{z}{0}_{\ell_u})$ in the polynomial expansion given in \eqref{eq:unrolling-formula-repeat}, (ii) does not contribute to the combinatorial factor which forms the first term in the definition of $\alpha(F)$ in \eqref{eq:unrolling-formula-repeat}, (iii) does not satisfy the conservation property given in \eref{conservation-eq} (cf. \defref{decorated-forests}).
        \item Like non-trivial roots but unlike leaves, $\pweight{u}$ is not defined for a  trivial root $u \in \rootset_0(F)$. Hence, it does not contribute a factor of $\fourier{\height{u}+1}{\pweight{u}}{\qweight{u}}$ in the Definition $\alpha(F)$ in \eqref{eq:unrolling-formula-repeat}.
    \end{enumerate}
    \item A leaf vertex $u \in \leaves{F}$ need not have $\height{u} = 0$. It is possible a leaf $u \in \leaves{F}$ has $\height{u} \geq 1$, if $\qweight{u} = 0$. 
    \item If a vertex $u$ has $\height{u} = 0$ then it must be that $u$ has no children and hence $u \in \rootset_0(F) \cup \leaves{F}$.
\end{enumerate}
\end{remark}
\subsection{Proof of \lemref{key-formula}}
This section is devoted to the proof of the expectation formula (\lemref{key-formula}). 
\begin{proof}[Proof of \lemref{key-formula}]
Since $\mM$ is semi-random, recall from \defref{matrix-ensemble-relaxed} that $\mM = \mS \mPsi \mS$, and hence, $m_{ij} = s_i \psi_{ij} s_j$. Comparing \eref{rearrangement} and \eqref{eq:key-formula}, we see that the formula claimed in the lemma follows if we show:
\begin{align}
    &\E\Big[\prod_{\substack{v \in \leaves{F} \cup \rootset_0(F)\\ \qweight{v} \geq 1}} \hermite{\qweight{v}}(\iter{z}{0}_{\ell_v}) \Big] = \prod_{j=1}^{|\pi|} \E \Big[ \prod_{\substack{v \in (\leaves{F} \cup \rootset_0(F))\cap B_j\\ \qweight{v} \geq 1}} \hermite{\qweight{v}}(Z) \Big], \label{eq:key-formula-intermediate-eq1}\\
    &\E\left[ \prod_{u \rightarrow v} s_{\ell_u}^{\pweight{v}} s_{\ell_v}^{\pweight{v}} \right]  = \prod_{j=1}^{|\pi|} \E\left[ \left( \prod_{r \in (\rootset(F) \cap B_j) \backslash \rootset_0(F)} S^{\qweight{r}} \right)  \cdot \left( \prod_{v \in B_j \cap \leaves{F}} S^{\pweight{v}} \right) \cdot \left( \prod_{\substack{u \in B_j \backslash (\leaves{F} \cup \rootset(F))}} S^{\pweight{u} + \qweight{u}} \right)\right]. \label{eq:key-formula-intermediate-eq2}
\end{align}
In order to obtain formula \eref{key-formula-intermediate-eq1}, we noted that $\iter{\vz}{0} \sim \gauss{\vzero}{\mI_\dim}$ (cf. \assumpref{initialization} and \assumpref{normalization}) and grouped the leaves assigned the same color in order to factorize the expectation over the independent coordinates of $\iter{\vz}{0}$. In order to obtain formula \eref{key-formula-intermediate-eq2} we observe that,
\begin{align*}
    \prod_{u \rightarrow v} s_{\ell_u}^{\pweight{v}} s_{\ell_v}^{\pweight{v}} &\explain{(a)}{=} \left( \prod_{r \in \rootset(F)\backslash \rootset_0(F)}\prod_{v: r \rightarrow v} s_{\ell_r}^{\pweight{v}} \right) \cdot \left( \prod_{\substack{u \in V \backslash (\rootset(F) \cup \leaves{F}) }} s_{\ell_u}^{\pweight{u}} \prod_{v: u \rightarrow v} s_{\ell_u}^{\pweight{v}}\right) \cdot \left( \prod_{u \in \leaves{F}} s_{\ell_u}^{\pweight{u}} \right) \\
    & \explain{(b)}{=}  \left( \prod_{r \in \rootset(F) \backslash \rootset_0(F)} s_{\ell_r}^{\qweight{r}} \right) \cdot \left( \prod_{\substack{u \in V \backslash( \rootset(F) \cup \leaves{F})}} s_{\ell_u}^{\pweight{u}+\qweight{u}} \right) \cdot \left( \prod_{u \in \leaves{F}} s_{\ell_u}^{\pweight{u}} \right),
\end{align*}
In the above display, in step (a), we reorganized the product over edges in order to collect the signs variables $s_{\ell_u}$ corresponding to the same node $u$ together. 
In step (b), we used the fact (cf. \defref{decorated-forests}) that for any node $u \in V \backslash (\rootset_0(F) \cup \leaves{F})$, 
$\qweight{u}  = \sum_{v: u \rightarrow v} \pweight{v}$.

Formula \eref{key-formula-intermediate-eq2} now follows by grouping the nodes assigned the same color in order to factorize the expectation over the independent coordinates of $s_{1:\dim}$. Observe that:
\begin{enumerate}
    \item $\alpha(F) = 0$ if there is a vertex $u \in V \backslash \rootset(F)$ such that $\pweight{u} = 1, \qweight{u} = 1$ or $\pweight{u} = 2, \qweight{u} = 0$. This follows from the assumption that $\E Z \nonlin_{t}(Z) = 0$ (\assumpref{divergence-free}) and $\E \nonlin^2_t(Z) = 1$ (\assumpref{normalization}). Hence if $\pweight{u} = 1, \qweight{u} = 1$ or $\pweight{u} = {2}, \qweight{u} = 0$, we have $\fourier{\height{u}}{\pweight{u}}{\qweight{u}} = 0$ and $\alpha(F) = 0$. 
    \item $\beta_1(T,\pi) = 0$ if there is a vertex $v \in \leaves{F} \cup \rootset_0(F)$ with $q_v \geq 1$ and $|B_{\pi(v)}| = 1$. This holds as $\E \hermite{q}(Z) = 0$ for any $q \geq 1$. 
    \item $\beta_2(T,\pi) = 0$ if there is a block $B$ in the partition $\pi$ such that the sum:
    \begin{align*}
        \left(\sum_{(B \cap \rootset(F)) \backslash \rootset_0(F)} \qweight{u}  \right)+ \left(\sum_{u \in B \backslash ( \leaves{F} \cup \rootset(F))} \pweight{u} + \qweight{u} \right) + \left( \sum_{v \in B \cap \leaves{F}} \pweight{v} \right) 
    \end{align*}
    has odd parity. This follows from the formula for $\beta_2(T,\pi)$ and the fact that $\E S^{w} = 0$ when $S \sim \unif{\{\pm 1\}}$ and $w$ is an odd number.
\end{enumerate}
As a consequence, unless $(F,\pi)$ is a relevant configuration (in the sense of \defref{relevant-config}), we have $\alpha(F) \cdot \valid(F,\pi) \cdot \beta_1(F,\pi) \cdot \beta_2(F,\pi) = 0$. Indeed, the \ref{root-rule} and \ref{sibling-rule} ensure $\valid(F,\pi) \neq 0$, \ref{forbidden-weights-rule} ensures $\alpha(F) \neq 0$, \ref{leaf-rule} and \ref{trivial-root-rule} ensure $\beta_1(F,\pi) \neq 0$ and the \ref{parity-rule} ensures $\beta_2(F,\pi) \neq 0$. This concludes the proof of this lemma.
\end{proof}

\section{Proof of \propref{key-estimate}} \label{sec:proof-key-estimate}
We begin by proving the following useful lemma. 

\begin{lemma}\label{lem:mobius} Let $\iter{\vu}{1}, \iter{\vu}{2}, \dotsc, \iter{\vu}{k}$ be a collection of $k$ vectors in $\R^\dim$. We have,
\begin{align*}
    \left|\sum_{\substack{\ell_{1:k} \in [\dim] \\ \ell_i \neq \ell_j \; \forall \; i \neq j}} \prod_{i=1}^k \iter{u}{i}_{\ell_i} \right| & \leq k^{2k} \cdot \max(\sqrt{\dim} U_\infty, \min(\overline{U},\dim U_\infty))^k,
\end{align*}
where,
\begin{align*}
    \overline{U} \explain{def}{=} \max_{i \in [k]} \left| \sum_{j=1}^\dim \iter{u}{i}_j \right|, \; U_\infty \explain{def}{=} \max_{i \in [k]} \|\iter{\vu}{i}\|_\infty
\end{align*}
\end{lemma}
\begin{proof} For any partition $\pi$ of the set $[k]$, we define,
\begin{align*}
    \wcolorings{}{\pi} \explain{def}{=} \{(\ell_1, \ell_2, \dotsc, \ell_k): \pi(i) = \pi(j) \implies \ell_i = \ell_j\}.
\end{align*}
The Mobius Inversion formula (see for e.g. \citep[Lemma 5]{tulino2010capacity}) states that there are coefficients $\mu(\pi)$ such that for any $\dim \in \N$ and any function $f: [\dim]^k \rightarrow \R$, we have,
\begin{align*}
    \sum_{\substack{\ell_{1:k} \in [\dim] \\ \ell_i \neq \ell_j \; \forall \; i \neq j}} f(\ell_1, \ell_2, \dotsc, \ell_k)  &= \sum_{\pi \in \part{[k]}} \mu(\pi) \sum_{\ell_{1:k} \in \wcolorings{}{\pi}}  f(\ell_1, \ell_2, \dotsc, \ell_k).
\end{align*}
Crucially, the same coefficients $\mu(\pi)$ work for any $f$ and any $\dim$. An explicit formula for $\mu(\pi)$ is available and in particular, $|\mu(\pi)|  \leq k^k$ (see for e.g. \citep[Section 5.1.4]{anderson2014asymptotically}). We will apply this formula to the function:
\begin{align*}
    f(\ell_1, \ell_2, \dotsc, \ell_k)  \explain{def}{=} \prod_{i=1}^k \iter{u}{i}_{\ell_i}.
\end{align*}
For a partition $\pi = \{B_1, B_2, \dotsc, B_{\pi}\}$, observe that we can simplify:
\begin{align*}
    \sum_{\ell_{1:k} \in \wcolorings{}{\pi}}  f(\ell_1, \ell_2, \dotsc, \ell_k) & = \sum_{a_1, a_2 \dotsc, a_{|\pi|} \in [\dim]} \prod_{i=1}^{|\pi|} \prod_{j \in B_i} \iter{u}{j}_{a_i} = \prod_{i=1}^{|\pi|} \left( \sum_{a=1}^\dim \prod_{j \in B_i} \iter{u}{j}_{a} \right).
\end{align*}
For any block $B_i$, we have,
\begin{align*}
    \left| \sum_{a=1}^\dim \prod_{j \in B_i} \iter{u}{j}_{a} \right| & \leq \dim U_\infty^{|B_i|}. 
\end{align*}
On the other hand, if a block $B_i$ has cardinality $1$, that is $B_i = \{j\}$ for some $j$, we can also obtain the following estimate:
\begin{align*}
    \left| \sum_{a=1}^\dim \prod_{j \in B_i} \iter{u}{j}_{a} \right| & = \left|\sum_{a=1}^\dim \iter{u}{j}_{a} \right|  \leq \overline{U}. 
\end{align*}
Let $|\pi|_1$ denote the number of blocks in $\pi$ with cardinality $1$. Hence, we have obtained the upper bound,
\begin{align*}
    \left| \sum_{\ell_{1:k} \in \wcolorings{}{\pi}}  f(\ell_1, \ell_2, \dotsc, \ell_k) \right| & \leq \min(\overline{U},\dim U_\infty)^{|\pi|_1} \cdot \dim^{|\pi| - |\pi|_1} \cdot U_\infty^{k - |\pi|_1}.
\end{align*}
Observe that since any block with cardinality more than $1$, has cardinality at least $2$, $$|\pi| \leq |\pi|_1 + (k-|\pi|_1)/2 \leq (k+|\pi|_1)/2.$$ Hence,
\begin{align*}
  \left|  \sum_{\ell_{1:k} \in \wcolorings{}{\pi}}  f(\ell_1, \ell_2, \dotsc, \ell_k) \right| & \leq \min(\overline{U},\dim U_\infty)^{|\pi|_1} \cdot (\sqrt{\dim} U_\infty)^{k-|\pi|_1} \\
  & \leq \max(\sqrt{\dim} U_\infty, \min(\overline{U},\dim U_\infty))^k.
\end{align*}
By the Mobius Inversion formula,
\begin{align*}
    \left|\sum_{\substack{\ell_{1:k} \in [\dim] \\ \ell_i \neq \ell_j \; \forall \; i \neq j}} \prod_{i=1}^k \iter{u}{i}_{\ell_i} \right| & \leq |\part{[k]}| \cdot \left( \max_{\pi \in \part{k}} |\mu(\pi)| \right) \cdot  \max(\sqrt{\dim} U_\infty, \min(\overline{U},\dim U_\infty))^k \\
    & \leq k^{2k} \cdot \max(\sqrt{\dim} U_\infty, \min(\overline{U},\dim U_\infty))^k,
\end{align*}
as claimed. 
\end{proof}
We now present the proof of \propref{key-estimate}.
\begin{proof}[Proof of \propref{key-estimate}]
Note that without loss of generality, we can assume that,
\begin{align*}
    \nullleaves{F,\pi} & = \{v_1, v_1^\prime, v_2, v_2^\prime, \dotsc, v_k, v_k^\prime\}. 
\end{align*}
where $v_i, v_i^\prime$ form a pair of nullifying leaves in the sense of \defref{nullifying}. Let $u_i$ and $u_i^\prime$ denote the parents of $v_i$ and $v_i^\prime$ respectively. Note that the edges $e_i \bydef u_i \rightarrow v_i, \; e_i^\prime \bydef u_i^\prime \rightarrow v_i^\prime$ form a pair of nullifying edges in the sense of \defref{nullifying}. Let $a_i$ denote the color assigned to block $B_i$ by $\ell$. Then we can write,
\begin{align*}
    \gamma(\mPsi; F, \ell)  &= \frac{1}{\dim} \prod_{\substack{e\in E\\ e = u \rightarrow v}} \psi_{a_{\pi(u)},a_{\pi(v)}}^{\pweight{v}}, \\
    \sum_{\ell \in \colorings{}{\pi}} \gamma(\mPsi; F, \ell) & = \sum_{\substack{a_1, a_2, \dotsc, a_{|\pi|}\\ a_i \neq a_j \forall i \neq j}} \frac{1}{\dim} \prod_{\substack{e\in E\\ e = u \rightarrow v}} \psi_{a_{\pi(u)},a_{\pi(v)}}^{\pweight{v}}.
\end{align*}
Without loss of generality let us further assume that $\{v_i, v_i^\prime\}$ occupy the last $k$ blocks. that is, $B_{i+|\pi|-k} = \{v_i, v_i^\prime \}$ for each $i \in [k]$. By the definition of nullifying edges (\defref{nullifying}), each of the colors $a_{|\pi| - k + i} \; i \in [k]$ appear exactly twice in $\gamma(\mPsi; F, \ell)$ with the edges $u_i \rightarrow v_i$ and $u_i^\prime \rightarrow v_i^\prime$. Hence, we can isolate their occurrences as follows:
\begin{align*}
     \gamma(\mPsi; F, \ell) & = \widetilde{\gamma}(\mPsi; F, a_{1}, a_{2}, \dotsc, a_{|\pi|-k}) \cdot \prod_{i=1}^k \left(\psi_{a_{\pi(u_i)}, a_{i+|\pi|-k}} \psi_{a_{\pi(u_i^\prime)}, a_{i+|\pi|-k}} \right),
\end{align*}
where we defined,
\begin{align*}
    \widetilde{\gamma}(\mPsi; F, a_{1}, a_{2}, \dotsc, a_{|\pi|-k}) &\bydef \frac{1}{\dim}\prod_{\substack{e \in E \backslash \{e_{1:k}, e^\prime_{1:k}\}\\ e = u \rightarrow v}}\psi_{a_{\pi(u)},a_{\pi(v)}}^{\pweight{v}}.
\end{align*}
By defining the indices $b_i = a_{i+|\pi| - k}$ for $i \in [k]$, we can obtain, 
\begin{align*}
     \left|\sum_{\ell \in \colorings{}{\pi}} \gamma(\mPsi; F, \ell) \right| &= \left|\sum_{\substack{a_1,  \dotsc, a_{|\pi| - k} \in [\dim] \\ a_i \neq a_j \forall i \neq j}} \widetilde{\gamma}(\mPsi; F, a_1, \dotsc, a_{|\pi| - k}) \cdot \sum_{\substack{b_{1:k} \in [\dim] \backslash \{a_{1}, \dotsc, a_{|\pi| - k}\}\\ b_i \neq b_j \forall i \neq j}} \prod_{i=1}^k \left(\psi_{a_{\pi(u_i)}, b_i} \psi_{a_{\pi(u_i^\prime)}, b_i} \right) \right| \\
     & \leq \sum_{\substack{a_1,  \dotsc, a_{|\pi| - k} \in [\dim] \\ a_i \neq a_j \forall i \neq j}} |\widetilde{\gamma}(\mPsi; F, a_1, \dotsc, a_{|\pi| - k})| \cdot \left| \sum_{\substack{b_{1:k} \in [\dim] \backslash \{a_{1}, \dotsc, a_{|\pi| - k}\}\\ b_i \neq b_j \forall i \neq j}} \prod_{i=1}^k \left(\psi_{a_{\pi(u_i)}, b_i} \psi_{a_{\pi(u_i^\prime)}, b_i} \right) \right|.
\end{align*}
Defining,
\begin{align*}
    \widetilde{\lambda}(\mPsi; F, a_1, \dotsc, a_{|\pi| - k}) \bydef \sum_{\substack{b_{1:k} \in [\dim] \backslash \{a_{1}, \dotsc, a_{|\pi| - k}\}\\ b_i \neq b_j \forall i \neq j}} \prod_{i=1}^k \left(\psi_{a_{\pi(u_i)}, b_i} \psi_{a_{\pi(u_i^\prime)}, b_i} \right),
\end{align*}
we can rewrite the bound obtained above as,
\begin{align}
     &\left|\sum_{\ell \in \colorings{}{\pi}} \gamma(\mPsi; F, \ell) \right|  \leq \sum_{\substack{a_1,  \dotsc, a_{|\pi| - k} \in [\dim] \\ a_i \neq a_j \forall i \neq j}} |\widetilde{\gamma}(\mPsi; F, a_1, \dotsc, a_{|\pi| - k})| \cdot |\widetilde{\lambda}(\mPsi; F, a_1, \dotsc, a_{|\pi| - k})|  \nonumber \\ 
     & \leq \dim^{|\pi| - k} \cdot \left( \max_{\substack{a_1, \dotsc, a_{|\pi|-k} \in [\dim] \\ a_i \neq a_j \forall i \neq j}} |\widetilde{\gamma}(\mPsi; F, a_1, \dotsc, a_{|\pi| - k})| \right) \cdot \left( \max_{\substack{a_1, \dotsc, a_{|\pi|-k} \in [\dim] \\ a_i \neq a_j \forall i \neq j}} |\widetilde{\lambda}(\mPsi; F, a_1, \dotsc, a_{|\pi| - k})| \right).  \label{eq:universal-forest-intermediate-estimate}
\end{align}
Next, we bound $|\widetilde{\gamma}(\mPsi; F, a_1, \dotsc, a_{|\pi| - k})|$ and $|\widetilde{\lambda}(\mPsi; F, a_1, \dotsc, a_{|\pi| - k})|$. The following bound on $|\widetilde{\gamma}|$  will be sufficient for our purposes:
\begin{align}\label{eq:gamma-tilde-bound}
    \left| \widetilde{\gamma}(\mPsi; F, a_1, \dotsc, a_{|\pi| - k})\right| & \bydef \left|  \frac{1}{\dim}\prod_{\substack{e \in E \backslash \{e_{1:k}, e^\prime_{1:k}\}\\ e = u \rightarrow v}}\psi_{a_{\pi(u)},a_{\pi(v)}}^{\pweight{v}} \right|  \leq \|\mPsi\|_\infty^{\alpha} \cdot \dim^{-1},
\end{align}
where,
   $ \alpha \bydef \sum_{v \in V \backslash (\rootset(F) \cup \nullleaves{F,\pi})} \pweight{v}$.
In order to control $|\widetilde{\lambda}(\mPsi; F, a_1, \dotsc, a_{|\pi| - k})|$ we will leverage the constraint $ \|\mPsi\mPsi\tran - \mI_\dim\|_\infty  \lesssim \dim^{-\frac{1}{2} + \epsilon}$ by applying \lemref{mobius} for a suitable choice of vectors $\iter{\vu}{1:k}\in \R^{\dim -(|\pi| -k)}$. We will index the entries of these vectors using the set $[\dim] \backslash \{a_1, a_2, \dotsc, a_{|\pi| - k}\}$. The entries of these vectors are defined as follows:
\begin{align*}
    \iter{u}{i}_j \bydef \psi_{a_{\pi(u_i)},j} \cdot \psi_{a_{\pi(u_i^\prime)},j} \; \forall \; j \; \in \; [\dim] \backslash \{a_1, a_2, \dotsc, a_{|\pi| - k}\}.
\end{align*}
In order to apply \lemref{mobius}, we need to bound:
\begin{align*}
    U_\infty &\bydef \max_{i \in [k]} \|\iter{\vu}{i}\|_\infty \leq \|\mPsi\|^2_\infty, 
\end{align*}
and,
\begin{align*}
    \overline{U} &\bydef \max_{i \in [k]} \left| \sum_{j\in [\dim] \backslash \{a_1, a_2, \dotsc, a_{|\pi| - k}\}} \iter{u}{i}_j\right| \\
    & = \max_{i \in [k]} \left| \sum_{j \in [\dim] \backslash \{a_1, a_2, \dotsc, a_{|\pi| - k}\}} \psi_{a_{\pi(u_i)},j} \cdot \psi_{a_{\pi(u_i^\prime)},j}\right| \\
    & \leq \max_{i \in [k]} \left(\left| \sum_{j =1}^\dim \psi_{a_{\pi(u_i)},j} \cdot \psi_{a_{\pi(u_i^\prime)},j}\right|  + \sum_{j \in \{a_1, a_2, \dotsc, a_{|\pi| - k}\}} |\psi_{a_{\pi(u_i)},j}| \cdot |\psi_{a_{\pi(u_i^\prime)},j}| \right) \\
    & \explain{}{\leq} \max_{i \in [k]} \left| (\mPsi\mPsi\tran)_{a_{\pi(u_i)},a_{\pi(u_i^\prime)}} \right| + (|\pi|-k)\|\mPsi\|_\infty^2 \\
    & \explain{(a)}{\leq} \max_{i \neq j} |(\mPsi\mPsi\tran)_{ij}| + (|\pi|-k)\|\mPsi\|_\infty^2\\
    & \leq \|\mPsi\mPsi\tran - \mI_\dim \|_\infty + (|\pi|-k)\|\mPsi\|_\infty^2.
\end{align*}
In the above display, the step (a) relies on the following fact: as $u_i \rightarrow v_i$ and $u_i^\prime \rightarrow v_i^\prime$ is a pair of nullifying edges (see \defref{nullifying}), $\pi(u_i) \neq \pi(u_i^\prime) \implies a_{\pi(u_i)} \neq a_{\pi(u_i^\prime)}$. Hence, by \lemref{mobius},
\begin{align}\label{eq:lambda-tilde-bound}
    |\widetilde{\lambda}(\mPsi; F, a_1, \dotsc, a_{|\pi| - k})| & \leq k^{2k} \cdot \max\left(\sqrt{\dim} \|\mPsi\|_\infty^2, \min\left(\|\mPsi\mPsi\tran - \mI_\dim\|_\infty+ (|\pi|-k)\|\mPsi\|_\infty^2, \dim \|\mPsi\|_\infty^2\right)\right)^k.
\end{align}
Plugging the bounds on $|\widetilde{\gamma}|, |\widetilde{\lambda}|$ obtained in \eref{gamma-tilde-bound} and \eref{lambda-tilde-bound} into \eref{universal-forest-intermediate-estimate}, we obtain,
\begin{align*}
     &\left|\sum_{\ell \in \colorings{}{\pi}} \gamma(\mPsi; F, \ell) \right|  \leq \\&\hspace{1cm}\dim^{|\pi| - k-1} \cdot \|\mPsi\|_\infty^\alpha \cdot \left(k^{2k} \cdot \max\left(\sqrt{\dim} \|\mPsi\|_\infty^2, \min\left(\|\mPsi\mPsi\tran - \mI_\dim\|_\infty+ (|\pi|-k)\|\mPsi\|_\infty^2, \dim \|\mPsi\|_\infty^2\right)\right)^k\right).
\end{align*}
Using the assumptions $\|\mPsi\|_\infty \lesssim \dim^{-\frac{1}{2} + \epsilon}$ and $\|\mPsi\mPsi\tran - \mI_\dim\|_\infty \lesssim \dim^{-\frac{1}{2} + \epsilon}$, we obtain the estimate,
\begin{align*}
    \left|\sum_{\ell \in \colorings{}{\pi}} \gamma(\mPsi; F, \ell) \right| & \lesssim \dim^{(|\pi| - k - 1) - \frac{\alpha}{2} - \frac{k}{2} + (2k+\alpha) \epsilon} \lesssim \dim^{-\exponent(F,\pi) + (2k + \alpha) \epsilon},
\end{align*}
where we defined $\exponent(F,\pi) = \alpha/2 + k/2 + 1 - (|\pi|-k)$. We can rewrite $\exponent(F,\pi)$ as follows:
\begin{align*}
    \exponent(F,\pi) & \explain{def}{=} 1 + \frac{\alpha}{2} - (|\pi| - k) + \frac{k}{2} \\& \explain{(c)}{=}   1 + \frac{\alpha}{2} - |\pi| + \frac{|\nullleaves{F,\pi}|}{2} + \frac{|\nullleaves{F,\pi}|}{4} \\
    &\explain{(d)}{=} 1 + \left(\frac{1}{2} \sum_{v \in V \backslash (\rootset(F) \cup \nullleaves{F,\pi})} \pweight{v} \right) + \frac{|\nullleaves{F,\pi}|}{2} - |\pi| + \frac{|\nullleaves{F,\pi}|}{4} \\
    & \explain{(e)}{=} 1+ \left(\frac{1}{2} \sum_{v \in V \backslash  \rootset(F)} \pweight{v} \right) - |\pi| + \frac{|\nullleaves{F,\pi}|}{4}.
\end{align*}
In the above display equality (c) follows because $\nullleaves{F,\pi} = 2k$, equality (d) follows from the definition of $\alpha$ in \eref{gamma-tilde-bound} and, equality (e) follows by observing that for any nullifying leaf $v \in \nullleaves{F,\pi}$, we have $\pweight{v} = 1$ (cf. \defref{nullifying}). Furthermore, since $\epsilon > 0$ was arbitrary replacing $\epsilon$ by $\epsilon/(2k+\alpha)$ completes the proof.
\end{proof}
\section{Proof of \propref{decomposition}}
\label{sec:proof-decomposition}
This section is devoted to the proof of \propref{decomposition}, which shows that any relevant configuration can be expressed as a linear combination of \universal configurations. We begin by showing that relevant configurations are ``almost'' \universal in the sense that they satisfy all properties of a \universal configuration (cf. \defref{universal-configuration}) with the sole exception of \ref{paired-leaf-property}. We call such configurations semi-\universal configurations.

\begin{definition} \label{def:semi-universal-config} A decorated forest $F$ and  a partition $\pi$ of its vertices form a semi-\universal configuration if they satisfy:
\begin{description}
    \myitem{\ref{root-property}}
    \myitem{\texttt{Sibling Property}}\label{sibling-property}: There are no siblings $u,v \in V \backslash \rootset(F)$ such that $B_{\pi(u)} = B_{\pi(v)} = \{u,v\}$.
    \myitem{\ref{singleton-leaf-property}}
    \myitem{\ref{forbidden-weights-property}}
    \myitem{\ref{parity-property}}
\end{description}
\end{definition}
\begin{lemma} Any relevant configuration $(F,\pi)$ is semi-\universal.
\end{lemma}
\begin{proof}
Recall the definition of relevant configurations from \defref{relevant-config}. We verify each of the requirements for a \universal configuration (\defref{semi-universal-config}):
\begin{enumerate}
    \item The \ref{root-property} is identical to the \ref{root-rule}. 
    \item The \ref{sibling-property} is implied by the \ref{sibling-rule}.
    \item Let $v \in \leaves{F}$ be such that $|B_{\pi(v)}| = 1$. By the \ref{leaf-rule}, we must have $\qweight{v} = 0$. By the \ref{parity-rule} $\pweight{v} \neq 1, 3$ and by the \ref{forbidden-weights-rule}, $\pweight{v} \neq 2$. Hence, $\pweight{v} \geq 4$, which verifies the \ref{singleton-leaf-property}. 
    \item The \ref{forbidden-weights-rule} implies the \ref{forbidden-weights-property}. 
    \item The \ref{parity-rule} is identical to the \ref{parity-property}. 
\end{enumerate}
\end{proof}

The above lemma shows that in order to convert a relevant configuration into a \universal configuration, one needs to eliminate all removable pairs of edges. The following lemma shows how this can be done: if a semi-\universal configuration $(F,\pi)$ has a pair of removable edges $u \rightarrow v$ and $u^\prime \rightarrow v^\prime$, then the structure of this configuration can be simplified by removing the block $\{v,v^\prime\}$ from $\pi$. 

\begin{lemma}(Removal Step) Let $(F,\pi)$ be a semi-\universal configuration with at least one pair of removable edges. Then there exist semi-\universal configurations $\{(\iter{F}{i}, \iter{\pi}{i}): i \in \{0, 1, 2, \dotsc, |\pi| - 1\}\}$ with $|\iter{\pi}{i}| = |\pi| - 1$ such that, 
\begin{align*}
    \sum_{\ell \in \colorings{}{\pi}} \gamma(\mPsi; F,\ell) =\sum_{\ell \in \colorings{}{\pi_0}} \gamma(\mPsi; F_0,\ell) -  \sum_{i=1}^{|\pi| - 1}   \sum_{\ell \in \colorings{}{\iter{\pi}{i}}} \gamma(\mPsi; \iter{F}{i},\ell),
\end{align*}
for any matrix $\mPsi \in \R^{\dim \times \dim}$ with $(\mPsi\mPsi\tran)_{ii} = 1$ for any $i \in [\dim]$. Furthermore, if the semi-\universal configuration has exactly one non-trivial root $(|\rootset(F) \backslash \rootset_0(F)| = 1)$, then each of semi-\universal configurations $\iter{F}{i}$ also has exactly one non-trivial root $(|\rootset(\iter{F}{i}) \backslash \rootset_0(\iter{F}{i})| = 1)$.   
\end{lemma}
\begin{proof}
Let $e_\star = u_\star \rightarrow v_\star$ and $e_\star^\prime = u^\prime_\star \rightarrow v^\prime_\star$ be the pair of removable edges. Since $(F,\pi)$ is semi-\universal it satisfies the \ref{sibling-property} and hence $u_\star \neq u_\star^\prime$. Without loss of generality, we assume that $\pi = \{B_1, B_2, \dotsc, B_{|\pi|}\}$ with,\begin{align*}
    B_1 & = \{v_\star,  v_\star^\prime \}, \; \{u_\star, u_\star^\prime\} \subset B_2.
\end{align*}
In particular, note that $|\pi| \geq 2$. Let $k_j$ denote the value of $\ell$ on any node in $B_j$. We can then write $\gamma(\mPsi; F,\ell)$ in terms of $k$ as:
\begin{align*}
    \gamma(\mPsi; F,\ell) \explain{def}{=} \frac{1}{\dim} \prod_{\substack{e \in E \\ e = u \rightarrow v}} (\psi)_{\ell_u\ell_v}^{p_{v}} = \frac{1}{\dim} \prod_{\substack{e \in E \\ e = u \rightarrow v}} \psi_{k_{\pi(u)}k_{\pi(v)}}^{p_{v}}.
\end{align*}
Note that in the product above the index $k_1$ appears exactly twice with the two edges $u_\star \rightarrow v_\star$ and $u_\star^\prime \rightarrow v_\star^\prime$ since $B_1 = \{v_\star, v_\star^\prime\}$ and $v_\star, v_\star^\prime$ are leaves. We extract these two terms:
\begin{align*}
    \gamma(F,\ell) & = \widetilde{\gamma}(F ; (k_2, k_3, \dots , k_{|\pi|})) \cdot \psi_{k_{2} k_1}^2, \\ \widetilde{\gamma}(F ; (k_2, k_3, \dots , k_{|\pi|})) &\explain{def}{=}  \frac{1}{\dim} \prod_{\substack{e \in E \\ e \neq e_\star, e \neq e_\star^\prime}} \psi_{k_{\pi(u)}k_{\pi(v)}}^{p_{{v}}}.
\end{align*}
Now we can compute,
\begin{align*}
     \sum_{ \substack{\ell \in \colorings{}{\pi}} }  \gamma(F,\ell) & =  \sum_{\substack{k_1, k_2, \dotsc, k_{|\pi|}\\ k_i \neq k_j \forall i \neq j}}  \widetilde{\gamma}(F ; (k_2, k_3, \dots , k_{|\pi|})) \cdot \psi_{k_{2} k_1} \psi_{k_{2} k_1}  \\
     & = \sum_{\substack{k_2, \dotsc, k_{|\pi|}\\ k_\alpha \neq k_\beta \forall \alpha \neq \beta}}  \widetilde{\gamma}(F ; (k_2, k_3, \dots , k_{|\pi|})) \cdot \left(\sum_{k_1 \in [\dim]\backslash\{k_2, k_3, \dotsc, k_{|\pi| -1}\}} \psi_{k_{2} k_1}^2 \right) \\
     & = \sum_{\substack{k_2, \dotsc, k_{|\pi|}\\ k_\alpha \neq k_\beta \forall \alpha \neq \beta}}  \widetilde{\gamma}(F ; (k_2, k_3, \dots , k_{|\pi|})) \cdot \left(\sum_{k_1 \in [\dim]} \psi_{k_{2} k_1}^2- \sum_{i=2}^{|\pi|} \psi_{k_2 k_i}^2 \right) \\
     & = \sum_{\substack{k_2, \dotsc, k_{|\pi|}\\ k_\alpha \neq k_\beta \forall \alpha \neq \beta}}  \widetilde{\gamma}(F ; (k_2, k_3, \dots , k_{|\pi|})) \cdot \left(\underbrace{(\mPsi\mPsi\tran)_{k_2,k_2}}_{=1}- \sum_{i=2}^{|\pi|} \psi_{k_2 k_i}^2 \right)
\end{align*}
Hence,
\begin{align} \label{eq:intermediate-decomp}
     \sum_{ \substack{\ell \in \colorings{}{\pi}} }  \gamma(F,\ell) & = \sum_{\substack{k_2, \dotsc, k_{|\pi|}\\ k_\alpha \neq k_\beta \forall \alpha \neq \beta}}  \widetilde{\gamma}(F ; (k_2, k_3, \dots , k_{|\pi|})) - \sum_{i=2}^{|\pi|} \sum_{\substack{k_2, \dotsc, k_{|\pi|}\\ k_\alpha \neq k_\beta \forall \alpha \neq \beta}}  \widetilde{\gamma}(F ; (k_2, k_3, \dots , k_{|\pi|})) \cdot \psi_{k_2 k_i}^2.
\end{align}
Next we define the configurations $(\iter{F}{i}, \iter{\pi}{i})$ for $i = 0, 1, \dotsc, |\pi| - 1$. Suppose that the original decorated forest was given by $F = (V, E, \height{}, \pweight{}, \qweight{})$. Then,
\begin{enumerate}
    \item We define the decorated forest $\iter{F}{0} = (\iter{V}{0}, \iter{E}{0}, \iter{\height{}}{0}, \iter{\pweight{}}{0}, \iter{\qweight{}}{0})$ as follows:
    \begin{enumerate}
        \item We set $\iter{V}{0} = V \backslash \{v_\star, v_\star^\prime\}$.
        \item We set $\iter{E}{0} = E \backslash \{e_\star, e_\star^\prime\}$.
    \end{enumerate}
    This defines a directed graph $(\iter{V}{0}, \iter{E}{0})$. It is straightforward to check that since $F$ was a directed forest, $\iter{F}{0}$ is also a directed forest with root set $\rootset(\iter{F}{0}) = \rootset(F)$. Next we define the functions $\iter{h}{0}, \iter{\pweight{}}{0}, \iter{\qweight{}}{0}$:
    \begin{enumerate}
        \setcounter{enumii}{2}
        \item We set $\iter{\height{}}{0}_v = \height{v}$ for any $v \in \iter{V}{0}$. 
        \item We set $\iter{\pweight{}}{0}_v = \pweight{v}$ for any $v \in \iter{V}{0} \backslash \rootset(\iter{F}{0})$. 
        \item We set  $\iter{\qweight{}}{0}_v = \qweight{v}$ for any $v \in \iter{V}{0} \backslash \{u_\star, u_\star^\prime\}$. and $\iter{\qweight{}}{0}_{u_\star} = \qweight{u_\star} - 1$ and $\iter{\qweight{}}{0}_{u_\star^\prime} = \qweight{u_\star^\prime} - 1$ \footnote{This ensures $\iter{F}{0}$ satisfies the conservation equation \eref{conservation-eq}}.
    \end{enumerate}
    It is straightforward to check that $\iter{F}{0}$ is a valid decorated forest (in the sense of \defref{decorated-forests}). We set $\iter{\pi}{0} = \{B_2, B_3, \dotsc, B_{|\pi|}\}$.
    \item For $i \geq 1$, we set $\iter{F}{i} = F$ and $\iter{\pi}{i} = \{B_2, B_3, \dotsc, B_{i}, B_{i+1} \cup\{v_\star, v_\star^\prime\}, B_{i+2}, \dotsc, B_{|\pi|}\}$.
\end{enumerate}
Observe that \eref{intermediate-decomp} can be re-expressed as:
\begin{align*}
    \sum_{\ell \in \colorings{}{\pi}} \gamma(\mPsi; F,\ell) =\sum_{\ell \in \colorings{}{\pi_0}} \gamma(\mPsi; F_0,\ell) -  \sum_{i=1}^{|\pi| - 1}   \sum_{\ell \in \colorings{}{\iter{\pi}{i}}} \gamma(\mPsi; \iter{F}{i},\ell).
\end{align*}
In order to complete the proof of the claim, we need to verify that the resulting configurations $(\iter{F}{i}, \iter{\pi}{i})$ are semi-\universal. Suppose the configurations are specified by \begin{align*}\iter{F}{i} &= (\iter{V}{i}, \iter{E}{i}, \iter{\height{}}{i}, \iter{\pweight{}}{i}, \iter{\qweight{}}{i}), \\
\iter{\pi}{i} &= \{\iter{B}{i}_1, \iter{B}{i}_2, \dotsc, \iter{B}{i}_{|\iter{\pi}{i}|}\}.
\end{align*} 
We begin by making the following observations for any $i \geq 0$:
\begin{description}
\item [Observation 1: ] For any vertex $v \in \iter{V}{i}$, either $\iter{B}{i}_{\iter{\pi}{i}(v)} = B_{\pi(v)}$ (that is, the block of $v$ is unchanged) or $|\iter{B}{i}_{\iter{\pi}{i}(v)}| \geq 3$ (that is, the block of $v$ has cardinality at least $3$). 
\item [Observation 2: ] It is possible that after the deletion of edges $u_\star \rightarrow v_\star$ and $u_\star^\prime \rightarrow v_\star^\prime$, $u_\star$ or $u_\star^\prime$ become leaves in $\iter{F}{0}$. Hence, $\leaves{\iter{F}{0}} \subset \leaves{F} \cup \{u_\star, u_\star^\prime\}$. However since $\iter{F}{i} = F$ for $i \geq 1$,  $\leaves{\iter{F}{i}} = \leaves{F}$ for $i \geq 1$.
\item [Observation 3: ] $u_\star$ is a leaf in $\iter{F}{0}$ iff $\qweight{u_\star} = 1$. Furthermore, in this situation $\iter{\qweight{}}{0}_{u_\star} = 0$. The same observation holds for $u_\star^\prime$.   
\item [Observation 4: ] It is possible that after the deletion of edges $u_\star \rightarrow v_\star$ and $u_\star^\prime \rightarrow v_\star^\prime$, $u_\star$ or $u_\star^\prime$ become trivial roots in $\iter{F}{0}$. Hence, $\rootset_0{(\iter{F}{0})} \subset \rootset({F}_0) \cup \{u_\star, u_\star^\prime\}$. However since $\iter{F}{i} = F$ for $i \geq 1$, $\rootset_0{(\iter{F}{i})} = \rootset_0({F})$ for $i \geq 1$.
\item [Observation 5: ] $u_\star$ is a trivial root $\iter{F}{0}$ iff $u_\star \in \rootset(F)$ and $\qweight{u_\star} = 1$. Furthermore, in this situation $\iter{\qweight{}}{0}_{u_\star} = 0$. The same observation holds for $u_\star^\prime$.  
\end{description}

In order to check $\{(\iter{F}{i}, \iter{\pi}{i})\}$ is a collection of semi-\universal configurations, we check each of the requirements of \defref{semi-universal-config}:
\begin{enumerate}
    \item Since $\rootset(\iter{F}{i}) = \rootset(F)$ for $i \geq 0$ and the partitions $\iter{\pi}{i}$ are contructed by either deleting a block of $\pi$ (consisting only of leaves $\{v_\star, v_\star^\prime\}$) entirely or by merging two blocks of $\pi$, $(\iter{F}{i}, \iter{\pi}{i})$ automatically satisfy the \ref{root-property} for $i \geq 0$.  
    \item In order to verify the \ref{sibling-property}, for the sake of contradiction assume that  $\iter{\pi}{i}$ has a block of the form $\{u,v\}$ where $\{u,v\}$ are siblings in $\iter{F}{i}$. Observe $\{u,v\}$ are also siblings in $F$. Furthermore, $\{u,v\}$ is also a block of $\pi$ (Observation 2). This contradicts the \ref{sibling-property} of $(F,\pi)$. 
    \item In order to verify the \ref{singleton-leaf-property}, for the sake of contradiction, assume that there is a leaf $v \in \leaves{\iter{F}{i}}$ with $|\iter{B}{i}_{\iter{\pi}{i}(v)}| = 1$ and $\iter{\pweight{}}{i}_v \in \{1,2,3\}$. By Observation 1, $|B_{\pi(v)}| = 1$. Observe that $v \not\in \{u_\star, u_\star^\prime, v_\star, v_\star^\prime\}$ since these vertices belong to blocks of size at least $2$ in $\pi$. Hence $v \in \leaves{F}$ (see Observation 2), $\pweight{v} = \iter{\pweight{}}{i}_v \in \{1,2,3\}$ and $|B_{\pi(v)}| = 1$ (see Observation 1). The existence of such a $v$ contradicts the \ref{singleton-leaf-property} of $(F,\pi)$. 
    \item In order to verify the \ref{forbidden-weights-property}, for the sake of contradiction suppose that there is a $v \in \iter{V}{i} \backslash \rootset(\iter{F}{i})$ such that $|\iter{B}{i}_{\iter{\pi}{i}(v)}| = 1$ and either $(\iter{\pweight{}}{i}_v = 2, \iter{\qweight{}}{i}_v = 0)$ or $(\iter{\pweight{}}{i}_v = 1, \iter{\qweight{}}{i}_v = 1)$.  By Observation 1, $|B_{{\pi}(v)}| = 1$ and hence $v \not\in \{u_\star, u_\star^\prime, v_\star, v_\star^\prime\}$. This means that $\pweight{v} = \iter{\pweight{}}{i}_{v}$ and $\qweight{v} = \iter{\qweight{}}{i}_{v}$. This leads to a contradiction of the \ref{forbidden-weights-property} for $(F,\pi)$.
    \item Finally, we verify the \ref{parity-property}. For the configuration $(\iter{F}{i}, \iter{\pi}{i})$ for $i \geq 1$, we observe that $\iter{F}{i} = F$ and $\iter{\pi}{i}$ is obtained by merging some blocks of $\pi$. Since the \ref{parity-property} is not disturbed by merging some blocks, $(\iter{F}{i}, \iter{\pi}{i})$ also satisfies the \ref{parity-property}. For the configuration $(\iter{F}{0}, \iter{\pi}{0})$, recall that $\iter{\pi}{0} = \{\iter{B}{0}_{1}, \iter{B}{0}_2, \dots, \iter{B}{0}_{|\iter{\pi}{0}|}\}$ where $\iter{B}{0}_i = B_{i+1}$. Observe that for any $i \geq 2$, since $u_\star, u_\star^\prime \not\in \iter{B}{0}_i$,
    \begin{align*}
         \iter{\pweight{}}{0}_u = \pweight{u}, \; \iter{\qweight{}}{0}_u &= \qweight{u}, \; \forall \; u \; \in \; \iter{B}{0}_i = B_{i+1}. 
    \end{align*}
    Furthermore, recalling Observation 2 and 4 we have,
    \begin{subequations} \label{eq:set-relations}
    \begin{align}
        (\iter{B}{0}_i \cap \rootset(\iter{F}{0})) \backslash \rootset_0(\iter{F}{0})  &= ({B}_{i+1} \cap \rootset({F})) \backslash \rootset_0({F}), \\
        \iter{B}{0}_i \cap \leaves{\iter{F}{0}}   &= {B}_{i+1} \cap \leaves{F}, \\
        \iter{B}{0}_i \backslash (\leaves{\iter{F}{0}} \cup \rootset(\iter{F}{0})) & = {B}_{i+1} \backslash (\leaves{{F}} \cup \rootset({F})). 
    \end{align}
    \end{subequations}
    Hence,
    \begin{subequations} \label{eq:verify-parity-property}
    \begin{align}
        &\sum_{r \in (\iter{B}{0}_i \cap \rootset(\iter{F}{0})) \backslash \rootset_0(\iter{F}{0}) } \qweight{r} + \sum_{u \in\iter{B}{0}_i \backslash (\leaves{\iter{F}{0}} \cup \rootset(\iter{F}{0}))} (\iter{\pweight{}}{0}_u + \iter{\qweight{}}{0}_u) +  \sum_{v \in \iter{B}{0}_i \cap \leaves{\iter{F}{0}}} \iter{\pweight{}}{0}_v \nonumber\\& \hspace{2cm} = \sum_{r \in ({B}_{i+1} \cap \rootset({F})) \backslash \rootset_0({F})} \qweight{r} + \sum_{u \in {B}_{i+1} \backslash (\leaves{{F}} \cup \rootset({F})) } (\pweight{u} + \qweight{u}) + \sum_{v \in {B}_{i+1} \cap \leaves{F} } \pweight{v}.
    \end{align}
    The RHS of the above display is even because $(F,\pi)$ satisfies the \ref{parity-property}. Hence we have verified the parity property for all blocks of the configuration $(\iter{F}{0}, \iter{\pi}{0})$ except $\iter{B}{0}_1$. Finally, we verify the \ref{parity-property} for this block. Recall that $\{u_\star, u_\star^\prime\} \subset \iter{B}{0}_1$ and $\iter{\qweight{}}{0}_{u_\star} = \qweight{u_\star} - 1$, $\iter{\qweight{}}{0}_{u_\star^\prime} = \qweight{u_\star^\prime} - 1$. If $u_\star, u_\star^\prime \not\in \leaves{\iter{F}{0}} \cup \rootset_0(\iter{F}{0})$, then the set equalities \eqref{eq:set-relations} continue to hold for $i=1$ and we have,
    \begin{align} \label{eq:parity-property-special-block}
         &\sum_{r \in (\iter{B}{0}_1 \cap \rootset(\iter{F}{0})) \backslash \rootset_0(\iter{F}{0}) } \qweight{r} + \sum_{u \in\iter{B}{0}_1 \backslash (\leaves{\iter{F}{0}} \cup \rootset(\iter{F}{0}))} (\iter{\pweight{}}{0}_u + \iter{\qweight{}}{0}_u) +  \sum_{v \in \iter{B}{0}_1 \cap \leaves{\iter{F}{0}}} \iter{\pweight{}}{0}_v \nonumber\\& \hspace{2cm} = \sum_{r \in ({B}_{2} \cap \rootset({F})) \backslash \rootset_0({F})} \qweight{r} + \sum_{u \in {B}_{2} \backslash (\leaves{{F}} \cup \rootset({F})) } (\pweight{u} + \qweight{u}) + \sum_{v \in {B}_{2} \cap \leaves{F} } \pweight{v} - 2.
    \end{align}
    \end{subequations}
    Infact, as a consequence of Observation 3 and 5, \eref{parity-property-special-block} holds even if $u_\star \in \leaves{\iter{F}{0}}\cup \rootset_0(\iter{F}{0})$ or $u_\star^\prime \in \leaves{\iter{F}{0}}\cup \rootset_0(\iter{F}{0})$ since $\iter{\qweight{}}{0}_{u_\star} = 0$ or $\iter{\qweight{}}{0}_{u_\star^\prime} = 0$ in this situation.  The RHS of \eqref{eq:parity-property-special-block} is even because $(F,\pi)$ satisfies the \ref{parity-property}. Hence we have verified the parity property for all blocks of the configuration $(\iter{F}{0}, \iter{\pi}{0})$.
\end{enumerate}
Finally, consider the situation when the semi-\universal configuration $F$ has exactly one non-trivial root $(|\rootset(F) \backslash \rootset_0(F)| = 1)$. Without loss of generality, we can assume that:
\begin{align*}
    \rootset(F) &= \{1, 2, 3, \dotsc, r\}, \\
    \rootset_0(F) &= \{1, 2, 3, \dotsc, r-1\}.
\end{align*}
Recall that for each $i \in \{1, 2, \dotsc, |\pi|-1\}$ since $\iter{F}{i} = F$, hence, $\iter{F}{i}$ also has exactly one non-trivial root. On the other hand, recall that $\iter{F}{0}$ was obtained by removing exactly the one pair of removable edges $u_\star \rightarrow v_\star$ and $u_\star^\prime \rightarrow v_\star^\prime$ along with the corresponding leaves $\{v_\star, v_\star^\prime\}$. This means that:
\begin{align*}
    \rootset(\iter{F}{0}) &= \rootset(F) = \{1, 2, \dotsc, r\}, \\
    \rootset_0(\iter{F}{0}) &\subset \rootset_0(F) = \{1, 2, \dotsc, r-1\}.
\end{align*}
Specifically, this implies that $\iter{F}{0}$ has at most one non-trivial root. Furthermore, recall that $u_\star \neq u_\star^\prime$ (due to the \ref{sibling-property}) and both $u_\star, u_\star^\prime$ have at least one child $(v_\star, v_\star^\prime)$. In particular, this means that at most one of $u_\star, u_\star^\prime$ can be a root of $F$, since $F$ had only one non-trivial root. Since $\rootset(F) = \rootset(\iter{F}{0})$ at most one of $u_\star, u_\star^\prime$ can be roots in $\iter{F}{0}$. As a consequence $\iter{F}{0}$ also has at least one edge, the edge connecting the non-root vertex among $\{u_\star, u_\star^\prime\}$ to its parent.  This means that $\iter{F}{0}$ has at least one non-trivial root, since a forest with no non-trivial roots has no edges. Hence we have shown that $\iter{F}{0}$ has exactly one non-trivial root, as claimed. This completes the proof of this lemma.  
\end{proof}

\begin{figure}[htb!]
\begin{algbox}
\textbf{Decomposition Algorithm} 

\vspace{1mm}

\textit{Input}: $(F_0,\pi_0)$, a relevant configuration. \\
\textit{Output}: $\calS$: A collection of \universal configurations, and a map $a: \calS \rightarrow \{\pm 1\}$. \\
\textit{Initialization} : $\iter{\calS}{0}:= \{(F_0,\pi_0)\}$, $\iter{a}{0}(F_0,\pi_0) := 1$.
\begin{itemize}
\item For $t \in \{1, 2, 3, \dotsc \}$
\begin{itemize}
    \item $\iter{\calV}{t-1} := \{(F,\pi) \in \iter{\calS}{t-1}: (F,\pi)\text{ has at least one removable edge.} \}$, $v_{t-1}: = |\iter{\calV}{t-1}|$. 
    \item If $v_{t-1} = 0$, end for loop. Otherwise,
    \begin{itemize}
        \item Let $\iter{\calV}{t-1} = \{(F_1,\pi_1), (F_2,\pi_2), \dotsc, (F_{v_{t-1}}, \pi_{v_{t-1}})\}$ be any enumeration of $\iter{\calV}{t-1}$. 
        \item For each $i \in [v_{t-1}]$, decompose configuration $(F_i,\pi_i)$ using \lemref{removal} to obtain configurations $\{(\iter{F_i}{j},\iter{\pi_i}{j}): j \in \{0, 1, \dotsc, |\pi_i|-1\}\}$.
        \item Update:
        \begin{subequations} \label{eq:update-eq-decomposition-algo}
        \begin{align}
            \iter{\calS}{t} &:= (\iter{\calS}{t-1} \backslash \iter{\calV}{t-1}) \cup \left( \bigcup_{i=1}^{v_{t-1}}\{(\iter{F_i}{j},\iter{\pi_i}{j}): j \in \{0, 1, \dotsc, |\pi_i|-1\}\} \right)  \\
            \iter{a}{t}(F,\pi) &:= \iter{a}{t-1}(F,\pi) \; \forall \; (F,\pi) \; \in \; \iter{\calS}{t-1} \backslash \iter{\calV}{t-1}, \\
            \iter{a}{t}(\iter{F}{0}_i,\iter{\pi}{0}_i) &:= \iter{a}{t-1}(F_i,\pi_i) \; \forall \; i \; \in \; [v_{t-1}], \\
            \iter{a}{t}(\iter{F}{j}_i,\iter{\pi}{j}_i) &:= -\iter{a}{t-1}(F_i,\pi_i) \; \forall \; j \; \in \; [|\pi_i|-1], \; \forall \; i \; \in \; [v_{t-1}].
        \end{align}
        \end{subequations}
    \end{itemize}
\end{itemize}
\item \textit{Return} $\calS := \iter{\calS}{t-1}$, $a := \iter{a}{t-1}$.
\end{itemize}
\vspace{1mm}
\end{algbox}
\caption{Decomposing a relevant configuration into a collection of \universal configurations} \label{fig:decomposition-algorithm}
\end{figure}

Now we are ready to prove \propref{decomposition}. The basic idea is that in order to decompose a relevant configuration into a linear combination of \universal configurations, we will repeatedly apply \lemref{removal} to eliminate all pairs of removable edges. This leads to the algorithm shown in \fref{decomposition-algorithm}. \propref{decomposition} follows from the analysis of this algorithm presented in the following lemma. 

\begin{lemma}\label{lem:removal} The algorithm in \fref{decomposition-algorithm} when run on a relevant configuration $(F_0,\pi_0)$ terminates in $t \leq |\pi_0|$ steps and returns a collection $\calS$ of \universal configurations of size $|\calS| \leq |\pi_0|!$ and a map $a: \calS \rightarrow \{\pm 1\}$ such that,
\begin{align*}
    \sum_{\ell \in \colorings{}{\pi_0}} \gamma(\mPsi; F_0, \ell) & = \sum_{(F,\pi) \in \calS} a(F,\pi) \cdot \sum_{\ell \in \colorings{}{\pi}}\gamma(\mPsi ; F, \ell),
\end{align*}
for any matrix $\mPsi \in \R^{\dim \times \dim}$ with $(\mPsi\mPsi\tran)_{ii} = 1$ for all $i \in [\dim]$. Furthermore, if the relevant configuration $(F_0, \pi_0)$ has exactly one non-trivial root $(|\rootset(F_0) \backslash \rootset_0(F_0)| = 1)$, then each of \universal configurations $(F,\pi) \in \calS$ also has exactly one non-trivial root $(|\rootset(F) \backslash \rootset_0({F})| = 1)$. 
\end{lemma}
\begin{proof}
The proof follows from the following sequence of arguments:
\begin{enumerate}
    \item First, we observe that for any $t\geq 0$, $\iter{\calS}{t}$ is a collection of semi-\universal configurations. This is true when $t=0$ since relevant configurations are semi-\universal. Furthermore, since \lemref{removal} is guaranteed to decompose a semi-\universal configuration into other semi-\universal configurations, the claim holds for $t\geq 1$ by induction. 
    \item Second, consider the situation when $(F_0, \pi_0)$ is a relevant configuration with exactly one non-trivial root. Since \lemref{removal} is guaranteed to decompose a semi-\universal configuration with exactly one non-trivial root into other semi-\universal configurations with exactly one non-trivial root, $\calS_t$ consists of configurations with exactly one non-trivial root for $t\geq 1$ by induction.
    \item Third, we observe that for any $t \geq 0$, we have,
    \begin{align} \label{eq:decomp-property-step-t}
    \sum_{\ell \in \colorings{}{\pi_0}} \gamma(\mPsi; F_0, \ell) & = \sum_{(F,\pi) \in \iter{\calS}{t}} \iter{a}{t}(F,\pi) \cdot \sum_{\ell \in \colorings{}{\pi}}\gamma(\mPsi ; F, \ell).
    \end{align}
    This is clearly true when $t=0$. For $t\geq 1$, this can be shown by induction. Suppose that \label{eq:decomp-property} holds at iteration $t$. Let $\iter{\calV}{t} = \{(F_1,\pi_1), (F_2,\pi_2), \dotsc, (F_{v_{t-1}}, \pi_{v_{t}})\}$. Then, by  \lemref{removal},
    \begin{align*}
          \sum_{\ell \in \colorings{}{\pi_i}} \gamma(\mPsi; F_i, \ell) & = \sum_{\ell \in \colorings{}{\iter{\pi}{0}_i}} \gamma(\mPsi; \iter{F}{0}_i, \ell) - \sum_{j=1}^{|\pi_i| - 1}\sum_{\ell \in \colorings{}{\iter{\pi}{j}_i}} \gamma(\mPsi; \iter{F}{j}_i, \ell).
    \end{align*}
    Substituting this in \eref{decomp-property-step-t} and using the update equations \eref{update-eq-decomposition-algo} yields:
    \begin{align*}
       \sum_{\ell \in \colorings{}{\pi_0}} \gamma(\mPsi; F_0, \ell) & = \sum_{(F,\pi) \in \iter{\calS}{t+1}} \iter{a}{t+1}(F,\pi) \cdot \sum_{\ell \in \colorings{}{\pi}}\gamma(\mPsi ; F, \ell) ,
    \end{align*}
    as desired.
    \item Fourth, we observe that if the algorithm terminates at iteration $t$, $\iter{\cal{S}}{t-1}$ consists of semi-\universal configurations with no removable edges. By comparing \defref{universal-configuration}, \defref{semi-universal-config} and \defref{removable-pair}, we immediately conclude that if the algorithm terminates at iteration $t$, $\iter{\calS}{t-1}$ is a collection of \universal configurations. 
    \item Next, we bound the maximum number of iterations the algorithm runs for. To this end, for $t\geq 0$, we define:
    \begin{align*}
        w_t \explain{def}{=} \max \{ |\pi| : (F,\pi) \in \iter{\calV}{t} \}.
    \end{align*}
    Observe that at each iteration, each configuration $(F,\pi) \in \iter{\calV}{t}$ is removed from $\iter{\calS}{t}$ and replaced with several configurations $(F^\prime, \pi^\prime)$ with $|\pi^\prime| = |\pi| - 1$ by application of \lemref{removal}. Hence $w_{t+1} \leq w_t - 1$. Furthermore $w_0 = |\pi_0|$. Hence, $w_t \leq |\pi_0| - t$. Additionally if $\iter{\calV}{t}\neq \varnothing$, then $w_t \geq 2$. This is because any configuration with a removable pair of edges $u \rightarrow v, u^\prime \rightarrow v^\prime$ has at least 2 blocks (one block containing $\{v,v^\prime\}$ and another one containing $\{u,u^\prime\}$). Hence if the algorithm does not terminate at iteration $t$, then $w_{t-1} \geq 2$. Hence the algorithm must terminate by iteration $t \leq |\pi|_0$. 
    \item Finally we bound the cardinality $|\iter{\calS}{t-1}|$ at termination. Using the update equation \eref{update-eq-decomposition-algo}, we obtain the recursion,
    \begin{align*}
        |\iter{\calS}{t}| & \leq |\iter{\calS}{t-1}| - |\iter{\calV}{t-1}| + \sum_{(F,\pi) \in \iter{\calV}{t-1}} |\pi| \\
        & \explain{(a)}{\leq} |\iter{\calS}{t-1}|  + (w_{t-1}-1) \cdot  |\iter{\calV}{t-1}| \\
        & \leq w_{t-1} \cdot |\iter{\calS}{t-1}| \\
        & \leq (|\pi_0| - t + 1) |\iter{\calS}{t-1}| \\
        & \leq |\pi_0| \cdot (|\pi_0| - 1) \cdot \dotsb \cdot (|\pi_0| - t + 1).
    \end{align*}
    In the above display, we used the definition of $w_{t-1}$ in step (a). Since $t \leq |\pi_0|$, we have that $|\calS| \leq |\pi_0|!$ at termination. 
\end{enumerate}
This concludes the proof of this result.

\end{proof}

\section{Proof of \propref{universality-universal-config}}
\label{sec:proof-universal-config-estimate}
In this section, we present the proof of \propref{universality-universal-config}. Consider the following claim regarding the structure of \universal configurations.
\begin{proposition}  \label{prop:universal-config-estimate} For a \universal configuration $(F,\pi)$  with $F = (V, E, \height{}, \pweight{}, \qweight{})$ and $\pi = \{B_1, B_2, \dotsc, B_{|\pi|}\}$, we have,
\begin{align*}
     \exponent(F,\pi) \explain{def}{=}  \frac{|\nullleaves{F,\pi}|}{4} + 1 - |\pi| + \frac{1}{2} \sum_{v \in V \backslash \rootset(F)}  \pweight{v} \geq \frac{|\rootset(F)| - |\rootset_0(F)|}{4}.
\end{align*}
\end{proposition}
We prove \propref{universality-universal-config} from the above claim and \propref{key-estimate}.
\begin{proof}[Proof of \propref{universality-universal-config}]
If a \universal configuration $(F,\pi)$ has a non-empty edge set, then it must have at least one non-trivial root. Hence, $|\rootset(F) \backslash \rootset_0(F)| \geq 1$ and \propref{universal-config-estimate} guarantees that $\exponent(F,\pi) \geq 1/4$. Then, by \propref{key-estimate},  $|\invpoly(\mPsi; F,\pi)| \lesssim \dim^{-1/4 + \epsilon}$. In particular, $\lim_{\dim \rightarrow \infty } \invpoly(\mPsi; F,\pi) = 0$, as claimed in \propref{universality-universal-config}. On the other hand, if the \universal configuration $(F,\pi)$ has no edges, the vertex set of $F$ is simply the set of roots $V = \{1, 2, 3, \dotsc, |\rootset(F)|\}$.  By the \ref{root-property}, $\pi$ is the trivial partition consisting of a single block $\pi = \{\{1, 2, 3, \dotsc, |\rootset(F)|\}\}$. In this case,
\begin{align*}
    \gamma(\mPsi; F,\ell) & \explain{def}{=}  \frac{1}{\dim} \prod_{\substack{e \in E \\ e = u \rightarrow v}} \psi_{\ell_u\ell_v}^{\pweight{v}} = \frac{1}{\dim}, 
\end{align*}
and,
\begin{align*}
    \invpoly(\mPsi; F,\pi) = \sum_{\ell \in \colorings{}{\pi}}   \gamma(\mPsi; F,\ell) & = \sum_{\ell=1}^\dim \frac{1}{\dim} = 1,
\end{align*}
which completes the proof of \propref{universality-universal-config}.
\end{proof} 

The remainder of the section is devoted to the proof of \propref{universal-config-estimate}. Since $(F,\pi)$ satisfy the \ref{root-property}, all roots of $F$ lie in a single block of $\pi$. Without loss of generality, throughout this section, we will assume that this block is $B_1$: $\rootset(F) \subset B_1$. We can rewrite $\exponent(F,\pi)$ as follows:
\begin{align}
     \exponent(F,\pi) &= 1+ \left(\frac{1}{2} \sum_{v \in V \backslash  \rootset(F)} \pweight{v} \right) - |\pi| + \frac{|\nullleaves{F,\pi}|}{4}  \nonumber \\
     & \explain{(a)}{=} \left( \frac{1}{2} \sum_{v \in B_1 \backslash \rootset(F)} \pweight{v}  \right) + \sum_{i=2}^{|\pi|} \left( - 1 + \frac{1}{2}\sum_{v \in B_i}  \pweight{v} \right) + \frac{|\nullleaves{F,\pi}|}{4}  \nonumber \\
     & \explain{(b)}{=} \sum_{v \in V} \effect{v} + \frac{|\nullleaves{F,\pi}|}{4}. \label{eq:exponent-new-formula}
\end{align}
In the above display, in step (a) we regrouped the sum over vertices $v \in V \backslash \rootset(F)$ by the blocks of the partition $\pi$. In order to obtain (b), we define the function $\effect{v}$ for each $v \in V$ as follows:
\begin{align}\label{eq:effect-def}
    \effect{v} \explain{def}{=} \begin{cases} \frac{\pweight{v}}{2} - \frac{1}{|B_{\pi(v)}|} &: v \in V \backslash B_1  \\ \frac{\pweight{v}}{2} &: v \in B_1 \backslash \rootset(F) \\ 0 &: v \in \rootset(F)  \end{cases}
\end{align}
The advantage of the formula \eqref{eq:exponent-new-formula} is that it expresses the exponent $\exponent(F,\pi)$ in terms of local contributions of each node:
\begin{enumerate}
    \item Nullifying leaves and vertices with $\effect{v} > 0$ have a positive effect on $\exponent(F,\pi)$.
    \item Vertices with $\effect{v}<0$ have a negative contribution to $\exponent(F,\pi)$. 
\end{enumerate}
The proof of \propref{universal-config-estimate} uses structural properties of \universal configurations to argue that there are enough nullifying leaves and vertices with $\effect{v} > 0$ to cancel out the negative effect of vertices with $\effect{v}<0$. The proof is organized as follows:
\begin{enumerate}
    \item In \sref{universal-forest-classification}, we develop a classification for vertices and leaves of a \universal configuration. This will set up the necessary terminology to state various structural properties of \universal configurations.
    \item In \sref{universal-forest-structural-prop}, we collect some structural properties of \universal configurations. 
    \item Finally, in \sref{universal-forest-final-proof}, we provide a proof of \propref{universal-config-estimate} using the structural properties derived in the previous section. 
\end{enumerate}
\subsection{Classification of Vertices}\label{sec:universal-forest-classification} Recall the definition $\effect{v}$ from \eqref{eq:effect-def}. 
Observe that since $\pweight{v} \in \{1, 2, 3, \dotsc, \}$ and $|B_{\pi(v)}| \in \{1, 2, 3, \dotsc, \}$, the possible values for $\effect{v}$ in increasing order are:
\begin{align} \label{eq:effect-possibilities}
    \effect{v} \in \left\{ - \frac{1}{2}, 0, \frac{1}{6}, \frac{1}{4}, \frac{3}{10}, \dotsc \right\}
\end{align}
We introduce the following classification of vertices of a \universal configuration based on the value of $\effect{v}$.

\begin{definition}[Vertex Classification]\label{def:vertex-classification} For a \universal configuration $(F,\pi)$ with $F = (V, E, \height{}, \pweight{}, \qweight{})$, $\pi=\{B_1, B_2, \dotsc, B_{|\pi|}\}$, and $\rootset(F) \subset B_1$ we define the following categories of vertices:
\begin{description}
\item [Excellent vertex:] An excellent vertex is a non-root vertex $v \in V \backslash \rootset(F)$ with $\effect{v} \geq 1/4$. We denote the set of all excellent vertices by $\enodes$. 
\item [Good Vertex:] A good vertex is a non-root vertex $v \in V \backslash B_1$ with $\effect{v} = 1/6$, $\pweight{v} = 1$ and $|B_{\pi(v)}| = 3$. We denote the set of all good vertices by $\gnodes$.
\item [Ok vertex:] A ok vertex is a vertex $v \in V$ with $\effect{v} = 0$. We denote the set of all ok vertices by $\onodes$. Ok vertices must be exactly one of the following mutually exclusive types:
\begin{enumerate}
    \item A root vertex $v \in \rootset(F)$.
    \item A non-root vertex $v \in V \backslash B_1$ with $\pweight{v} = 2, |B_{\pi(v)}|=1$. 
    \item A non-root vertex $v \in V \backslash B_1$ with $\pweight{v} = 1, |B_{\pi(v)}|=2$.
\end{enumerate}
\item [Bad vertex:] A bad vertex is a non-root vertex $v \in V \backslash B_1$ with $\effect{v} = -1/2$ and $(\pweight{v} = 1, |B_{\pi(v)}| = 1)$. We denote the set of all bad vertices by $\bnodes$.
\end{description}
\end{definition}

\begin{lemma}\label{lem:vertex-classification} Let $(F,\pi)$ be a \universal configuration with $F = (V, E, \height{}, \pweight{}, \qweight{})$ and $\pi = \{B_1, \dotsc, B_{|\pi|}\}$ such that $\rootset(F) \subset B_1$. Then, the sets $\enodes, \gnodes, \onodes, \bnodes$ form a partition of $V$. 
\end{lemma}
\begin{proof}
Recalling \eqref{eq:effect-possibilities},  we consider the various possibilities for $\effect{v}$ for a vertex $v$:
\begin{enumerate}
   \item $\effect{v} \geq 1/4$: Observe that in this situation $v \notin \rootset(F)$ (since otherwise $\effect{v} = 0)$. Hence, this vertex is an excellent vertex.  
   \item $\effect{v} = 1/6$: Observe that in this situation $v \notin \rootset(F)$ (since otherwise $\effect{v}=0$) and $v \notin B_1 \backslash \rootset(F)$ (since otherwise $\effect{v} \geq 1/2$). Hence $v \in V \backslash B_1$. In this situation $\effect{v}=1/6$ iff $\pweight{v} = 1$ and $|B_{\pi(v)}| = 3$. Hence, this vertex is a good vertex. 
   \item $\effect{v} = 0:$ Such a vertex satisfies the definition of an ok vertex. Furthermore,  in this situation, we can consider the following exhaustive subcases:
   \begin{enumerate}
       \item $v \in \rootset(F)$.
       \item $v \in B_1 \backslash \rootset(F)$: This is not possible since in this case $\effect{v} \geq 1/2$.
       \item $v \in V \backslash B_1$: In this case $\effect{v} = 0$ iff $(\pweight{v}=2, |B_{\pi(v)}| = 1)$ or $(\pweight{v}=1, |B_{\pi(v)}| = 2)$.
   \end{enumerate}
   These are precisely the 3 possibilities for ok vertices mentioned in \defref{vertex-classification}.
   \item $\effect{v} = -1/2$: Observe that $v \not\in B_1$ (since otherwise, $\effect{v} \geq 0$). In this situation $\effect{v} = -1/2$ happens iff $\pweight{v} = 1$ and $|B_{\pi(v)}|=1$. This vertex satisfies the definition of a bad node. 
\end{enumerate}
The above possibilities are disjoint and exhaustive and hence $\enodes, \gnodes, \onodes, \bnodes$ form a partition of $V$. 
\end{proof}
Similarly, we develop a classification of the leaves of a configuration. 
\begin{definition}[Leaf Classification]\label{def:leaf-classification} For a \universal configuration $(F,\pi)$ with $F = (V, E, \height{}, \pweight{}, \qweight{})$, $\pi = \{B_1, B_2, \dotsc, B_{|\pi|}\}$, and $\rootset(F) \subset B_1$, we define the following categories of leaves:
\begin{description}
\item [Excellent leaves:] An excellent leaf is a leaf vertex $v \in \leaves{F}$ with $\effect{v} \geq 1/4$. The set of all excellent leaves is denoted by $\eleaves$.
\item [Good leaves:] A good leaf is a leaf vertex $v \in \leaves{F} \backslash B_1$ with $p_v =1, |B_{\pi(v)}| = 3$. The set of all good leaves is denoted by $\gleaves$. 
\item [Type-1 leaves:] A type-1 leaf is a leaf vertex $v \in \leaves{F} \backslash B_1$ with $\pweight{v} = 1$ and $B_{\pi(v)} = \{v,v^\prime\}$ where $v^\prime \in V \backslash B_1$ is a non-root vertex with $\pweight{v^\prime} \geq 2$. The set of all type-1 leaves is denoted by $\tileaves$.
\item [Type-2 leaves:] A type-2 leaf is a leaf vertex $v \in \leaves{F} \backslash B_1$ with $\pweight{v} = 1$ and $B_{\pi(v)} = \{v,v^\prime\}$ where $v^\prime \in V \backslash B_1$ is a non-root vertex with exactly one child ($\chrn{v^\prime} = 1$) and satisfies $\pweight{v^\prime} = 1$. The set of all type-2 leaves is denoted by $\tiileaves$.
\item [Type-3 leaves] A type-3 leaf is a leaf vertex $v \in \leaves{F} \backslash B_1$ with $\pweight{v} = 1$ and $B_{\pi(v)} = \{v,v^\prime\}$ where $v^\prime \in V \backslash B_1$ is a non-root vertex with two or more children ($\chrn{v^\prime} \geq 2$) and satisfies $\pweight{v^\prime} = 1$. The set of all type-3 leaves is denoted by $\tiiileaves$.
\item [Nullifying leaves:] See \defref{nullifying}. The set of all nullifying leaves is denoted by $\nleaves{}$. 
\end{description}
\end{definition}
\begin{lemma}\label{lem:leaf-classification} For a \universal configuration, the sets $\eleaves, \gleaves, \tileaves, \tiileaves, \tiiileaves, \nleaves$ form a partition of $\leaves{F}$, the set of all leaves of $F$. 
\end{lemma}
\begin{proof}
Recall that the sets of excellent, good, ok and bad nodes form a partition of $V$. Hence we consider the following mutually exclusive and exhaustive cases for a leaf $v \in \leaves{F}$:
\begin{enumerate}
    \item $v$ is an excellent vertex: such leaves satisfy the definition of excellent leaves.
    \item $v$ is a good vertex: Such leaves satisfy the definition of good leaves.
    \item $v$ is an ok vertex: Recall the 3 different categories of ok vertices from \defref{vertex-classification}. Since $v$ is a leaf, $v \not\in \rootset(F)$. This means that $v \in V \backslash B_1$ and either $\pweight{v} = 2, |B_{\pi(v)}| = 1$ or $\pweight{v} = 1, |B_{\pi(v)}| = 2$. However, since $(F,\pi)$ is a \universal configuration, by the \ref{singleton-leaf-property}, any leaf with $|B_{\pi(v)}| = 1$ must have $\pweight{v} \geq 4$. Hence it must be that $\pweight{v} = 1, |B_{\pi(v)}| = 2$. Let $B_{\pi(v)} = \{v,v^\prime\}$. Observe that since $B_{\pi(v)} = B_{\pi(v^\prime)} \neq B_1, \; v^\prime \in V \backslash B_1$. In particular $v^\prime$ is a non-root vertex. We consider the following mutually exclusive and exhaustive sub-cases (depending on the classification of $v^\prime$). 
    \begin{enumerate}
        \item $v^\prime$ is an excellent node. Since $v^\prime \in V \backslash B_1$ and $|B_{\pi(v^\prime)}| = 2$ and $\effect{v^\prime} \geq 1/4$ we must have $\pweight{v^\prime}  = 2(\effect{v^\prime} + 1/|B_{\pi(v^\prime)}|) \geq 3/2$. Since $\pweight{v^\prime} \in \N$, it must be that $\pweight{v^\prime}\geq 2$. Hence, in this $v$ satisfies the definition of a type-1 leaf. 
        \item $v^\prime$ cannot be a good node since $|B_{\pi(v^\prime)}| = 2$. 
        \item $v^\prime$ is a ok node. Since $v^\prime \in V \backslash B_1$, it is not a root vertex. Furthermore, since $|B_{\pi(v^\prime)}| = 2$, we must have $\pweight{v^\prime} = 1$. We consider the two further mutually exclusive and exhaustive sub-cases depending on whether $v^\prime$ is a leaf or not:
        \begin{enumerate}
            \item $v^\prime$ is a leaf: Let $u,u^\prime$ denote the parents of $v,v^\prime$ respectively. By the \ref{paired-leaf-property}, we must have $\pi(u) \neq \pi(u^\prime)$. Hence $u \rightarrow v$ and $u^\prime \rightarrow v^\prime$ are a pair of nullifying edges (recall \defref{nullifying}) and $v$ becomes a nullifying leaf in this case.
            \item $v^\prime$ is not a leaf: Since $v^\prime$ is not a root either, it must be that $\chrn{v^\prime} \geq 1$. If $\chrn{v^\prime} = 1$, then $v$ satisfies all the requirements of a type-2 leaf in this case. If $\chrn{v^\prime} \geq 2$, then $v$ satisfies all the requirements of a type-3 leaf. 
        \end{enumerate}
        \item $v^\prime$ is a bad node: this is not possible since $|B_{\pi(v^\prime)}| = 2$. 
    \end{enumerate}
    \item $v$ is a bad node: This is not possible since bad nodes have $\pweight{v} = 1, |B_{\pi(v)}| = 1$. However by the \ref{singleton-leaf-property}, any leaf with $|B_{\pi(v)}| = 1$ must have $\pweight{v} \geq 4$.  
\end{enumerate}
Since the above cases were mutually exclusive and exhaustive, the sets $\eleaves, \gleaves, \tileaves, \tiileaves, \nleaves$ form a partition of the set of all leaves of $F$. 
\end{proof}

\subsection{Structural Properties of \universal forests} \label{sec:universal-forest-structural-prop}
In this subsection, we collect some useful structural properties of \universal forests. The following lemma provides a lower bound on $\exponent(F,\pi)$ in a \universal configuration.

\begin{lemma}\label{lem:exponent-lb} For a \universal configuration $(F,\pi)$, we have,
\begin{align*}
    \exponent(F,\pi) & \geq \frac{|\rootset(F)| - |\rootset_0(F)|}{4} + \frac{\onodes_{\geq 2}}{4} - \frac{(|\eleaves| + |\gleaves| + |\tileaves| + |\tiileaves| + |\tiiileaves|+ |\bnodes_{2}|)}{4} - \frac{|\bnodes_1|}{2} + \sum_{v \in \onodes \cup \gnodes \cup \enodes} \effect{v},
\end{align*}
where,
\begin{align*}
     \onodes_{\geq 2} \bydef \{v \in \onodes: \chrn{v} \geq 2\}, \bnodes_1 &\bydef \{v \in \bnodes: \chrn{v} = 1\},  \bnodes_{2} \bydef \{v \in \bnodes: \chrn{v} = 2\}.
\end{align*}
\end{lemma}
\begin{proof}
Recall that for any undirected graph, the sum of degrees of the vertices is twice the number of edges. We can apply this to the forest $F$ by viewing it as an undirected graph. Let $\mathsf{degree}(v)$ denote the degree of a vertex $v$ in the forest $F$ (when viewed as an undirected graph). Let $\chrn{v}$ denote the number of children of $v$ (when $F$ is viewed as a directed forest). Observe that for any $v \not\in \rootset(F)$, $\mathsf{degree}(v) = \chrn{v} + 1$ since $v$ has $\chrn{v}$ children and 1 parent. On the other hand, for any root vertex $r\in \rootset(F)$, $\mathsf{degree}(r) = \chrn{r}$ since a root node has no parent. Hence by the degree-sum formula,
\begin{align*}
    2 |E| & = -|\rootset(F)| + \sum_{v\in V} (\chrn{v}+1),
\end{align*}
where the term $-|\rootset(F)|$ ``corrects'' for the exceptional nature of root vertices. Recall that in a forest $|E| = |V| - |\rootset(F)|$. Hence,
\begin{align*}
    |V| & = |\rootset(F)| +  \sum_{v\in V} \chrn{v}. 
\end{align*}
Next, we categorize vertices according to the number of children they have. For each $i \in \W$, we define:
\begin{align*}
    V_i &\explain{def}{=} \{v\in V: \chrn{v} = i\}, \\
    \onodes_i &\bydef \{v \in \onodes: \chrn{v} = i\},  \onodes_{\geq i} \bydef \{v \in \onodes: \chrn{v} \geq i\}, \\
    \bnodes_i &\bydef \{v \in \bnodes: \chrn{v} = i\},  \bnodes_{\geq i} \bydef \{v \in \bnodes: \chrn{v} \geq i\}.
\end{align*}
Hence,
\begin{align*}
    \sum_{i=0}^\infty |V_i| = |\rootset(F)| + \sum_{i=0}^\infty i |V_i|.
\end{align*}
Recall that $\rootset_0(F)$ is the set of all trivial roots (i.e. roots with no children). Since $V_0 = \leaves{F} \cup \rootset_0(F)$ and $\leaves{F} \cap \rootset_0(F) = \emptyset$, we have $|V_0| = |\leaves{F}| + |\rootset_0(F)|$. Using this  and rearranging the expression in the previous display gives:
\begin{align*}
    |\leaves{F}|& = |\rootset(F)| - |\rootset_0(F)| + \sum_{i=2}^\infty (i-1) |V_i| \\
    & \geq |\rootset(F)| - |\rootset_0(F)| + |\onodes_{\geq 2}| + |\bnodes_{2}| + 2|\bnodes_{\geq 3}|.
\end{align*}
Appealing to \lemref{leaf-classification}, 
\begin{align*}
    |\leaves{F}|& = |\eleaves| + |\gleaves| + |\tileaves| + |\tiileaves| + |\tiiileaves| + |\nleaves|. 
\end{align*}
Hence,
\begin{align} \label{eq:nleaves-lb}
    |\nleaves| & \geq |\rootset(F)| - |\rootset_0(F)| + |\onodes_{\geq 2}| + |\bnodes_{2}| + 2|\bnodes_{\geq 3}| - ( |\eleaves| + |\gleaves| + |\tileaves| + |\tiileaves|+ |\tiiileaves|). 
\end{align}
Recalling the formula for $\exponent(F,\pi)$ in \eref{exponent-new-formula}, we have,
\begin{align*}
    &\exponent(F,\pi) = \frac{|\nullleaves{F,\pi}|}{4} + \sum_{v \in V}  \effect{v}    \\
    & \explain{(a)}{\geq} \frac{|\rootset(F)| - |\rootset_0(F)|}{4} + \frac{|\onodes_{\geq 2}| + |\bnodes_{2}| + 2|\bnodes_{\geq 3}|}{4} - \frac{(|\eleaves| + |\gleaves| + |\tileaves| + |\tiileaves| + |\tiiileaves|)}{4} +\sum_{v \in V} \effect{v}  \\
    & \explain{(b)}{=}\frac{|\rootset(F)| - |\rootset_0(F)|}{4} + \frac{|\onodes_{\geq 2}| + |\bnodes_{2}| + 2|\bnodes_{\geq 3}|}{4} - \frac{(|\eleaves| + |\gleaves| + |\tileaves| + |\tiileaves| + |\tiiileaves|)}{4} - \frac{|\bnodes|}{2} + \sum_{v \in \onodes \cup \gnodes \cup \enodes} \effect{v} \\
    & \explain{(c)}{=} \frac{|\rootset(F)| - |\rootset_0(F)|}{4} + \frac{\onodes_{\geq 2}}{4} - \frac{(|\eleaves| + |\gleaves| + |\tileaves| + |\tiileaves| + |\tiiileaves|+ |\bnodes_{2}|)}{4} - \frac{|\bnodes_1|}{2} + \sum_{v \in \onodes \cup \gnodes \cup \enodes} \effect{v}.
\end{align*}
In the above display, step (a) relies on \eref{nleaves-lb}, step (b) uses the fact (cf. \defref{vertex-classification}) that $\effect{v}= -1/2 \; \forall \; v \; \in \; \bnodes$. Finally in step (c) we used the fact that $\bnodes_0, \bnodes_1, \; \bnodes_2, \bnodes_{\geq 3}$ form a partition of $\bnodes$ and hence $|\bnodes| = |\bnodes_0| +|\bnodes_1| + |\bnodes_2| +  |\bnodes_{\geq 3}|$. Furthermore $|\bnodes_0| = 0$ since nodes with zero children are either trivial roots or leaves, but, there are no bad roots (cf. \defref{vertex-classification}) and no bad leaves (cf. \defref{leaf-classification} and \lemref{leaf-classification}).
\end{proof}
We observe that in the lower bound derived in \lemref{exponent-lb}, vertices in the sets $\enodes, \gnodes, \onodes_{\geq 2}$ seem to have a positive effect whereas vertices in the sets $\eleaves, \gleaves, \tileaves, \tiileaves, \tiiileaves, \bnodes_1, \bnodes_2$ have a negative effect. We will argue that the positive effect is sufficient to cancel out the negative effect. In order to do this sytematically, we will find it helpful to construct an \emph{injective} map:
\begin{align*}
    \map: \bnodes_1 \cup \bnodes_2 \cup \tiileaves \cup \tiiileaves \rightarrow \enodes \cup \onodes_{\geq 2} \backslash \rootset(F). 
\end{align*}
We will use this map to show that for each node $v \in \bnodes_1 \cup \bnodes_2 \cup \tiileaves \cup \tiiileaves$, there is a node $\map(v)$ whose positive effect is sufficient to cancel out the negative effect of $v$. In order to construct $\map$, we will find the following structural property of \universal configurations useful.

\begin{lemma}\label{lem:structural-prop} Consider a \universal configuration $(F,\pi)$ with $F = (V, E, \height{}, \pweight{}, \qweight{})$ and $\pi=\{B_1, B_2, \dotsc, B_{|\pi|}\}$ such that $\rootset(F) \subset B_1$. Let $v_0 \in V \backslash \rootset(F)$ be a non-root vertex such that $\pweight{v_0} \geq 2$. Then there is a $t \in \N_0$ and a path $v_0 \rightarrow v_1 \rightarrow v_2 \rightarrow \dotsb \rightarrow v_t$ with the following properties:
\begin{enumerate}
    \item $v_i \in V \backslash B_1, \; \pweight{v_i} = 2, \; |B_{\pi(v_i)}| = 1, \; \chrn{v_i} = 1$ for all $0 \leq i \leq t-1$.
    \item $v_t \in V \backslash \rootset(F)$ with $\pweight{v_t} \geq 2$. 
    \item Exactly one of the following is true:
    \begin{enumerate}
        \item $\effect{v_t} \geq 1/2$ 
        \item $v_t \in V\backslash B_1$ and $\pweight{v_t} = 2, |B_{\pi(v_t)}| = 1, \chrn{v_t} \geq 2$.
    \end{enumerate}
\end{enumerate}
\end{lemma}
\fref{path} illustrates the claim of this lemma. 
\begin{figure}[ht]
\centering
\includegraphics[width=0.8\textwidth,right]{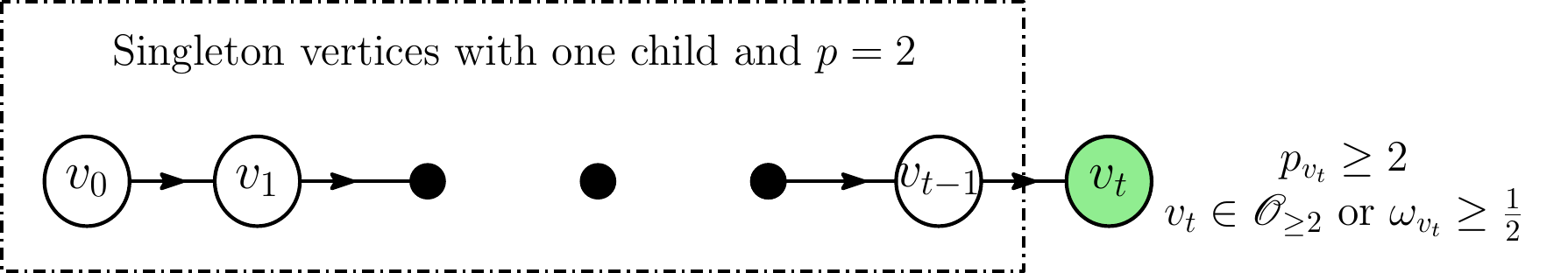}
\caption{A vertex $v_0 \in V \backslash \rootset(F)$ with $\pweight{v_0} \geq 2$ is connected to a vertex $v_t$ with the properties $(\pweight{v_t} \geq 2)$ and $(v_t \in \onodes_{\geq 2} \text{ or } \effect{v_t} \geq 1/2)$ via a path $v_0 \rightarrow v_1 \rightarrow \dotsb \rightarrow v_{t-1}$ consisting of singleton vertices $(|B_{\pi(v_i)}| = 1)$ with $\pweight{v_i} = 2$ that have exactly one child. Vertex colors do not represent the blocks of $\pi$.}  
\label{fig:path}
\end{figure}

\begin{proof}
We begin by making the following claim.
\paragraph{Claim.}  Let $u \in V \backslash \rootset(F)$ be a non-root vertex with $\pweight{u} \geq 2$. Then exactly one of the following is true:
\begin{description}
    \item [Case 1.] $\effect{u} \geq 1/2$
    \item [Case 2.] $u \in V \backslash B_1$ with $\pweight{u}=2, |B_{\pi(u)}| = 1$ and  $u$ has at least two children.
    \item [Case 3.] $u \in V \backslash B_1$ with $\pweight{u}=2, |B_{\pi(u)}| = 1$ and $u$ has exactly one child $v \in V\backslash \rootset(F)$ which satisfies $\pweight{v} \geq 2$. 
\end{description}
Before we present the proof of this claim, we use it to prove the lemma. Consider the following procedure which constructs the required path starting from the given vertex $v_0 \in V \backslash \rootset(F)$ with $\pweight{v_0} \geq 2$. By the claim, the following cases are exhaustive:
\begin{enumerate}
    \item If $v_0$ satisfies Case 1 of the claim, we have  $\effect{v_0} \geq 1/2$, we can terminate the procedure and set $t = 0$ and return trivial path $v_0$ (consisting of a single vertex). This path has no edges and satisfies the claimed properties. 
    \item If $v_0$ satisfies Case 2 of the claim, then $v_0 \in V \backslash B_1$ with $\pweight{v_0}=2, |B_{\pi(v_0)}| = 1$ and  $v_0$ has at least two children. We can again terminate the procedure and set $t = 0$ and return trivial path $v_0$ (consisting of a single vertex). This path has no edges and satisfies the claimed properties. 
    \item Otherwise $v_0$ must satisfy Case 3 of the claim: $v_0 \in V \backslash B_1$ with $\pweight{v_0}=2, |B_{\pi(v_0)}| = 1$ and  has a single child, which we label $v_1$, with the properties $v_1 \in V \backslash \rootset(F)$ and $\pweight{v_1} \geq 2$. In particular, the claim can now be applied to $v_1$ to continue the procedure recursively:
    \begin{enumerate}
            \item If $v_1$ satisfies Case 1 of the claim, we terminate the procedure and set $t = 1$ and return the path $v_0 \rightarrow v_1$.
            \item If $v_1$ satisfies Case 2 of the claim, we terminate the procedure and set $t = 1$ and return the path $v_0 \rightarrow v_1$.
            \item If not, then the claim guarantees that $v_1$ must have exactly one child, which we label $v_2$, which satisfies $v_2 \in V \backslash \rootset(F)$ and  $\pweight{v_2}\geq 2$. We continue this procedure by applying the claim to $v_2$.
    \end{enumerate}
\end{enumerate}
As we execute the above procedure we construct a path $v_0 \rightarrow v_1 \rightarrow v_2 \rightarrow \dotsb$. Since the forest is finite, the above procedure must terminate at some $t$ and we would have a path $v_0 \rightarrow v_1 \rightarrow v_2 \dotsb \rightarrow v_t$. Observe that $\pweight{v_i} \geq 2$ for each node on the path. In particular claim (2) of the lemma is verified. Furthermore, since the path terminated at $t$, either item (3a) or (3b) in the statement of the lemma must hold since these are the only conditions under which the procedure terminates. Furthermore since the procedure did not terminate at step $i \leq t-1$, applying the claim to $v_i$, we must have $v_i$ satisfies Case 3 of the claim which verifies item (1) of the lemma. This concludes the proof of the lemma. We now provide the proof of the claim.
\paragraph{Proof of Claim. } Consider a non-root vertex $u \in V \backslash \rootset(F)$ with $\pweight{u} \geq 2$. Consider the two exhaustive cases:
\begin{enumerate}
    \item $\effect{u} \geq 1/2$. This is Case 1 of the claim. 
    \item $\effect{u} < 1/2$. Observe that $u \not \in B_1 \backslash \rootset(F)$ since otherwise by \eqref{eq:effect-def}, $\effect{u} \geq 1/2$. Hence $u \in V \backslash B_1$. This means that $\effect{u} = \pweight{u}/2 - 1/|B_{\pi(u)}| < 1/2$. However since $\pweight{u} \geq 2$, it must be that $\pweight{u} = 2$ and $|B_{\pi(u)}| = 1$. Observe that $u$ cannot be a leaf (otherwise, the \ref{singleton-leaf-property} would imply $\pweight{u} \geq 4$). Hence $\chrn{u} \geq 1$. Furthermore, by the \ref{parity-property} $\qweight{u}$ must be even and by \ref{forbidden-weights-property}, $\qweight{u} \neq 0$. Hence $\qweight{u} \geq 2$. Consider the following 3 exhaustive sub-cases. 
    \begin{enumerate}
        \item $u$ has two or more children: this leads to Case 2 of the claim. 
        \item $u$ has one child denoted by $v$. Hence, by the conservation equation (cf. \eqref{eq:conservation-eq} in \defref{decorated-forests}) $\pweight{v} = \qweight{u} \geq 2$. This leads to Case 3 of the claim. 
    \end{enumerate}
\end{enumerate}
Since the above case analysis is exhaustive, the claim has been proved.
\end{proof}
We are now ready to construct the map $\map$.

\begin{lemma}\label{lem:map-construction} For any \universal configuration $(F,\pi)$ with $F = (V, E, \height{}, \pweight{}, \qweight{})$, $\pi=\{B_1, B_2, \dotsc, B_{|\pi|}\}$, and $\rootset(F) \subset B_1$ there is a map $\map: \bnodes_1 \cup \bnodes_2 \cup \tiileaves \cup \tiiileaves \rightarrow \enodes \cup \onodes_{\geq 2} \backslash \rootset(F)$ with the following properties:
\begin{enumerate}
    \item If $u \in \bnodes_1$, then $v \bydef \map(u)$ satisfies $v \in \enodes$ with $\pweight{v} \geq 3$. 
    \item If $u \in \bnodes_2 \cup \tiileaves $, then $v \bydef \map(u)$ satisfies $v \in V \backslash \rootset(F)$ with $\pweight{v} \geq 2$ and exactly one of the following properties:
    \begin{enumerate}
        \item $v \in \enodes$ with $\effect{v} \geq 1/2$.
        \item $v \in \onodes_{\geq 2} \backslash B_1$ with $\pweight{v} = 2, \; |B_{\pi(v)}| = 1, \chrn{v} \geq 2$.
    \end{enumerate}
    \item If $u \in \tiiileaves$, then $v \bydef \map(u)$ satisfies $v \in \onodes_{\geq 2} \backslash B_1$ with $\pweight{v} = 1, |B_{\pi(v)}| = 2$ and $\chrn{v} \geq 2$.
    \item The map $\map$ is injective. 
\end{enumerate}
\end{lemma}
\begin{figure}
    \begin{subfigure}[b]{0.49\textwidth}
         \centering
         \includegraphics[width=0.5\textwidth]{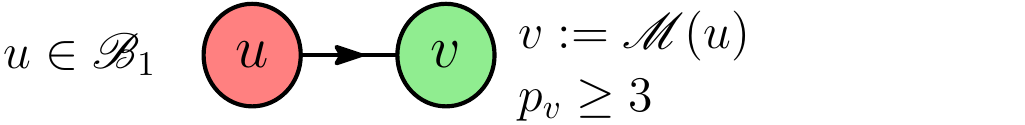}
         \caption{$u \in \bnodes_1$ (Case 1).}
         \label{fig:map-b1}
     \end{subfigure}
     \begin{subfigure}[b]{0.49\textwidth}
         \centering
         \includegraphics[width=0.5\textwidth]{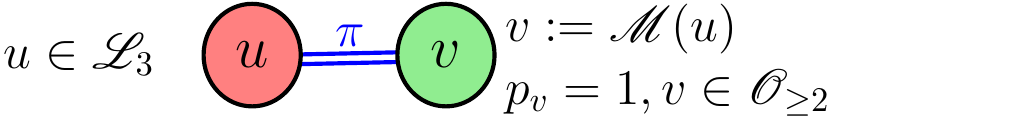}
         \caption{$u \in \tiiileaves$ (Case 4).}
         \label{fig:map-l3}
     \end{subfigure}
      \par\bigskip
     \begin{subfigure}[b]{0.49\textwidth}
         \includegraphics[width=\textwidth]{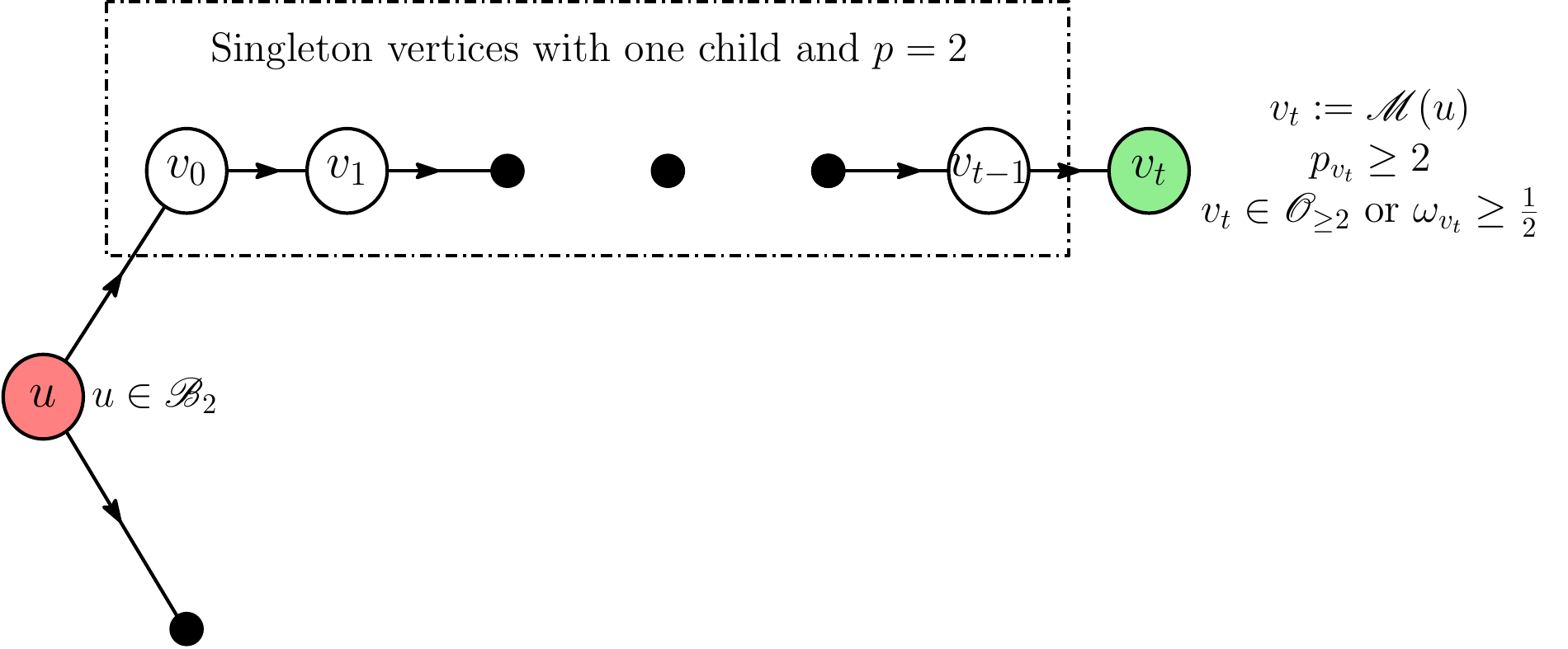}
         \caption{$u \in \bnodes_2$ (Case 2).}
         \label{fig:map-b2}
     \end{subfigure}
    \begin{subfigure}[b]{0.49\textwidth}
         \centering
         \includegraphics[width=\textwidth]{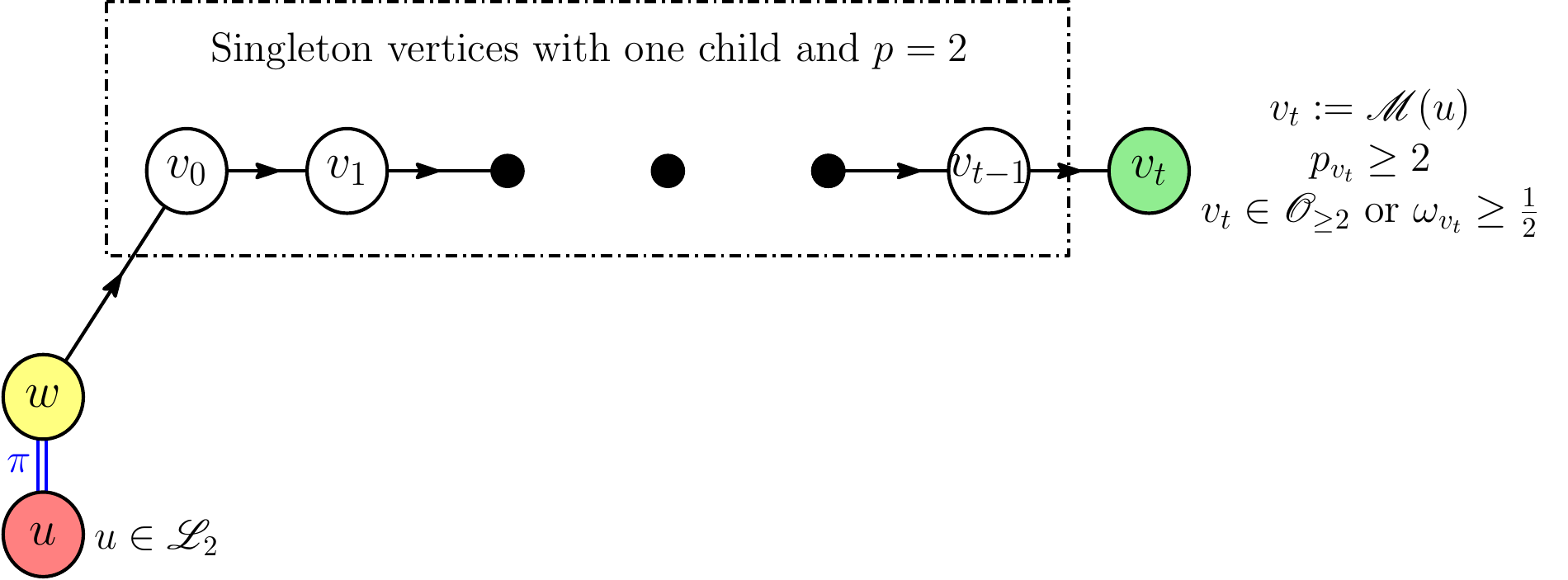}
         \caption{$u \in \tiileaves$ (Case 3).}
         \label{fig:map-l2}
     \end{subfigure}
        \caption{Construction of $\map$. Vertices in the same block of $\pi$ are connected with blue double edges $(\explain{$\pi$}{=\joinrel=})$. Vertex colors do not represent blocks of $\pi$.}
        \label{fig:map}
\end{figure}
\begin{proof}
We first describe the construction of $\map$. We consider 4 cases $u \in \bnodes_1$, $u \in \bnodes_2$, $u \in \tiileaves$ and $u \in \tiiileaves$. 
\begin{enumerate}
    \item $u \in \bnodes_1$: In this case we know that $u \in V \backslash B_1$ with $\pweight{u} = 1$, $\chrn{u} = 1$ and $|B_{\pi(u)}| = 1$. By the \ref{parity-property}, we know that $\qweight{u} \notin \{0,2\}$. Furthermore by \ref{forbidden-weights-property}, $\qweight{u} \neq 1$. Hence, $\qweight{u} \geq 3$.  Let $v$ denote the unique child of $u$. By the conservation constraint \eref{conservation-eq}, we must have $\pweight{v} \geq 3$. This also means that $\effect{v} \geq 1/2$ (cf. \eqref{eq:effect-def}), and in particular $v\in \enodes$. We set $\map(u):= v$, which verifies claim (1). This case is depicted in \fref{map-b1}. 
    \item $u \in \bnodes_2$: In this case we know that $u \in V \backslash B_1$ with $\pweight{u} = 1, |B_{\pi(u)}| = 1$ and $\chrn{u} = 2$. By the same argument as the previous case, we must have $\qweight{u} \not\in \{0,1,2\}$. Hence $\qweight{u} \geq 3$. Since $u$ has exactly two children, by the conservation constraint \eref{conservation-eq}, it has a child $v_0$ with $\pweight{v_0} \geq 2$. Now by \lemref{structural-prop}, there is a $t \in \N_0$ and a path $v_0 \rightarrow v_1 \rightarrow \dotsb \rightarrow v_t$ with the following properties:
    \begin{enumerate}
    \item $v_i \in V \backslash B_1, \; \pweight{v_i} = 2, \; |B_{\pi(v_i)}| = 1, \; \chrn{v_i} = 1$ for all $0 \leq i \leq t-1$.
    \item $v_t \in V \backslash \rootset(F)$ with $\pweight{v_t} \geq 2$. 
    \item Exactly one of the following is true:
    \begin{enumerate}
        \item $\effect{v_t} \geq 1/2$ 
        \item $v_t \in V\backslash B_1$ and $\pweight{v_t} = 2, |B_{\pi(v_t)}| = 1, \chrn{v_t} \geq 2$.
    \end{enumerate}
\end{enumerate}
We set $\map(u):= v_t$ and observe that $\pweight{v_t} \geq 2$ and either $v_t \in \enodes$ with $\effect{v_t} \geq 1/2$ or $v_t \in \onodes_{\geq 2} \backslash B_1$ with $\pweight{v_t} = 2, |B_{\pi(v_t)}| = 1, \chrn{v_t} \geq 2$, as claimed. This case is depicted in \fref{map-b2}. 
\item $u \in \tiileaves$: In this case, we know that $u \in \leaves{F} \backslash B_1$ with $\pweight{u} = 1, B_{\pi(u)} = \{u,w\}$ where $w \in V \backslash B_1$ satisfies $\pweight{w} = 1, \chrn{w} = 1$. By the conservation constraint $\qweight{w} \neq 0$. By the \ref{parity-property}, $\qweight{w} \neq 1$ (since $\pweight{u} + \pweight{w} + \qweight{w}$ has even parity). Hence we must $\qweight{w} \geq 2$. Let $v_0$ be the unique child of $w$. By the conservation constraint \eref{conservation-eq}, $\pweight{v_0} \geq 2$. Now by \lemref{structural-prop}, there is a $t \in \N_0$ and a path $v_0 \rightarrow v_1 \rightarrow \dotsb \rightarrow v_t$ with the following properties:
     \begin{enumerate}
    \item $v_i \in V \backslash B_1, \; \pweight{v_i} = 2, \; |B_{\pi(v_i)}| = 1, \; \chrn{v_i} = 1$ for all $0 \leq i \leq t-1$.
    \item  $v_t \in V \backslash \rootset(F)$ with $\pweight{v_t} \geq 2$. 
    \item Exactly one of the following is true:
    \begin{enumerate}
        \item $\effect{v_t} \geq 1/2$ 
        \item $v_t \in V\backslash B_1$ and $\pweight{v_t} = 2, |B_{\pi(v_t)}| = 1, \chrn{v_t} \geq 2$.
    \end{enumerate}
    \end{enumerate}
We set $\map(u):= v_t$ and observe that $\pweight{v_t} \geq 2$ and either $v_t \in \enodes$ with $\effect{v_t} \geq 1/2$ or $v_t \in \onodes_{\geq 2} \backslash B_1$ with $\pweight{v_t} = 2, |B_{\pi(v_t)}| = 1, \chrn{v_t} \geq 2$, as claimed. This case is depicted in \fref{map-l2}. 
\item If $u \in \tiiileaves$, then we know that $u \in \leaves{F} \backslash B_1$ with $\pweight{u} = 1, B_{\pi(u)} = \{u,v\}$ with $v\in V \backslash B_1$ and $\pweight{v} = 1$ and $\chrn{v} \geq 2$. In particular observe that $v \in \onodes_{\geq 2}$ and hence we can set $\map{(u)}:= v$. This case is depicted in \fref{map-l3}. 
\end{enumerate}
Next, we need to show the map defined above is injective. Define the following subset of vertices:
\begin{align*}
    \mathcal{S} \explain{def}{=} \{w \in V \backslash B_1: \pweight{w} = 2, |B_{\pi(w)}| = 1, \chrn{w} = 1\}.
\end{align*}
Observe that in \fref{map}, the vertices in $\mathcal{S}$ are exactly the white vertices. From the construction of $\map$  defined (refer to \fref{map}) above we observe the following claims are true:
\begin{enumerate}
    \item If $v= \map{(u)}$ for some $u \in \tiiileaves$ then $\pweight{v} = 1$ and $B_{\pi(v)} = \{u,v\}$.  On the other hand, if $v = \map{(u)}$ for some $u \in \bnodes_1 \cup \bnodes_2 \cup \tiileaves$, then $\pweight{v} \geq 2$. 
    \item If $v = \map{(u)}$ for some $u \in \bnodes_1$, then the parent of $v$ is $u$ and $u \not\in \mathcal{S}$. In particular, the closest ancestor of $v$ which is not in the set $\mathcal{S}$ is precisely $u$. 
    \item If $v = \map{(u)}$ for some $u \in \bnodes_2$, then the closest ancestor of $v$ which is not in the set $\mathcal{S}$ is precisely $u$. 
    \item If $v = \map{(u)}$ for some $u \in \tiileaves$, then the closest ancestor of $v$ which is not in the set $\mathcal{S}$ is an internal node $w$ with the property $\pweight{w} =1, \chrn{w} = 1, B_{\pi(w)} = \{w, u\}$. Furthermore $w \not\in \bnodes_1$ and $w \not\in \bnodes_2$. 
\end{enumerate}
The above three claims provide an algorithm to compute $\map^{-1}(v)$: 
\begin{enumerate}
    \item If $\pweight{v} = 1$, then it must be that $B_{\pi(v)} = \{u,v\}$ and $\map^{-1}(v):= u$. 
    \item If $\pweight{v} \geq 2$, find the closest ancestor of $v$ that is not in the set $\mathcal{S}$. Let this ancestor be $w$. 
    \begin{itemize}
    \item If $w \in \bnodes_1$ then, $\map^{-1}(v) := w$. 
    \item If $w \in \bnodes_2$ then, $\map^{-1}(v) : = w$.
    \item If none of the above conditions are met, then it must be that $B_{\pi(w)} = \{u,w\}$ and $\map^{-1}(v): = u$.
    \end{itemize}
\end{enumerate}
Since we have constructed an inverse for $\map$, $\map$ is injective. 
\end{proof}

\subsection{Proof of \propref{universal-config-estimate}}\label{sec:universal-forest-final-proof}
\begin{proof}[Proof of \propref{universal-config-estimate}]
We are now ready to prove for \propref{universal-config-estimate}. It relies on several results proved previously, which we reproduce below for convenience.
\begin{enumerate}
    \item In \lemref{exponent-lb}, we showed that,
    \begin{align} \label{eq:exponent-lb-recall}
         \exponent(F,\pi) & \geq \frac{|\rootset(F)| - |\rootset_0(F)|}{4} + \frac{|\onodes_{\geq 2}|}{4} - \frac{(|\eleaves| + |\gleaves| + |\tileaves| + |\tiileaves| + |\tiiileaves|+ |\bnodes_{2}|)}{4} - \frac{|\bnodes_1|}{2} + \sum_{v \in \onodes \cup \gnodes \cup \enodes} \effect{v},
    \end{align}
    where,
    \begin{align}\label{eq:effect-def-recall}
    \effect{v} \explain{def}{=} \begin{cases} \frac{\pweight{v}}{2} - \frac{1}{|B_{\pi(v)}|} &: v \in V \backslash B_1  \\ \frac{\pweight{v}}{2} &: v \in B_1 \backslash \rootset(F) \\ 0 &: v \in \rootset(F)  \end{cases}
\end{align}
    \item In \lemref{map-construction}, we constructed an \emph{injective} map $\map: \bnodes_1 \cup \bnodes_2 \cup \tiileaves \cup \tiiileaves \rightarrow \enodes \cup \onodes_{\geq 2} \backslash \rootset(F)$ with the following properties:
\begin{enumerate}
  \item If $u \in \bnodes_1$, then $v \bydef \map(u)$ satisfies $v \in \enodes$ with $\pweight{v} \geq 3$. 
    \item If $u \in \bnodes_2 \cup \tiileaves $, then $v \bydef \map(u)$ satisfies $v \in V \backslash \rootset(F)$ with $\pweight{v} \geq 2$ and exactly one of the following properties:
    \begin{enumerate}
        \item $v \in \enodes$ with $\effect{v} \geq 1/2$.
        \item $v \in \onodes_{\geq 2} \backslash B_1$ with $\pweight{v} = 2, \; |B_{\pi(v)}| = 1, \chrn{v} \geq 2$.
    \end{enumerate}
    \item If $u \in \tiiileaves$, then $v \bydef \map(u)$ satisfies $v \in \onodes_{\geq 2} \backslash B_1$ with $\pweight{v} = 1, |B_{\pi(v)}| = 2$ and $\chrn{v} \geq 2$.
\end{enumerate}
\end{enumerate}
As mentioned previously, the rational of constructing the map $\map$ is to cancel the negative contribution of a vertex $u \in \bnodes_1 \cup \bnodes_2 \cup \tiileaves \cup \tiiileaves$ in the lower bound in \eqref{eq:exponent-lb-recall} by the positive effect of $\map({u}) \in \onodes_{\geq 2} \cup \enodes$. To this end, we make the following observations. First, we observe that we can partition the sets $\bnodes_2$ and $\tiileaves$ as $\bnodes_2 = \bnodes_2^{o} \cup \bnodes_2^e$ and $\tiileaves = \tiileaves^{o} \cup \tiileaves^{e}$ where:
    \begin{align*}
        \bnodes_2^o & = \{u \in \bnodes_2 : \map(u) \in \onodes_{\geq 2}\}, \; \bnodes_2^e = \{u \in \bnodes_2 : \map(u) \in \enodes\}, \\
         \tiileaves^o  &= \{u \in \tiileaves : \map(u) \in \onodes_{\geq 2}\}, \; \tiileaves^e = \{u \in \tiileaves : \map(u) \in \enodes\}
    \end{align*}
    By the injectivity of $\map$, $|\bnodes_2^o| + |\tiileaves^o| + |\tiiileaves| \leq |\onodes_{\geq 2}|$. Hence,
    \begin{align*}
        \exponent(F,\pi) & \geq \frac{|\rootset(F)| - |\rootset_0(F)|}{4}  - \frac{(|\eleaves| + |\gleaves| + |\tileaves| + |\tiileaves^e| + |\bnodes_{2}^e|)}{4} - \frac{|\bnodes_1|}{2} + \sum_{v \in \onodes \cup \gnodes \cup \enodes} \effect{v}
    \end{align*}
    \item Next, we transfer the negative contribution of a vertex $u \in \bnodes_2^e \cup \tiileaves^e \cup \bnodes_1$ to $\map(u)$ by defining weights $\mu_v$ for each $v \in V$ as follows:
    \begin{align*}
    \mu_v & = \begin{cases} 2 &: v \in \map(\bnodes_1) \\
    1 &: v \in \map(\bnodes_2 \cup \tiileaves) \cap \enodes \\ 0&: \text{ otherwise}\end{cases}. 
\end{align*}
The injectivity of $\map$ again guarantees:
\begin{align}
    \exponent(F,\pi) & \geq \frac{|\rootset(F)| - |\rootset_0(F)|}{4} + \sum_{v \in \onodes \cup \gnodes \cup \enodes} \left(\effect{v} - \frac{\mu_v}{4} \right)  -  \frac{(|\eleaves| + |\gleaves| + |\tileaves|)}{4}. \label{eq:exponent-lb-intermediate}
\end{align}
 
We also record the following implications that we use several times in the proof: 
\begin{subequations}\label{eq:mu-implication}
\begin{align} 
    \mu_v = 1 &\implies \effect{v} \geq 1/2   &\text{[follows from  \lemref{map-construction} recalled in item (2b) above]},\\
    \mu_v = 2 &\implies \pweight{v} \geq 3 \implies \effect{v} \geq 3/2 - 1 = 1/2  &\text{[follows from  \lemref{map-construction} recalled in (2a) and \eqref{eq:effect-def-recall}]}.
\end{align}
\end{subequations}
We further rewrite \eqref{eq:exponent-lb-intermediate} by defining the weights:
\begin{align*}
    \lambda_v & = \begin{cases} 1 &: v \in \eleaves \cup \gleaves \cup \tileaves  \\ 0&: \text{ otherwise}\end{cases},
\end{align*}
Since $\eleaves \cup \gleaves \cup \tileaves \subset \onodes \cup \gnodes \cup \enodes$ we can rexpress \eqref{eq:exponent-lb-intermediate} as:
\begin{align*}
     \exponent(F,\pi) & \geq \frac{|\rootset(F)| - |\rootset_0(F)|}{4} + \sum_{v \in \onodes \cup \gnodes \cup \enodes} \left( \effect{v} - \frac{\mu_v}{4} - \frac{\lambda_v}{4} \right).  
\end{align*}
In order to show that $\exponent(F,\pi) \geq (|\rootset(F)| - |\rootset_0(F)|)/4$, we will show that,
\begin{align*}
    \sum_{v \in \onodes \cup \gnodes \cup \enodes} \left( \effect{v} - \frac{\mu_v}{4} - \frac{\lambda_v}{4} \right) \geq 0.
\end{align*}
In order to show the above inequality, for each node $u \in  \onodes \cup \gnodes \cup \enodes$, we will show that either,
\begin{align*}
    \effect{u} - \frac{\mu_u}{4} - \frac{\lambda_u}{4} \geq 0,
\end{align*}
or for $B: = B_{\pi(u)}$,
\begin{align*}
    \sum_{v \in (\onodes \cup \gnodes \cup \enodes)\cap B} \left( \effect{v} - \frac{\mu_v}{4} - \frac{\lambda_v}{4} \right) \geq 0.
\end{align*}
We consider the following 4 exhaustive cases for $u$: $\lambda_u = 0$ or $\lambda_u = 1, \mu_u = 0$ or $\lambda_u = 1, \mu_u = 1$ or $\lambda_u = 1, \mu_u = 2$.
\begin{enumerate}
    \item $\lambda_u = 0:$ In this case if $\mu_u = 0$, we know by \defref{vertex-classification}, that for any $u \in \onodes \cup \gnodes \cup \enodes$,
    \begin{align*}
          \effect{u} \geq 0 \implies  \effect{u} - \frac{\mu_u}{4} - \frac{\lambda_u}{4} \geq 0. 
    \end{align*}
    On the other hand if $\mu_u \geq 1$, \eqref{eq:mu-implication} guarantees that,
    \begin{align*}
          \effect{u} \geq 1/2  \implies  \effect{u} - \frac{\mu_u}{4} - \frac{\lambda_u}{4} \geq 0.
    \end{align*}
    \item $\lambda_u = 1, \mu_u = 1$: In this case, \eqref{eq:mu-implication} yields:
    \begin{align*}
        \effect{u} \geq 1/2 \implies \effect{u} - \frac{\mu_u}{4} - \frac{\lambda_u}{4} \geq 0.
    \end{align*}
    \item $\lambda_u = 1, \mu_u = 2$: In this case, we know that $u$ is a leaf and furthermore by \eqref{eq:mu-implication} $\pweight{u} \geq 3$. We consider the following two sub-cases:
    \begin{enumerate}
        \item $|B_{\pi(u)}| = 1$. In this case, by the \ref{parity-property} we cannot have $\pweight{u} = 3$. Hence, we must have, $\pweight{u} \geq 4$. And hence,
        \begin{align*}
             \effect{u} - \frac{\mu_u}{4} - \frac{\lambda_u}{4} \geq \frac{\pweight{u}}{2} - \frac{1}{|B_{\pi(u)}|} - \frac{\mu_u}{4} - \frac{\lambda_u}{4} \geq 2 - 1 - \frac{2}{4} - \frac{1}{4} = \frac{1}{4}. 
        \end{align*}
        \item $|B_{\pi(u)}| \geq  2$. In this case, we have,
        \begin{align*}
            \effect{u} - \frac{\mu_u}{4} - \frac{\lambda_u}{4} \geq \frac{\pweight{u}}{2} - \frac{1}{|B_{\pi(u)}|} - \frac{\mu_u}{4} - \frac{\lambda_u}{4} \geq \frac{3}{2} - \frac{1}{2} - \frac{2}{4} - \frac{1}{4} = \frac{1}{4}.
        \end{align*}
    \end{enumerate}
    \item $\lambda_u = 1, \mu_u = 0$: In this case, we already know that $u \in \eleaves  \cup \tileaves \cup \gleaves$. We consider 3 exhaustive sub-cases corresponding to whether $u \in \eleaves$ or $u \in \tileaves$ or $u \in \gleaves$.
    \begin{enumerate}
        \item $u \in \eleaves$: In this case the definition of excellent leaves (\defref{leaf-classification}) guarantees:
        \begin{align*}
            \effect{u} \geq 1/4 \implies 
            \effect{u} - \frac{\mu_u}{4} - \frac{\lambda_u}{4} \geq 0. 
        \end{align*}
        \item $u \in \tileaves$: In this case, we know that $u \in \leaves{F} \backslash B_1$ with $\pweight{u} = 1, B_{\pi(u)} = \{u,v\}$ for some node $v \in V \backslash B_1$ with $\pweight{v} \geq 2$. In this case we will show that for $B = B_{\pi(u)} = \{u,v\}$ we have,
        \begin{align*}
           \left( \effect{u} - \frac{\mu_u}{4} - \frac{\lambda_u}{4} \right) +  \left( \effect{v} - \frac{\mu_v}{4} - \frac{\lambda_v}{4} \right) &=\left( \frac{\pweight{u}}{2} - \frac{1}{|B_{\pi(u)}|} - \frac{\mu_u}{4} - \frac{\lambda_u}{4} \right) +  \left( \frac{\pweight{v}}{2} - \frac{1}{|B_{\pi(v)}|} - \frac{\mu_v}{4} - \frac{\lambda_v}{4} \right) \\ &= \frac{\pweight{v}}{2} - \frac{3}{4} - \frac{\lambda_v + \mu_v}{4} 
        \end{align*}We consider two sub-cases depending on whether $v$ is a leaf or not.
        \begin{enumerate}
            \item $v$ is a leaf: Then by the \ref{parity-property}, $\pweight{u} + \pweight{v}$ must be even and so infact, $\pweight{v} \geq 3$. Since $\lambda_v + \mu_v \leq 3$, we have,
            \begin{align*}
                 \frac{\pweight{v}}{2} - \frac{3}{4} - \frac{\lambda_v + \mu_v}{4} \geq 0.
            \end{align*}
            \item $v$ is not a leaf: Note that $\lambda_v = 0$. If $\mu_v \leq 1$, then again since we know $\pweight{v} \geq 2$,
            \begin{align*}
                 \frac{\pweight{v}}{2} - \frac{3}{4} - \frac{\lambda_v + \mu_v}{4} \geq 0.
            \end{align*}
            On the other hand if $\mu_v = 2$, then \eqref{eq:mu-implication} guarantees $\pweight{v} \geq 3$. Hence,
            \begin{align*}
                 \frac{\pweight{v}}{2} - \frac{3}{4} - \frac{\lambda_v + \mu_v}{4} \geq \frac{1}{4}. 
            \end{align*}
        \end{enumerate}
        \item $u \in \gleaves:$ In this case, from the definition of good leaves (\defref{leaf-classification}), we know that $u \in \leaves{F} \backslash B_1$ with $\pweight{u} = 1$ and $B_{\pi(u)} = \{u, v, w\}$. Since $B_1$ and $B_{\pi(u)}$ are both blocks of a partition, they must either be identical or disjoint. Since $u \not \in B_1$, $B_1$ and $B_{\pi(u)}$ must be disjoint. Hence, $v,w \in V \backslash B_1$. Consequently, recalling the definition of $\effect{}$ from \eqref{eq:effect-def-recall} in this case, we need to show,
        \begin{align*}
         &\left( \effect{u} - \frac{\mu_u}{4} - \frac{\lambda_u}{4} \right) +  \left( \effect{v} - \frac{\mu_v}{4} - \frac{\lambda_v}{4} \right)  + \left( \effect{w} - \frac{\mu_w}{4} - \frac{\lambda_w}{4} \right) = \\
            &\left( \frac{\pweight{u}}{2} - \frac{1}{|B_{\pi(u)}|} - \frac{\mu_u}{4} - \frac{\lambda_u}{4} \right) +  \left( \frac{\pweight{v}}{2} - \frac{1}{|B_{\pi(v)}|} - \frac{\mu_v}{4} - \frac{\lambda_v}{4} \right)  + \left( \frac{\pweight{w}}{2} - \frac{1}{|B_{\pi(w)}|} - \frac{\mu_w}{4} - \frac{\lambda_w}{4} \right) \geq 0. 
        \end{align*}
        Since we know $\pweight{u} = 1, B_{\pi(u)} = \{u, v, w\}, \lambda_u = 1, \mu_u = 0$ we can simplify the above formula and we need to show:
        \begin{align} \label{eq:final-case-goal}
            \frac{1+ \pweight{v} + \pweight{w}}{2} - 1 - \frac{1 + \lambda_v + \lambda_w}{4} - \frac{\mu_v + \mu_w}{4} \geq 0.
        \end{align}
        We will consider the following 4 exhaustive cases: 
        \begin{enumerate}
            \item $(\mu_v \geq 1, \mu_w \geq 1)$: Observe that since $\mu_v \geq 1$, the construction of $\map$ guarantees (cf. \eqref{eq:mu-implication}):
            \begin{align*}
               \effect{v} =  \frac{\pweight{v}}{2} - \frac{1}{|B_{\pi(v)}|} \geq \frac{1}{2} \implies \pweight{v} \geq \frac{5}{3} \implies \pweight{v} \geq 2. 
            \end{align*}
            Analogously, we can argue that $\pweight{w} \geq 2$. If there are at most $2$ leaves among $\{u,v,w\}$ (and hence $\lambda_u + \lambda_v + \lambda_w \leq 2)$, we can lower bound \eref{final-case-goal} as follows:
            \begin{align*}
                 \frac{1+ \pweight{v} + \pweight{w}}{2} - 1 - \frac{1 + \lambda_v + \lambda_w}{4} - \frac{\mu_v + \mu_w}{4} \geq \frac{5}{2} - 1 - \frac{2}{4} - \frac{4}{4} \geq 0. 
            \end{align*}
            On the other hand, if all 3 of $\{u,v,w\}$ are leaves, by the \ref{parity-property}, we must have $\pweight{u}  + \pweight{v} + \pweight{w}$ is even. Since we know that $\pweight{u} + \pweight{v} + \pweight{w} \geq 5$ we must in fact have $\pweight{u} + \pweight{v} + \pweight{w} \geq 6$. We can now lower bound \eref{final-case-goal} as follows:
            \begin{align*}
                 \frac{1+ \pweight{v} + \pweight{w}}{2} - 1 - \frac{1 + \lambda_v + \lambda_w}{4} - \frac{\mu_v + \mu_w}{4} \geq \frac{6}{2} - 1 - \frac{3}{4} - \frac{4}{4} \geq \frac{1}{4}. 
            \end{align*}
            \item $(\mu_v = 2, \mu_w = 0)$ or $(\mu_v = 0, \mu_w = 2)$: Observe that these cases are symmetric. Hence we only need to consider the case $(\mu_v = 2, \mu_w = 0)$. In this situation, the construction of $\map$ guarantees $\pweight{v} \geq 3$ (cf. \eqref{eq:mu-implication}). Additionally we have, $\pweight{w} \geq 1$ and $\lambda_v + \lambda_w \leq 2$ and $\mu_v + \mu_w = 2$. Hence, we can now lower bound \eref{final-case-goal} as follows:
            \begin{align*}
                 \frac{1+ \pweight{v} + \pweight{w}}{2} - 1 - \frac{1 + \lambda_v + \lambda_w}{4} - \frac{\mu_v + \mu_w}{4} \geq \frac{5}{2} - 1 - \frac{3}{4} - \frac{2}{4} \geq \frac{1}{4}. 
            \end{align*}
            \item $(\mu_v = 1, \mu_w = 0)$ or $(\mu_v = 0, \mu_w = 1)$: Observe that these cases are symmetric. Hence we only need to consider the case $(\mu_v = 1, \mu_w = 0)$. In this situation, since $\mu_v = 1$, the construction of $\map$ guarantees (cf. \eqref{eq:mu-implication}):
            \begin{align*}
                \effect{v} =\frac{\pweight{v}}{2} - \frac{1}{|B_{\pi(v)}|} \geq \frac{1}{2} \implies \pweight{v} \geq \frac{5}{3} \implies \pweight{v} \geq 2. 
            \end{align*}
            Hence, by observing $p_w\geq 1, \lambda_v + \lambda_w \leq 2$ and $\mu_v + \mu_w = 1$ we can now lower bound \eref{final-case-goal} as follows:
            \begin{align*}
                 \frac{1+ \pweight{v} + \pweight{w}}{2} - 1 - \frac{1 + \lambda_v + \lambda_w}{4} - \frac{\mu_v + \mu_w}{4} \geq \frac{4}{2} - 1 - \frac{3}{4} - \frac{1}{4} \geq 0.
            \end{align*}
            \item $(\mu_v = 0, \mu_w = 0)$: Observe that since $\pweight{v} \geq 1$ and $\pweight{w} \geq 1$ we have $\pweight{u} + \pweight{v} + \pweight{w} \geq 3$. Furthermore, $\mu_v + \mu_w = 0$.  Now if there are at most $2$ leaves among $\{u, v, w\}$, we have $\lambda_u + \lambda_v + \lambda_w \leq 2$ and hence we can lower bound \eref{final-case-goal} as follows:
            \begin{align*}
                 \frac{1+ \pweight{v} + \pweight{w}}{2} - 1 - \frac{1 + \lambda_v + \lambda_w}{4} - \frac{\mu_v + \mu_w}{4} \geq \frac{3}{2} - 1 - \frac{2}{4} - \frac{0}{4} \geq 0.
            \end{align*}
            On the other hand, if all 3 of $\{u,v,w\}$ are leaves, by the \ref{parity-property} $\pweight{u} + \pweight{v} + \pweight{w}$ is even. Since $\pweight{u} + \pweight{v} + \pweight{w} \geq 3$, we infact have $\pweight{u} + \pweight{v} + \pweight{w} \geq 4$. Hence we obtain the following lower bound on \eref{final-case-goal}:
            \begin{align*}
                 \frac{1+ \pweight{v} + \pweight{w}}{2} - 1 - \frac{1 + \lambda_v + \lambda_w}{4} - \frac{\mu_v + \mu_w}{4} \geq \frac{4}{2} - 1 - \frac{3}{4} - \frac{0}{4} \geq \frac{1}{4}.
            \end{align*}
        \end{enumerate}
    \end{enumerate}
\end{enumerate}
This concludes the proof. 
\end{proof}

\section{Proof of  \thref{SE-final}}
\label{proof_SE-final} 

In this section, we provide a proof for \thref{SE-final}. Using a simple rescaling argument, we can without loss of generality assume that the matrix ensemble $\mM$ is semi-random with $\sigpsi=1$ (\defref{matrix-ensemble-relaxed}), the initialization $\iter{\vz}{0} \sim \gauss{0}{1}$ (that is, \assumpref{initialization} holds with $\sigma_0^2 = 1$) and the non-linearities satisfy $\E \nonlin_t^2(Z) = 1$ for $t \; \in \; [T]$ and $Z\sim \gauss{0}{1}$. We record this rescaling argument in the lemma below and defer its proof to \appref{rescaling} for the reader interested in the full details. 

\begin{lemma}[Rescaling]\label{lem:rescaling} It is sufficient to prove \thref{SE-final} under the additional assumptions: (a) the matrix ensemble $\mM$ is semi-random with $\sigpsi=1$ (\defref{matrix-ensemble-relaxed}), (b) the initialization $\iter{\vz}{0} \sim \gauss{0}{1}$ (that is, \assumpref{initialization} holds with $\sigma_0^2 = 1$) and, (c) the non-linearities satisfy $\E \nonlin_t^2(Z) = 1$ for $t \; \in \; [T]$ and $Z\sim \gauss{0}{1}$.
\end{lemma}

\thref{SE-normalized} will play a critical role in our proof. 
We first complement \thref{SE-normalized} with the following concentration estimate.

\begin{theorem}\label{thm:concentration} Under the assumptions of \thref{SE-normalized}, for any polynomial test function $h: \R^{t+1} \rightarrow \R$ with degree at most $\degree$, we have, 
\begin{align*}
    \Var\left[ \frac{1}{\dim} \sum_{i=1}^\dim h(\iter{\vz}{0}_i, \iter{\vz}{1}_i, \dotsc, \iter{\vz}{t}_i) \right] & \lesssim \dim^{-1 + \epsilon}.
\end{align*}
\end{theorem}
\begin{proof}
This variance bound is derived using the Efron-Stein Inequality \citep{efron1981jackknife}. The complete derivation is provided in \appref{concentration} for the interested reader. 
\end{proof}

In order to deduce \thref{SE-final} from \thref{SE-normalized}, we will rely on the following approximation result, which shows that an iteration which satisfies the assumptions of \thref{SE-final} can always be approximated by an iteration which satisfies the assumptions of \thref{SE-normalized}. 

\begin{proposition}[Approximation]\label{prop:approx} Fix a non-negative integer $T \in \W$ and functions $\nonlin_1, \nonlin_2, \dotsc, \nonlin_T : \R \rightarrow \R$ and a test function $h:\R^{T+1} \rightarrow \R$. Consider the iteration:
\begin{align*}
    \iter{\vz}{t+1} = \mM \nonlin_{t+1}(\iter{\vz}{t}),
\end{align*}
initialized at $\iter{\vz}{0}$ that satisfies the assumptions of \thref{SE-final} with constants $\sigpsi = 1$ and $\sigma_0^2 = 1$. Suppose that the test function $h$ satisfies the following continuity estimate for some fixed constant $L \geq 0$:
    \begin{align*}
     |h(\vx) - h(\vx^\prime)| & \leq L \cdot  \|\vx - \vx^\prime \| \cdot (1+ \|\vx\| + \|\vx^\prime\|),
    \end{align*}
Then, there is a matrix ensemble $\hat{\mM}$ and sequences of polynomial approximating functions  $\iter{\nonlin}{k}_1, \iter{\nonlin}{k}_2, \dotsc, \iter{\nonlin}{k}_T : \R \rightarrow \R$ and $\iter{h}{k}:\R^{T+1} \rightarrow \R$ indexed by $k \in \N$ such that for each $k \in \N$, the iteration:
\begin{align*}
     \iter{\hat{\vz}}{t+1;k} = \hat{\mM} \iter{\nonlin}{k}_{t+1}(\iter{\hat{\vz}}{t;k}),
\end{align*}
initialized at $ \iter{\hat{\vz}}{0;k} = \iter{\vz}{0}$ satisfies the assumptions of \thref{SE-normalized} and,
\begin{align*}
    \lim_{k \rightarrow \infty}  \limsup_{\dim \rightarrow \infty} \E\left[ \left| \frac{1}{\dim} \sum_{i=1}^\dim h(\iter{z}{0}_i, \iter{z}{1}_i, \dotsc, \iter{z}{T}_i) - \frac{1}{\dim} \sum_{i=1}^\dim \iter{h}{k}(\iter{\hat{z}}{0;k}_i, \iter{\hat{z}}{1;k}_i, \dotsc, \iter{\hat{z}}{T;k}_i)  \right| \right] & = 0.
\end{align*}
\end{proposition}
The proof of this result is based on routine approximation arguments and is presented in \appref{approximation}. We now present the proof of \thref{SE-final}. \begin{proof}[Proof of \thref{SE-final}] Consider an iteration that satisfies the assumptions of \thref{SE-final} and the additional assumptions described in \lemref{rescaling}:
\begin{align*}
     \iter{\vz}{t+1} = \mM \nonlin_{t+1}(\iter{\vz}{t}).
\end{align*}
with initialization $\iter{\vz}{0} \sim \gauss{0}{1}$. In order to understand the dynamics of the above iteration, we will compare it to another iteration which uses a rotationally invariant matrix instead of $\mM$. The dynamics of such iterations are well-understood from prior work. Specifically, we consider the iteration:
\begin{align*}
     \iter{\vw}{t+1} = \mQ \nonlin_{t+1}(\iter{\vw}{t}),
\end{align*}
with
\begin{align}\label{eq:Q-def}
    \mQ = \mS \mU \mP \mLambda \mP\tran \mU\tran \mS, 
\end{align}
where:
\begin{enumerate}
    \item ${\mS} = \diag(s_1, s_2, \dotsc, s_\dim)$ with $s_i \explain{i.i.d.}{\sim} \unif{\{\pm 1\}}$
    \item $\mU \sim \unif{\mathbb{O}(\dim)}$.
    \item $\mP$ is a uniformly random $\dim \times \dim$ permutation matrix.
    \item $\mLambda = \diag(\lambda_1, \lambda_2, \dotsc, \lambda_\dim)$ with $\lambda_i \explain{i.i.d.}{\sim} \unif{\{\pm 1\}}$. 
\end{enumerate}
This choice of $\mQ$ has two useful properties:
\begin{enumerate}
    \item $\mQ$ satisfies is semi-random with constant $\sigpsi = 1$ in the sense of \defref{matrix-ensemble-relaxed} (cf. \lemref{sign-perm-inv-ensemble}). 
    \item Since the Haar measure on the orthogonal group is invariant to left and right multiplication by orthogonal matrices, $\mQ \explain{d}{=} \mU \mLambda \mU\tran$. Matrices of the form with $\mU \mLambda \mU\tran$ are called rotationally invariant matrices and the dynamics of approximate message passing algorithms for rotationally invariant ensembles has been studied in prior work. In particular, \citet[Corollary 4.4]{fan2020approximate} yields the result (see for e.g. \citet[Appendix A]{fan2021replica} for additional details):
    \begin{align*}
    \frac{1}{\dim} \sum_{i=1}^\dim h(\iter{w}{0}_i, \iter{w}{1}_i, \dotsc, \iter{w}{T}_i) \explain{P}{\rightarrow} \E h(Z_0, Z_1, \dotsc, Z_T), 
\end{align*}
where the law of the random vector $(Z_0, Z_1, \dotsc Z_T)$ is as defined in \thref{SE-final}.
\end{enumerate}
As a consequence, \thref{SE-final} follows if we can show that the random variables:
\begin{align*}
    X_\dim &\explain{def}{=}  \frac{1}{\dim} \sum_{i=1}^\dim h(\iter{z}{0}_i, \iter{z}{1}_i, \dotsc, \iter{z}{T}_i), \\
    Y_\dim &\explain{def}{=} \frac{1}{\dim} \sum_{i=1}^\dim h(\iter{w}{0}_i, \iter{w}{1}_i, \dotsc, \iter{w}{T}_i) 
\end{align*}
satisfy $X_{\dim} - Y_{\dim} \explain{P}{\rightarrow 0}$. In order to show this we use \propref{approx} to construct two sequences of approximating iterations $\iter{\hat{\vz}}{t;k}, \iter{\hat{\vw}}{t;k}$ and a sequence of approximating polynomial test functions $\iter{h}{k}: \R^{T+1} \rightarrow \R$ indexed by $k \in \N$ such that the random variables:
\begin{align*}
     \iter{X}{k}_\dim &\explain{def}{=}  \frac{1}{\dim} \sum_{i=1}^\dim \iter{h}{k}(\iter{\hat{z}}{0;k}_i, \iter{\hat{z}}{1;k}_i, \dotsc, \iter{\hat{z}}{T;k}_i), \\
    \iter{Y}{k}_\dim &\explain{def}{=} \frac{1}{\dim} \sum_{i=1}^\dim \iter{h}{k}(\iter{\hat{w}}{0;k}_i, \iter{\hat{w}}{1;k}_i, \dotsc, \iter{\hat{w}}{T;k}_i),
\end{align*}
approximate $X_\dim, Y_\dim$ in the sense:
\begin{align}\label{eq:appx-guarantee}
    \lim_{k \rightarrow \infty } \limsup_{\dim \rightarrow \infty} \E \;  |X_N - \iter{X}{k}_\dim| = \lim_{k \rightarrow \infty } \limsup_{\dim \rightarrow \infty} \E \;  |Y_N - \iter{Y}{k}_\dim| = 0. 
\end{align}
Furthermore, since \propref{approx} guarantees that for each $k \in \N$, the two iterations $\iter{\hat{\vz}}{t;k}, \iter{\hat{\vw}}{t;k}$ satisfy the requirements of \thref{SE-normalized} and \thref{concentration} and since $\iter{h}{k}$ is a polynomial,
\begin{align*}
    \lim_{\dim \rightarrow \infty} \E[\iter{X}{k}_\dim]  &= \lim_{\dim \rightarrow \infty} \E[\iter{Y}{k}_\dim] \; \forall \; k \; \in \N, \\
   \lim_{\dim \rightarrow \infty} \Var[\iter{X}{k}_\dim]  &= \lim_{\dim \rightarrow \infty} \Var[\iter{Y}{k}_\dim]  = 0 \; \forall \; k \; \in \N.
\end{align*}
By Chebychev's Inequality, we have,
\begin{align} \label{eq:chebychev-conclusion}
    \iter{X}{k}_\dim - \iter{Y}{k}_\dim \explain{P}{ \rightarrow} 0 \; \forall \; k \; \in \N.
\end{align}
With these observations in hand, we can conclude the claim of the proposition. For any $\epsilon > 0$ we have,
\begin{align*}
    \lim_{\dim \rightarrow \infty} \; \P( |X_\dim - Y_\dim| > 3\epsilon) & \leq  \lim_{k \rightarrow \infty} \lim_{\dim \rightarrow \infty} \left( \P( |X_\dim - \iter{X}{k}_\dim| > \epsilon) + \P( |Y_\dim - \iter{Y}{k}_\dim| > \epsilon) + \P( |\iter{X}{k} - \iter{Y}{k}_\dim| > \epsilon) \right) \\
    & \explain{\eqref{eq:chebychev-conclusion}}{=} \lim_{k \rightarrow \infty} \lim_{\dim \rightarrow \infty} \left( \P( |X_\dim - \iter{X}{k}_\dim| > \epsilon) + \P( |Y_\dim - \iter{Y}{k}_\dim| > \epsilon) \right) \\
    & \leq \lim_{k \rightarrow \infty} \lim_{\dim \rightarrow \infty} \left( \E \; |X_\dim - \iter{X}{k}_\dim|  +  \E \; |Y_\dim - \iter{Y}{k}_\dim|  \right)/\epsilon \\
    &\explain{\eqref{eq:appx-guarantee}}{=} 0. 
\end{align*}
This concludes the proof of \thref{SE-normalized}.
\end{proof}

\subsection*{Acknowledgements}
We are grateful to Giorgio Cipolloni, Jiaoyang Huang, Benjamin Landon, and Dominik Schröder for helpful discussions regarding the local law for Wigner matrices. The work of YML is supported by a Harvard FAS Dean's competitive fund award for promising scholarship, and by the US National Science Foundation under grant CCF-1910410.  SS gratefully acknowledges support from a Harvard FAS Dean's competitive fund award. 
\bibliographystyle{plainnat}
\bibliography{refs}

\begin{thebibliography}{74}
\providecommand{\natexlab}[1]{#1}
\providecommand{\url}[1]{\texttt{#1}}
\expandafter\ifx\csname urlstyle\endcsname\relax
  \providecommand{\doi}[1]{doi: #1}\else
  \providecommand{\doi}{doi: \begingroup \urlstyle{rm}\Url}\fi

\bibitem[Abbara et~al.(2020)Abbara, Baker, Krzakala, and
  Zdeborov{\'a}]{abbara2020universality}
Alia Abbara, Antoine Baker, Florent Krzakala, and Lenka Zdeborov{\'a}.
\newblock On the universality of noiseless linear estimation with respect to
  the measurement matrix.
\newblock \emph{Journal of Physics A: Mathematical and Theoretical},
  53\penalty0 (16):\penalty0 164001, 2020.

\bibitem[Anderson and Farrell(2014)]{anderson2014asymptotically}
Greg~W Anderson and Brendan Farrell.
\newblock Asymptotically liberating sequences of random unitary matrices.
\newblock \emph{Advances in Mathematics}, 255:\penalty0 381--413, 2014.

\bibitem[Bai and Yin(1988)]{bai1988necessary}
Zhi-Dong Bai and Yong-Qua Yin.
\newblock Necessary and sufficient conditions for almost sure convergence of
  the largest eigenvalue of a {W}igner matrix.
\newblock \emph{The Annals of Probability}, pages 1729--1741, 1988.

\bibitem[Bai and Yin(2008)]{bai2008limit}
Zhi-Dong Bai and Yong-Qua Yin.
\newblock Limit of the smallest eigenvalue of a large dimensional sample
  covariance matrix.
\newblock In \emph{Advances In Statistics}, pages 108--127. World Scientific,
  2008.

\bibitem[Bayati and Montanari(2011)]{bayati2011dynamics}
Mohsen Bayati and Andrea Montanari.
\newblock The dynamics of message passing on dense graphs, with applications to
  compressed sensing.
\newblock \emph{IEEE Transactions on Information Theory}, 57\penalty0
  (2):\penalty0 764--785, 2011.

\bibitem[Bayati et~al.(2015)Bayati, Lelarge, and
  Montanari]{bayati2015universality}
Mohsen Bayati, Marc Lelarge, and Andrea Montanari.
\newblock Universality in polytope phase transitions and message passing
  algorithms.
\newblock \emph{The Annals of Applied Probability}, 25\penalty0 (2):\penalty0
  753--822, 2015.

\bibitem[Bean et~al.(2013)Bean, Bickel, El~Karoui, and Yu]{bean2013optimal}
Derek Bean, Peter~J Bickel, Noureddine El~Karoui, and Bin Yu.
\newblock Optimal {M}-estimation in high-dimensional regression.
\newblock \emph{Proceedings of the National Academy of Sciences}, 110\penalty0
  (36):\penalty0 14563--14568, 2013.

\bibitem[Benaych-Georges and Knowles(2017)]{benaych2016lectures}
Florent Benaych-Georges and Antti Knowles.
\newblock Lectures on the local semicircle law for {W}igner matrices.
\newblock In \emph{Advanced topics in random matrices, Panoramas et Synthèses
  53}, pages 1--90. Soc. Math. France, Paris, 2017.

\bibitem[Bercu et~al.(2015)Bercu, Delyon, and Rio]{bercu2015concentration}
Bernard Bercu, Bernard Delyon, and Emmanuel Rio.
\newblock \emph{Concentration inequalities for sums and martingales}.
\newblock Springer, 2015.

\bibitem[Berthier et~al.(2020)Berthier, Montanari, and
  Nguyen]{berthier2020state}
Raphael Berthier, Andrea Montanari, and Phan-Minh Nguyen.
\newblock State evolution for approximate message passing with non-separable
  functions.
\newblock \emph{Information and Inference: A Journal of the IMA}, 9\penalty0
  (1):\penalty0 33--79, 2020.

\bibitem[Billingsley(2008)]{billingsley2008probability}
Patrick Billingsley.
\newblock \emph{Probability and measure}.
\newblock John Wiley \& Sons, 2008.

\bibitem[Bloemendal et~al.(2014)Bloemendal, Erd{\H{o}}s, Knowles, Yau, and
  Yin]{alex2014isotropic}
Alex Bloemendal, L{\'a}szl{\'o} Erd{\H{o}}s, Antti Knowles, Horng-Tzer Yau, and
  Jun Yin.
\newblock Isotropic local laws for sample covariance and generalized {W}igner
  matrices.
\newblock \emph{Electronic Journal of Probability}, 19:\penalty0 1--53, 2014.

\bibitem[Bolthausen(2014)]{bolthausen2014iterative}
Erwin Bolthausen.
\newblock An iterative construction of solutions of the {TAP} equations for the
  {S}herrington--{K}irkpatrick model.
\newblock \emph{Communications in Mathematical Physics}, 325\penalty0
  (1):\penalty0 333--366, 2014.

\bibitem[Boucheron et~al.(2013)Boucheron, Lugosi, and
  Massart]{boucheron2013concentration}
St{\'e}phane Boucheron, G{\'a}bor Lugosi, and Pascal Massart.
\newblock \emph{Concentration inequalities: A nonasymptotic theory of
  independence}.
\newblock Oxford university press, 2013.

\bibitem[Britanak et~al.(2010)Britanak, Yip, and Rao]{britanak2010discrete}
Vladimir Britanak, Patrick~C Yip, and Kamisetty~Ramamohan Rao.
\newblock \emph{Discrete cosine and sine transforms: general properties, fast
  algorithms and integer approximations}.
\newblock Elsevier, 2010.

\bibitem[{\c{C}}akmak and Opper(2019)]{ccakmak2019memory}
Burak {\c{C}}akmak and Manfred Opper.
\newblock Memory-free dynamics for the {T}houless-{A}nderson-{P}almer equations
  of {I}sing models with arbitrary rotation-invariant ensembles of random
  coupling matrices.
\newblock \emph{Physical Review E}, 99\penalty0 (6):\penalty0 062140, 2019.

\bibitem[Carmona and Hu(2006)]{carmona2006universality}
Philippe Carmona and Yueyun Hu.
\newblock Universality in {S}herrington–{K}irkpatrick's spin glass model.
\newblock \emph{Annales de l'Institut Henri Poincare (B) Probability and
  Statistics}, 42\penalty0 (2):\penalty0 215--222, 2006.
\newblock ISSN 0246-0203.
\newblock \doi{https://doi.org/10.1016/j.anihpb.2005.04.001}.

\bibitem[Celentano et~al.(2021)Celentano, Cheng, and
  Montanari]{celentano2021high}
Michael Celentano, Chen Cheng, and Andrea Montanari.
\newblock The high-dimensional asymptotics of first order methods with random
  data.
\newblock \emph{arXiv preprint arXiv:2112.07572}, 2021.

\bibitem[Chatterjee(2005)]{chatterjee2005simple}
Sourav Chatterjee.
\newblock A simple invariance theorem.
\newblock \emph{arXiv preprint math/0508213}, 2005.

\bibitem[Chen and Lam(2021)]{chen2021universality}
Wei-Kuo Chen and Wai-Kit Lam.
\newblock Universality of approximate message passing algorithms.
\newblock \emph{Electronic Journal of Probability}, 26:\penalty0 1--44, 2021.

\bibitem[Cover(1965)]{cover1965geometrical}
Thomas~M. Cover.
\newblock Geometrical and statistical properties of systems of linear
  inequalities with applications in pattern recognition.
\newblock \emph{IEEE Transactions on Electronic Computers}, EC-14\penalty0
  (3):\penalty0 326--334, 1965.
\newblock \doi{10.1109/PGEC.1965.264137}.

\bibitem[Donoho and Montanari(2016)]{donoho2016high}
David Donoho and Andrea Montanari.
\newblock High dimensional robust {M}-estimation: Asymptotic variance via
  approximate message passing.
\newblock \emph{Probability Theory and Related Fields}, 166\penalty0
  (3):\penalty0 935--969, 2016.

\bibitem[Donoho and Tanner(2009)]{donoho2009observed}
David Donoho and Jared Tanner.
\newblock Observed universality of phase transitions in high-dimensional
  geometry, with implications for modern data analysis and signal processing.
\newblock \emph{Philosophical Transactions of the Royal Society A:
  Mathematical, Physical and Engineering Sciences}, 367\penalty0
  (1906):\penalty0 4273--4293, 2009.

\bibitem[Donoho and Tanner(2010)]{donoho2010counting}
David~L Donoho and Jared Tanner.
\newblock Counting the faces of randomly-projected hypercubes and orthants,
  with applications.
\newblock \emph{Discrete \& computational geometry}, 43\penalty0 (3):\penalty0
  522--541, 2010.

\bibitem[Donoho et~al.(2009)Donoho, Maleki, and Montanari]{donoho2009message}
David~L Donoho, Arian Maleki, and Andrea Montanari.
\newblock Message-passing algorithms for compressed sensing.
\newblock \emph{Proceedings of the National Academy of Sciences}, 106\penalty0
  (45):\penalty0 18914--18919, 2009.

\bibitem[Dudeja and Bakhshizadeh(2020)]{dudeja2020universality}
Rishabh Dudeja and Milad Bakhshizadeh.
\newblock Universality of linearized message passing for phase retrieval with
  structured sensing matrices.
\newblock \emph{arXiv preprint arXiv:2008.10503}, 2020.

\bibitem[Efron and Stein(1981)]{efron1981jackknife}
Bradley Efron and Charles Stein.
\newblock The jackknife estimate of variance.
\newblock \emph{The Annals of Statistics}, pages 586--596, 1981.

\bibitem[Erd{\H{o}}s and Yau(2017)]{erdHos2017dynamical}
L{\'a}szl{\'o} Erd{\H{o}}s and Horng-Tzer Yau.
\newblock \emph{A dynamical approach to random matrix theory}, volume~28.
\newblock American Mathematical Soc., 2017.

\bibitem[Erd{\H{o}}s et~al.(2009)Erd{\H{o}}s, Schlein, and
  Yau]{erdHos2009local}
L{\'a}szl{\'o} Erd{\H{o}}s, Benjamin Schlein, and Horng-Tzer Yau.
\newblock Local semicircle law and complete delocalization for {W}igner random
  matrices.
\newblock \emph{Communications in Mathematical Physics}, 287\penalty0
  (2):\penalty0 641--655, 2009.

\bibitem[Fan(2022)]{fan2020approximate}
Zhou Fan.
\newblock Approximate message passing algorithms for rotationally invariant
  matrices.
\newblock \emph{The Annals of Statistics}, 50\penalty0 (1):\penalty0 197--224,
  2022.

\bibitem[Fan and Wu(2021)]{fan2021replica}
Zhou Fan and Yihong Wu.
\newblock The replica-symmetric free energy for ising spin glasses with
  orthogonally invariant couplings.
\newblock \emph{arXiv preprint arXiv:2105.02797}, 2021.

\bibitem[Fan et~al.(2022)Fan, Li, and Sen]{fan2022tap}
Zhou Fan, Yufan Li, and Subhabrata Sen.
\newblock {TAP} equations for orthogonally invariant spin glasses at high
  temperature.
\newblock \emph{arXiv preprint arXiv:2202.09325}, 2022.

\bibitem[Farrell(2011)]{farrell2011limiting}
Brendan Farrell.
\newblock Limiting empirical singular value distribution of restrictions of
  discrete {F}ourier transform matrices.
\newblock \emph{Journal of Fourier Analysis and Applications}, 17\penalty0
  (4):\penalty0 733--753, 2011.

\bibitem[Gerbelot et~al.(2020{\natexlab{a}})Gerbelot, Abbara, and
  Krzakala]{gerbelot2020asymptotic}
Cedric Gerbelot, Alia Abbara, and Florent Krzakala.
\newblock Asymptotic errors for teacher-student convex generalized linear
  models (or: How to prove {K}abashima's replica formula).
\newblock \emph{arXiv preprint arXiv:2006.06581}, 2020{\natexlab{a}}.

\bibitem[Gerbelot et~al.(2020{\natexlab{b}})Gerbelot, Abbara, and
  Krzakala]{pmlr-v125-gerbelot20a}
C\'{e}dric Gerbelot, Alia Abbara, and Florent Krzakala.
\newblock Asymptotic errors for high-dimensional convex penalized linear
  regression beyond gaussian matrices.
\newblock In Jacob Abernethy and Shivani Agarwal, editors, \emph{Proceedings of
  Thirty Third Conference on Learning Theory}, volume 125 of \emph{Proceedings
  of Machine Learning Research}, pages 1682--1713. PMLR, 09--12 Jul
  2020{\natexlab{b}}.

\bibitem[Hastie et~al.(2019)Hastie, Montanari, Rosset, and
  Tibshirani]{hastie2019surprises}
Trevor Hastie, Andrea Montanari, Saharon Rosset, and Ryan~J Tibshirani.
\newblock Surprises in high-dimensional ridgeless least squares interpolation.
\newblock \emph{arXiv preprint arXiv:1903.08560}, 2019.

\bibitem[Hopfield(1982)]{hopfield1982neural}
John~J Hopfield.
\newblock Neural networks and physical systems with emergent collective
  computational abilities.
\newblock \emph{Proceedings of the National Academy of Sciences}, 79\penalty0
  (8):\penalty0 2554--2558, 1982.

\bibitem[Hu and Lu(2020)]{hu2020universality}
Hong Hu and Yue~M Lu.
\newblock Universality laws for high-dimensional learning with random features.
\newblock \emph{arXiv preprint arXiv:2009.07669}, 2020.

\bibitem[Javanmard and Montanari(2013)]{javanmard2013state}
Adel Javanmard and Andrea Montanari.
\newblock State evolution for general approximate message passing algorithms,
  with applications to spatial coupling.
\newblock \emph{Information and Inference: A Journal of the IMA}, 2\penalty0
  (2):\penalty0 115--144, 2013.

\bibitem[Kabashima(2003)]{kabashima2003cdma}
Yoshiyuki Kabashima.
\newblock A {CDMA} multiuser detection algorithm on the basis of belief
  propagation.
\newblock \emph{Journal of Physics A: Mathematical and General}, 36\penalty0
  (43):\penalty0 11111, 2003.

\bibitem[Korada and Montanari(2011)]{korada2011applications}
Satish~Babu Korada and Andrea Montanari.
\newblock Applications of the {L}indeberg principle in communications and
  statistical learning.
\newblock \emph{IEEE transactions on information theory}, 57\penalty0
  (4):\penalty0 2440--2450, 2011.

\bibitem[Lu(2021)]{lu2021householder}
Yue~M Lu.
\newblock Householder dice: A matrix-free algorithm for simulating dynamics on
  gaussian and random orthogonal ensembles.
\newblock \emph{IEEE Transactions on Information Theory}, 67\penalty0
  (12):\penalty0 8264--8272, 2021.

\bibitem[Ma and Ping(2017)]{ma2017orthogonal}
Junjie Ma and Li~Ping.
\newblock Orthogonal {AMP}.
\newblock \emph{IEEE Access}, 5:\penalty0 2020--2033, 2017.

\bibitem[Ma et~al.(2021)Ma, Dudeja, Xu, Maleki, and Wang]{ma2021spectral}
Junjie Ma, Rishabh Dudeja, Ji~Xu, Arian Maleki, and Xiaodong Wang.
\newblock Spectral method for phase retrieval: an expectation propagation
  perspective.
\newblock \emph{IEEE Transactions on Information Theory}, 67\penalty0
  (2):\penalty0 1332--1355, 2021.

\bibitem[Maillard et~al.(2020)Maillard, Loureiro, Krzakala, and
  Zdeborov{\'a}]{maillard2020phase}
Antoine Maillard, Bruno Loureiro, Florent Krzakala, and Lenka Zdeborov{\'a}.
\newblock Phase retrieval in high dimensions: Statistical and computational
  phase transitions.
\newblock \emph{Advances in Neural Information Processing Systems},
  33:\penalty0 11071--11082, 2020.

\bibitem[Mar{\v{c}}enko and Pastur(1967)]{marvcenko1967distribution}
Vladimir~A Mar{\v{c}}enko and Leonid~Andreevich Pastur.
\newblock Distribution of eigenvalues for some sets of random matrices.
\newblock \emph{Mathematics of the USSR-Sbornik}, 1\penalty0 (4):\penalty0 457,
  1967.

\bibitem[Marinari et~al.(1994)Marinari, Parisi, and
  Ritort]{marinari1994replica}
Enzo Marinari, Giorgio Parisi, and Felix Ritort.
\newblock Replica field theory for deterministic models {II}. a non-random spin
  glass with glassy behaviour.
\newblock \emph{Journal of Physics A: Mathematical and General}, 27\penalty0
  (23):\penalty0 7647, 1994.

\bibitem[Mei and Montanari(2022)]{mei2022generalization}
Song Mei and Andrea Montanari.
\newblock The generalization error of random features regression: Precise
  asymptotics and the double descent curve.
\newblock \emph{Communications on Pure and Applied Mathematics}, 75\penalty0
  (4):\penalty0 667--766, 2022.

\bibitem[Mezard and Montanari(2009)]{mezard2009information}
Marc Mezard and Andrea Montanari.
\newblock \emph{Information, physics, and computation}.
\newblock Oxford University Press, 2009.

\bibitem[Minka(2013)]{minka2013expectation}
Thomas~P Minka.
\newblock Expectation propagation for approximate {B}ayesian inference.
\newblock \emph{arXiv preprint arXiv:1301.2294}, 2013.

\bibitem[Monajemi et~al.(2013)Monajemi, Jafarpour, Gavish, Donoho, Ambikasaran,
  Bacallado, Bharadia, Chen, Choi, Chowdhury, Chowdhury, Damle, Fithian, Goetz,
  Grosenick, Gross, Hills, Hornstein, Lakkam, Lee, Li, Liu, Sing-Long, Marx,
  Mittal, Monajemi, No, Omrani, Pekelis, Qin, Raines, Ryu, Saxe, Shi, Siilats,
  Strauss, Tang, Wang, Zhou, and Zhu]{monajemi2013deterministic}
Hatef Monajemi, Sina Jafarpour, Matan Gavish, David~L. Donoho, Sivaram
  Ambikasaran, Sergio Bacallado, Dinesh Bharadia, Yuxin Chen, Young Choi,
  Mainak Chowdhury, Soham Chowdhury, Anil Damle, Will Fithian, Georges Goetz,
  Logan Grosenick, Sam Gross, Gage Hills, Michael Hornstein, Milinda Lakkam,
  Jason Lee, Jian Li, Linxi Liu, Carlos Sing-Long, Mike Marx, Akshay Mittal,
  Hatef Monajemi, Albert No, Reza Omrani, Leonid Pekelis, Junjie Qin, Kevin
  Raines, Ernest Ryu, Andrew Saxe, Dai Shi, Keith Siilats, David Strauss, Gary
  Tang, Chaojun Wang, Zoey Zhou, and Zhen Zhu.
\newblock Deterministic matrices matching the compressed sensing phase
  transitions of gaussian random matrices.
\newblock \emph{Proceedings of the National Academy of Sciences}, 110\penalty0
  (4):\penalty0 1181--1186, 2013.

\bibitem[Mondelli and Venkataramanan(2021)]{mondelli2021pca}
Marco Mondelli and Ramji Venkataramanan.
\newblock {PCA} initialization for approximate message passing in rotationally
  invariant models.
\newblock \emph{Advances in Neural Information Processing Systems}, 34, 2021.

\bibitem[Montanari and Saeed(2022)]{montanari2022universality}
Andrea Montanari and Basil Saeed.
\newblock Universality of empirical risk minimization.
\newblock \emph{arXiv preprint arXiv:2202.08832}, 2022.

\bibitem[Montanari and Venkataramanan(2021)]{montanari2021estimation}
Andrea Montanari and Ramji Venkataramanan.
\newblock Estimation of low-rank matrices via approximate message passing.
\newblock \emph{The Annals of Statistics}, 49\penalty0 (1):\penalty0 321--345,
  2021.

\bibitem[O'Donnell(2014)]{o2014analysis}
Ryan O'Donnell.
\newblock \emph{Analysis of {B}oolean functions}.
\newblock Cambridge University Press, 2014.

\bibitem[Opper and Winther(2001)]{opper2001adaptive}
Manfred Opper and Ole Winther.
\newblock Adaptive and self-averaging thouless-anderson-palmer mean-field
  theory for probabilistic modeling.
\newblock \emph{Physical Review E}, 64\penalty0 (5):\penalty0 056131, 2001.

\bibitem[Opper et~al.(2005)Opper, Winther, and Jordan]{opper2005expectation}
Manfred Opper, Ole Winther, and Michael~J Jordan.
\newblock Expectation consistent approximate inference.
\newblock \emph{Journal of Machine Learning Research}, 6\penalty0 (12), 2005.

\bibitem[Oymak and Hassibi(2014)]{oymak2014case}
Samet Oymak and Babak Hassibi.
\newblock A case for orthogonal measurements in linear inverse problems.
\newblock In \emph{2014 IEEE International Symposium on Information Theory},
  pages 3175--3179. IEEE, 2014.

\bibitem[Panahi and Hassibi(2017)]{panahi2017universal}
Ashkan Panahi and Babak Hassibi.
\newblock A universal analysis of large-scale regularized least squares
  solutions.
\newblock \emph{Advances in Neural Information Processing Systems}, 30, 2017.

\bibitem[Parisi and Potters(1995)]{parisi1995mean}
Giorgio Parisi and Marc Potters.
\newblock Mean-field equations for spin models with orthogonal interaction
  matrices.
\newblock \emph{Journal of Physics A: Mathematical and General}, 28\penalty0
  (18):\penalty0 5267, 1995.

\bibitem[Rangan et~al.(2019)Rangan, Schniter, and Fletcher]{rangan2019vector}
Sundeep Rangan, Philip Schniter, and Alyson~K Fletcher.
\newblock Vector approximate message passing.
\newblock \emph{IEEE Transactions on Information Theory}, 65\penalty0
  (10):\penalty0 6664--6684, 2019.

\bibitem[Rush et~al.(2017)Rush, Greig, and Venkataramanan]{rush2017capacity}
Cynthia Rush, Adam Greig, and Ramji Venkataramanan.
\newblock Capacity-achieving sparse superposition codes via approximate message
  passing decoding.
\newblock \emph{IEEE Transactions on Information Theory}, 63\penalty0
  (3):\penalty0 1476--1500, 2017.

\bibitem[Schm{\"u}dgen(2017)]{schmudgen2017moment}
Konrad Schm{\"u}dgen.
\newblock \emph{The moment problem}, volume~9.
\newblock Springer, 2017.

\bibitem[Sherrington and Kirkpatrick(1975)]{sherrington1975solvable}
David Sherrington and Scott Kirkpatrick.
\newblock Solvable model of a spin-glass.
\newblock \emph{Physical review letters}, 35\penalty0 (26):\penalty0 1792,
  1975.

\bibitem[Sodin(2014)]{sodin2014several}
Sasha Sodin.
\newblock Several applications of the moment method in random matrix theory.
\newblock \emph{arXiv preprint arXiv:1406.3410}, 2014.

\bibitem[Sur and Cand{\`e}s(2019)]{sur2019modern}
Pragya Sur and Emmanuel~J Cand{\`e}s.
\newblock A modern maximum-likelihood theory for high-dimensional logistic
  regression.
\newblock \emph{Proceedings of the National Academy of Sciences}, 116\penalty0
  (29):\penalty0 14516--14525, 2019.

\bibitem[Sur et~al.(2019)Sur, Chen, and Cand{\`e}s]{sur2019likelihood}
Pragya Sur, Yuxin Chen, and Emmanuel~J Cand{\`e}s.
\newblock The likelihood ratio test in high-dimensional logistic regression is
  asymptotically a rescaled chi-square.
\newblock \emph{Probability theory and related fields}, 175\penalty0
  (1):\penalty0 487--558, 2019.

\bibitem[Takeuchi(2017)]{takeuchi2017rigorous}
Keigo Takeuchi.
\newblock Rigorous dynamics of expectation-propagation-based signal recovery
  from unitarily invariant measurements.
\newblock In \emph{2017 IEEE International Symposium on Information Theory
  (ISIT)}, pages 501--505. IEEE, 2017.

\bibitem[Tulino et~al.(2010)Tulino, Caire, Shamai, and
  Verd{\'u}]{tulino2010capacity}
Antonia~M Tulino, Giuseppe Caire, Shlomo Shamai, and Sergio Verd{\'u}.
\newblock Capacity of channels with frequency-selective and time-selective
  fading.
\newblock \emph{IEEE Transactions on Information Theory}, 56\penalty0
  (3):\penalty0 1187--1215, 2010.

\bibitem[Voiculescu(1991)]{voiculescu1991limit}
Dan Voiculescu.
\newblock Limit laws for random matrices and free products.
\newblock \emph{Inventiones mathematicae}, 104\penalty0 (1):\penalty0 201--220,
  1991.

\bibitem[Voiculescu(1998)]{voiculescu1998strengthened}
Dan Voiculescu.
\newblock A strengthened asymptotic freeness result for random matrices with
  applications to free entropy.
\newblock \emph{International Mathematics Research Notices}, 1998\penalty0
  (1):\penalty0 41--63, 1998.

\bibitem[Wendel(1962)]{wendel1962problem}
James~G Wendel.
\newblock A problem in geometric probability.
\newblock \emph{Mathematica Scandinavica}, 11\penalty0 (1):\penalty0 109--111,
  1962.

\bibitem[Wigner(1958)]{wigner1958distribution}
Eugene~P Wigner.
\newblock On the distribution of the roots of certain symmetric matrices.
\newblock \emph{Annals of Mathematics}, pages 325--327, 1958.

\bibitem[Winder(1966)]{winder1966partitions}
Robert~O Winder.
\newblock Partitions of {N}-space by hyperplanes.
\newblock \emph{SIAM Journal on Applied Mathematics}, 14\penalty0 (4):\penalty0
  811--818, 1966.

\end{thebibliography}

\appendix

\section{Omitted Proofs from \sref{matrix-ensemble-eg}}
\label{appendix:examples}
\subsection{Proof of \lemref{sign-perm-inv-ensemble}}
\label{appendix:sign-perm-inv-ensemble}
\begin{proof}[Proof of \lemref{sign-perm-inv-ensemble}] Let $\mS = \diag{(\vs)}, \vs \sim \unif{\{\pm 1\}^\dim}$ be a uniformly random signed diagonal matrix and $\mP$ be a uniformly random $\dim \times \dim$ permutation matrix independent of each other and independent of $\mF$.  By the assumptions of the lemma, we have, $\mM \explain{d}{=} \widetilde{\mM} \explain{d}{=} \mS \mF \mP \mLambda \mP\tran \mF \tran \mS$.  Let $\mPsi$ denote the matrix:
\begin{align*}{\mPsi \explain{def}{=} \mF \mP\cdot  \mLambda \cdot  \mP\tran \mF}.
\end{align*}
As a consequence of the distributional equivalence of $\mM$ and $\widetilde{\mM} \explain{def}{=} \mS \mPsi \mS$, it is sufficient to verify that $\widetilde{\mM}$ satisfies the requirements of \defref{matrix-ensemble-relaxed} with probability 1. Note that $\widetilde{\mM} \explain{def}{=} \mS \mPsi \mS$ satisfies item (1) of  \defref{matrix-ensemble-relaxed} by construction. In order to verify item (2) of  \defref{matrix-ensemble-relaxed} observe that we can write:
\begin{align*}
   {\mPsi} & = \E[ {\mPsi}] + ({\mPsi} -  \E[ {\mPsi}]), \\
     {{\mPsi}}^2 & = \E[ {\mPsi}^2] + ({{\mPsi}}^2- \E[ {{\mPsi}}^2]).
\end{align*}
Next, we compute $\E[ {\mPsi}]$ and $\E[ {\mPsi}^2]$, where the expectation is with respect to the randomness in $\mP$.  Let $\tau$ be the uniformly random permutation associated with $\mP$ and let $\lambda_{1:\dim}$ denote the diagonal entries of $\mLambda$. We have,
\begin{subequations}\label{eq:random-permutation-expectation}
\begin{align}
    \E[\psi_{ij}] & = \sum_{k=1}^\dim f_{ik} f_{jk} \E[\lambda_{\tau(k)}] = \frac{\Tr(\mLambda)}{\dim} \cdot (\mF \mF\tran)_{ij}. \\
    \E[(\mPsi^2)_{ij}] & = \sum_{k=1}^\dim f_{ik} f_{jk} \E[\lambda_{\tau(k)}^2] = \frac{\Tr(\mLambda^2)}{\dim} \cdot (\mF \mF\tran)_{ij}.
\end{align}
\end{subequations}
Next, we obtain a concentration estimate for $\|{\mPsi} -  \E[ {\mPsi}]\|_\infty$ and $\|{\mPsi^2} -  \E[ {\mPsi^2}]\|_\infty$. For any matrix $\mA \in \R^{\dim \times \dim}$, \citet[Theorem 4.3]{bercu2015concentration} show that the statistic :
\begin{align*}
    X_\dim(\mA) \bydef \sum_{\ell=1}^\dim A_{\ell, \tau(\ell)} ,
\end{align*}
constructed using a uniformly random permutation $\tau$ satisfies the following concentration estimate for any $\delta \in (0,1/4)$:
\begin{align*}
    \P\left( |X_\dim(\mA) - \E X_\dim(\mA) | > \frac{C\cdot \|\mA\|_{\fr}}{\sqrt{\dim}} \cdot \sqrt{\ln\left( \frac{1}{\delta}\right)} + C \cdot \|\mA\|_\infty \ln\left( \frac{1}{\delta}\right) \right) & \leq 4\delta,
\end{align*}
for a suitable constant $0<C<\infty$. Fix any $i,j \in [\dim]$. We apply the above concentration inequality with the choice $\delta = 1/\dim^4$ and $\mA = \iter{\mA}{ij}$ where the entries of $\iter{\mA}{ij}$ are $ \iter{A}{ij}_{k\ell} \explain{def}{=} f_{ik} f_{jk} \lambda_\ell $. Observe that,
\begin{align*}
    X_\dim(\iter{\mA}{ij}) &= \psi_{ij}, \\
    \|\iter{\mA}{ij}\|_{\fr}^2 & = \left(\sum_{k=1}^\dim f_{ik}^2 f_{jk}^2 \right) \cdot \left( \sum_{\ell = 1}^\dim \lambda_\ell^2 \right) \leq \dim^2 \|\mF\|_\infty^4 \|\mLambda\|_{\op}^2, \\
    \|\iter{\mA}{ij}\|_{\infty} & \leq \|\mF\|_\infty^2 \|\mLambda\|_{\op}.
\end{align*}
Hence,
\begin{align*}
  \P\left( |\psi_{ij} - \E \psi_{ij}| > 2C \cdot \|\mF\|_\infty^2 \cdot \|\mLambda\|_{\op} \cdot \left( \sqrt{\dim \ln(\dim)} + 2\ln(\dim) \right) \right)  & \leq 1/\dim^4.
\end{align*}
Using a union bound,
\begin{align*}
     \P\left( \|\mPsi - \E \mPsi\|_\infty > 2C \cdot \|\mF\|_\infty^2 \cdot \|\mLambda\|_{\op} \cdot \left( \sqrt{\dim \ln(\dim)} + 2\ln(\dim) \right) \right)  & \leq 1/\dim^2.
\end{align*}
Using the assumption $\|\mF\|_{\infty} \lesssim \dim^{-1/2+\epsilon}$ and the observation $\|\mLambda\|_{\op} \explain{\eqref{eq:op-norm-bound}}{\lesssim} 1$ along with the Borel-Cantelli lemma we obtain,
\begin{subequations}\label{eq:random-permutation-concentration}
\begin{align}
    \P\left(  \|{\mPsi} - \E {\mPsi}\|_\infty \lesssim  \dim^{-\frac{1}{2} + \epsilon} \; \forall \; \epsilon > 0\ \right) & = 1. 
\end{align}
By an analogous argument, we obtain,
\begin{align}
    \P\left(  \|{\mPsi}^2 - \E {{\mPsi}}\|_\infty \lesssim  \dim^{-\frac{1}{2} + \epsilon} \; \forall \; \epsilon > 0\ \right) & = 1.
\end{align}
\end{subequations}
We can now verify requirements (2a), (2b), (2c) and (2d) of \defref{matrix-ensemble-relaxed}.
\begin{enumerate}
    \item To verify (2a), we observe:
    \begin{align*}
        \|{\mPsi}\|_\infty &\leq \|{\mPsi} - \E{\mPsi}|_\infty + \|\E[\mPsi]\|_\infty \\
        & \explain{(a)}{\lesssim} \|\mPsi - \E[\mPsi]\|_\infty + \dim^{-\frac{1}{2} + \epsilon}\\
        & \explain{(b)}{\lesssim} \dim^{-\frac{1}{2} + \epsilon}.
    \end{align*}
    \item To verify (2b), observe that $\|\mPsi\|_{\op} = \|\mLambda\|_{\op} \explain{}{\lesssim} 1$. 
    \item To verify (2c), we observe:
    \begin{align*}
        \max_{i\neq j}|({\mPsi}^2)_{ij}| &\leq \|\mPsi^2 - \E{\mPsi}^2\|_\infty + \max_{i\neq j}|\E({\mPsi}^2)_{ij}| \\
        & \explain{(a)}{=} \|{\mPsi}^2 - \E{\mPsi}^2\|_\infty \\
        & \explain{(b)}{\lesssim} \dim^{-\frac{1}{2} + \epsilon}.
    \end{align*}
    \item  To verify (2d), we observe:
    \begin{align*}
        \max_{i\in [\dim]}|({\mPsi}^2)_{ii} - \sigpsi| &\leq \|{\mPsi}^2 - \E{\mPsi}^2\|_\infty + \max_{i\in [\dim]}|\E({\mPsi}^2)_{ii} -   \sigpsi| \\
        & \explain{(a)}{=} \|{\mPsi}^2 - \E{\mPsi}^2\|_\infty + \left| \frac{\Tr(\mLambda^2)}{\dim} - \sigpsi \right| \\
        & \explain{(b)}{\ll} 1.
    \end{align*}
\end{enumerate}
In each of the above displays, we used the formulae for $\E[\mPsi]$ and $\E[\mPsi^2]$ from \eref{random-permutation-expectation} in steps marked (a) and the concentration estimates from \eref{random-permutation-concentration} in the steps marked (b). This proves the claim of the lemma. 
\end{proof}
\subsection{Proof of \lemref{wigner}}
\label{appendix:local-law}
Recall that $\mM(\lambda)$ was given by:
\begin{align*}
    \mM(\lambda) = (\lambda \mI_{\dim} - \mJ )^{-1} - \frac{\Tr((\lambda \mI_{\dim} - \mJ )^{-1})}{\dim} \cdot \mI_{\dim}.
\end{align*}
In the above display, $\mJ$ be a Wigner matrix with symmetric entries, that is, $\mJ = \mW/\sqrt{\dim}$ for a symmetric matrix $\mW$ whose entries $W_{ij}$ are i.i.d. symmetric ($W_{ij} \explain{d}{=} - W_{ij})$ random variables with $\E W_{ij} = 0$, $\E W_{ij}^2 = 1 + \delta_{ij}$ and finite moments of all orders. Our goal is to verify that $\mM(\lambda)$ satisfies \defref{matrix-ensemble-relaxed}.  Since the entries of $\mW$ are assumed to be symmetric, $\mM(\lambda)$ satisfies item (1) of \defref{matrix-ensemble-relaxed}. The remaining requirements can be verified using the local law for Wigner matrices, as stated in \citet{benaych2016lectures}. In particular, \lemref{wigner} follows immediately from the following result. 

\begin{lemma}[{\citet[Theorem 10.3 and Remark 2.7]{benaych2016lectures}}]\label{lem:local-law} Under the above hypotheses, for any fixed $\lambda > 2$ (independent of $\dim$), $\mR(\lambda) \explain{def}{=} (\lambda \mI_{\dim} - \mJ)^{-1}$ satisfies:
\begin{subequations}
\begin{align}
    \P \left( \| \mR(\lambda) - G_{\mathrm{sc}}(\lambda) \cdot \mI_{\dim} \|_{\infty} \leq \dim^{-1/2 + \epsilon}\right) \geq 1 - \dim^{-D}, \label{eq:local-law-resolvent} \\
    \P \left( \| \mR^2(\lambda) + G_{\mathrm{sc}}^\prime(\lambda) \cdot \mI_{\dim} \|_{\infty} \leq \dim^{-1/2 + \epsilon}\right) \geq 1 - \dim^{-D}. \label{eq:local-law-resolvent-sq}
\end{align}
\end{subequations}
for any fixed constants $\epsilon > 0, \; D \in \N$ (independent of $\dim$) and $\dim \geq \dim_0(\epsilon,D, \lambda)$.  In the above display, $G_{\mathrm{sc}}$ denotes the Cauchy Transform of the semi-circle distribution supported on $[-2,2]$ and $\dim_0(\epsilon, D, \lambda) \in \N$ is a finite integer that depends only on $\epsilon, D, \lambda$.
\end{lemma}
\begin{proof}  Define the constant $\delta \explain{def}{=} (\lambda - 2) > 0$ and let $A = [-2,2]$ denote the support of the semicircle law.  A direct consequence of   \citet[Theorem 10.3, Remark 2.7]{benaych2016lectures} is that for any fixed $\epsilon > 0 , D \in \N$, there is a finite integer $N_0(\epsilon, D, \delta) \in \N$ such that,
\begin{align}\label{eq:local-law}
    \P \Bigg( \sup_{\substack{z \in \C \; :\;  \mathrm{dist}(A, z) \in [\delta/6,\delta] }}  \|\mR(z)   - G_{\mathrm{sc}}(z) \mI_{\dim}\|_{\infty} \leq \dim^{-1/2 + \epsilon} \Bigg) \geq 1 - \dim^{-D} \; \forall \; \dim \geq N_0(\epsilon, D, \delta).
\end{align}
Since $\mathrm{dist}(A, \lambda) = \delta$, the first claim \eqref{eq:local-law-resolvent} of the lemma follows immediately. In order to obtain the second claim \eqref{eq:local-law-resolvent-sq}, we will follow the argument employed in the proof of \citet[Theorem 10.3, Case 2]{benaych2016lectures}. By a polarization argument, \eqref{eq:local-law-resolvent-sq} follows if we show:
\begin{align}\label{eq:local-law-goal}
     \P \Bigg(   |\vv\tran \mR^2(\lambda) \vv  + G_{\mathrm{sc}}^\prime(\lambda)| \leq C\dim^{-1/2 + \epsilon} \Bigg) \geq 1 - 2\dim^{-D} \; \forall \; \dim \geq N_0(\epsilon, D, \delta),
\end{align}
for a fixed unit vector $\vv \in \R^{\dim}$ with at most $2$ non-zero coordinates and fixed finite constant $C$ (independent of $\dim$).  Let $\{(\lambda_i(\mJ), \vu_i): i \in [\dim]\}$ denote the eigenvalues and eigenvectors of $\mJ$. Following \citep{benaych2016lectures} we define the signed measure:
\begin{align*}
    \widetilde{\xi}_{\dim} & \explain{def}{=} \sum_{i=1}^{\dim} |\ip{\vv}{\vu_i}|^2 \cdot  \delta_{\lambda_i} - \xi_{\mathrm{sc}}.
\end{align*}
We also define the holomorphic function $f_{\lambda} : \C \backslash \{\lambda\} \rightarrow \C$ as $f_\lambda(x) \explain{def}{=} 1/(\lambda - x)^2$. Observe that:
\begin{align}
    \vv\tran \mR^2(\lambda) \vv  + G_{\mathrm{sc}}^\prime(\lambda) & = \int_{\R} f_{\lambda}(x) \; \widetilde{\xi}_{\dim}(\diff x) ,\label{eq:local-law-sq-integral} \\
    \vv\tran \mR(w) \vv  - G_{\mathrm{sc}}(w) & = \int_{\R} \frac{1}{x-w} \; \widetilde{\xi}_{\dim}(\diff x). \label{eq:local-law-integral}
\end{align}
In order to upper bound the RHS of the above display, we will use the Helffer-Sjostrand formula (see for e.g. \citep[Appendix C, Proposition C.1]{benaych2016lectures}). Let $\chi: \C \rightarrow [0,1]$ be a infinitely differentiable cutoff function such that:
\begin{align*}
    \chi(z)  = 1  \text{ if } \mathrm{dist}(A,z) \leq \delta/6, \;   \chi(z)  = 0  \text{ if } \mathrm{dist}(A,z) \geq \delta/3.
\end{align*}
By the Helffer-Sjostrand formula, we have,
\begin{align*}
    f_{\lambda}(x) \chi(x) & = \frac{1}{\pi} \int_{\C} \frac{\partial_{\overline{w}}[ f_{\lambda}(w) \chi(w)] }{x-w} \diff^2 w,
\end{align*}
where $\partial_{\overline{w}}$ is the Wirtinger derivative with respect to  $\overline{w}$ and the integral is with respect to the Lebesgue measure on $\C$. Since $f_{\lambda}$ is holomorphic on the domain where $\chi \neq 0$, $\chi \cdot \partial_{\overline{w}} f_{\lambda} = 0$. Hence,
\begin{align}\label{eq:HS-representation}
     f_{\lambda}(x) \chi(x) & = \frac{1}{\pi} \int_{\C} \frac{ f_{\lambda}(w) \partial_{\overline{w}}[\chi(w)] }{x-w} \diff^2 w.
\end{align}
With this formula, we can upper bound \eqref{eq:local-law-sq-integral}:
\begin{align*}
    |\vv\tran \mR^2(\lambda) \vv  + G_{\mathrm{sc}}^\prime(\lambda)| & =  \left|\int_{\R} f_{\lambda}(x) \; \widetilde{\xi}_{\dim}(\diff x) \right| \\
    & \explain{(a)}{=} \left|\int_{\R} f_{\lambda}(x) \cdot \chi(x) \; \widetilde{\xi}_{\dim}(\diff x) \right| \\
    &  \explain{\eqref{eq:HS-representation}}{=} \frac{1}{\pi}\left| \int_{\R}  \widetilde{\xi}_{\dim}(\diff x) \;  \int_{\C} \frac{ f_{\lambda}(w) \partial_{\overline{w}}[\chi(w)] }{x-w} \diff^2 w   \right| \\
    & \explain{\eqref{eq:local-law-integral}}{=} \frac{1}{\pi}\left| \int_{\C} f_{\lambda}(w) \cdot (\vv\tran \mR(w) \vv  - G_{\mathrm{sc}}(w)) \cdot  \partial_{\overline{w}}[\chi(w)]  \diff^2 w \right| \\
    & \explain{(b)}{\leq} \left( \frac{1}{\pi} \int_{\C} |f_{\lambda}(w)| |\partial_{\overline{w}}[\chi(w)]|  \diff^2 w \right) \cdot  2\dim^{-1/2+\epsilon} \\
    & \explain{(c)}{\leq} 2C(\lambda,\delta) \dim^{-1/2+\epsilon}.
\end{align*}
In the above display, in the step marked (a), we used the fact that $\widetilde{\xi}_{\dim}$ is supported on the set $$\{x \in \R: \mathrm{dist}(x,A) \leq \delta/6\}$$ with probability $1- \dim^{-D}$ (see for e.g. \citep[Theorem 2.9]{benaych2016lectures}). In the step marked (b), we relied on the fact that on the domain where $\partial_{\overline{w}}[\chi(w)]  \neq 0$, \eqref{eq:local-law} guarantees that $|\vv\tran \mR(w) \vv  - G_{\mathrm{sc}}(w)| \leq 2\dim^{-1/2+\epsilon}$ with probability $1-\dim^{-D}$ for any unit vector $\vv$ with at most two non-zero coordinates. Finally in the step marked (c) we observed that the integral in step (b) can be bounded by a finite and fixed constant $C(\lambda,\delta)$ since $\partial_{\overline{w}}[\chi(w)] \neq 0$ on a compact subset of $\C$ and on this subset, $f_{\lambda}(w)$ is holomorphic and hence uniformly bounded. This proves the desired claim \eqref{eq:local-law-goal}. 
\end{proof}

\subsection{Proof of \lemref{wishart}}\label{appendix:local-law-mp}
Recall that $\mM(\lambda)$ was given by:
\begin{align*}
    \mM(\lambda) = (\lambda \mI_{\dim} - \mJ )^{-1} - \frac{\Tr((\lambda \mI_{\dim} - \mJ )^{-1})}{\dim} \cdot \mI_{\dim}.
\end{align*}
In the above display, $\mJ$  is a covariance matrix of the form $\mJ = \mX \tran \mX / \sqrt{M \dim}$ for a $M \times \dim$ matrix $\mX$ with a converging aspect ratio $M/N \rightarrow \phi \in (0, \infty)$ whose entries $X_{ij}$ are i.i.d. symmetric ($X_{ij} \explain{d}{=} - X_{ij}$) random variables with $\E X_{ij} = 0$, $\E X_{ij}^2 = 1$ and finite moments of all orders. Our goal is to verify that $\mM(\lambda)$ satisfies \defref{matrix-ensemble-relaxed}. Indeed, since the entries of $\mX$ are assumed to be symmetric, $\mM(\lambda)$ satisfies requirement (1) of \defref{matrix-ensemble-relaxed}. The remaining requirements follow immediately from the local law for sample covariance matrices obtained by \citet{alex2014isotropic}, stated in the lemma given below. 

\begin{lemma}[\citet{alex2014isotropic}] Under the above hypotheses, for any fixed $\lambda > \lambda_{+}^{\mathrm{MP}}$ (independent of $\dim$), $\mR(\lambda) \explain{def}{=} (\lambda \mI_{\dim} - \mJ)^{-1}$ satisfies:
\begin{subequations}
\begin{align}
    \P \left( \| \mR(\lambda) - G_{\mathrm{MP}}(\lambda) \cdot \mI_{\dim} \|_{\infty} \leq \dim^{-1/2 + \epsilon}\right) \geq 1 - \dim^{-D}, \label{eq:local-law-resolvent-mp} \\
    \P \left( \| \mR^2(\lambda) + G_{\mathrm{MP}}^\prime(\lambda) \cdot \mI_{\dim} \|_{\infty} \leq \dim^{-1/2 + \epsilon}\right) \geq 1 - \dim^{-D}. \label{eq:local-law-resolvent-sq-mp}
\end{align}
\end{subequations}
for any fixed constants $\epsilon > 0, \; D \in \N$ (independent of $\dim$) and $\dim \geq \dim_0(\epsilon,D, \lambda)$. In the above display, $G_{\mathrm{MP}}$ denotes the Cauchy Transform of the Marchenko-Pastur distribution and $\dim_0(\epsilon, D, \lambda) \in \N$ is a finite integer that depends only on $\epsilon, D, \lambda$.
\end{lemma}
\begin{proof} The claim \eqref{eq:local-law-resolvent-mp} is a direct consequence \citet[Theorem 2.5 and Remark 2.7]{alex2014isotropic}. \eqref{eq:local-law-resolvent-sq-mp} can be obtained from the results of \citet[Theorem 2.5]{alex2014isotropic} by following the same argument of \citet{benaych2016lectures} used to obtain \lemref{local-law} (described in \appref{local-law}). We omit the details. 
\end{proof}

\section{Equivalence of \thref{SE-alt} and \thref{SE-final}}
\label{appendix:equivalence}
This section is devoted to the proof of \lemref{equivalence}. 
\begin{proof}[Proof of \lemref{equivalence}]
Let us assume that \thref{SE-final} holds. Consider the iteration:
\begin{align*}
     \iter{\vz}{t+1} & = \mM \cdot \left( \nonlin_{t+1}(\iter{\vz}{t}) - \frac{\langle{\nonlin_{t+1}(\iter{\vz}{t})}\; , \; {\iter{\vz}{t}} \rangle}{\|\iter{\vz}{t}\|^2} \cdot \iter{\vz}{t} \right),
\end{align*}
initialized at $\iter{\vz}{0}$, which satisfies the assumptions of \thref{SE-alt}. Since the non-linearities $\nonlin_t$ need not be divergence free, we center them suitably to make them divergence free. Define $\ctr{\nonlin}_t$ as:
\begin{align*}
    \ctr{\nonlin}_t(x) = {\nonlin}_t(x) - \frac{\E[Z \nonlin_t(\sigma_{t-1}Z)]}{\sigma_{t-1}} \cdot x, \; Z \sim \gauss{0}{1},
\end{align*}
where $\sigma_t^2$ is as defined in the state evolution recursion \eqref{eq:SE-memory-free}. In particular $\E[Z\ctr{\nonlin}_t(\sigma_{t-1} Z)] = 0$ for $Z \sim \gauss{0}{1}$ and hence the non-linearities $\ctr{\nonlin}_t$ are divergence-free (\assumpref{divergence-free}). Hence, by \thref{SE-final}, the iteration:
\begin{align*}
    \iter{\vz}{t+1} & = \mM  \ctr{\nonlin}_{t+1}(\iter{\vz}{t}),
\end{align*}
initialized at $\iter{\ctr{\vz}}{0} = \iter{\vz}{0}$ satisfies:
\begin{align*}
    \frac{1}{\dim} \sum_{i=1}^\dim h(\iter{\ctr{z}}{0}_i, \iter{\ctr{z}}{1}_i, \dotsc, \iter{\ctr{z}}{T}_i) \explain{P}{\rightarrow} \E h(Z_0, Z_1, \dotsc, Z_T),
\end{align*}
where $(Z_0, Z_1, \dotsc, Z_T) \sim \gauss{\vzero}{\mSigma_T}$ with $\mSigma_T$ as defined in the state evolution recursion \eqref{eq:SE-memory-free}. Hence, if we show that:
\begin{align} \label{eq:equivalence} 
     \frac{1}{\dim} \sum_{i=1}^\dim h(\iter{\ctr{z}}{0}_i, \iter{\ctr{z}}{1}_i, \dotsc, \iter{\ctr{z}}{T}_i) -  \frac{1}{\dim} \sum_{i=1}^\dim h(\iter{{z}}{0}_i, \iter{{z}}{1}_i, \dotsc, \iter{{z}}{T}_i) \explain{P}{\rightarrow}  0,
\end{align}
\thref{SE-alt} is immediate. By the continuity hypothesis on $h$, we have,
\begin{align}
   &\frac{1}{\dim} \sum_{i=1}^\dim |h(\iter{\ctr{z}}{0}_i, \iter{\ctr{z}}{1}_i, \dotsc, \iter{\ctr{z}}{T}_i) - h(\iter{{z}}{0}_i, \iter{{z}}{1}_i, \dotsc, \iter{{z}}{T}_i) | \nonumber\\ &\leq  \frac{L}{\dim} \sum_{i=1}^\dim (1 + ( |\iter{\ctr{z}}{0}_i|^2 + \dotsb + |\iter{\ctr{z}}{T}_i|^2)^{1/2} + ( |\iter{{z}}{0}_i|^2 + \dotsb + |\iter{{z}}{T}_i|^2)^{1/2}) \cdot (|\iter{\ctr{z}}{0}_i-\iter{{z}}{0}_i|^2 + \dotsb + |\iter{\ctr{z}}{T}_i-\iter{{z}}{T}_i|^2  )^{1/2} \nonumber \\
   & \leq \sqrt{3} L \cdot \left(1 + \sum_{t=0}^T \frac{\|\iter{\vz}{t}\|^2 +\|\iter{\ctr{\vz}}{t}\|^2 }{\dim}   \right)^{1/2} \cdot \left( \sum_{t=0}^T \frac{\|\iter{\vz}{t}-\iter{\ctr{\vz}}{t}\|^2 }{\dim}  \right)^{1/2} \nonumber \\
   & \leq \sqrt{3} L \cdot \left(1 + \sum_{t=0}^T \frac{\|\iter{\vz}{t}-\iter{\ctr{\vz}}{t}\|^2 +3\|\iter{\ctr{\vz}}{t}\|^2 }{\dim}   \right)^{1/2} \cdot \left( \sum_{t=0}^T \frac{\|\iter{\vz}{t}-\iter{\ctr{\vz}}{t}\|^2 }{\dim}  \right)^{1/2} \label{eq:pl-2-estimate}
\end{align}
Since $\|\iter{\ctr{\vz}}{t}\|^2/\dim \explain{P}{\rightarrow} \E Z_t^2$ by \thref{SE-final}, in order to conclude \eqref{eq:equivalence}, it is sufficient to show that for any $t \in [T]$,
\begin{align*}
    \frac{\|\iter{\vz}{t}-\iter{\ctr{\vz}}{t}\|^2 }{\dim}  \explain{P}{\rightarrow} 0.
\end{align*}
We will show this by induction. As our induction hypothesis, we assume that the above claim holds. In order to show this claim for $t+1$ observe that:
\begin{align*}
    &\frac{\|\iter{\vz}{t+1}-\iter{\ctr{\vz}}{t+1}\|^2 }{\dim}  \leq \frac{\|\mM\|_{\op}^2}{\dim} \cdot \left\| \nonlin_{t+1}(\iter{\vz}{t}) - \nonlin_{t+1}(\iter{\ctr{\vz}}{t}) + \frac{\E[Z \nonlin_{t+1}(\sigma_t Z)]}{\sigma_t} \cdot \iter{\ctr{\vz}}{t}  -  \alpha_t \cdot \iter{\vz}{t}  \right\|^2 \\
    & \hspace{1.5cm}\leq \frac{3\|\mM\|_{\op}^2}{\dim} \cdot \left( \left\|\nonlin_{t+1}(\iter{\vz}{t}) - \nonlin_{t+1}(\iter{\ctr{\vz}}{t}) \right\|^2+ \left( \alpha_t - \frac{\E[Z \nonlin_{t+1}(\sigma_t Z)]}{\sigma_t}  \right)^2 \|\iter{\ctr{\vz}}{t}\|^2 + \alpha_t^2 \cdot \|\iter{\vz}{t}-\iter{\ctr{\vz}}{t}\|^2\right) \\
    & \hspace{1.5cm}\leq \frac{3\|\mM\|_{\op}^2}{\dim} \cdot \left( \|\nonlin_{t+1}^\prime\|_{\infty}^2\left\|\iter{\vz}{t} - \iter{\ctr{\vz}}{t} \right\|^2+ \left( \alpha_t - \frac{\E[Z \nonlin_{t+1}(\sigma_t Z)]}{\sigma_t}  \right)^2 \|\iter{\ctr{\vz}}{t}\|^2 + \alpha_t^2 \cdot \|\iter{\vz}{t}-\iter{\ctr{\vz}}{t}\|^2\right)
\end{align*}
where,
\begin{align*}
    \alpha_t &\explain{def}{=} \frac{\langle{\nonlin_{t+1}(\iter{\vz}{t})}\; , \; {\iter{\vz}{t}} \rangle}{\|\iter{\vz}{t}\|^2}.
\end{align*}
Observe that to complete the induction step, we need to show that \begin{align} \label{eq:alpha-plim}
    \alpha_t &\explain{def}{=} \frac{\langle{\nonlin_{t+1}(\iter{\vz}{t})}\; , \; {\iter{\vz}{t}} \rangle}{\|\iter{\vz}{t}\|^2} \explain{P}{\rightarrow} \frac{\E[Z \nonlin_{t+1}(\sigma_t Z)]}{\sigma_t}.
\end{align}
In order to conclude this, we observe that by the induction hypothesis, $\|\iter{\vz}{t}\|^2/\dim - \|\iter{\ctr{\vz}}{t}\|^2/\dim \explain{P}{\rightarrow} 0$. On the other hand, \thref{SE-final} shows that $\|\iter{\ctr{\vz}}{t}\|^2/\dim \explain{P}{\rightarrow} \E Z_t^2 = \sigma_t^2$. Hence, $\|\iter{\vz}{t}\|^2/\dim  \explain{P}{\rightarrow}\sigma_t^2$. Similarly, by repeating the argument used to obtain \eqref{eq:pl-2-estimate} and using the induction hypothesis, we can show that $\langle{\nonlin_{t+1}(\iter{\vz}{t})}\; , \; {\iter{\vz}{t}} \rangle/\dim - \langle{\nonlin_{t+1}(\iter{\ctr{\vz}}{t})}\; , \; {\iter{\ctr{\vz}}{t}} \rangle/\dim \explain{P}{\rightarrow} 0$. By \thref{SE-final}, $\langle{\nonlin_{t+1}(\iter{\ctr{\vz}}{t})}\; , \; {\iter{\ctr{\vz}}{t}} \rangle/\dim \explain{P}{\rightarrow}  \E[ Z_t \nonlin_{t+1}(Z_t)] = \sigma_t \cdot \E[Z \nonlin_{t+1}(\sigma_t Z)]$. Hence, $\langle{\nonlin_{t+1}(\iter{\vz}{t})}\; , \; {\iter{\vz}{t}} \rangle/\dim  \explain{P}{\rightarrow}  \sigma_t  \cdot \E[Z \nonlin_{t+1}(\sigma_t Z)]$. \eqref{eq:alpha-plim} now follows from the continuous mapping theorem. This concludes the proof of this lemma. 
\end{proof}

\section{Proof of \lemref{rescaling}} \label{appendix:rescaling}
\begin{proof}[Proof of \lemref{rescaling}]
Suppose that \thref{SE-final} holds under the additional assumptions (a-c) stated in the claim of the lemma. We need to show that \thref{SE-final} also holds without these additional assumptions. Consider an iteration:
\begin{align*}
    \iter{\vz}{t+1} = \mM \nonlin_{t+1}(\iter{\vz}{t}),
\end{align*}
initialized at $\iter{\vz}{0}$ that satisfies all the assumptions of \thref{SE-final}, but not necessarily the additional assumptions (a-c) in the statement of the lemma:
\begin{enumerate}
    \item $\iter{\vz}{0} \sim \gauss{\vzero}{\sigma_0^2 \mI_\dim}$ (\assumpref{initialization}),
    \item The matrix ensemble $\mM = \mS \mPsi\mS$ is semi-random with constant $\sigpsi$ in the sense of  \defref{matrix-ensemble-relaxed}.
    \item  The non-linearities $\nonlin_t$ are Lipschitz functions that satisfy \assumpref{divergence-free}.  
\end{enumerate}
In order to show \thref{SE-final} holds for this iteration, we need to show that any test function $h:\R^{T+1} \rightarrow \R$ which satisfies:
\begin{align*}
    |h(\vx) - h(\vx^\prime)| & \leq L \cdot  \|\vx - \vx^\prime \| \cdot (1+ \|\vx\| + \|\vx^\prime\|),
\end{align*}
for some fixed constant $L \geq 0$, we have,
\begin{align}\label{eq:rescaling-goal}
    \frac{1}{\dim} \sum_{i=1}^\dim h(\iter{z}{0}_i, \iter{z}{1}_i, \dotsc, \iter{z}{T}_i) \explain{P}{\rightarrow} \E h(Z_0, Z_1, \dotsc, Z_T),
\end{align}
where $(Z_0, Z_1, \dotsc, Z_T) \sim \gauss{\vzero}{\mSigma_T}$. The matrix $\mSigma_T$ is defined using the state evolution recursions:
\begin{align}\label{eq:SE-memory-free-rescaling}
    \sigma_{t+1}^2 & = \sigpsi \cdot \E[\nonlin_{t+1}^2(Z_t)] \\
    \rho_{{s,t+1}} & = \sigpsi \cdot \E[ \nonlin_{s}(Z_{s-1})\nonlin_{t+1}(Z_t) ] \; \forall \; s \leq t.
\end{align}
initialized at $\sigma_0^2: = \sigma_0^2$ (the variance parameter of the initialization), $\rho_{0,t} := 0$ for any $t \in [T]$ and the matrix $\mSigma_T \in \R^{(T+1) \times (T+1)}$ is given by:
\begin{align}
    \mSigma_T & \explain{def}{=} \begin{bmatrix} \sigma_0^2 & \rho_{0,1} & \rho_{0,2} & \hdots  & \rho_{0,T}  \\ \rho_{0,1} & \sigma_1^2 & \rho_{1,2} & \hdots &\rho_{1,T} \\ \vdots & \vdots & \vdots & \vdots & \vdots \\ \rho_{0,T} & \rho_{1,T} & \rho_{2,T} & \hdots & \sigma_T^2 \end{bmatrix}.
\end{align}
In order to show \eqref{eq:rescaling-goal}, we will define a rescaled iteration which satisfies the additional assumptions (a-c) along with all the assumptions required by \thref{SE-final}. We define the normalized iterates:
\begin{align*}
    \iter{\widehat{\vz}}{t} = \frac{\iter{\vz}{t}}{{\sigma}_t},
\end{align*}
where ${\sigma}_t$ are given by SE recursion \eref{SE-memory-free-rescaling}.
Observe that the normalized iterates follow the dynamics:
\begin{align*}
    \iter{\hat{\vz}}{t+1} &= \hat{\mM} \hat{f}_{t+1}({\iter{\hat{\vz}}{t}}), \\
    \iter{\hat{\vz}}{0} &\sim \gauss{\vzero}{\mI_\dim}.
\end{align*}
where,
\begin{align*}
    \hat{\mM} &= \mS \widehat{\mPsi} \mS, \; \widehat{\mPsi} = \frac{1}{\sigma_{\psi}} \cdot \mPsi, \\
    \widehat{f}_{t}(z) &= \sigma_{\psi} \cdot  \frac{f_t({\sigma}_{t-1} z)}{{\sigma}_t} =   \frac{f_t( {\sigma}_{t-1} z)}{\sqrt{ \E[ \nonlin_{t}^2({\sigma}_{t-1} Z)]}}.
\end{align*}
For this iteration observe that:
\begin{enumerate}
    \item $\hat{\mM}$ is semi-random with constant $\hat{\sigma}_{\psi}^2 =1$ in the sense of \defref{matrix-ensemble-relaxed}, as required by the additional assumption (a). 
    \item The intialization $\iter{\widehat{\vz}}{0} \sim \gauss{\vzero}{\mI_\dim}$. In particular, \assumpref{initialization} and the additional assumption (b) are satisfied.
    \item $\E[ \widehat{\nonlin}_{t}^2(Z)] = 1$ and $\E[ \widehat{f}_{t}(Z)Z] = 0$ for $Z \sim \gauss{0}{1}$. Hence \assumpref{divergence-free} and additional assumption (c) are fulfilled. 
\end{enumerate}
This means that, for any test function $\widehat{h}:\R^{T+1} \rightarrow \R$ which satisfies:
\begin{align*}
    |\widehat{h}(\vx) - \widehat{h}(\vx^\prime)| & \leq L \cdot  \|\vx - \vx^\prime \| \cdot (1+ \|\vx\| + \|\vx^\prime\|),
\end{align*}
we have,
\begin{align*}
     \frac{1}{\dim} \sum_{i=1}^\dim \widehat{h}(\iter{\hat{z}}{0}_i, \iter{\hat{z}}{1}_i, \dotsc, \iter{\hat{z}}{T}_i) \explain{P}{\rightarrow} \E \widehat{h}(\hat{Z}_0, \hat{Z}_1, \dotsc, \hat{Z}_T),
\end{align*}
where $(\hat{Z}_0, \hat{Z}_1, \dotsc, \hat{Z}_T) \sim \gauss{\vzero}{\widehat{\mSigma}_T}$. The matrix $\widehat{\mSigma}_T$ is defined using the state evolution recursions corresponding to the rescaled iterations:
\begin{align*}
    \widehat{\sigma}_{t}^2 & = 1 \; \forall \; t \; \in \; \{0, 1, \dotsc, T\} \\
    \widehat{\rho}_{{s,t+1}} & = \E[ \widehat{\nonlin}_{s}(\hat{Z}_{s-1})\widehat{\nonlin}_{t+1}(\hat{Z}_t) ].
\end{align*}
initialized at $\sigma_0^2: = 1$ and $\rho_{0,t} := 0$ for any $t \in [T]$ and the matrix $\widehat{\mSigma}_T \in \R^{(T+1) \times (T+1)}$ is given by:
\begin{align*}
    \widehat{\mSigma}_T & \explain{def}{=} \begin{bmatrix} 1 & \widehat{\rho}_{0,1} & \widehat{\rho}_{0,2} & \hdots  & \widehat{\rho}_{0,T}  \\ \widehat{\rho}_{0,1} & 1 & \widehat{\rho}_{1,2} & \hdots &\widehat{\rho}_{1,T} \\ \vdots & \vdots & \vdots & \vdots & \vdots \\ \widehat{\rho}_{0,T} & \widehat{\rho}_{1,T} & \widehat{\rho}_{2,T} & \hdots & 1 \end{bmatrix}.
\end{align*}
Applying the above result to the test function $\widehat{h}(x_0, x_1, \dotsc, x_T) = h(\sigma_0 x_0, \sigma_1 x_1, \dotsc, \sigma_T x_T)$ we obtain:
\begin{align*}
      \frac{1}{\dim} \sum_{i=1}^\dim h(\iter{z}{0}_i, \iter{z}{1}_i, \dotsc, \iter{z}{T}_i) =   \frac{1}{\dim} \sum_{i=1}^\dim h(\sigma_0\iter{\hat{z}}{0}_i, \sigma_1\iter{\hat{z}}{1}_i, \dotsc, \sigma_T\iter{\hat{z}}{T}_i) \explain{P}{\rightarrow} \E h(\sigma_0 \hat{Z}_0, \sigma_1 \hat{Z}_1, \dotsc, \sigma_T \hat{Z}_T)
\end{align*}
It is straightforward to check by induction that,
\begin{align*}
    (\sigma_0 \hat{Z}_0, \sigma_1 \hat{Z}_1, \dotsc, \sigma_T \hat{Z}_T) \explain{d}{=} (Z_0, Z_1, \dotsc, Z_T),
\end{align*}
which yields the desired result \eqref{eq:rescaling-goal}. 
\end{proof}

\section{Concentration Analysis}
\label{appendix:concentration} 
In this section, we study the concentration properties of the iterates: $\iter{\vz}{t} = \mM \nonlin_t(\iter{\vz}{t-1})$ and provide a proof of \thref{concentration}. We define the following notation, which we use through out this section. Let $\iter{\mZ}{T}$ be a $\dim \times (T+1)$ matrix constucted by arranging the iterates $\iter{\vz}{0:T}$ along the columns. Call the rows of $\iter{\mZ}{T}$ as $\iter{\mZ}{T}_1, \iter{\mZ}{T}_2, \dotsc, \iter{\mZ}{T}_\dim$.
\begin{align*}
    \iter{\mZ}{T} = \begin{bmatrix} \iter{\vz}{0} & \iter{\vz}{1} & \hdots & \iter{\vz}{T} \end{bmatrix} = \begin{bmatrix} {\iter{\mZ}{T}_1}\tran \\ {\iter{\mZ}{T}_2}\tran \\ \vdots \\{\iter{\mZ}{T}_\dim}\tran  \end{bmatrix}
\end{align*}
Note that under the assumptions of \thref{concentration}, there are two sources of randomness in the iterates $\iter{\vz}{t}$: the random diagonal sign matrix $\mS$ used to construct the matrix $\mM = \mS \mPsi \mS$ (recall \defref{matrix-ensemble-relaxed}) and the standard Gaussian vector $\vg \sim \gauss{\vzero}{\mI_\dim}$ used to construct the initialization $\iter{\vz}{0} =  \vg$ (recall \assumpref{initialization}). In this section, we will find it convenient to use the notation $\iter{\vz}{t}(\vg, \mS)$ to make the dependence of $\iter{\vz}{t}$ on these two sources of randomness explicit. Analogously, we will use the notations $\iter{\mZ}{T}(\vg, \mS)$ to denote the matrix formed by arranging the iterates along the columns and refer to the rows of this matrix using $\iter{\mZ}{T}_{j}(\vg, \mS) \; j \in [\dim]$.

The proof of \thref{concentration} follows from an application of the Efron-Stein Inequality (in order to control the variance contribution from the random signs $\mS$) and the Gaussian Poincare Inequality ((in order to control the variance contribution from the Gaussian initialization $\iter{\vz}{0} = \vg$). In order to apply the Efron-Stein Inequality, one needs to control the perturbation introduced in the iterates if a single sign in $\mS$ is flipped. Hence, we introduce the following setup. Let $\mS = \diag(s_{1:\dim})$ and $\mS^\prime = \diag({s}_{1:\dim}^\prime)$ be two diagonal sign matrices. For each $i \in [\dim]$, define the perturbation vectors:
\begin{align*}
    \iter{\vDelta}{i ; t}(\vg, \mS, \mS^\prime) \explain{def}{=}  \iter{\vz}{t}(\vg, \iter{\mS}{i}) - \iter{\vz}{t}(\vg, \mS),
\end{align*}
where,
\begin{align*}
    \iter{\mS}{i} \explain{def}{=} \diag(s_1, s_2, \dotsc, s_{i-1}, s_i^\prime, s_{i+1}, \dotsc, s_\dim). 
\end{align*}
We will find the following estimates useful while applying the Efron-Stein Inequality.

\begin{lemma}\label{lem:perturbation-estimate-efron-stein} For any fixed $T \in \N$, there is a finite constant $C$ that depends only on $T$, $\degree$, $\nonlin_{1:T}$ and  $\|\mPsi\|_{\op}$, such that, we have,
\begin{align*}
   \max_{t \leq T} \max_{i \in [\dim]} \|\iter{\vDelta}{i;t}(\vg, \mS, \mS^\prime)\| &\leq C \cdot (1+ \rho_T^{T\degree}(\vg, \mS, \mS^\prime)),  \\
    \max_{t \leq T} \left\| \sum_{i=1}^\dim \iter{\vDelta}{i;t}(\vg, \mS, \mS^\prime) \cdot \iter{\vDelta}{i;t}(\vg, \mS, \mS^\prime)\tran  \right\|_{\op} & \leq C\cdot(1 + \rho_T^{2\degree T(T+1)}(\vg, \mS, \mS^\prime)),
\end{align*}
where,
\begin{align}\label{eq:rho-def}
    \rho_T(\vg, \mS, \mS^\prime) \explain{def}{=} \max_{t \leq T} \max_{i \in [\dim]} \|\iter{\vz}{t}(\vg, \iter{\mS}{i})\|_\infty
\end{align}
\end{lemma}
In order to control the variance contribution from the Gaussian initialization $\iter{\vz}{0} = \vg$ using the Gaussian Poincare Inequality, we will rely on the following gradient estimate. 
\begin{lemma}\label{lem:gradient-estimate} For any polynomial $h : \R^{T+1} \rightarrow \R$ of degree atmost $\degree$, we have,
\begin{align*}
    \left\|\nabla_{\vg} \left( \frac{1}{\dim} \sum_{j=1}^\dim h(\iter{\mZ}{T}_j(\vg,\mS))\right) \right\| & \leq \frac{C \cdot (1+\rho_T(\vg, \mS, \mS)^{(T+1)(\degree -1)})}{\sqrt{\dim}},
\end{align*}
where $C$ is a suitably large constant that depends on $T,\degree$, $\|\mPsi\|_{\op}$ and the coefficients of the polynomials $\nonlin_{1:T}, h$.
\end{lemma}
Finally, since the above perturbation estimates are stated in terms of $\rho_T(\vg, \mS, \mS^\prime)$ defined in \eref{rho-def}, we will need to estimate the moments of $\rho_T(\vg, \mS, \mS^\prime)$, which is the content of the following lemma.

\begin{lemma}\label{lem:rho-moment-estimates} 
For any fixed $\ell, T \in \N$ and $\epsilon \in (0,1)$, we have,
\begin{align*}
   \E \left[ \rho_T^{2\ell}(\vg, \mS, \mS)\right] \leq \E \left[ \rho_T^{2\ell}(\vg, \mS, \mS^\prime)\right]  & \lesssim \dim^\epsilon,
\end{align*}
where $\vg \sim \gauss{\vzero}{\mI_\dim}$ and $\mS = \diag(s_{1:\dim}), \mS^\prime = \diag(s^\prime_{1:\dim})$ with $s_{1:\dim}, s_{1:\dim}^\prime \explain{i.i.d.}{\sim} \unif{\{\pm 1\}}$.
\end{lemma}
We postpone the proof of \lemref{perturbation-estimate-efron-stein}, \lemref{gradient-estimate} and \lemref{rho-moment-estimates} to the end of this section and present the proof of \thref{concentration}.
\begin{proof}[Proof of \thref{concentration}]
Let $\iter{\vz}{0} = \vg$ with $\vg \sim \gauss{\vzero}{\mI_\dim}$ (cf. \assumpref{initialization} and \assumpref{normalization}) and $\mM = \mS \mPsi \mS$ with $\mS = \diag(s_{1:\dim})$ where $s_{1:\dim} \explain{i.i.d.}{\sim} \unif{\{\pm 1\}}$.  As discussed previously, we make the dependence of $\iter{\vz}{t}$ on the random variables $\vg, \mS$ explicit by using the notation $\iter{\vz}{t}(\vg, \mS)$. For convenience, we define,
\begin{align*}
    H_T(\vg, \mS) \explain{def}{=}  \frac{1}{\dim} \sum_{j=1}^\dim h(\iter{\mZ}{T}_j(\vg, \mS)).
\end{align*}
By the law of total variance,
\begin{align*}
    \Var[H_T(\vg, \mS)] & = \E[\Var[ H_T(\vg, \mS)  \big|  \vg]] + \Var[ \E[H_T(\vg, \mS) \big| \vg]].
\end{align*}
We control each of the terms on the RHS. In order to control the expected conditional variance, we use the Efron-Stein Inequality. Let $\mS^\prime = \diag(s_{1:\dim}^\prime)$ be an independent copy of $\mS$ and $$\iter{\mS}{i} = \diag(s_1, \dotsc, s_{i-1}, s_i^\prime, s_{i+1}, \dotsc, s_\dim).$$ By the Efron-Stein Inequality,
\begin{align*}
    2 \E[\Var[ H_T(\vg, \mS)  \big|  \vg]] & \leq \sum_{i=1}^\dim \E[(H_T(\vg, \iter{\mS}{i}) - H_T(\vg, \mS))^2]
\end{align*}
For any $t \in \{0, 1, 2, \dotsc, T\}$, define $\iter{\vDelta}{i;t} = \iter{\vz}{t}(\vg, \iter{\mS}{i}) - \iter{\vz}{t}(\vg, \mS)$. By Taylor's theorem, 
\begin{align} \label{eq: taylor-expansion}
    H_T(\vg, \iter{\mS}{i}) - H_T(\vg, \mS) & \explain{(a)}{=} \frac{1}{\dim} \sum_{j=1}^\dim \sum_{t=1}^T \partial_t h(\iter{\mZ}{T}_j(\vg, \mS)) \cdot \iter{\Delta}{i;t}_j + \frac{1}{\dim}\sum_{j=1}^\dim \iter{\epsilon}{i}_j \nonumber \\
    & \explain{(b)}{=} \frac{1}{\dim} \sum_{t=1}^T \ip{\partial_t h(\iter{\mZ}{T}(\vg, \mS))}{\iter{\vDelta}{i;t}} + \frac{1}{\dim}\sum_{j=1}^\dim \iter{\epsilon}{i}_j.
\end{align}
In the above display, in the step marked (a), $\partial_t h(x_0, x_1, \dotsc, x_T)$ denotes the partial derivative of $h(x_0, \dotsc, x_T)$ with respect to $x_t$. In the step marked (b), $\partial_t h(\iter{\mZ}{T}(\vg,\mS)) \in \R^{\dim}$ denotes the vector with coordinates $(\partial_t h(\iter{\mZ}{T}(\vg,\mS)))_j \explain{def}{=} \partial_t h(\iter{\mZ}{T}_j(\vg,\mS))$. Since $h$ is a polynomial of degree at most $\degree$, the second order error term in the Taylor's expansion can be controlled by:
\begin{align*} 
    |\iter{\epsilon}{i}_j| & \leq C \cdot \left(1 + \|\iter{\mZ}{T}_j(\vg,\mS)\|^{D} + \|\iter{\mZ}{T}_j(\vg,\iter{\mS}{i})\|^{D} \right)\cdot \left(\sum_{t=1}^T |\iter{\Delta}{i;t}_j|^2 \right), \\
\end{align*}
where the constant $C$ depends on $\degree, T$ and the coefficients of the polynomial $h$. Recalling the definition of $\rho_T(\vg, \mS, \mS^\prime)$:
\begin{align*}
    \rho_T(\vg, \mS, \mS^\prime) \explain{def}{=} \max_{t \leq T} \max_{i \in [\dim]} \|\iter{\vz}{t}(\vg, \iter{\mS}{i})\|_\infty,
\end{align*}
the above bound can be expressed as:
\begin{align}\label{eq:taylor-thm-remainder-bound}
     |\iter{\epsilon}{i}_j| & \leq C \cdot(1+ \rho_T(\vg, \mS, \mS^\prime)^D) \cdot \left(\sum_{t=1}^T |\iter{\Delta}{i;t}_j|^2 \right).
\end{align}
Hence,
\begin{align*}
     \E[\Var[ H_t(\vg, \mS)  \big|  \vg]]  &\leq \underbrace{\E\left[ \sum_{i=1}^\dim \left(\frac{1}{\dim} \sum_{t=1}^T \ip{\partial_t h(\iter{\mZ}{T}(\vg, \mS))}{\iter{\vDelta}{i;t}} \right)^2 \right]}_{(\star)} + \underbrace{\E\left[\sum_{i=1}^\dim \left( \frac{1}{\dim}\sum_{j=1}^\dim \iter{\epsilon}{i}_j\right)^2 \right]}_{(\dagger)}.
\end{align*}
In order to upper bound the term $(\star)$, we observe that,
\begin{align*}
    (\star) & \explain{def}{=} \E\left[ \sum_{i=1}^\dim \left(\frac{1}{\dim} \sum_{t=1}^T \ip{\partial_t h(\iter{\mZ}{T}(\vg, \mS))}{\iter{\vDelta}{i;t}} \right)^2 \right] \\
    & \explain{(a)}{\leq} \frac{T}{\dim^2} \cdot \sum_{t=1}^T \E\left[ \sum_{i=1}^\dim \ip{\partial_t h(\iter{\mZ}{T}(\vg, \mS))}{\iter{\vDelta}{i;t}}^2  \right] \\
    & \leq \frac{T}{\dim^2} \cdot \sum_{t=1}^T \E\left[  \left\| \sum_{i=1}^\dim \iter{\vDelta}{i;t}{\iter{\vDelta}{i;t}}\tran \right\|_{\op} \cdot \|\partial_t h(\iter{\mZ}{T}(\vg, \mS)) \|^2  \right] \\
    &\explain{(b)}{\leq} \frac{T}{\dim^2} \cdot \sum_{t=1}^T \E\left[  \left\| \sum_{i=1}^\dim \iter{\vDelta}{i;t}{\iter{\vDelta}{i;t}}\tran \right\|_{\op} \cdot C \cdot\dim \cdot  (1+ \rho_T^\degree(\vg, \mS, \mS^\prime))    \right] \\
    & \explain{(c)}{\leq} \frac{C}{\dim} \sum_{t=1}^T \E\left[  (1 + \rho_T^{2\degree T(T+1)}(\vg, \mS, \mS^\prime)) \cdot  (1+ \rho_T^\degree(\vg, \mS, \mS^\prime))    \right] \\
    & \explain{(d)}{\lesssim} \dim^{-1+\epsilon}.
\end{align*}
In the above display, step (a) is obtained by the Cauchy-Schwarz Inequality, step (b) uses the fact that each entry of  $\partial_t h(\iter{\mZ}{T}(\vg, \mS)$ is a polynomial of degree at most $D-1$, step (c) relied on the estimates on the operator norm from \lemref{perturbation-estimate-efron-stein} and finally step (d) relied on the moment estimates from \lemref{rho-moment-estimates}. 
In order to upper bound the term $(\dagger)$, we observe that,
\begin{align*}
    (\dagger) &\explain{def}{=}  \E\left[\sum_{i=1}^\dim \left( \frac{1}{\dim}\sum_{j=1}^\dim \iter{\epsilon}{i}_j\right)^2 \right] \\
    & \leq \E\left[\sum_{i=1}^\dim \left( \frac{1}{\dim}\sum_{j=1}^\dim |\iter{\epsilon}{i}_j|\right)^2 \right] \\
    & \explain{(a)}{\leq} \frac{C^2}{\dim^2} \sum_{i=1}^\dim \E\left[ (1+ \rho_T(\vg, \mS, \mS^\prime)^D)^2 \cdot \left(\sum_{t=1}^T \|\iter{\vDelta}{i;t}\|^2 \right)^2 \right] \\
    & \explain{(b)}{\leq} \frac{C^2 T}{\dim^2} \sum_{t=1}^T\sum_{i=1}^\dim \E\left[ (1+ \rho_T(\vg, \mS, \mS^\prime)^D)^2 \cdot  \|\iter{\vDelta}{i;t}\|^4 \right] \\
    & \explain{(c)}{\leq} \frac{C}{\dim} \sum_{t=1}^T \E\left[ (1+ \rho_T(\vg, \mS, \mS^\prime)^D)^2 \cdot (1 + \rho_T^{2\degree T(T+1)}(\vg, \mS, \mS^\prime))^4 \right]
    \;  \explain{(d)}{\lesssim} \;  \dim^{-1+\epsilon}.
\end{align*}
In the above display, step (a) uses the estimate from \eqref{eq:taylor-thm-remainder-bound}, step (b) is obtained using Cauchy-Schwarz Inequality, step (c) uses the norm estimate from \lemref{perturbation-estimate-efron-stein} and step (d) relies on the moment estimates from \lemref{rho-moment-estimates}.
Hence, the estimates on $(\star)$ and $\dagger$ yield:
\begin{align*}
    \E[\Var[ H_t(\vg, \mS)  \big|  \vg]]  & \lesssim  \dim^{-1+\epsilon}.
\end{align*}
This concludes our analysis of the expected conditional variance and we have shown that,
\begin{align*}
     \Var[H_t(\vg, \mS)] & \lesssim \dim^{-1+\epsilon} + \Var[ \E[H_t(\vg, \mS) \big| \vg]].
\end{align*}
In order to control the variance of the conditional expectation, we will rely on the Gaussian Poincare Inequality which states that,
\begin{align*}
    \Var[ \E[H_t(\vg, \mS) \big| \vg]] & \leq \E[ \|\nabla_{\vg} \E[H_t(\vg, \mS) \big| \vg]\|^2].
\end{align*}
Since $\vg, \mS$ are independent, we have $\nabla_{\vg} \E[H_t(\vg, \mS) \big| \vg] =  \E[ \nabla_{\vg} H_t(\vg, \mS) \big| \vg]$. By Jensen's Inequality,
\begin{align*}
     \Var[ \E[H_t(\vg, \mS) \big| \vg]]  &\leq  \E[ \| \E[\nabla_{\vg} H_t(\vg, \mS) \big| \vg]\|^2] \leq \E[  \E[\|\nabla_{\vg} H_t(\vg, \mS)\|^2 \big| \vg]] = \E[\|\nabla_{\vg} H_t(\vg, \mS)\|^2] 
\end{align*}
Using the gradient estimates from \lemref{gradient-estimate} and the moment bounds from \lemref{rho-moment-estimates}, we obtain,
\begin{align*}
    \Var[ \E[H_t(\vg, \mS) \big| \vg]] & \leq  \frac{C^2 \cdot \E[(1+\rho_T(\vg, \mS, \mS)^{(T+1)(\degree -1)})^2]}{\dim} \lesssim \dim^{-1+\epsilon}.
\end{align*}
This concludes the proof of the claim $\Var[H_t(\vg, \mS)] \lesssim \dim^{-1+\epsilon}$. 
\end{proof}

\subsection{Proof of \lemref{perturbation-estimate-efron-stein}}
\begin{proof}[Proof of \lemref{perturbation-estimate-efron-stein}] For ease of notation, we will not make the dependence of $\iter{\vDelta}{i;t}, \rho_T$ on $\vg, \mS, \mS^\prime$ explicit. We begin by considering the decomposition:
\begin{align*}
    &\iter{\vDelta}{i;t} \explain{def}{=} \iter{\vz}{t}(\vg, \iter{\mS}{i}) - \iter{\vz}{t}(\vg, {\mS}) \\
    & = \iter{\mS}{i} \mPsi \iter{\mS}{i} \nonlin_t(\iter{\vz}{t-1}(\vg, \iter{\mS}{i})) - \mS \mPsi \mS \nonlin_t(\iter{\vz}{t-1}(\vg, \mS)) \\
    & = \iter{\mS}{i} \mPsi \iter{\mS}{i} \nonlin_t(\iter{\vz}{t-1}(\vg, \iter{\mS}{i})) - \mS \mPsi \mS \nonlin_t(\iter{\vz}{t-1}(\vg, \iter{\mS}{i})) + \mS \mPsi \mS \nonlin_t(\iter{\vz}{t-1}(\vg, \iter{\mS}{i}))  - \mS \mPsi \mS \nonlin_t(\iter{\vz}{t-1}(\vg, \mS)).
\end{align*}
We define,
\begin{align}\label{eq:alpha-beta-def}
    \iter{\valpha}{i;t} &\explain{def}{=} \iter{\mS}{i} \mPsi \iter{\mS}{i} \nonlin_t(\iter{\vz}{t-1}(\vg, \iter{\mS}{i})) -  {\mS} \mPsi \iter{\mS}{i} \nonlin_t(\iter{\vz}{t-1}(\vg, \iter{\mS}{i})) = (1- s_i^\prime s_i) \cdot \iter{z}{t}_i(\vg, \iter{\mS}{i}) \cdot \ve_i, \\
    \iter{\vbeta}{i;t} &\explain{def}{=} {\mS} \mPsi \iter{\mS}{i} \nonlin_t(\iter{\vz}{t-1}(\vg, \iter{\mS}{i})) - {\mS} \mPsi {\mS} \nonlin_t(\iter{\vz}{t-1}(\vg, \iter{\mS}{i})) = (s_i^\prime- s_i) \cdot \nonlin_t(\iter{z}{t-1}_i(\vg, \iter{\mS}{i}))) \cdot \mS \mPsi \ve_i. 
\end{align}
where, $\ve_i$ denotes the i-th coordinate vector. Hence,
\begin{align}\label{eq:delta-decomp}
    \iter{\vDelta}{i;t} & = \iter{\valpha}{i;t} + \iter{\vbeta}{i;t} + \mS \mPsi \mS \cdot( \nonlin_t(\iter{\vz}{t-1}(\vg, \iter{\mS}{i}))  -  \nonlin_t(\iter{\vz}{t-1}(\vg, \mS))).
\end{align}
We now consider each of the claims of the lemma one-by-one:
\begin{enumerate}
    \item Applying the triangle inequality on \eref{delta-decomp} yields:
    \begin{align*}
         \|\iter{\vDelta}{i;t}\| & \leq \|\iter{\valpha}{i;t}\| + \|\iter{\vbeta}{i;t}\| + \| \mPsi\|_{\op}\| \nonlin_t(\iter{\vz}{t-1}(\vg, \iter{\mS}{i}))  -  \nonlin_t(\iter{\vz}{t-1}(\vg, \mS))\|.
    \end{align*}
    We estimate each of the norms in the above display. Observe that:
    \begin{align*}
        \|\iter{\valpha}{i;t}\| & \leq 2 \cdot |\iter{z}{t}_i(\vg, \iter{\mS}{i})| \leq 2 \rho_T, \\
        \| \iter{\vbeta}{i;t}\| &\leq 2 \cdot |\nonlin_t(\iter{z}{t-1}_i(\vg, \iter{\mS}{i}))| \cdot \| \mPsi \|_{\op}
    \end{align*}
    \begin{subequations}\label{eq:poly-PL}
    Since $\nonlin_t$ is a polynomial of degree at most $\degree$, we have, for any $x,y \in \R$:
    \begin{align}
        |\nonlin_t(x)| & \leq C \cdot (1+|x|^\degree), \\
        |\nonlin_t(x) - \nonlin_t(y)| & \leq C \cdot |x-y| \cdot (1+ |x|^{\degree-1} + |y|^{\degree-1}),
    \end{align}
    \end{subequations}
    where $C$ is a suitably constant that is determined by $D$ and the particular polynomial $\nonlin_t$. Hence,
    \begin{align*}
         \| \iter{\vbeta}{i;t}\| & \leq 2C \cdot \|\mPsi\|_{\op} \cdot (1 + |\iter{z}{t-1}_i(\vg, \iter{\mS}{i})|^\degree) \\&\leq 2C \cdot \|\mPsi\|_{\op} \cdot (1 + \rho_T^\degree), \\
         \| \nonlin_t(\iter{\vz}{t-1}(\vg, \iter{\mS}{i}))  -  \nonlin_t(\iter{\vz}{t-1}(\vg, \mS))\| &\leq C\cdot (1 + \|\iter{\vz}{t-1}(\vg, \iter{\mS}{i})\|_\infty^{\degree -1} +  \|\iter{\vz}{t-1}(\vg, \mS)\|_{\infty}^{\degree -1}) \cdot  \|\iter{\vDelta}{i;t-1}\| \\
         & \leq 2 \cdot C \cdot (1+\rho_T^{\degree-1}) \cdot \|\iter{\vDelta}{i;t-1}\|.
    \end{align*}
    Hence,
    \begin{align*}
         \|\iter{\vDelta}{i;t}\| & \leq  C \cdot (1+\rho_T^\degree) \cdot ( 1+ \|\iter{\vDelta}{i;t}\|),
    \end{align*}
    where $C$ is a suitably large constant determined by $D$, the particular polynomial $\nonlin_t$ and $\|\mPsi\|_{\op}$.
    Unrolling the above recursion, we obtain,
    \begin{align*}
        \max_{t \leq T}  \|\iter{\vDelta}{i;t}(\vg, \mS, \mS^\prime)\| &\leq C \cdot (1+ \rho_T^{T\degree}),
    \end{align*}
    where $C$ is another suitably large constant that depends on $T, \degree, \nonlin_{1:T},  \|\mPsi\|_{\op}$. This proves the first estimate claimed in the lemma. 
\item In order to prove the second claim of the lemma, we further develop the decomposition in \eref{delta-decomp} using the Taylor Series expansion of $\nonlin_t$:
\begin{subequations}\label{eq:delta-decomp-improved}
\begin{align}
    \iter{\vDelta}{i;t} & = \iter{\valpha}{i;t} + \iter{\vbeta}{i;t} + \mS \mPsi \mS \cdot \mD_t \cdot  \iter{\vDelta}{i;t-1} + \mS \mPsi \mS \cdot \iter{\vepsilon}{i;t}.
\end{align}
In the above display, $\mD_t$ is a diagonal matrix given by:
\begin{align}
    \mD_t \explain{def}{=} \diag(\nonlin_t^\prime(\iter{z}{t-1}_{1}(\vg, \mS)), \nonlin_t^\prime(\iter{z}{t-1}_{2}(\vg, \mS)), \dotsc, \nonlin_t^\prime(\iter{z}{t-1}_{\dim}(\vg, \mS))),
\end{align}
and $\vepsilon_t$ is the Taylor's remainder. Since
\begin{align}
    \iter{\vepsilon}{i;t} \explain{def}{=} \nonlin_t(\iter{\vz}{t-1}(\vg, \iter{\mS}{i}))  -  \nonlin_t(\iter{\vz}{t-1}(\vg, \mS)) - \mD_t \cdot  \iter{\vDelta}{i;t-1}.
\end{align}
Since $\nonlin_t$ was a polynomial of degree at most $\degree$, this remainder can be bounded entry-wise by:
\begin{align}
     |\iter{\epsilon}{i;t}_j| & \leq C \cdot (1+|\iter{z}{t-1}_j(\vg, \iter{\mS}{i})|^{\degree-2} + |\iter{z}{t-1}_j(\vg, \mS)|^{\degree-2})  \cdot |\iter{\Delta}{i;t-1}_j|^2, \\
     & \leq C \cdot (1 + \rho_T^{\degree -2}) \cdot |\iter{\Delta}{i;t-1}_j|^2,
\end{align}
where $C$ is a constant that depends on $\nonlin_t, D$. 
\end{subequations}
With this recursive decomposition of $\iter{\vDelta}{i;t}$, we can prove the second estimate in the lemma. For convenience, let us define:
\begin{align*}
    L_t \explain{def}{=}  \left\| \sum_{i=1}^\dim \iter{\vDelta}{i;t} \cdot {\iter{\vDelta}{i;t}}\tran  \right\|_{\op}. 
\end{align*}
For a collection of vectors $\iter{\vv}{1:\dim}$ we use the notation $[\iter{\vv}{1:\dim}]$ to denote the matrix whose columns are $\iter{\vv}{1:\dim}$. We observe that:
\begin{align}
    L_t & = \|[\iter{\vDelta}{1:\dim; t}]\|_{\op}^2 \nonumber\\
    & \explain{}{\leq} 4 \cdot  \left( \|[\iter{\valpha}{1:\dim; t}]\|_{\op}^2 +  \|[\iter{\vbeta}{1:\dim; t}]\|_{\op}^2 +  \|\mS \mPsi \mS \mD_t [\iter{\vDelta}{1:\dim; t-1}]\|_{\op}^2 + \|\mS \mPsi \mS  [\iter{\vepsilon}{1:\dim; t}]\|_{\op}^2 \right). \label{eq:L-recursion-initial}
\end{align}
In the above display, we used the recursive decomposition in \eref{delta-decomp-improved} and the triangle inequality to obtain the final inequality. Next, we estimate each of terms appearing in the above inequality. Recalling the definitions of $\iter{\valpha}{i;t}, \iter{\vbeta}{i;t}$, we obtain:
\begin{align}
     \|[\iter{\valpha}{1:\dim; t}]\|_{\op}^2 & \leq 4 \max_{i \in [\dim]} |\iter{z}{t}_i(\vg, \iter{\mS}{i})|^2 \leq 4 \rho_T^2, \label{eq:alpha-norm-ub} \\
      \|[\iter{\vbeta}{1:\dim; t}]\|_{\op}^2 & \leq 4 \cdot \|\mPsi\|_{\op}^2 \cdot \max_{i \in \dim} |\nonlin_t(\iter{\vz}{t-1}(\vg, \iter{\mS}{i}))|^2  \nonumber \\ 
      & \explain{\eqref{eq:poly-PL}}{\leq} 4 \cdot C^2 \cdot (1+\rho_T^\degree)^2.  \label{eq:beta-norm-ub}
\end{align}
Notice that:
\begin{align}
    \|\mS \mPsi \mS \mD_t [\iter{\vDelta}{1:\dim; t-1}]\|_{\op}^2 & = \left\| \mS \mPsi \mS \mD_t \cdot \left(\sum_{i=1}^\dim \iter{\vDelta}{i;t-1} \cdot {\iter{\vDelta}{i;t-1}}\tran  \right) \cdot \mD_t \mS  \mPsi\tran \mS \right\|_{\op} \nonumber\\
    & \leq \|\mD_t\|_{\op}^2 \cdot \|\mPsi\|_{\op}^2 \cdot L_{t-1}\nonumber \\
    & \leq \|\mPsi\|_{\op}^2 \cdot L_{t-1} \cdot \max_{i \in [\dim]} |\nonlin_t^\prime(\iter{z}{t-1}_{i}(\vg, \mS))|^2 \nonumber \\
    & \explain{\eqref{eq:poly-PL}}{\leq} C^2 \cdot \|\mPsi\|_{\op}^2 \cdot (1 + \rho_T^{\degree-1})^2 \cdot L_{t-1}.\label{eq:delta-norm-ub}
\end{align}
Lastly, we control:
\begin{align*}
    \|\mS \mPsi \mS  [\iter{\vepsilon}{1:\dim; t}]\|_{\op}^2 & = \left\| \mS \mPsi \mS \cdot \left( \sum_{i=1}^\dim \iter{\vepsilon}{i;t} {\iter{\vepsilon}{i;t}}\tran\right) \mS \mPsi\tran \mS \right\|_{\op} \leq \|\mPsi\|_{\op}^2 \cdot \left\| \sum_{i=1}^\dim \iter{\vepsilon}{i;t} {\iter{\vepsilon}{i;t}}\tran \right\|_{\op}.
\end{align*}
By Gershgorin's Circle theorem, the spectral norm of a symmetric matrix can be bounded by the largest row $\ell_1$ norm. Hence,
\begin{align*}
    \|\mS \mPsi \mS  [\iter{\vepsilon}{1:\dim; t}]\|_{\op}^2 &  \leq \|\mPsi\|_{\op}^2 \cdot  \max_{j\in [\dim]} \left( \sum_{k=1}^{\dim}\left| \sum_{i=1}^\dim \iter{\epsilon}{i;t}_j \iter{\epsilon}{i;t}_k \right| \right) \leq \|\mPsi\|_{\op}^2 \cdot  \max_{j\in [\dim]} \left( \sum_{k=1}^{\dim} \sum_{i=1}^\dim |\iter{\epsilon}{i;t}_j| \cdot  |\iter{\epsilon}{i;t}_k|  \right).
\end{align*}
Recalling the estimates on $|\iter{\epsilon}{i;t}_j|$ from \eref{delta-decomp-improved}, we obtain,
\begin{align}
    \|\mS \mPsi \mS  [\iter{\vepsilon}{1:\dim; t}]\|_{\op}^2 & \leq C^2 \cdot (1 + \rho_T^{\degree -2})^2 \cdot \|\mPsi\|_{\op}^2 \cdot \max_{j\in [\dim]} \left( \sum_{k=1}^{\dim} \sum_{i=1}^\dim |\iter{\Delta}{i;t-1}_j|^2 \cdot  |\iter{\Delta}{i;t-1}_k|^2  \right) \nonumber\\
    & = C^2 \cdot (1 + \rho_T^{\degree -2})^2 \cdot \|\mPsi\|_{\op}^2 \cdot \max_{j\in [\dim]} \left(  \sum_{i=1}^\dim |\iter{\Delta}{i;t-1}_j|^2 \cdot  \|\iter{\vDelta}{i;t-1}\|^2  \right) \nonumber\\
    & \explain{(a)}{\leq}  C^4 \cdot (1 + \rho_T^{\degree -2})^2 \cdot \|\mPsi\|_{\op}^2 \cdot (1+\rho_T^{T\degree})^2 \cdot \max_{j\in [\dim]} \left(  \sum_{i=1}^\dim |\iter{\Delta}{i;t-1}_j|^2   \right) \nonumber\\
    & \leq C^4 \cdot (1 + \rho_T^{\degree -2})^2 \cdot \|\mPsi\|_{\op}^2 \cdot (1+\rho_T^{T\degree})^2  \cdot  \left\|  \sum_{i=1}^\dim \iter{\vDelta}{i;t-1}{\iter{\vDelta}{i;t-1}}\tran   \right\|_{\op} \nonumber\\
    & =C^4 \cdot (1 + \rho_T^{\degree -2})^2 \cdot \|\mPsi\|_{\op}^2 \cdot (1+\rho_T^{T\degree})^2  \cdot  L_{t-1}. \label{eq:epsilon-norm-ub}
\end{align}
In the above display, the inequality marked (a) follows from the first estimate claimed in this lemma. Plugging in the estimates in \eref{alpha-norm-ub}, \eref{beta-norm-ub}, \eref{delta-norm-ub} and \eref{epsilon-norm-ub} into \eref{L-recursion-initial} gives the following recursive estimate for $L_t$:
\begin{align*}
    L_t & \leq C \cdot (1+ \rho_T^{2\degree(T+1)}) \cdot (1+L_{t-1})\; \forall \; t \leq T,
\end{align*}
where $C$ is a suitably large constant that depends on $T$, $D$, $\nonlin_{1:T}$  and  $\|\mPsi\|_{\op}$. Unrolling this recursive upper bound gives:
\begin{align*}
    \max_{t \leq T} L_t & \leq C(1 + \rho_T^{2\degree T(T+1)}),
\end{align*}
for another suitably large constant $C$ that depends on $T$, $D$, $\nonlin_{1:T}$  and  $\|\mPsi\|_{\op}$.
\end{enumerate}
\end{proof}

\subsection{Proof of \lemref{gradient-estimate}}
\begin{proof}[Proof of \lemref{gradient-estimate}] For convenience, we define,
\begin{align*}
    H_T(\vg, \mS) \explain{def}{=}  \frac{1}{\dim} \sum_{j=1}^\dim h(\iter{\mZ}{T}_j(\vg, \mS)).
\end{align*}
We will estimate $\nabla_{\vg} H_T(\vg, \mS)$ by estimating $|H_T(\vg, \mS) -H_T(\vg^\prime, \mS)|$ in terms of $\|\vg - \vg^\prime\|$ for any two vectors $\vg, \vg^\prime \in \R^\dim$:
\begin{align*}
   |H_T(\vg, \mS)  - H_T(\vg^\prime, \mS)| & = \left| \frac{1}{\dim} \sum_{j=1}^\dim (h(\iter{\mZ}{T}_j(\vg, \mS) - h(\iter{\mZ}{T}_j(\vg^\prime, \mS))\right|. 
\end{align*}
Since $h: \R^{T+1} \rightarrow \R$ is a polynomial of degree at most $D$, we have,
\begin{align*}
    |h(\vx) - h(\vx^\prime)| & \leq C \cdot \|\vx - \vx\| \cdot (1 + \|\vx\|^{\degree - 1} + \|\vx\|^{\degree -1}),
\end{align*}
where $C$ is a suitably large constant that depends on $\degree, T$ and the coefficients of the polynomial $h$.  Hence,
\begin{align*}
    |h(\iter{\mZ}{T}_j(\vg, \mS) - h(\iter{\mZ}{T}_j(\vg^\prime, \mS)| & \leq C \cdot \|\iter{\mZ}{T}_j(\vg, \mS) - \iter{\mZ}{T}_j(\vg^\prime, \mS) \| \cdot (1 + \|\iter{\mZ}{T}_j(\vg, \mS)\|^{\degree -1} + \|\iter{\mZ}{T}_j(\vg^\prime, \mS)\|^{\degree -1}).
\end{align*}
Recalling that,
\begin{align*}
    \rho_T(\vg, \mS, \mS) & \explain{def}{=}  \max_{t \leq T} \|\iter{\vz}{t}(\vg, \mS)\|_\infty,
\end{align*}
the above bound can be written as:
\begin{align*}
     |h(\iter{\mZ}{T}_j(\vg, \mS) - h(\iter{\mZ}{T}_j(\vg^\prime, \mS)| & \leq C \cdot ( 1 + \rho_T(\vg, \mS, \mS)^{\degree -1} + \rho_T(\vg^\prime, \mS, \mS)^{\degree -1}) \cdot \sum_{t=0}^T |\iter{z}{t}_i(\vg, \mS) - \iter{z}{t}_i(\vg^\prime, \mS)|. 
\end{align*}
Hence we obtain,
\begin{align*}
    |H_T(\vg, \mS)  - H_T(\vg^\prime, \mS)| & \leq \frac{ C \cdot ( 1 + \rho_T(\vg, \mS, \mS)^{\degree -1} + \rho_T(\vg^\prime, \mS, \mS)^{\degree -1})}{\dim} \cdot \sum_{t=0}^T \|\iter{\vz}{t}(\vg, \mS) - \iter{\vz}{t}(\vg^\prime, \mS)\|_1 \\
    & \leq \frac{ C \cdot ( 1 + \rho_T(\vg, \mS, \mS)^{\degree -1} + \rho_T(\vg^\prime, \mS, \mS)^{\degree -1})}{\sqrt{\dim}} \cdot \sum_{t=0}^T \|\iter{\vz}{t}(\vg, \mS) - \iter{\vz}{t}(\vg^\prime, \mS)\|.
\end{align*}
Next, we develop a recursive upper bound on $\|\iter{\vz}{t}(\vg, \mS) - \iter{\vz}{t}(\vg^\prime, \mS)\|$. We have,
\begin{align*}
    \|\iter{\vz}{t}(\vg, \mS) - \iter{\vz}{t}(\vg^\prime, \mS)\| & = \|\mS \mPsi \mS \cdot (\nonlin_t(\iter{\vz}{t-1}(\vg, \mS)) - \nonlin_t(\iter{\vz}{t-1}(\vg^\prime, \mS)))\| \\
    & \leq \|\mPsi\|_{\op} \cdot \| \nonlin_t(\iter{\vz}{t-1}(\vg, \mS)) - \nonlin_t(\iter{\vz}{t-1}(\vg^\prime, \mS))\| \\
    & \explain{(a)}{\leq}  C \cdot  \|\mPsi\|_{\op} \cdot (1+\rho_T(\vg, \mS, \mS)^{\degree -1} + \rho_T(\vg^\prime,\mS, \mS)^{\degree -1}) \cdot  \|\iter{\vz}{t-1}(\vg, \mS) - \iter{\vz}{t-1}(\vg^\prime, \mS)\|,
\end{align*}
where step (a) follows because $\nonlin_t$ is a polynomial of degree atmost $\degree$.  Unrolling the above recursion, we obtain,
\begin{align*}
        \|\iter{\vz}{t}(\vg, \mS) - \iter{\vz}{t}(\vg^\prime, \mS)\| &  \leq  C \cdot (1+\rho_T(\vg, \mS, \mS)^{T(\degree -1)} + \rho_T(\vg^\prime,\mS, \mS)^{T(\degree -1)}) \cdot \|\vg - \vg^\prime\|   \; \forall \; t \leq T,
\end{align*}
where $C$ is a suitably large constant that depends on $T,\degree$, $\|\mPsi\|_{\op}$ and the coefficients of the polynomials $\nonlin_{1:T}, h$. Hence we have obtained,
\begin{align*}
      |H_T(\vg, \mS)  - H_T(\vg^\prime, \mS)| & \leq \frac{C \cdot (1+\rho_T(\vg, \mS, \mS)^{(T+1)(\degree -1)} + \rho_T(\vg^\prime,\mS, \mS)^{(T+1)(\degree -1)})}{\sqrt{\dim}} \cdot  \|\vg - \vg^\prime\|.
\end{align*}
Taking $\vg^\prime \rightarrow \vg$ in the above display yields:
\begin{align*}
    \|\nabla_{\vg} \; H_T(\vg, \mS)\| & \leq \frac{C \cdot (1+\rho_T(\vg, \mS, \mS)^{(T+1)(\degree -1)})}{\sqrt{\dim}},
\end{align*}
as claimed.
\end{proof}

\subsection{Proof of \lemref{rho-moment-estimates}}
\begin{proof}[Proof of \lemref{rho-moment-estimates}]
Define $q \explain{def}{=} \lceil 2/\epsilon \rceil$. By Jensen's Inequality,
\begin{align*}
    \left( \E \left[ \rho_T^{2\ell}(\vg, \mS, \mS^\prime)\right] \right)^{q} & \leq \E \left[ \rho_T^{2q\ell}(\vg, \mS, \mS^\prime)\right].
\end{align*}
Recall that,
\begin{align*}
    \rho_T^{2q\ell}(\vg, \mS, \mS^\prime) &\explain{def}{=} \max_{t  \leq T} \max_{i \in [\dim]} \max_{j \in [\dim]} |\iter{z}{t}_j(\vg, \iter{\mS}{i})|^{2q\ell} \\
    & \leq \sum_{t=0}^T \sum_{i=1}^\dim \sum_{j=1}^\dim \iter{z}{t}_j(\vg, \iter{\mS}{i})^{2q\ell}.
\end{align*}
Since $\iter{\mS}{i} \explain{d}{=} \mS$, we have $\iter{\vz}{t}(\vg, \iter{\mS}{i}) \explain{d}{=} \iter{\vz}{t}(\vg, \mS)$. Hence,
\begin{align*}
    \E[\rho_T^{2q\ell}(\vg, \mS, \mS^\prime)] & \leq \dim^2 \cdot \left( \sum_{t=0}^T \E \left[ \frac{1}{\dim} \sum_{j=1}^\dim \iter{z}{t}_j(\vg, \mS)^{2q\ell} \right] \right).
\end{align*}
By \thref{SE-normalized}, for any $t \in \{0, 1, \dotsc, T\}$,
\begin{align*}
   \lim_{\dim \rightarrow \infty} \E \left[ \frac{1}{\dim} \sum_{j=1}^\dim \iter{z}{t}_j(\vg, \mS)^{2q\ell} \right] & = \E[Z^{2q\ell}] < \infty, \; Z \sim \gauss{0}{1}. 
\end{align*}
Hence,
\begin{align*}
    \E[\rho_T^{2q\ell}(\vg, \mS, \mS^\prime)] & \lesssim \dim^2 \implies   \E \left[ \rho_T^{2\ell}(\vg, \mS, \mS^\prime)\right] \lesssim \dim^{\frac{2}{q}} \lesssim \dim^\epsilon,
\end{align*}
as claimed. 
\end{proof}

\section{Approximation Analysis}
\label{appendix:approximation}

We prove \propref{approx} in this section. The following lemma provides the construction of the matrix $\hat{\mM}$ used to prove the above proposition. 

\begin{lemma}\label{lem:matrix-approx} Let ${\mM}{} = {\mS}{} {\mPsi}{} {\mS} \in \R^{\dim \times \dim}$ be semi-random (\defref{matrix-ensemble-relaxed}) with constant $\sigpsi =1$. Then, there is a deterministic matrix ${\hat{\mPsi}}$ such that ${\hat{\mM}} \explain{def}{=} {\mS} {\hat{\mPsi}} {\mS}$ is semi-random with constant $\sigpsi=1$ and satisfies \assumpref{normalization}, and $\|{\mM} - {\hat{\mM}}\|_{\op}  \ll 1$.
\end{lemma}

The following two lemmas provide constructions of polynomial functions that approximate the non-linearities $\nonlin_t$ and the test function $h$. 

\begin{lemma}\label{lem:nonlin-approx} Let $\nonlin: \R \rightarrow \R$ be a function such that:
\begin{align*}
    \E[\nonlin^2(Z)] = 1, \; \E[Z \nonlin(Z)] = 0, \; Z \sim \gauss{0}{1}. 
\end{align*}
Then, there is a sequence of approximating functions $\iter{\nonlin}{k}: \R \rightarrow \R$ indexed by $k \in \N$ such that $\iter{\nonlin}{k}$ is a polynomial of degree at most $k$ for each $k \in \N$ and the sequence $\iter{\nonlin}{k}$ satisfies:
\begin{align*}
    \E[{\iter{\nonlin}{k}}(Z)^2] = 1, \; \E[Z \iter{\nonlin}{k}(Z)] = 0 \; \forall \; k \; \in \N, \; \lim_{k \rightarrow \infty} \E[(\iter{\nonlin}{k}(Z) - \nonlin(Z))^2] = 0, \; Z \sim \gauss{0}{1}.
\end{align*}
\end{lemma}
\begin{lemma}\label{lem:test-func-approx} Let $T \in \N$ and let $\mZ = (Z_0, Z_1, \dotsc, Z_T)$ be a random vector with standard Gaussian marginals:
\begin{align*}
    Z_i \sim \gauss{0}{1} \; \forall \; i \; \in \; \{0, 1, \dotsc, T\}.
\end{align*}
Let $h: \R^{T+1} \rightarrow \R$ be a  function such that $\E[h^2(Z_0, Z_1, \dotsc, Z_T)]<\infty$. Then, there is a sequence of approximating functions $\iter{h}{k} : \R^{T+1} \rightarrow \R$ indexed by $k \in \N$ such that $\iter{h}{k}$ is a polynomial with degree at most $k$ for each $k \in \N$ and the sequence $\iter{h}{k}$ satisfies:
\begin{align*}
    \lim_{k \rightarrow \infty} \E[(h(Z_0, Z_1, \dotsc, Z_T) - \iter{h}{k}(Z_0, Z_1, \dotsc, Z_T))^2] & = 0.
\end{align*}
\end{lemma}
Finally, the last ingredient in the proof of \propref{approx} will be the following corollary of \thref{SE-normalized}.

\begin{corollary}\label{cor:SE-normalized}  Fix non-negative integers $T \in \W$, $D \in \W$ and functions $\nonlin_1, \nonlin_2, \dotsc, \nonlin_T : \R \rightarrow \R$. Consider the iteration:
\begin{align*}
    \iter{\vz}{t+1} = \mM \nonlin_{t+1}(\iter{\vz}{t}),
\end{align*}
initialized at $\iter{\vz}{0}$. Suppose that:
\begin{enumerate}
    \item $\iter{\vz}{0}$ satisfies \assumpref{initialization} with $\sigma_0^2 = 1$,
    \item The matrix ensemble $\mM = \mS \mPsi\mS$ is semi-random (\defref{matrix-ensemble-relaxed}) with $\sigpsi = 1$ and satisfies \assumpref{normalization},
    \item  The non-linearities $\nonlin_t$ are polynomials of degree at most $\degree$ and satisfy \assumpref{divergence-free} and $\E[\nonlin_t^2(Z)] = 1$ for each $t \in [T]$ where $Z \sim \gauss{0}{1}$.
\end{enumerate}
Then, there is a random vector $(Z_0, Z_1, \dotsc, Z_T)$ with standard Gaussian marginals $Z_i \sim \gauss{0}{1} \; \forall \; i \; \in \; \{0, 1, 2, \dotsc, T\}$, whose law is determined by $T, f_{1:T}$, such that for any continuous test function $h: \R^{T+1} \rightarrow \R$ which satisfies $$|h(\vx)| \leq L(1+\|\vx\|^D) \; \forall \; \vx \; \in \; \R^{T+1}$$ for some fixed constant $L$, we have,
\begin{align*}
   \lim_{\dim \rightarrow \infty} \E\left[ \frac{1}{\dim} \sum_{i=1}^\dim h(\iter{z}{0}_i, \iter{z}{1}_i, \dotsc, \iter{z}{T}_i)  \right] & = \E h(Z_0, Z_1, \dotsc, Z_T).
\end{align*}
\end{corollary}

With these intermediate results at hand, we first finish the proof of \propref{approx}. The proof of the intermediate results quoted above will appear at the end of this section. 

\begin{proof}[Proof of \propref{approx}] First, we specify our choice of the approximating functions and matrices:
\begin{enumerate}
    \item We choose $\hat{\mM}$ to be the matrix ensemble constructed in \lemref{matrix-approx}. In particular,  $\iter{\hat{\mM}}{\dim}$ satisfies \defref{matrix-ensemble-relaxed} with constant $\sigpsi=1$, \assumpref{normalization} and, \begin{align}\label{eq:matrix-approximation}
   \|\hat{\mM}\|_{\op} \lesssim 1, \;  \|\hat{\mM}- \mM\|_{\op} \ll 1.
    \end{align}
    \item For each $t \in [T]$, we choose $\{\iter{f}{k}_t, \; k \in \N\}$ to be the approximating sequence for $f_t$ constructed in \lemref{nonlin-approx}. For each $k \in \N$, 
    $\iter{\nonlin}{k}$ is a polynomial of degree at most $k$ with $\E[{\iter{\nonlin}{k}}(Z)^2] = 1, \; \E[Z \iter{\nonlin}{k}(Z)] = 0$  for $Z \sim \gauss{0}{1}$ and the sequence $\iter{\nonlin}{k}$ satisfies:
\begin{align} \lim_{k \rightarrow \infty} \label{eq:nonlin-approximation} \E[(\iter{\nonlin}{k}(Z) - \nonlin(Z))^2] = 0, \; Z \sim \gauss{0}{1}.
\end{align}
\end{enumerate}
Next, we construct the approximating sequence for the test function $h$. Consider the sequence of iterations (indexed by $k \in \N$) generated by the choices made above:
\begin{align*}
     \iter{\hat{\vz}}{t+1;k} = \hat{\mM} \iter{\nonlin}{k}_{t+1}(\iter{\hat{\vz}}{t;k}),
\end{align*}
initialized at $ \iter{\hat{\vz}}{0;k} = \iter{\vz}{0}$. For each $k \in \N$, let $(\iter{Z}{0; k}_\dim, \iter{Z}{1; k}_\dim, \dotsc, \iter{Z}{T; k}_\dim)$ denote the sequence of random vectors with law $\iter{\mu}{k}_\dim$:
\begin{align*}
   \iter{\mu}{k}_\dim & = \E \left[ \frac{1}{\dim} \sum_{i=1}^\dim \delta_{\iter{\hat{z}}{0;k}_i, \iter{\hat{z}}{1;k}_i, \dotsc, \iter{\hat{z}}{T;k}_i} \right].
\end{align*}
By \corref{SE-normalized}, for each $k \in \N$, there is a random vector $(\iter{Z}{k}_0, \iter{Z}{k}_1, \dotsc,  \iter{Z}{k}_T)$ with $\gauss{0}{1}$ marginals such that,
\begin{align*}
    (\iter{Z}{0; k}_\dim, \iter{Z}{1; k}_\dim, \dotsc, \iter{Z}{T; k}_\dim) \explain{d}{\rightarrow} (\iter{Z}{k}_0, \iter{Z}{k}_1, \dotsc,  \iter{Z}{k}_T).
\end{align*}
By \lemref{test-func-approx}, for each $k \in \N$, there is a sequence of approximating functions $\{\iter{\hat{h}}{\ell}_k: \ell \in \N\}$ with the property that $\iter{\hat{h}}{\ell}_k$ is polynomial of degree $\ell$ and,
\begin{align*}
    \lim_{\ell \rightarrow \infty } \E[(\iter{\hat{h}}{\ell}_k(\iter{Z}{k}_0, \iter{Z}{k}_1, \dotsc,  \iter{Z}{k}_T) - h(\iter{Z}{k}_0, \iter{Z}{k}_1, \dotsc,  \iter{Z}{k}_T))^2] & = 0.
\end{align*}
\begin{enumerate}
  \setcounter{enumi}{2}
    \item We choose $\iter{h}{k} : = \iter{\hat{h}}{\ell_k}_k$ where $\ell_k \in \N$ is large enough to guarantee:
    \begin{align}\label{eq:testfunc-approx}
        \E[(\iter{h}{k}(\iter{Z}{k}_0, \iter{Z}{k}_1, \dotsc,  \iter{Z}{k}_T) - h(\iter{Z}{k}_0, \iter{Z}{k}_1, \dotsc,  \iter{Z}{k}_T))^2] & \leq 1/k.
    \end{align}
\end{enumerate}
Now that we have constructed the desired approximations, we can provide a proof for the claim of this proposition. We define  another intermediate iteration:
\begin{align*}
    \iter{\tilde{\vz}}{t+1} & = \hat{\mM} \nonlin_{t}(\iter{\tilde{\vz}}{t}), \\
    \iter{\tilde{\vz}}{0} &:= \iter{{\vz}}{0}.
\end{align*}
Using the Triangle Inequality and Jensen's Inequality we obtain the following decomposition:
\begin{align*}
    &\E\left[ \left| \frac{1}{\dim} \sum_{i=1}^\dim h(\iter{z}{0}_i, \iter{z}{1}_i, \dotsc, \iter{z}{T}_i) - \frac{1}{\dim} \sum_{i=1}^\dim \iter{h}{k}(\iter{\hat{z}}{0;k}_i, \dotsc, \iter{\hat{z}}{T;k}_i)  \right| \right] \\& \hspace{6cm}\leq \E\left[ \frac{1}{\dim} \sum_{i=1}^\dim \left|h(\iter{z}{0}_i, \dotsc, \iter{z}{T}_i) -  \iter{h}{k}(\iter{\hat{z}}{0;k}_i, \dotsc, \iter{\hat{z}}{T;k}_i) \right|   \right] \\
    & \hspace{6cm}\leq  \sqrt{(i)} + (ii) + (iii),
\end{align*}
where the terms ($i$-$iii$) are defined as follows:
\begin{align*}
    (i) &\explain{def}{=} \E\left[ \frac{1}{\dim} \sum_{i=1}^\dim \left(h(\iter{\hat{z}}{0;k}_i, \dotsc, \iter{\hat{z}}{T;k}_i) -  \iter{h}{k}(\iter{\hat{z}}{0;k}_i, \dotsc, \iter{\hat{z}}{T;k}_i) \right)^2   \right] \\
    (ii) &\explain{def}{=} \E\left[ \frac{1}{\dim} \sum_{i=1}^\dim \left|h(\iter{\tilde{z}}{0}_i, \dotsc, \iter{\tilde{z}}{T}_i) -  h(\iter{\hat{z}}{0;k}_i, \dotsc, \iter{\hat{z}}{T;k}_i) \right|   \right] \\
    (iii) & \explain{def}{=} \E\left[ \frac{1}{\dim} \sum_{i=1}^\dim \left|h(\iter{{z}}{0}_i, \dotsc, \iter{{z}}{T}_i) - h(\iter{\tilde{z}}{0}_i, \dotsc, \iter{\tilde{z}}{T}_i)\right|   \right]
\end{align*}
Since the term $(iii)$ does not depend on $k$, observe that the claim of the proposition follows if we show:
\begin{align*}
      \lim_{k \rightarrow \infty} \lim_{\dim \rightarrow \infty} \;  (i) & = 0, \;   \lim_{k \rightarrow \infty} \limsup_{\dim \rightarrow \infty} \;  (ii)  = 0, \;   \lim_{\dim \rightarrow \infty} \;  (iii)  = 0.
\end{align*}
The remainder of the proof is devoted to showing each of the above claims. 
\paragraph{Analysis of $(i)$:} 
Observe that:
\begin{align*}
    (i) & = \E\left[ \left(h(\iter{Z}{0; k}_\dim, \iter{Z}{1; k}_\dim, \dotsc, \iter{Z}{T; k}_\dim) -  \iter{h}{k}(\iter{Z}{0; k}_\dim, \iter{Z}{1; k}_\dim, \dotsc, \iter{Z}{T; k}_\dim)\right)^2 \right].
\end{align*}
Using the continuity hypothesis on $h$, it is easy to check that the function $(h(z_0, z_1, \dotsc, z_T) - \iter{h}{k}(z_0, z_1, \dotsc, z_T))^2$ can be bounded by a polynomial and hence by \corref{SE-normalized}, 
\begin{align*}
    &\lim_{\dim \rightarrow \infty} \E\left[ \left(h(\iter{Z}{0; k}_\dim, \iter{Z}{1; k}_\dim, \dotsc, \iter{Z}{T; k}_\dim) -  \iter{h}{k}(\iter{Z}{0; k}_\dim, \iter{Z}{1; k}_\dim, \dotsc, \iter{Z}{T; k}_\dim)\right)^2 \right]  \\&\hspace{5cm}= \E\left[ \left(h(\iter{Z}{k}_0, \iter{Z}{k}_1, \dotsc, \iter{Z}{k}_T) -  \iter{h}{k}(\iter{Z}{k}_0, \iter{Z}{k}_1, \dotsc, \iter{Z}{k}_T)\right)^2 \right] \explain{\eqref{eq:testfunc-approx}}{\leq} 1/k.
\end{align*}
Hence,
\begin{align*}
     \lim_{k \rightarrow \infty} \lim_{\dim \rightarrow \infty} \;  (i) & = 0.
\end{align*}
\paragraph{Analysis of $(ii)$.} Using the continuity hypothesis on $h$,
\begin{align*}
    &|h(\iter{\tilde{z}}{0}_i, \dotsc, \iter{\tilde{z}}{T}_i) -  h(\iter{\hat{z}}{0;k}_i, \dotsc, \iter{\hat{z}}{T;k}_i)| \\&\hspace{3cm}\leq \left((\iter{\hat{z}}{0;k}_i-\iter{\tilde{z}}{0}_i)^2 + \dotsb + (\iter{\hat{z}}{T;k}_i-\iter{\tilde{z}}{T}_i)^2  \right)^{1/2} \\ &\hspace{5cm} \times \left( 1 + \left((\iter{\hat{z}}{0;k}_i)^2 + \dotsb + (\iter{\hat{z}}{T;k}_i)^2 \right)^{1/2} + \left((\iter{\tilde{z}}{0}_i)^2 + \dotsb + (\iter{\tilde{z}}{T}_i)^2 \right)^{1/2} \right).
\end{align*}
Using Jensen's and Cauchy-Schwarz Inequality, we obtain,
\begin{subequations} \label{eq:recycle-approx-argument}
\begin{align}
    (ii) &\explain{def}{=}  \E\left[ \frac{1}{\dim} \sum_{i=1}^\dim \left|h(\iter{\tilde{z}}{0}_i, \dotsc, \iter{\tilde{z}}{T}_i) -  h(\iter{\hat{z}}{0;k}_i, \dotsc, \iter{\hat{z}}{T;k}_i) \right|   \right] \leq 3 \sqrt{(ii.a)} \sqrt{(ii.b)},
\end{align}
where the terms $(ii.a), (ii.b)$ are defined as:
\begin{align}
    (ii.a) &\explain{def}{=} \frac{1}{\dim} \sum_{t=1}^T \E[ \|\iter{\hat{\vz}}{t;k} - \iter{\tilde{\vz}}{t}\|^2] \\
    (ii.b) &\explain{def}{=} 1 + \frac{1}{\dim} \sum_{t=0}^T \E[\|\iter{\hat{\vz}}{t;k}\|^2] + \frac{1}{\dim} \sum_{t=0}^T \E[\|\iter{\tilde{\vz}}{t}\|^2].
\end{align}
\end{subequations}
In order to control $(ii.a)$, we develop a recursive upper bound on $\|\iter{\hat{\vz}}{t;k} - \iter{\tilde{\vz}}{t}\|$:
\begin{align*}
    \|\iter{\tilde{\vz}}{t} - \iter{\hat{\vz}}{t;k}\| & = \|\hat{\mM} \cdot (\nonlin_t(\iter{\tilde{\vz}}{t-1}) - \iter{\nonlin}{k}_{t}(\iter{\hat{\vz}}{t-1;k}))\| \\
    & \leq\|\hat{\mM}\|_{\op} \cdot \| \nonlin_t(\iter{\tilde{\vz}}{t-1}) - \iter{\nonlin}{k}_{t}(\iter{\hat{\vz}}{t-1;k})\| \\
    & \leq \|\hat{\mM}\|_{\op} \cdot \| \nonlin_t(\iter{\tilde{\vz}}{t-1}) - \nonlin_{t}(\iter{\hat{\vz}}{t-1;k})\| + \|\hat{\mM}\|_{\op} \cdot \| \nonlin_t(\iter{\hat{\vz}}{t-1;k}) - \iter{\nonlin}{k}_{t}(\iter{\hat{\vz}}{t-1;k})\| \\
    & \explain{(a)}{\leq} L\cdot\|\hat{\mM}\|_{\op}   \cdot  \|\iter{\tilde{\vz}}{t-1} - \iter{\hat{\vz}}{t-1;k}\| + \|\hat{\mM}\|_{\op} \cdot \| \nonlin_t(\iter{\hat{\vz}}{t-1;k}) - \iter{\nonlin}{k}_{t}(\iter{\hat{\vz}}{t-1;k})\|.
\end{align*}
In the step marked (a) we used the assumption that $\nonlin_t$ was $L$-Lipschitz. Unrolling the above inequality, we obtain,
\begin{align*}
    \|\iter{\tilde{\vz}}{t} - \iter{\hat{\vz}}{t;k}\| &\leq \sum_{i=1}^t L^{i-1} \cdot \|\hat{\mM}\|_{\op}^i \cdot \| \nonlin_{t-i+1}(\iter{\hat{\vz}}{t-i;k}) - \iter{\nonlin}{k}_{t-i+1}(\iter{\hat{\vz}}{t-i;k})\|.
\end{align*}
Hence,
\begin{align*}
    \frac{\E\|\iter{\tilde{\vz}}{t} - \iter{\hat{\vz}}{t;k}\|^2}{\dim} & \leq t \cdot \sum_{i=1}^t L^{2(i-1)} \cdot \|\hat{\mM}\|_{\op}^{2i} \cdot  \frac{\E\|\nonlin_{t-i+1}(\iter{\hat{\vz}}{t-i;k}) - \iter{\nonlin}{k}_{t-i+1}(\iter{\hat{\vz}}{t-i;k})\|^2}{\dim} \\
    & = t \cdot \sum_{i=1}^t L^{2(i-1)} \cdot \|\hat{\mM}\|_{\op}^{2i} \cdot \E[(\nonlin_{t-i+1}(\iter{Z}{t-i;k}_\dim) - \iter{\nonlin}{k}_{t-i+1}(\iter{Z}{t-i;k}_\dim))^2].
\end{align*}
Since $\nonlin_{1:T}$ were assumed to be Lipchitz, the functions $z \mapsto (\nonlin_{i}(z) - \iter{\nonlin}{k}_{i}(z))^2$ can be bounded by polynomials and hence, by \corref{SE-normalized},
\begin{align*} 
    \limsup_{\dim \rightarrow \infty} \frac{\E\|\iter{\tilde{\vz}}{t} - \iter{\hat{\vz}}{t;k}\|^2}{\dim} & \leq t \cdot \sum_{i=1}^t C^{2i} \cdot L^{2(i-1)}   \cdot \E[(\nonlin_{t-i+1}(Z) - \iter{\nonlin}{k}_{t-i+1}(Z))^2] < \infty.
\end{align*}
where $Z \sim \gauss{0}{1}$ and $C = \limsup_{\dim \rightarrow \infty} \|\hat{\mM}\|_{\op} < \infty$ (cf. \eqref{eq:matrix-approximation}). Combining the above display with \eqref{eq:nonlin-approximation}, we obtain,
\begin{align} \label{eq:iterate-norm-error-estimate}
    \lim_{k \rightarrow \infty} \limsup_{\dim \rightarrow \infty} \frac{\E\|\iter{\tilde{\vz}}{t} - \iter{\hat{\vz}}{t;k}\|^2}{\dim} & = 0. 
\end{align}
Hence, 
\begin{align*}
   \lim_{k \rightarrow \infty} \limsup_{\dim \rightarrow \infty} \;  (ii.a) &\explain{def}{=} \lim_{k \rightarrow \infty} \limsup_{\dim \rightarrow \infty} \frac{1}{\dim} \sum_{t=1}^T \E[ \|\iter{\hat{\vz}}{t;k} - \iter{\tilde{\vz}}{t}\|^2] = 0.
\end{align*}
In order to control $(ii.b)$ we observe that:
\begin{align*}
    \limsup_{k \rightarrow \infty} \limsup_{\dim \rightarrow \infty} \; (ii.b) &\explain{def}{=} \limsup_{k \rightarrow \infty} \limsup_{\dim \rightarrow \infty} \left(1 + \frac{1}{\dim} \sum_{t=0}^T \E[\|\iter{\hat{\vz}}{t;k}\|^2] + \frac{1}{\dim} \sum_{t=0}^T \E[\|\iter{\tilde{\vz}}{t}\|^2] \right) \\ & \leq \limsup_{k \rightarrow \infty} \limsup_{\dim \rightarrow \infty} \left(1 + \frac{2}{\dim} \sum_{t=0}^T \E[\|\iter{\hat{\vz}}{t;k}-\iter{\tilde{\vz}}{t}\|^2] +  \frac{3}{\dim} \sum_{t=0}^T \E[\|\iter{\tilde{\vz}}{t}\|^2] \right)  \\ & \explain{\eqref{eq:iterate-norm-error-estimate}}{=} 1 + \limsup_{\dim \rightarrow \infty} \left( \frac{3}{\dim} \sum_{t=0}^T \E[\|\iter{\tilde{\vz}}{t}\|^2]\right).
\end{align*}
Observe that since $\nonlin_t$ are Lipschitz with constant $L$, we can control  $\|\iter{\tilde{\vz}}{t}\|$ as follows:
\begin{align}\label{eq:recycle-norm-bound}
    \|\iter{\tilde{\vz}}{t}\| & \leq  \| \hat{\mM} \nonlin_t(\iter{\tilde{\vz}}{t-1})\| \leq \|\hat{\mM}\|_{\op}\|\nonlin_t(\iter{\tilde{\vz}}{t-1})\| \leq\|\hat{\mM}\|_{\op} \cdot (\|\nonlin_t(\vzero)\| + L\cdot \|\iter{\tilde{\vz}}{t-1}\|).
\end{align}
Unrolling the above recursion yields:
\begin{align*}
    \|\iter{\tilde{\vz}}{t}\| & \leq \sum_{i=1}^t \|\hat{\mM}\|_{\op}^{i} \cdot L^{i-1} \cdot \|\nonlin_{t-i+1}(\vzero)\| + \|\hat{\mM}\|_{\op}^{t} \cdot L^{t} \cdot  \|\iter{\vz}{0}\|.
\end{align*}
Since $\|\hat{\mM}\|_{\op} \lesssim 1$, $\|\nonlin_{i}(\vzero)\|^2 = \dim\cdot \nonlin_i(0)^2$ and $\E \|\iter{\vz}{0}\|^2= \dim$, we obtain $ \limsup_{k \rightarrow \infty} \limsup_{\dim \rightarrow \infty} \; (ii.b)  < \infty$ which yields:
\begin{align*}
     \limsup_{k \rightarrow \infty} \limsup_{\dim \rightarrow \infty} \; (ii) & \leq 3 \limsup_{k \rightarrow \infty} \limsup_{\dim \rightarrow \infty} \sqrt{(ii.a)} \sqrt{(ii.b)} = 0.
\end{align*}
\paragraph{Analysis of $(iii)$.} 
By repeating the argument used to obtain \eqref{eq:recycle-approx-argument} we obtain:
\begin{subequations} \label{eq:recycled-approx-argument}
\begin{align}
    (iii) & \explain{def}{=} \E\left[ \frac{1}{\dim} \sum_{i=1}^\dim \left|h(\iter{{z}}{0}_i, \dotsc, \iter{{z}}{T}_i) - h(\iter{\tilde{z}}{0}_i, \dotsc, \iter{\tilde{z}}{T}_i)\right|   \right] \leq 3 \sqrt{(iii.a)} \sqrt{(iii.b)},
\end{align}
where the terms $(iii.a), (iii.b)$ are defined as:
\begin{align}
    (iii.a) &\explain{def}{=} \frac{1}{\dim} \sum_{t=1}^T \E[ \|\iter{{\vz}}{t} - \iter{\tilde{\vz}}{t}\|^2] \\
    (iii.b) &\explain{def}{=} 1 + \frac{1}{\dim} \sum_{t=0}^T \E[\|\iter{{\vz}}{t}\|^2] + \frac{1}{\dim} \sum_{t=0}^T \E[\|\iter{\tilde{\vz}}{t}\|^2].
\end{align}
\end{subequations}
We can control $\| \iter{{\vz}}{t}-\iter{\tilde{\vz}}{t} \|$ as follows:
\begin{align*}
    \| \iter{{\vz}}{t}-\iter{\tilde{\vz}}{t} \| & = \|\mM \nonlin_t(\iter{{\vz}}{t-1}) - \hat{\mM} \nonlin_t(\iter{\tilde{\vz}}{t-1})\| \\
    & \leq \|\mM \nonlin_t(\iter{{\vz}}{t-1}) - {\mM} \nonlin_t(\iter{\tilde{\vz}}{t-1})\| + \|\mM \nonlin_t(\iter{\tilde{\vz}}{t-1}) - \hat{\mM} \nonlin_t(\iter{\tilde{\vz}}{t-1})\| \\
    & \leq \|\mM\|_{\op} \cdot L \cdot  \| \iter{{\vz}}{t-1}-\iter{\tilde{\vz}}{t-1} \| + \|\mM - \hat{\mM}\|_{\op} \cdot \|\nonlin_t(\iter{\tilde{\vz}}{t-1})\|.
\end{align*}
Unrolling the above inequality yields,
\begin{align*}
     \| \iter{{\vz}}{t}-\iter{\tilde{\vz}}{t} \| & \leq  \|\mM - \hat{\mM}\|_{\op} \cdot \sum_{i=1}^t \|\mM\|_{\op}^{i-1} \cdot L^{i-1} \cdot \|\nonlin_{t-i+1}(\iter{\tilde{\vz}}{t-i})\| \\
     & \leq \|\mM - \hat{\mM}\|_{\op} \cdot \sum_{i=1}^t \|\mM\|_{\op}^{i-1} \cdot L^{i-1} \cdot (\sqrt{\dim} \cdot |\nonlin_{t-i+1}(0)| + L \cdot \|\iter{\tilde{\vz}}{t-i}\|).
\end{align*}
From our analysis of the term $(ii)$, we know that $\E[\|\iter{\tilde{\vz}}{t-i}\|^2] \lesssim \dim$. Since $\|\mM - \hat{\mM}\|_{\op}  \ll 1$, we immediately obtain:
\begin{align*}
   \limsup_{\dim \rightarrow \infty} \;  (iii.a) & =   \sum_{t=1}^T  \limsup_{\dim \rightarrow \infty}  \; \frac{\E[ \|\iter{{\vz}}{t} - \iter{\tilde{\vz}}{t}\|^2]}{\dim} = 0.
\end{align*}
In order to upper bound $(iii.b)$, as we recalled previously, we had already showed that $\E[\|\iter{\tilde{\vz}}{t}\|^2] \lesssim \dim$ in our analysis of the term $(ii)$. Furthermore the same argument (cf. \eqref{eq:recycle-norm-bound}) yields  $\E[\|\iter{{\vz}}{t}\|^2] \lesssim \dim$. This shows that $(iii.b) \lesssim 1$. Hence,
\begin{align*}
     \lim_{\dim \rightarrow \infty} \; (iii) = 0.
\end{align*}
This completes the proof of the proposition.
\end{proof}

\subsection{Proof of \lemref{matrix-approx}}
\begin{proof}[Proof of \lemref{matrix-approx}]
Since ${\mM}{}$ is semi-random (\defref{matrix-ensemble-relaxed}), $\mM = \mS \mPsi \mS$. Define the matrix $\hat{\mM} = \mS \hat{\mPsi} \mS$ where,
\begin{align*}
    \hat{\mPsi} \explain{def}{=} \mD^{-\frac{1}{2}} \mPsi, \; \mD = \diag\left((\mPsi \mPsi\tran)_{11}, (\mPsi \mPsi\tran)_{22}, \dotsc,(\mPsi \mPsi\tran)_{\dim \dim} \right). 
\end{align*}
First, we check that $\hat{\mM}$ satisfies all the requirements of \defref{matrix-ensemble-relaxed} and \assumpref{normalization} 
\begin{enumerate}
    \item The requirement (1) of \defref{matrix-ensemble-relaxed} is satisfied by construction. 
    \item Notice that,
    \begin{align*}
        \|\hat{\mPsi}\|_\infty & \leq \left(\max_{i \in [\dim]}  \frac{1}{(\mPsi \mPsi\tran)_{ii}}  \right)^{\frac{1}{2}} \cdot \|\mPsi\|_\infty. 
    \end{align*}
    \defref{matrix-ensemble-relaxed} guarantees that,
    \begin{align*}
        \|\mPsi\|_\infty \lesssim \dim^{-\frac{1}{2} + \epsilon}, \; \max_{i \in [\dim]} \left| \frac{1}{(\mPsi \mPsi\tran)_{ii}} \right| & = 1 + o_\dim(1),
    \end{align*}
    which yields $\|\hat{\mPsi}\|_\infty\lesssim \dim^{-\frac{1}{2} + \epsilon}$ for any $\epsilon > 0$. This verifies requirement (2a) in \defref{matrix-ensemble-relaxed}. 
     \item To we verify requirement (2b) in \defref{matrix-ensemble-relaxed} we observe that
    \begin{align*}
        \|\hat{\mPsi}\|_{\op} & \leq \|\mD^{-\frac{1}{2}}\|_{\op}^2 \|\mPsi\|_{\op} \leq \left(\max_{i \in [\dim]} \left| \frac{1}{(\mPsi \mPsi\tran)_{ii}} \right| \right) \cdot \|\mPsi\|_{\op} \lesssim 1, 
    \end{align*}
    as required.
    \item To verify requirement (2c) in in \defref{matrix-ensemble-relaxed}, we note that $\hat{\mPsi} \hat{\mPsi}\tran = \mD^{-\frac{1}{2}} {\mPsi} {\mPsi}\tran \mD^{-\frac{1}{2}}$. 
    Hence,
    \begin{align*}
        \max_{i \neq j} |(\hat{\mPsi} \hat{\mPsi}\tran)_{ij}|  & \leq \left( \max_{i \in [\dim]} \frac{1}{(\mPsi \mPsi\tran)_{ii}}  \right) \cdot  \max_{i \neq j} |{\mPsi} {\mPsi}\tran|. 
    \end{align*}
     \defref{matrix-ensemble-relaxed} guarantees that,
    \begin{align*}
        \max_{i \neq j} |({\mPsi} {\mPsi}\tran)_{ij}| \lesssim \dim^{-\frac{1}{2} + \epsilon}, \; \max_{i \in [\dim]} \left| \frac{1}{(\mPsi \mPsi\tran)_{ii}} \right| & = 1 + o_\dim(1).
    \end{align*}
    Hence, $\max_{i \neq j} |\hat{\mPsi} \hat{\mPsi}\tran| \lesssim \dim^{-\frac{1}{2} + \epsilon}$ for any $\epsilon > 0$, which verifies item (2b) of \defref{matrix-ensemble-relaxed}. 
    \item Observe that,
    \begin{align*}
         (\hat{\mPsi}\hat{\mPsi}\tran)_{ii} & =  \sum_{j=1}^\dim  \hat{\Psi}_{ij}^2 = \left( \sum_{j=1}^\dim  \frac{{\Psi}_{ij}^2}{(\mPsi \mPsi\tran)_{ii}} \right) = \frac{(\mPsi \mPsi\tran)_{ii}}{(\mPsi \mPsi\tran)_{ii}}   = 1. 
    \end{align*}
    This verifies requirements (2d) in \defref{matrix-ensemble-relaxed} and (2) in \assumpref{normalization}.
\end{enumerate}
Finally, we control $ \|{\mM}{} -{\hat{\mM}}{}\|_{\op}$ as follows:
\begin{align*}
     \|{\mM}{} - {\hat{\mM}}{}\|_{\op} & =  \|{\mPsi}{} - {\hat{\mPsi}}{}\|_{\op} = \|(\mD^{-\frac{1}{2}} - \mI_\dim) \cdot \mPsi \|_{\op} \leq \|\mPsi\|_{\op} \cdot \max_{i \in [\dim]} \left| \sqrt{\frac{1}{(\mPsi \mPsi\tran)_{ii}}} - 1 \right|. 
\end{align*}
\defref{matrix-ensemble-relaxed} guarantees that,
\begin{align*}
    \|\mPsi\|_{\op} \lesssim 1, \;  \max_{i \in [\dim]} \left| \sqrt{\frac{1}{(\mPsi \mPsi\tran)_{ii}}} - 1 \right| \ll 1,
\end{align*}
and hence, we have $\|{\mM}{} - {\hat{\mM}}{}\|_{\op} \ll 1$, as claimed. This concludes the proof of the lemma. 
\end{proof}

\subsection{Proof of \lemref{nonlin-approx}}
\begin{proof}[Proof of \lemref{nonlin-approx}]
Since the Hermite polynomials form a complete orthonormal basis for the Gaussian Hilbert space $L^2(\gauss{0}{1})$, we can write down the Hermite decomposition of $\nonlin$:
\begin{align*}
    \nonlin(w) & = \sum_{\ell=0}^\infty \alpha(\ell) \cdot \hermite{\ell}(w),
\end{align*}
for square-summable coefficients $\alpha(\ell) = \E[\hermite{\ell}(Z) \nonlin_i(Z)], \; Z \sim \gauss{0}{1}$ which satisfy:
\begin{align*}
    \sum_{\ell = 1} \alpha^2(\ell) & = \E \nonlin^2(Z) = 1.
\end{align*}
Define the polynomial:
\begin{align*}
    \iter{\hat{f}}{k}(w) = \sum_{\ell=0}^k \alpha(\ell) \cdot  \hermite{\ell}(w). 
\end{align*}
The polynomial with the desired properties is:
\begin{align*}
     \iter{{f}}{k}(w) =  \frac{\iter{\hat{f}}{k}(w)}{\sqrt{\E[\iter{\hat{f}}{k}(Z)^2]}}.
\end{align*}
It is straight forward to check that:
\begin{align*}
    \E\left[|\iter{{f}}{k}(Z) - \nonlin(Z)|^2\right] & = 2 \cdot \left( 1 - \sqrt{\sum_{\ell=0}^k \alpha^2(\ell)} \right) \rightarrow 0, \text{ as } k \rightarrow \infty.
\end{align*}
Also notice that:
\begin{align*}
    \E[\iter{f}{k}(Z)^2] = 1 , \; \E[Z \iter{f}{k}(Z)] = \frac{\alpha(1)}{(\sum_{\ell=0}^k \alpha^2(\ell))^{1/2}} = \frac{\E[Z \nonlin(Z)]}{(\sum_{\ell=0}^k \alpha^2(\ell))^{1/2}}  = 0.
\end{align*}
This constructs the desired approximating sequence for $\nonlin$.
\end{proof}

\subsection{Proof of \lemref{test-func-approx}}
\begin{proof}[Proof of \lemref{test-func-approx}]
In order to an approximating sequence for $h$, we observe that for any $\lambda \geq 0$, the vector $\mZ$ satisfies:
\begin{align*}
    \E \exp(\lambda \|\mZ\|_\infty) & \leq  \sum_{i=0}^{T+1}\E \exp(\lambda |Z_i|) <\infty.
\end{align*}
Hence, by \citet[Corollary 14.24, Definition 14.1]{schmudgen2017moment}  polynomials are dense in $L^2(\mu)$ where $\mu$ is the law of $\mZ$. Hence, there is a sequence of polynomials $\iter{h}{k}$ with the desired properties. 
\end{proof}

\subsection{Proof of \corref{SE-normalized}}
\begin{proof}[Proof of \corref{SE-normalized}]
Let $(\iter{Z}{0}_\dim, \iter{Z}{1}_\dim, \dotsc, \iter{Z}{T}_\dim)$ denote the sequence of random vectors with law ${\mu}_\dim$:
\begin{align*}
   {\mu}_\dim & \explain{def}{=} \E \left[ \frac{1}{\dim} \sum_{i=1}^\dim \delta_{\iter{\hat{z}}{0}_i, \iter{\hat{z}}{1}_i, \dotsc, \iter{\hat{z}}{T}_i} \right].
\end{align*}
Since the ordinary monomials $z_0^{k_0} z_1^{k_1} \dotsb z_T^{k_T}$ can be expressed as a finite linear combination of the Hermite basis polynomials, \thref{SE-normalized} guarantees that for any $k_0, k_1, \dotsc, k_T \in \W$ the limits:
\begin{align*}
    m(k_0, k_1, \dotsc, k_T) \explain{def}{=} \lim_{\dim \rightarrow \infty} \E[(\iter{Z}{0}_\dim)^{k_0} \cdot (\iter{Z}{1}_\dim)^{k_1} \cdot \dotsb \cdot (\iter{Z}{T}_\dim)^{k_T}] =  \E\left[ \frac{1}{\dim} \sum_{i=1}^\dim(\iter{z}{0}_i)^{k_0} \cdot (\iter{z}{1}_i)^{k_1}, \dotsb \cdot (\iter{z}{T}_i)^{k_T}  \right]
\end{align*}
exist and are determined by $T, \nonlin_{0:T}$. We claim that $ m(k_0, k_1, \dotsc, k_T)$ is the moment sequence of some random vector $(Z_0, Z_1, \dotsc, Z_T)$.  This is follows because since all moments of the sequence $(\iter{Z}{0}_\dim, \iter{Z}{1}_\dim, \dotsc, \iter{Z}{T}_\dim)$ are bounded, this sequence of random vectors is tight and has a subsequence that converges in distribution. Let  $(Z_0, Z_1, \dotsc, Z_T)$ denote the limit of this convergent subsequence. By the continuous mapping theorem, $(\iter{Z}{0}_\dim)^{k_0} \cdot (\iter{Z}{1}_\dim)^{k_1} \cdot \dotsb \cdot (\iter{Z}{T}_\dim)^{k_T}$ converges weakly to ${Z}_0^{k_0} \cdot {Z}_1^{k_1} \cdot \dotsb \cdot {Z}_T^{k_T}$ along this subsequence. Since all moments of the random vector $(\iter{Z}{0}_\dim, \iter{Z}{1}_\dim, \dotsc, \iter{Z}{T}_\dim)$ are bounded, $(\iter{Z}{0}_\dim)^{k_0} \cdot (\iter{Z}{1}_\dim)^{k_1} \cdot \dotsb \cdot (\iter{Z}{T}_\dim)^{k_T}$ is uniformly integrable and $\E[\iter{Z}{0}_\dim)^{k_0} \cdot (\iter{Z}{1}_\dim)^{k_1} \cdot \dotsb \cdot (\iter{Z}{T}_\dim)^{k_T}]$ converges along this subsequence to $\E[{Z}_0^{k_0} \cdot {Z}_1^{k_1} \cdot \dotsb \cdot {Z}_T^{k_T}]$. Hence,
\begin{align*}
     m(k_0, k_1, \dotsc, k_T) \explain{def}{=} \lim_{\dim \rightarrow \infty} \E[(\iter{Z}{0}_\dim)^{k_0} \cdot (\iter{Z}{1}_\dim)^{k_1} \cdot \dotsb \cdot (\iter{Z}{T}_\dim)^{k_T}] = \E[{Z}_0^{k_0} \cdot {Z}_1^{k_1} \cdot \dotsb \cdot {Z}_T^{k_T}].
\end{align*}
Hence, $m(k_0, k_1, \dotsc, k_T)$ is the moment sequence of the random vector $(Z_0, Z_1, \dotsc, Z_T)$. In particular, we have shown that,
\begin{align*}
    (\iter{Z}{0}_\dim, \iter{Z}{1}_\dim, \dotsc, \iter{Z}{T}_\dim) \rightarrow (Z_0, Z_1, \dotsc, Z_T) \text{ in moments}.
\end{align*}
By the second claim in \thref{SE-normalized} for any $k \in \N$,
\begin{align*}
    m(k, 0, 0, \dotsc, 0) = m(0, k, 0, \dotsc, 0) = \dotsb = m(0, 0, 0, \dotsc, 0, k) = \E Z^k, \; Z \sim \gauss{0}{1}.
\end{align*}
Since the Gaussian distribution is uniquely determined by its moments, $(Z_0, Z_1, \dotsc, Z_T)$ must have $\gauss{0}{1}$ marginals. Since the marginal distribution of each coordinate of $(Z_0, Z_1, \dotsc, Z_T)$ is uniquely determined by its moments, the joint law of $(Z_0, Z_1, \dotsc, Z_T)$ is also uniquely determined by its joint moments by Petersen's Theorem (see for e.g. \citep[Theorem 14.6]{schmudgen2017moment}). Since convergence in moments to a random variable uniquely determined by its moments implies weak convergence (see for e.g. \citep[Theorem 30.2]{billingsley2008probability}), we have,
\begin{align*}
    (\iter{Z}{0}_\dim, \iter{Z}{1}_\dim, \dotsc, \iter{Z}{T}_\dim) \explain{d}{\rightarrow} (Z_0, Z_1, \dotsc, Z_T).
\end{align*}
By the continuous mapping theorem, $h(\iter{Z}{0}_\dim, \iter{Z}{1}_\dim, \dotsc, \iter{Z}{T}_\dim) \explain{d}{\rightarrow} h(Z_0, Z_1, \dotsc, Z_T)$. Furthermore, since $h$ is bounded by a polynomial function and all moments of $(\iter{Z}{0}_\dim, \iter{Z}{1}_\dim, \dotsc, \iter{Z}{T}_\dim)$ remain bounded as $\dim \rightarrow \infty$, $h(\iter{Z}{0}_\dim, \iter{Z}{1}_\dim, \dotsc, \iter{Z}{T}_\dim)$ is uniformly integrable and,
\begin{align*}
    \lim_{\dim \rightarrow \infty} \E[h(\iter{Z}{0}_\dim, \iter{Z}{1}_\dim, \dotsc, \iter{Z}{T}_\dim)] = \E[h(Z_0, Z_1, \dotsc, Z_T)],
\end{align*}
as claimed. 
\end{proof}

\section{Proof of \factref{hermite-property}}\label{appendix:hermite-property}
\begin{proof}[Proof of \factref{hermite-property}]
This formula is derived using the generating function of Hermite polynomials. The generating function of Hermite polynomials (see e.g. \citet[Section 11.2]{o2014analysis}) is given by:
\begin{align} \label{eq:generating}
    e^{\xi \lambda - \frac{\lambda^2}{2}} & = \sum_{i=0}^\infty \frac{\lambda^i}{\sqrt{i!}} \cdot \hermite{q}(\xi).
\end{align}
Applying this formula to $\xi = \ip{\vu}{\vx}$ we obtain:
\begin{align} \label{eq:formula-2-hermite}
     e^{\ip{\vu}{\vx} \lambda - \frac{\lambda^2}{2}}  & = \sum_{q=0}^\infty \frac{\lambda^q}{\sqrt{q!}} \cdot \hermite{q}(\ip{\vu}{\vx}).
\end{align}
On the other hand, we can obtain an alternate formula for $e^{\ip{\vu}{\vx} \lambda - \frac{\lambda^2}{2}}$ by observing:
\begin{align}\label{eq:formula-1-hermite}
    e^{\ip{\vu}{\vx} \lambda - \frac{\lambda^2}{2}}  \explain{(a)}{=} e^{\ip{\vu}{\vx} \lambda - \frac{\lambda^2 \|\vu\|^2}{2}} & = \prod_{i=1}^\dim e^{ x_i u_i \lambda  - \frac{u_i^2 \lambda^2}{2}} 
     \explain{(b)}{=} \prod_{i=1}^\dim \sum_{a_i = 0}^\infty \frac{ (\lambda u_i)^{a_i}}{\sqrt{a_i!}} \cdot \hermite{a_i}(x_i) 
    = \sum_{\va \in \W^\dim} \frac{\lambda^{\|\va\|_1}}{\sqrt{\va !}} \cdot \vu^{\va} \cdot \hermite{\va}(\vx). 
\end{align}
In the above display equality (a) follows from $\|\vu\|^2 = 1$ and equality (b) follows from the generating function formula for Hermite polynomials given in \eqref{eq:generating}. Comparing the coefficient of $\lambda^q$ in \eqref{eq:formula-2-hermite} and the \eqref{eq:formula-1-hermite} gives the claimed formula.  
\end{proof}

\end{document}